\documentclass[a4paper]{amsart}

\usepackage[applemac]{inputenc}
\usepackage[T1]{fontenc}
\usepackage[english]{babel}
\usepackage{babelbib}
\usepackage{enumerate}
\usepackage{url}
\usepackage{amssymb}
\usepackage[cal=boondoxo]{mathalfa}
\usepackage{scalerel}
\usepackage{aliascnt}

\DeclareFontFamily{U}{skulls}{}
\DeclareFontShape{U}{skulls}{m}{n}{ <-> skull }{}

\usepackage[hypertexnames=false,
 	pdfpagemode=UseNone,
 	breaklinks=true,
 	extension=pdf,
 	colorlinks=true,
 	linkcolor=blue,
 	citecolor=red,
 	urlcolor=blue,
 ]{hyperref} 
\usepackage{color}

\usepackage{tikz}
\usepackage{tikz-cd}

\usetikzlibrary{matrix}
\usetikzlibrary{arrows}

\usepackage{graphicx}

\usepackage{cleveref}

\mathcode`A="7041 \mathcode`B="7042 \mathcode`C="7043 \mathcode`D="7044
\mathcode`E="7045 \mathcode`F="7046 \mathcode`G="7047 \mathcode`H="7048
\mathcode`I="7049 \mathcode`J="704A \mathcode`K="704B \mathcode`L="704C
\mathcode`M="704D \mathcode`N="704E \mathcode`O="704F \mathcode`P="7050
\mathcode`Q="7051 \mathcode`R="7052 \mathcode`S="7053 \mathcode`T="7054
\mathcode`U="7055 \mathcode`V="7056 \mathcode`W="7057 \mathcode`X="7058
\mathcode`Y="7059 \mathcode`Z="705A

\newcommand{\bbA}{\mathbb{A}}

\newcommand{\bbC}{\mathbb{C}}

\newcommand{\bbF}{\mathbb{F}}
\newcommand{\bbG}{\mathbb{G}}

\newcommand{\bbN}{\mathbb{N}}

\newcommand{\bbP}{\mathbb{P}}
\newcommand{\bbQ}{\mathbb{Q}}
\newcommand{\bbR}{\mathbb{R}}

\newcommand{\bbV}{\mathbb{V}}

\newcommand{\bbZ}{\mathbb{Z}}

\newcommand{\Gm}{\mathbb{G}_m}

\newcommand{\cA}{\mathcal{A}}
\newcommand{\cB}{\mathcal{B}}
\newcommand{\cC}{\mathcal{C}}

\newcommand{\cE}{\mathcal{E}}
\newcommand{\cF}{\mathcal{F}}
\newcommand{\cG}{\mathcal{G}}
\newcommand{\cH}{\mathcal{H}}
\newcommand{\cI}{\mathcal{I}}

\newcommand{\cL}{\mathcal{L}}

\newcommand{\cO}{\mathcal{O}}
\newcommand{\cP}{\mathcal{P}}
\newcommand{\cQ}{\mathcal{Q}}
\newcommand{\cR}{\mathcal{R}}
\newcommand{\cS}{\mathcal{S}}
\newcommand{\cT}{\mathcal{T}}

\newcommand{\cV}{\mathcal{V}}
\newcommand{\cW}{\mathcal{W}}
\newcommand{\cX}{\mathcal{X}}
\newcommand{\cY}{\mathcal{Y}}
\newcommand{\cZ}{\mathcal{Z}}

\newcommand{\rF}{\textup{F}}

\newcommand{\rH}{\textup{H}}

\newcommand{\rR}{\textup{R}}

\newcommand{\rT}{\textup{T}}

\newcommand{\rX}{\textup{X}}

\newcommand{\frH}{\mathfrak{H}}

\newcommand{\sfP}{\mathsf{P}}

\newcommand{\rd}{\textup{d}}

\newcommand{\an}{\textup{an}}
\newcommand{\s}{\textup{s}}

\renewcommand{\ss}{\textup{ss}}

\newcommand{\rs}{\textup{rs}}

\newcommand{\et}{\textup{\'et}}
\newcommand{\dR}{\textup{dR}}

\newcommand{\cris}{\textup{crys}}

\newcommand{\sm}{\textup{sm}}
\newcommand{\nd}{\textup{nd}}
\renewcommand{\sp}{\operatorname{sp}}
\newcommand{\ad}{\textup{ad}}
\renewcommand{\top}{\textup{top}}

\newcommand{\onto}{\twoheadrightarrow}
\newcommand{\ontoo}{\too\hspace{-14pt}\too}

\newcommand{\into}{\hookrightarrow}
\newcommand{\too}{\longrightarrow}
\renewcommand{\phi}{\varphi}
\renewcommand{\epsilon}{\varepsilon}
\renewcommand{\ker}{\Ker}

\newcommand{\iso}{\simeq}

\newcommand{\Rep}{\textup{Rep}}

\newcommand{\R}{\textup{R}}

\newcommand{\PLambda}{\scalerel*{{\rotatebox[origin=c]{180}{$\bbV$}}}{\bbV}}

\newcommand{\intoo}{\lhook\joinrel\longrightarrow}

\DeclareMathOperator{\Vect}{Vect}
\DeclareMathOperator{\Alb}{Alb}

\DeclareMathOperator{\Gal}{Gal}

\DeclareMathOperator{\rk}{rk}
\DeclareMathOperator{\pr}{pr}

\DeclareMathOperator{\Spec}{Spec}

\DeclareMathOperator{\Spf}{Spf}

\DeclareMathOperator{\Hom}{Hom}

\DeclareMathOperator{\Mon}{Mon}
\DeclareMathOperator{\im}{Im}
\DeclareMathOperator{\Ker}{Ker}
\DeclareMathOperator{\Frac}{Frac}

\DeclareMathOperator{\Flag}{Flag}

\newcommand{\red}{{\mathrm{red}}}

\DeclareMathOperator{\Pic}{Pic}
\DeclareMathOperator{\Stab}{Stab}
\DeclareMathOperator{\Iso}{Iso}
\DeclareMathOperator{\End}{End}
\DeclareMathOperator{\Ext}{Ext}
\DeclareMathOperator{\Sym}{Sym}

\DeclareMathOperator{\Aut}{Aut}
\DeclareMathOperator{\GL}{GL}
\DeclareMathOperator{\SL}{SL}
\DeclareMathOperator{\SO}{SO}
\DeclareMathOperator{\Sp}{Sp}
\DeclareMathOperator{\GO}{GO}

\DeclareMathOperator{\GSp}{GSp}

\DeclareMathOperator{\NS}{NS}
\DeclareMathOperator{\Tran}{Tran}

\DeclareMathOperator{\Lie}{Lie}

\DeclareMathOperator{\gr}{gr}
\DeclareMathOperator{\id}{id}

\DeclareMathOperator{\Hilb}{Hilb}

\DeclareMathOperator{\alb}{alb}

\DeclareMathOperator{\characteristic}{char}

\DeclareMathOperator{\tr}{tr}

\DeclareMathOperator{\rad}{rad}

\DeclareMathOperator{\codim}{codim}

\DeclareMathOperator{\Res}{Res}
\DeclareMathOperator{\Par}{Par}
\DeclareMathOperator{\wt}{wt}

\DeclareMathOperator{\Alt}{Alt}

\DeclareMathOperator{\Type}{Type}

\DeclareMathOperator{\DE}{DE}
\DeclareMathOperator{\PV}{PV}
\DeclareMathOperator{\MT}{MT}
\DeclareMathOperator{\Ad}{Ad}

\renewcommand{\le}{\leqslant}
\renewcommand{\ge}{\geqslant}

\theoremstyle{plain}
\newtheorem{theoremintro}{Theorem}

\newtheorem*{maintheorem-monodromy}{Main theorem (monodromy version)}
\newtheorem*{maintheorem-tannaka}{Main theorem (Tannaka version)}
\newtheorem*{corollaryintro}{Corollary}

\newtheorem{theorem}{Theorem}[section]

\newaliascnt{lemma}{theorem}
\newtheorem{lemma}[lemma]{Lemma}
\aliascntresetthe{lemma}
\crefname{lemma}{lemma}{lemmas}

\newaliascnt{proposition}{theorem}
\newtheorem{proposition}[proposition]{Proposition}
\aliascntresetthe{proposition}
\crefname{proposition}{proposition}{propositions}

\newaliascnt{corollary}{theorem}
\newtheorem{corollary}[corollary]{Corollary}
\aliascntresetthe{corollary}
\crefname{corollary}{corollary}{corollaries}

\newaliascnt{claim}{theorem}

\aliascntresetthe{claim}
\crefname{claim}{claim}{claims}

\newaliascnt{fact}{theorem}

\aliascntresetthe{fact}
\crefname{fact}{fact}{facts}

\newtheorem*{bigmonodromyintro}{Big Monodromy Criterion}

\newaliascnt{assumption}{equation}

\aliascntresetthe{assumption}
\crefname{assumption}{assumption}{assumptions}

\newaliascnt{condition}{equation}

\aliascntresetthe{condition}
\crefname{condition}{condition}{conditions}

\theoremstyle{definition}
\newtheorem*{definitionintro}{Definition}

\newaliascnt{definition}{theorem}
\newtheorem{definition}[definition]{Definition}
\aliascntresetthe{definition}
\crefname{definition}{definition}{definitions}

\newaliascnt{example}{theorem}
\newtheorem{example}[example]{Example}
\aliascntresetthe{example}
\crefname{example}{example}{examples}

\newaliascnt{notation}{theorem}

\aliascntresetthe{notation}
\crefname{notation}{notation}{notations}

\newaliascnt{remark}{theorem}
\newtheorem{remark}[remark]{Remark}
\aliascntresetthe{remark}
\crefname{remark}{remark}{remarks}

\numberwithin{equation}{section}

\begin{document}

\title{Arithmetic finiteness of very irregular varieties}

\begin{abstract}
We prove the Shafarevich conjecture for varieties that embed in their Albanese variety with ample normal bundle, subject to mild numerical conditions. Our proof relies on the Lawrence-Venkatesh method as in the work of Lawrence-Sawin, together with the big monodromy criterion from our work with Javanpeykar and Lehn.
\end{abstract}

\author[T. Kr\"amer]{Thomas Kr\"amer}
\address{Thomas Kr\"amer \\
	Fakult\"{a}t f\"{u}r Mathematik\\
	Technische Universit\"{a}t Chemnitz\\
	Rei\-chenhainer Stra\ss e 39, 09126 Chemnitz\\
	Germany.}
\email{thomas.kraemer@math.tu-chemnitz.de}

\author[M. Maculan]{Marco Maculan}

\address{ Marco Maculan \\
	Institut de Math\'ematiques de Jussieu \\
Sorbonne Universit\'e \\
4, place Jussieu \\
75005 Paris \\
France}
\email{marco.maculan@imj-prg.fr}

\makeatletter
\@namedef{subjclassname@2020}{
	\textup{2020} Mathematics Subject Classification}
\makeatother

\subjclass[2020]{11G35, 14G25 (primary), 11F80, 14D07, 14G05, 14K12 (secondary).}
\keywords{Shafarevich conjecture, irregular variety, Albanese variety, good reduction, Galois representation, period map, monodromy}

\maketitle

\setcounter{tocdepth}{1}

\tableofcontents

\section{Introduction}

\thispagestyle{empty} 

In this paper we obtain finiteness results \`a la Shafarevich for a large class of varieties over number fields, thus proving the Lang-Vojta conjecture for the moduli stacks of such varieties. One of the ingredients of the proof is the big monodromy criterion in our work with Javanpeykar and Lehn \cite{JKLM, KLM}. \medskip

\subsection{The Shafarevich conjecture} To put our results in perspective, let us briefly recall some history. At the ICM in Stockholm in 1962, Shafarevich conjectured that over any number field there are only finitely many isomorphism classes of smooth projective curves of given genus~$g\ge 2$ that have good reduction outside a given finite set~$\Sigma$ of places; the latter roughly means that the discriminant of an equation for the curve is divisible at most by the primes in~$\Sigma$. Since then, analogous finiteness statements for other types of varieties have become known as instances of the \emph{Shafarevich conjecture} (not to be confused with the eponymous conjecture on holomorphic convexity of the universal cover of complex projective manifolds). The conjecture does not hold for all types of varieties, for instance it fails for genus~$1$ curves without rational points. The correct statement is obtained by regarding varieties with good reduction as integral points of moduli spaces. \medskip

The scarcity of integral points is the object of the \emph{Lang-Vojta conjecture}, one of the major unsolved problems in diophantine geometry. It predicts that on an algebraic variety of log general type over a number field, the integral points are not Zariski dense. In dimension one this is the Mordell conjecture that was proven by Faltings~\cite{FaltingsMordell}: Any smooth projective curve of genus~$g\ge 2$ has only finitely many rational points. In higher dimension, the Lang-Vojta conjecture implies in particular that any hyperbolic variety over a number field has only finitely many integral points. Here `hyperbolic' means that all positive-dimensional subvarieties are of log general type. By the work of Campana-P\u{a}un \cite[cor.~4.6]{CampanaPaun}, moduli spaces\footnote{Here we ignore issues arising from the coarseness of moduli spaces; see~\cite[sect.~6.3]{JavanpeykarLoughran} and~\cite[conj.~1.1]{JavanpeykarLondon} for a detailed discussion.} of canonically polarized varieties are hyperbolic in this sense. Therefore the Shafarevich conjecture for a given class of canonically polarized varieties is equivalent to the Lang-Vojta conjecture for the corresponding moduli space.\medskip

In the case of curves and abelian varieties, the Shafarevich conjecture was proven by Faltings~\cite{FaltingsMordell} to deduce the Mordell conjecture (the fact that the Shafarevich conjecture would imply the Mordell conjecture had been noticed earlier by Kodaira and Parshin). Since then, there has been only little progress on the Shafarevich conjecture as we will recall below. Recently, Lawrence and Sawin \cite{LS20} used the method of Lawrence and Venkatesh \cite{LV} to verify the Shafarevich conjecture in the special case of hypersurfaces in abelian varieties. In this paper we will prove it for varieties that embed in their Albanese variety with ample normal bundle. More precisely, we prove it for the following class of varieties:

\begin{definitionintro} Let~$X$ be a smooth projective geometrically connected variety over a field~$k$ of characteristic~$0$. When~$k$ is algebraically closed we say that~$X$ is \emph{very irregular} if the Albanese morphism given by some point in~$X(k)$ is a closed embedding with nonzero ample normal bundle. When~$k$ is arbitrary, we say that~$X$ is very irregular if the base change to an algebraic closure of~$k$ is very irregular.
\end{definitionintro}

We confirm the Lang-Vojta conjecture for the moduli spaces of very irregular varieties. 
Very irregular varieties form a very large class of canonically polarized varieties which includes all smooth projective curves of genus~$\ge 2$, all smooth complete intersections of ample divisors in an abelian variety, and more generally~$d$-dimensional smooth subvarieties of abelian varieties whose normal bundle is a direct sum of ample vector bundles of rank~$< d$. There are plenty of other examples such as Fano surfaces of lines of cubic threefolds, Schoen surfaces, etc. Higher dimensional abelian varieties and their subvarieties are hard to describe by explicit equations but have a rich geometry: In particular, subvarieties of abelian varieties with ample normal bundle have been studied by many authors, including Hartshorne \cite{Har71}, Ran \cite{Ran} and Debarre \cite{Deb95}. We chose the name `very irregular' in reminiscence of  the classical notion of irregular varieties, that is, smooth projective complex varieties~$S$ whose irregularity~$h^1(S, \cO_S) = h^0(S, \Omega^1_S)$ does not vanish.

\subsection{Intrinsic results} \label{sec:IntrinsicResults}
Let~$K$ be a number field,~$\Sigma$ a finite set of places and~$\cO_{K, \Sigma}$ the ring of~$\Sigma$-integers of~$K$.

\begin{definitionintro} A canonically polarized variety~$X$ over~$K$ has \emph{good reduction (outside~$\Sigma$)} if there is a smooth proper scheme~$\cX$ over~$\cO_{K, \Sigma}$ with generic fiber~$X$ and relatively ample relative canonical bundle~$\omega_{\cX}$. 
We then call~$\cX$ a \emph{good model} for~$X$. 
\end{definitionintro}

If a good model exists, then it is unique up to isomorphism; see \cref{sec:ModelsOfVeryIrregularVarieties}. 
For surfaces our main result can be stated as follows:

\begin{theoremintro}  \label{IntroThmIntrinsicSurfaces}
Fix~$c \ge 1$. Up to~$K$-isomorphism there are only finitely many very irregular surfaces~$X$ with good reduction,~$c_2(X) = c$ and~$h^0(X, \Omega^1_X) \ge 6$.
\end{theoremintro}

The above finiteness statement is the first one that holds for such a large class of surfaces. 
Before it, the Shafarevich conjecture for surfaces was known only in the following specific cases: Abelian surfaces \cite{FaltingsMordell}, del Pezzo surfaces~\cite{SchollDelPezzo}, K3 surfaces \cite{AndreK3,She,TakamatsuK3}, Enriques surfaces \cite{TakamatsuEnriques}, sextic surfaces \cite{JavLoughPisa}, fibrations of smooth curves over smooth curves \cite{JavanpeykarLondon}, surfaces that geometrically are products of curves \cite{JavLoughLondon} and, under mild numerical conditions, very irregular surfaces~$X$ with~$h^0(X, \Omega^1_X) = 3$  \cite{LS20}. \medskip

To state our results in higher dimension, note that a very irregular variety is canonically polarized; unless the polarization is specified, its Hilbert polynomial will be understood to be taken with respect to the canonical bundle. Also, for a~$d$-dimensional smooth variety~$X$ we write~$\chi_{\top}(X) = c_d(X)$ for the topological Euler characteristic of~$X$ and~$\chi(X, \cF)$ for the Euler characteristic of a coherent sheaf~$\cF$ on~$X$. We prove the following generalization of Faltings' theorem:

\begin{theoremintro} \label{IntroThmIntrinsicGeneral}
Fix~$P \in \bbQ[z]$ of degree~$d$. Then up to~$K$-isomorphism  there are only finitely many very irregular varieties~$X$ with good reduction, Hilbert polynomial~$P$,~$h^0(X, \Omega^1_X) \ge 2 d + 2$ and
\begin{equation} \label{SkullInequality}
(-1)^d \chi(X \times X, \Omega^d_{X \times X}) \; \le \; \tfrac{1}{2}\chi_{\top}(X \times X),
\end{equation}
satisfying the numerical condition \eqref{Eq:NumericalConditions}  below with~$A = \Alb(X)$. 
\end{theoremintro}

Instead of fixing the Hilbert polynomial of~$X$ one may equivalently fix the dimension of~$X$ and the top self intersection of its canonical bundle \cite{KollarMatsusaka}. For~$d = 2$ the conditions \eqref{SkullInequality} and \eqref{Eq:NumericalConditions} are always satisfied, thus \cref{IntroThmIntrinsicSurfaces} follows from \cref{IntroThmIntrinsicGeneral}. In higher dimension \cref{IntroThmIntrinsicGeneral} covers a vast class of previously unknown cases: For instance, the inequality \eqref{SkullInequality} holds for all odd~$d$ by~\cref{inequality-in-odd-dimension}.
Very recently, the inequality \eqref{SkullInequality} has been proved for any smooth complete intersection in an abelian variety; see \cite[prop.~3.7]{LuSkull}. It seems plausible that \eqref{SkullInequality} always holds for smooth subvarieties of an abelian variety. The condition \eqref{Eq:NumericalConditions} is  only needed when~$X$ is symmetric in the sense of \cref{sec:NumericalConditions}, and it is empty for~$h^0(X, \Omega^1_X) \ge 4 d + 2$. Thus we obtain:

\begin{corollaryintro} Fix~$P \in \bbQ[z]$ of odd degree~$d$. Then up to~$K$-isomorphism  there are only finitely many very irregular varieties~$X$ with good reduction, Hilbert polynomial~$P$ and~$h^0(X, \Omega^1_X) \ge 4 d + 2$.
\end{corollaryintro}

In dimension~$\ge 3$, such finiteness results were previously known beyond Faltings' theorem only for hyper-K\"ahler varieties \cite{AndreK3, TakamatsuHyperkahler}, flag varieties~\cite{JavLoughFlag}, complete intersections of Hodge level~$\le 1$ \cite{JavanpeykarLoughran}, Fano threefolds \cite{JavLoughPisa, LichtHyperbolicity, TakamatsuMukaiGenusSeven}, fibrations in smooth hyperbolic curves~\cite{TakamatsuPolycurves} and very irregular varieties~$X$ with~$h^0(X, \Omega^1_X)= \dim X + 1$ by~\cite{LS20}. \medskip 

Theorems~\ref{IntroThmIntrinsicSurfaces},~\ref{IntroThmIntrinsicGeneral} and their corollary are the first Shafarevich-type finiteness result that apply to varieties for which a direct classification is impossible (e.g. their Hodge numbers cannot be explicitly computed).

\subsection{Numerical condition} \label{sec:NumericalConditions} To formulate the numerical condition in~\cref{IntroThmIntrinsicGeneral}, let~$X$ be a smooth subvariety of an abelian variety~$A$ of dimension~$g$ over a field~$k$. 
If~$k$ is algebraically closed, we say that~$X \subset A$ is \emph{symmetric up to translation} if there is a point~$a \in A(k)$ such that~$X = - X + a$. If~$k$ is an arbitrary field, we say that~$X \subset A$ is symmetric up to translation if its base change to an algebraic closure of~$k$ is so. By the \emph{stabilizer} of a subvariety~$X \subset A$ we mean the algebraic group~$\Stab_A(X)$ whose points in a~$k$-scheme~$S$ are \[ \Stab_A(X)(S)\;=\;\{ a \in A(S) \colon X_S + a \;=\; X_S\}.\] 
We say that~$X$ is \emph{divisible} if~$\Stab_A(X)$ is nonzero. When~$X \subset A$ is of general type (e.g., if~$X$ has ample normal bundle), then the algebraic group~$\Stab_A(X)$ is finite. Its order~$n$ divides~$\chi_\top(X)$ because~$\chi_\top(X) / n$ is the topological Euler characteristic of the quotient variety \[\bar{X} = X/\Stab_A(X).\] In~\cref{IntroThmIntrinsicGeneral} we require:

\begin{equation} \label{Eq:NumericalConditions}
\begin{array}{cc}
|\chi_{\top}(\bar{X})| \neq 2^{2m - 1} 
&
\begin{tabular}{l}
\text{if~$X$ is symmetric up to translation,}\\
\text{$d \ge (g-1)/4$,~$ m\in \{3, \dots, d\}$, and~$m\equiv d$ mod~$2$.}
\end{tabular}
\end{array}
\end{equation}

We do not know any example of a smooth subvariety~$X \subset A$ with ample normal bundle for which the condition \eqref{Eq:NumericalConditions} fails.

\subsection{Extrinsic result} \label{sec:ExtrinsicResults} We will deduce \cref{IntroThmIntrinsicGeneral} from a general result about subvarieties of arbitrary abelian varieties. To state it, let us say that a smooth projective variety~$X$ over a field~$k$ is a \emph{product} resp.~a \emph{symmetric power of a curve} if over an algebraic closure of~$k$ it becomes isomorphic to a product of smooth projective varieties of dimension~$> 0$ resp. a symmetric power of a smooth projective curve. Suppose that $X$ is a subvariety of an abelian variety~$A$. When~$k$ is algebraically closed, one easily sees that~$X$ is
\begin{itemize}
 \item a product if and only if there are subvarieties~$X_1, X_2 \subset A$ of positive dimension such that the sum induces an isomorphism
 \[ X_1 \times X_2 \stackrel{\sim}{\too} X = X_1 + X_2;\]
 \item a symmetric power of a curve if and only if there is a curve~$C \subset A$ such that for~$d = \dim X$ the sum induces an isomorphism
 \[ \Sym^d C \stackrel{\sim}{\too} X= C + \cdots + C. \]
 \end{itemize} 
 For $k$ arbitrary, we say that $X$ is \emph{isogenous to a product} resp.~\emph{isogenous to symmetric power of a curve} if $X / \Stab_A(X)$ is a product resp.~a symmetric power of a curve.
\medskip
 
  This being said, let~$A$ be an abelian variety of dimension~$g$ over the number field~$K$ and let~$L$ be an ample line bundle on~$A$. 
  In \cite{LS20} Lawrence and Sawin prove that up to translation, there are only finitely many hypersurfaces~$X \subset A$ algebraically equivalent to~$L$ and with good reduction. In the opposite case of subvarieties of high codimension, we show:

\begin{theoremintro} \label{IntroThmExtrinsic}
Fix~$P \in \bbQ[z]$ of degree~$d < (g-1)/2$. Then up to translation by points in~$A(K)$ there are only finitely many smooth subvarieties~$X \subset A$ over~$K$ with ample normal bundle, good reduction and Hilbert polynomial~$P$ with respect to~$L$ which satisfy \eqref{SkullInequality} and \eqref{Eq:NumericalConditions} and are neither isogenous to a product  nor isogenous to a symmetric power of a curve.  
\end{theoremintro}

We say that a subvariety~$X\subset A$ is a \emph{geometric complete intersection of ample divisors} if its base change to an algebraic closure of~$K$ is a complete intersection of ample divisors. Smooth complete intersections of ample divisors in an abelian variety are neither a product \cite[rem.~6.3]{JKLM} nor a symmetric power of a curve, except for the theta divisor of a nonhyperelliptic genus~$3$ curve \cite[prop. 2.9 (2)]{JKLM}. Moreover, if they are nondivisible, then they satisfy \eqref{Eq:NumericalConditions} by~\cite[prop. 2.16, cor. 2.17]{JKLM}. As mentioned above, the inequality~\eqref{SkullInequality} holds for complete intersections in abelian varieties, hence as a direct consequence of \cref{IntroThmExtrinsic} we obtain:

\begin{corollaryintro} \label{IntroCorExtrinsic}
Fix~$P \in \bbQ[z]$ of degree~$d < (g-1)/2$. Then up to translation by points in~$A(K)$ there exist only finitely many smooth subvarieties~$X \subset A$ with good reduction, Hilbert polynomial~$P$ with respect to~$L$, 
which are nondivisible geometric complete intersection of ample divisors.
\end{corollaryintro}

The main geometric input for our proof of \cref{IntroThmExtrinsic} is the Big Monodromy Criterion from \cite{JKLM, KLM} that we recall in the next section.

\subsection{Big monodromy} \label{subsec:bigmonodromy} 
In this section we work over~$\bbC$. Let~$S$ be a smooth variety and~$A$ an abelian variety of dimension~$g$. Consider a morphism of complex varieties~$f \colon \cX \to A_S := A\times S$ whose composite  with the projection~$\pr_S \colon A_S \to S$ is smooth over~$S$ with connected fibers of dimension~$d$.  The situation is summarized in the following diagram:
\[
\begin{tikzcd}
&  \cX 
\ar[d,"f"]
\ar[dl, bend right=30, swap, "\pi_A"] \ar[dr, bend left=30, "\pi_S"]  & \\
A & A_S \ar[l, swap, "\pr_A"] \ar[r, "\pr_S"]  & S
\end{tikzcd}
\]
Let~$\pi_1(A, 0)$ be the topological fundamental group and denote the group of its characters by 
$
 \Pi(A)=\Hom(\pi_1(A, 0), \bbC^\times).
$
For~$\chi \in \Pi(A)$ let~$L_\chi$ denote the associated rank one local system on~$A$. For~$\underline{\chi} = (\chi_1, \dots, \chi_n) \in \Pi(A)^n$ we consider the local system
\[ V_{\underline{\chi}} \;:=\; R^d \pi_{S \ast} \pi_A^* L_{\underline{\chi}} \quad \textup{where} \quad L_{\underline{\chi}} \;:=\; L_{\chi_1} \oplus \cdots \oplus L_{\chi_n}.\]
For~$s\in S(\bbC)$ let~$\rho \colon \pi_1(S, s) \to \GL(V_{\underline{\chi}, s})$ be  the monodromy representation  on the fiber~$V_{\underline{\chi}, s}$. The \emph{algebraic monodromy group} of the local system~$V_{\underline{\chi}}$ is the Zariski closure of the image of~$\rho$. By construction it is an algebraic subgroup of
\[
  \GL(V_{\chi_1, s}) \times \cdots \times \GL(V_{\chi_n, s})
 \;\subset\; 
 \GL(V_{\underline{\chi}, s}).
\]
We say~$f\colon \cX \to A_S$ is \emph{symmetric up to translation} if there are an automorphism~$\iota$ of~$\cX$ as an~$S$-scheme and~$a\colon S\to A$ such that~$f(\iota(x)) = a -f(x)$ for all~$x \in \cX$.
In this case Poincar\'e duality gives as in~\cite[sect. 1.1]{JKLM} for~$i=1,\dots, n$ a nondegenerate bilinear pairing
\[ 
\theta_{\chi_i,s} \colon 
\quad V_{\chi_i,s} \otimes  V_{\chi_i, s} \too L_{\chi_i, a(s)}
\]
which is symmetric if~$d$ is even, and alternating otherwise. The pairing is compatible with the monodromy operation, thus the algebraic monodromy group is contained in the orthogonal  resp. symplectic group if~$d$ is even resp. odd. We say that we have big monodromy if the algebraic monodromy group is a big as possible. More precisely:

\begin{definitionintro} 
We say that~$V_{\underline{\chi}}$ has {\em big monodromy} if its algebraic monodromy group contains~$G_1\times \cdots \times G_n$ as a normal subgroup, where~$G_i \subset \GL(V_{\chi_i, s})$ is defined by
\[
G_i \;:=\; 
\begin{cases}
\SL(V_i) & \textup{if~$\cX$ is not symmetric up to translation}, \\
\SO(V_i, \theta_i)  & \textup{if~$\cX$ is symmetric up to translation and~$d$ is even},  \\
\Sp(V_i, \theta_i)  & \textup{if~$\cX$ is symmetric up to translation and~$d$ is odd},
\end{cases}
\]
where $V_i = V_{\chi_i, s}$ and $\theta_i = \theta_{\chi_i,s}$.
We say that~$\cX \to S$ \emph{has big monodromy for most tuples of torsion characters} if, for each~$n \ge 1$ and all torsion~$n$-tuples~$\underline{\chi} \in \Pi(A)^n$ outside a finite union of torsion translates of linear subvarieties, we have
\begin{enumerate}
\item $R^j \pi_{S*} \pi_A^* L_{\underline{\chi}} = 0$ for~$j\neq d$;\smallskip
\item $V_{\underline{\chi}} = R^d \pi_{S*} \pi_A^* L_{\underline{\chi}}$ has big monodromy.
\end{enumerate}
By a \emph{linear subvariety} we mean a subset~$\Pi(B)\subset \Pi(A^n)$ for an abelian quotient variety~$A^n\twoheadrightarrow B$ with~$\dim B < n \dim A$. 
\end{definitionintro} 

By generic vanishing~\cite{BSS, KWVanishing, SchnellHolonomic}  condition (1) holds if~$f$ is  finite and more generally in the semismall case as described in \cref{ex:semismall}. 
Of course, the monodromy of~$V_{\underline{\chi}}$ is not big if the family~$\cX \to S$ is isotrivial. To prevent this, let~$\cX_{\bar{\eta}}$ be the fiber of~$\cX \to S$ over a geometric generic point~$\bar{\eta}$ of~$S$. We suppose that the morphism $f_{\bar{\eta}}$ is a closed embedding through which we see~$\cX_{\bar{\eta}}$ as a subvariety of $A_{S, \bar \eta}$. We say that~$\cX_{\bar{\eta}}$ is \emph{constant up to a translation} if it is the translate of a subvariety~$Y \subset A$ along a point in~$A(\bar{\eta})$. Other obstructions to big monodromy are~$\cX_{\bar{\eta}}$ being divisible, or a product or a symmetric power of a curve. By \cite{JKLM, KLM} these are the only obstructions to big monodromy:

\begin{bigmonodromyintro} 
Suppose~$f_{\bar{\eta}}$ is a closed embedding and the subvariety $X := f(\cX_{\bar{\eta}}) \subset A_{S, \bar \eta}$ has ample normal bundle, dimension~$d < (g-1)/2$, and satisfies the numerical assumption~\eqref{Eq:NumericalConditions}. Then the following conditions are equivalent:
\begin{itemize}
\item $X$ is nondivisible, not constant up to translation, not a symmetric power of a curve and not a product. \smallskip
\item $\cX \to A_S$ has big monodromy for most torsion tuples of characters.
\end{itemize}
\end{bigmonodromyintro}

Note that in \cite{JKLM, KLM} the morphism $f$ is assumed to be a closed embedding, but this hypothesis is only needed for the generic fiber. Indeed, the big monodromy property is insensitive to shrinking the base $S$, since $S$ is smooth. For the proof of~\cref{IntroThmExtrinsic} we will combine the Big Monodromy Criterion with a new nondensity result for integral points given by \cref{Thm:MainTheoremForFamiliesWithBigMonodromy} below.

\subsection{Nondensity of integral points} \label{sec:NonDensityIntro} We now work again over the number field~$K$. Given a variety~$S$ over~$K$ we say that its \emph{integral points are not Zariski dense} if for any finite extension~$K'$ of~$K$, any finite set of primes~$\Sigma'$ in~$K'$ and any flat scheme~$\cS$ over~$\cO_{K', \Sigma'}$ with generic fiber~$S_{K'}$, the subset \[\cS(\cO_{K', \Sigma'}) \; \subset \; S(K')\]
is not Zariski dense in~$S_{K'}$. Let~$A$ be an abelian variety over~$K$,~$S$ a smooth geometrically connected variety over~$K$, and~$\cX$ a smooth~$S$-scheme of finite type with geometrically connected fibers of dimension~$d$. We say that a morphism of~$S$-schemes~$f\colon \cX \to A_S$ has \emph{big monodromy for most tuples of torsion characters} if for some (equivalently,~any) embedding~$\sigma\colon K\into \bbC$ the family~$\cX_\sigma \to A_{S,\sigma}$ does so. We prove the following new nondensity result for integral points, whose range of application goes far beyond families of subvarieties of abelian varieties:

\begin{theoremintro} \label{Thm:MainTheoremForFamiliesWithBigMonodromy} Suppose that\smallskip
\begin{itemize}
\item $\cX \to A_S$ has big monodromy for most tuples of torsion characters; \smallskip
\item every geometric fiber~$X$ of~$\cX \to S$ satisfies \eqref{SkullInequality}.\smallskip 
\end{itemize}
Then the integral points of~$S$ are not Zariski dense.
\end{theoremintro}

Arithmetic consequences of big monodromy phenomena have a long history going back to the Deligne's proof of the Weil conjectures \cite[sect.~5]{DeligneWeilI}. Our proof of~\cref{Thm:MainTheoremForFamiliesWithBigMonodromy} will occupy sections~\ref{sec:buildings} through~\ref{sec:ProofOfNonDensityThm}; in the rest of this section we briefly outline the main ideas (for a comparison with previous work see~\cref{sec:comparison}).\medskip

Assume for simplicity~$K = \bbQ$. For any~$N \ge 1$ divisible enough, the abelian variety~$A$ extends to an abelian scheme over~$\bbZ[1/N]$ that we still denote~$A$. Similarly, without changing notation, we spread out the varieties~$S$ and~$\cX$ to~$\bbZ[1/N]$, as well as the morphism~$f \colon \cX \to A_S$. Suppose that~$S$ and~$\cX$ are smooth respectively over~$\bbZ[1/N]$ and~$S$. To see that
\[ S(\bbZ[1/N]) \subset S(\bbQ)\]
is not Zariski dense, pick a prime~$p\nmid N$, let~$\bar{\bbQ}_p$ be an algebraic closure of~$\bbQ_p$ and let~$\bar{\bbQ} \subset \bar{\bbQ}_p$ be the algebraic closure of~$\bbQ$. We then have an embedding 
\[ \Gamma_p :=\Gal(\bar{\bbQ}_p/ \bbQ_p) \; \intoo \; \Gamma := \Gal(\bar{\bbQ}/\bbQ).\]
First of all, the Big Monodromy Criterion suggests to consider cohomology twisted by a sum of rank one local systems as in \cite{LS20}. To do this suppose that we can find a character 
$\chi_0 \colon \pi_1(A(\bbC), 0) \to \bbC^\times$ of order~$r$ with~$p\nmid r$ such that 
\begin{enumerate}
\item $\rH^i(\cX_{s}(\bbC), \pi_A^\ast L_{ \chi}) = 0$ for all~$i\neq d$ and all~$\chi \in \Gamma.\chi_0$,~$s \in S(\bbC)$,\smallskip
\item $V_{\underline{\chi}}$ has big monodromy for the tuple~$\underline{\chi}$ of all characters~$\chi \in \Gamma.\chi_0$.\smallskip
\end{enumerate}
where~$\pi_A \colon \cX \to A$ is the projection. Here the assumption (2) is unrealistic and only serves to simplify the presentation; we will explain how to bypass it in \cref{sec:ProofOfNonDensityThm}. The orbit~$\Gamma.\chi_0$ corresponds to a direct summand
\[ E \subset [r]_\ast \bbQ_{p, A}\]
where~$[r] \colon A \to A$ denotes the multiplication by~$r$ and~$\bbQ_{p, A}$ is the constant~$p$-adic \'etale sheaf. For each~$s \in S(\bbZ[1/N])$ the Galois representation 
\[ V_s := \rH^d_{\et}(\cX_s \times {\bar{\bbQ}}, \pi^\ast_A E) \]
of~$\Gamma$ is pure of weight~$d$ by \cite{DeligneWeilI}. It is conjectured to be semisimple, but with our current state of knowledge a main task of this paper will be to work without assuming this. For this we consider the semisimplification filtration~$V^\bullet_s$ on~$V_s$. By Faltings' finiteness of pure global Galois representations \cite[th. 3.1]{DeligneBourbakiFaltings}, the image of the composite map
\[
\begin{tikzcd}[row sep=0pt]
S(\bbZ[1/N]) \ar[r]
& { \left\{ \begin{array}{c} \textup{filtered global}\\\textup{Galois rep's} \end{array} \right\} / \iso} \ar[r] 
& { \left\{ \begin{array}{c} \textup{graded global}\\\textup{Galois rep's} \end{array} \right\} / \iso} \\
s \ar[r, mapsto]& {(V_s, V^\bullet_s)} \ar[r, mapsto] & {\gr V^\bullet_s}
\end{tikzcd}
\]
is finite. Hence to show that~$S(\bbZ[1/N])$ is not Zariski dense in~$S$, it suffices to show that the fibers of the map \[ S(\bbZ[1/N]) \ni s \longmapsto \textup{isomorphism class of }\gr V^\bullet_s\] are not Zariski dense. 
A natural attempt would be to restrict representations to the local Galois group at a prime~$\ell\neq p$ with~$\ell\nmid N$. To see why this does not suffice, fix an embedding~$\bar{\bbQ} \into \bar{\bbQ}_\ell$ and consider the inclusion 
\[ \Gamma_\ell :=\Gal(\bar{\bbQ}_\ell/ \bbQ_\ell) \; \intoo \; \Gamma := \Gal(\bar{\bbQ}/\bbQ).\]
For~$s \in S(\bbZ[1/N])$ let~$V_{s, \ell}$ be the~$p$-adic representation of~$\Gamma_\ell$ obtained from~$V_{s}$ by restriction. By proper base change in \'etale cohomology, the isomorphism class of~$V_{s, \ell}$ only depends on the image of~$s$ in~$S(\bbF_\ell)$. In particular the set
\[  \left\{ V_{s, \ell} \mid s \in S(\bbZ[1/N]) \right\} / \iso  \]
is finite, so there is no hope of proving nondensity by restricting to the subgroup~$\Gamma_\ell$ for a single~$\ell \neq p$. Instead the heart of the Lawrence-Venkatesh method \cite{LV} lies in the insight that the restriction of~$V_s$ to~$\Gamma_p$ moves enough when~$s \in S(\bbZ[1/N])$ varies, provided that the family~$\cX \to A_S$ has big monodromy.

\medskip
To make this precise we will apply~$p$-adic Hodge theory. Recall that~$p$-adic Hodge theory associates to every finite dimensional \emph{crystalline} representation~$V$ of~$\Gamma_p$ a triple
\[ D_{\cris}(V) :=  (V_\dR, \phi_V, h^\bullet V) \]
of a finite dimensional~$\bbQ_p$-vector space~$V_\dR$, an~$\bbQ_p$-linear endomorphism~$\phi_V$ of~$V_\dR$ and a filtration~$h^\bullet V_\dR$ of~$V_\dR$ by vector subspaces (not necessarily stable under~$\phi_V$). Such triples are called filtered isocrystals, but in this introduction we will informally call them~$p$-adic Hodge structures. Rather than defining the functor~$D_{\cris}$ let us only mention what it does to the representations in question. To do so consider the vector bundle~$[r]_\ast \cO_A$ with the connection~$ \nabla$ induced by the canonical derivation on~$\cO_A$. There is a well-defined direct summand
 \[\cE \subset [r]_\ast \cO_{A_{\bbZ_p}}\]
associated to the local system~$E$ and stable under~$\nabla$. For~$s \in S(\bbZ[1/N])$ the representation~$V_s$ is crystalline since~$\cX_s$ has good reduction at~$p$. Then Faltings' \'etale-de Rham comparison theorem \cite[th. 5.6]{FaltingsCristallineEtale} gives
 \[ D_{\cris}(V_s) \; = \; (\rH^d_{\dR}(\cX_{s} \times \bbQ_p, \pi_A^\ast \cE), \; \phi_s, \; \textup{Hodge filtration})\]
 where~$\phi_s$ is the endomorphism induced by the Frobenius operator on crystalline cohomology via the de Rham-crystalline comparison theorem. This should give the idea of how to prove that the representations~$V_s$ move. In a sense to be made precise below, we can identify de Rham cohomology groups for~$s$ and~$s'$ close enough by integrating the Gauss-Manin connection. It then remains to show that the Hodge filtration moves enough, i.e.~the associated period mapping has large enough image, which will be implied by the hypothesis of big monodromy. 
 
\medskip 

In \cref{sec:Realizations,,sec:PAdicHodgeTheory} we develop a new framework for Galois representations and filtered isocrystals which permits to deal with operations such as taking the semisimplification of a representation and applying~$p$-adic Hodge theory. In particular, the semisimplication filtration~$V^\bullet_s$ on~$V_s$ induces a filtration~$D_\cris(V^\bullet_s)$ on~$D_\cris(V_s)$ which is stable under the crystalline Frobenius. This operation is compatible with taking the associated graded, leading to the following commutative diagram:
\[ \label{semisimplification-diagram}
\begin{tikzcd}
{ \left\{ \begin{array}{c} \textup{filtered global}\\\textup{Galois rep's}\\\textup{crystalline at~$p$} \end{array} \right\} } \ar[d, "\textup{restrict to~$\Gamma_p$}"']\ar[r,"\textup{graded}"] \ar[d]
& { \left\{ \begin{array}{c} \textup{graded global}\\\textup{Galois rep's}\\\textup{crystalline at~$p$}  \end{array} \right\}} \ar[d, "\textup{restrict to~$\Gamma_p$}"] \\
{ \left\{ \begin{array}{c} \textup{filtered crystalline}\\\textup{local Galois rep's} \end{array} \right\}} \ar[r,"\textup{graded}"] \ar[d, "{D_{\cris}}"']
& { \left\{ \begin{array}{c} \textup{graded crystalline}\\\textup{local Galois rep's} \end{array} \right\}} \ar[d, "{D_{\cris}}"] \\
{ \left\{ \begin{array}{c} \textup{filtered~$p$-adic}\\\textup{Hodge structures} \end{array} \right\}} \ar[r,"\textup{graded}"]
& { \left\{ \begin{array}{c} \textup{graded~$p$-adic}\\\textup{Hodge structures} \end{array} \right\}}
\end{tikzcd}
\]
To prove that~$S(\bbZ[1/N])$ is not Zariski dense, it will be enough to show that the fibers of the map
\[ S(\bbZ[1/N]) \ni s \longmapsto \textup{isomorphism class of~$\gr D_{\cris}(V^\bullet_s)$}\]
are not dense.
Consider the~$d$-th relative de Rham cohomology group
\[ \cV \; := \; \cH_{\dR}^d(\cX_{\bbZ_p}/S_{\bbZ_p}; \pi_A^\ast (\cE, \nabla))\]
By a careful choice of the prime~$p$ in the first place we may assume that~$\cV$ is a vector bundle over~$S_{\bbZ_p}$. This vector bundle~$\cV$ comes equipped with the Gauss-Manin connection that we still denote~$\nabla$. Fix a point~$o \in S(\bbZ[1/N])$ and consider the residue disk
\[ \Omega= \{ s \in S(\bbZ_p) \mid s \equiv o \mod p\}.\]
Since~$S(\bbF_p)$ is finite, it suffices to show that the intersection~$\Omega \cap S(\bbZ[1/N])$ is not Zariski dense for any such residue disk. Working on~$\Omega$ has the advantage that the Gauss-Manin connection can be integrated on~$\Omega$: We have a natural isomorphism of~$K$-vector spaces
\[ \tau_s \colon \cV_s= \rH^d_{\dR}(\cX_{s} \times \bbQ_p, \pi_A^\ast \cE) \stackrel{\sim}{\too} \cV_o= \rH^d_{\dR}(\cX_{o} \times \bbQ_p, \pi_A^\ast \cE), \]
see~\cref{sec:ParallelTransport}.
The filtration on~$\cV_o$ induced by the Hodge filtration on~$\cV_s$ via~$\tau_s$ is by definition the image of~$s$ via the~$p$-adic period mapping
\[
 \Phi_p\colon \quad \Omega \;\too\; \Flag_t(\cV_o)
\]
where~$\Flag_t(\cV_o)$ denotes the variety of flags on~$\cV_o$ whose type~$t$ is the type of the flag underlying the Hodge filtration. Define an equivalence  relation on~$\Omega$ by putting
\[ 
 s \sim s' \quad \textup{if} \quad \gr D_{\cris}(V_s^\bullet) \iso \gr D_{\cris}(V_{s'}^\bullet). \]
The equivalence classes give rise to constructible subsets~$Z \subset \Flag_t(\cV_o)$, and we win as soon as we show that the preimage~$\Phi_p^{-1}(Z)$ of each of these subsets is not Zariski dense in~$\Omega$. The Ax-Lindemann property of period mappings proved by Bakker-Tsimerman~\cite{BakkerTsimerman}, or rather its~$p$-adic version given in~\cref{Prop:PAdicBakkerTsimerman}, says that for this it will be enough to show that \smallskip 
\begin{itemize}
\item $\Phi_p$ has a Zariski dense image in~$\Flag_t(\cV_o)$, and \smallskip 
\item $\dim Z+ \dim S_{\bbQ} \le \dim \Flag_t(\cV_o)$. \smallskip
\end{itemize}
We now have to pay for the careless discussion given above: The period mapping~$\Phi_p$ will \emph{never} have a Zariski dense image in~$\Flag_t(\cV_o)$, for the simple reason that we forgot all the extra structures that~$\cE$ comes with. Indeed, over a finite extension~$K$ of~$\bbQ_p$ we have a splitting into a direct sum of line bundles
\[ \cE_{\rvert A_K} = \bigoplus_{\chi \in \Gamma.\chi_0}\cL_\chi\]
and the connection decomposes accordingly. So over~$K$ the image of~$\Phi_p$ is contained in
\[ \frH := \prod_{\chi \in \Gamma.\chi_0} \Flag_{t_\chi} (\rH^d_{\dR}(\cX_s \times K, \cL_\chi)). \]
The new framework developed in \cref{sec:buildings,,sec:Realizations} is flexible enough to enable to capture all the relevant extra structures on \'etale and de Rham cohomology. It will undoubtely have future applications in situations where the monodromy groups are not classical groups. In~\cref{lemma:FaltingsFiniteness} we generalize Faltings' finiteness theorem for pure global Galois representations in a way that keeps track of the extra structures. In \cref{prop:DescentViaPAdicHodgeTheory} we show that our general framework is compatible with~$p$-adic Hodge theory in the sense that the extra structures on \'etale cohomology are carried to those on de Rham cohomology via Faltings' comparison theorem. 

\medskip

Suppose now that with all the extra structures in place, we started the above argument all over again. By assumption the family~$\cX \to A_S$ has big monodromy with respect to the tuple of consisting of the characters in the orbit~$\Gamma.\chi_0$. This implies that~$\Phi_p$ has a Zariski dense image in~$\frH$. It remains to bound the dimension of the equivalence classes~$Z$ mentioned above. We do this in \cref{Thm:BoundInFlag} by giving such an upper bound in terms of a combinatorial function and of the dimension of the centralizer of the crystalline Frobenius. 
A crucial restriction on the dimension comes from the purity of the global Galois representation we started with: The filtration on each graded piece of~$\gr D_{\cris}(V_s^\bullet)$ induced by the Hodge filtration on~$D_{\cris}$ has the same weight~$d / 2$ as the original Hodge filtration. In the framework where extra structure is taken into account such a property carries over in a weaker form; see \cref{sec:PositivityParabolic} and \cref{Prop:PositivityFromGlobalPurity}. At this stage, to complete the proof we only need to bound the dimension of the centralizer of the crystalline Frobenius~$\varphi_o$. In order to apply \cref{Thm:BoundInFlag} we will need that this centralizer is very small compared to~$\frH$. We will obtain this by letting~$r$ go to infinity. This will conclude the proof of \cref{Thm:MainTheoremForFamiliesWithBigMonodromy}, hence of theorems~\ref{IntroThmIntrinsicSurfaces}, \ref{IntroThmIntrinsicGeneral} and \ref{IntroThmExtrinsic}.

\subsection{Comparison with previous approaches} \label{sec:comparison}
\Cref{Thm:MainTheoremForFamiliesWithBigMonodromy} is new and goes far beyond the application to subvarieties of abelian varieties. Its proof differs from that of \cite[th. 8.17]{LS20} in several aspects. The version of~$p$-adic Hodge theory that we set up works over arbitrary~$p$-adic fields rather than  over~$\bbQ_p$. This makes the proof of \cref{Thm:MainTheoremForFamiliesWithBigMonodromy} simpler and more natural, allowing for families of subvarieties over arbitrary number fields rather than  over~$\bbQ$.
We cast our theory for general reductive groups: This will permit to generalize \cref{Thm:MainTheoremForFamiliesWithBigMonodromy} to cases where the algebraic monodromy is not a classical group, but for instance a product of classical groups, or in fact any reductive group once a suitable version of the~$p$-adic Riemann-Hilbert correspondence is known.
Unlike~\cite[lemmas 5.50 and 5.52]{LS20}, our approach to the semisimplification of Galois representations does not require the choice of Levi subgroups. 
The absence of such a choice leads to a functorial construction, making sense of diagrams like the one on page~\pageref{semisimplification-diagram} (see \cref{prop:DescentViaPAdicHodgeTheory}). 
In comparison with the notion of Hodge-Deligne systems from loc.~cit., our approach in \cref{sec:Realizations} keeps track of only the minimal amount of linear algebra needed. Finally, we address a few inaccuracies in \cite[lemma 2.6]{LV} 
in the proof of Faltings' finiteness of semisimple representations with values in arbitrary reductive groups; see \cref{lemma:FaltingsFiniteness}.

\subsection*{Conventions} A \emph{variety} over a field~$k$ is a separated finite type~$k$-scheme, and a \emph{subvariety} is a closed subvariety. A projective scheme over a Noetherian scheme~$S$ is a proper~$S$-scheme admitting a relatively ample line bundle. Actions are left actions unless said otherwise.

\subsection*{Acknowledgements}We thank G. Ancona, J.-B.\ Bost, A.\ Javanpeykar, B.\ Klingler, H.\ Liu, and Q.\ Liu for helpful comments. We also thank the referees for their careful reading and for pointing out some inaccuracies in an earlier version of the paper. T.K.\ was supported by the DFG research grant Kr 4663/2-1.

\section{Moduli of symmetric powers of curves}

In the proof of \cref{IntroThmExtrinsic}, we will need the fact that being a product and being a symmetric power of a curve define constructible subsets of the Hilbert scheme of a given abelian variety. In this section, we prove something stronger: These loci are in fact closed. Even better, we show that a family of curves of gonality $> n$ can be reconstructed from the family of its $n$-th symmetric powers, which may be viewed as a version of Martens' theorem for families.

\subsection{Statements} 
Let~$S$ be a variety over a field~$k$ of characteristic~$0$, and let~$\cX \to S$ be a smooth projective morphism with geometrically connected fibers of dimension~$d$. Consider the subset
\[
\Sigma \; \subset \; S
\]
consisting of images of geometric points~$\bar{s}$ of~$S$ such that~$\cX_{\bar{s}}$ is a symmetric power of a (necessarily smooth, projective, and connected) curve. By usual spreading out arguments $\cX_{\bar{s}}$ is then the symmetric power of a curve defined over a finite extension of the residue field at the image of~$\bar{s} \to S$. Since we are in characteristic~$0$, if there is a relatively ample line bundle on~$\cX$, the relative Picard scheme $\Pic^0(\cX / S) \to S$ is a projective abelian scheme \cite[rem.~9.5.21]{FGAExplained}. We denote by
\[
\Alb(\cX / S) \too S
\]
its dual abelian scheme.  A morphism~$f \colon \cX \to \Alb(\cX / S)$ is a \emph{translate of an Albanese morphism} if there are an \'etale surjective morphism~$S' \to S$ and a section~$x \in \cX(S')$ such that
\[
f_{S'} = a + f(x),
\]
where~$a \colon \cX_{S'} \to \Alb(\cX / S)_{S'}$ is the Albanese morphism induced by~$x$. The main result is a version of Martens' theorem \cite{RanMartens} for families:

\begin{theorem} \label{Thm:SymmPowersComponent} 
Suppose that the geometric fibers of~$\cX \to S$ are canonically polarized and embed in their Albanese variety. Then the following hold:
\begin{enumerate}
\item The subset~$\Sigma \subset S$ is open and closed.
\item If there is a translate of an Albanese morphism~$\cX_\Sigma \to \Alb(\cX / S)_\Sigma$, then there exist a finite \'etale morphism~$\Sigma' \to \Sigma$, a smooth proper curve~$\cC \to \Sigma'$, and an isomorphism of~$\Sigma'$-schemes
\[
\Sym^d \cC \; \iso \; \cX_{\Sigma'}.
\]
\end{enumerate}
\end{theorem}

Here by a \emph{curve} over a Noetherian scheme we mean a flat, separated scheme of finite type with geometric fibers of pure dimension one. In fact, the finite \'etale cover~$\Sigma' \to \Sigma$ in the statement will be a principal bundle under the~$d$-torsion subgroup of the abelian scheme~$\Alb(\cX / S)_\Sigma$. \Cref{Thm:SymmPowersComponent} will be deduced in \cref{sec:ProofSymmPowersComp} from the following more precise statement, concerning the locus of symmetric powers of a curve in the Hilbert scheme of the abelian scheme~$\Alb(\cX/S)$. As we will need these results also for an arbitrary abelian variety over~$k$, we consider more generally a projective abelian scheme~$\cA \to S$. Let~$\Hilb_\cA$ be the Hilbert scheme of~$\cA$ and
\[
\Hilb^\sm_\cA \; \subset \; \Hilb_\cA \qquad \textup{resp.} \quad \Hilb^{\nd}_\cA \; \subset \; \Hilb_\cA
\]
be the open subsets where the fibers of the universal family over~$\Hilb_\cA$ are smooth with geometrically connected fibers, resp. are nondivisible. Note that nondivisibility is an open condition: the fiberwise stabilizer of the universal family is a closed group subscheme of~$\cA \times_S \Hilb_\cA$ and as such it is trivial over an open subset of~$\Hilb_\cA$ \cite[th. 5.22]{FGAExplained}. Let~$n \ge 1$ be an integer and
\[
H_n \; \subset \; \Hilb^\sm_\cA
\]
the open subset where the fibers of the universal family are curves~$C$ such that the sum induces an isomorphism
\[
\Sym^n C \; \stackrel{\sim}{\too} \; C + \cdots + C.
\]
The~$n$-fold sum~$C + \cdots + C$ is smooth, hence we have a morphism
\[
\sigma_n \colon H_n \;  \too \; \Hilb^\sm_\cA, \qquad C \; \longmapsto \; C + \cdots + C.
\]
The~$n$-torsion subgroup~$\cA[n]$ acts on~$H_n$ by translation and the morphism~$\sigma_n$ is invariant under this action.

\begin{theorem} \label{Thm:SymmetricPowerLocus}
The morphism~$\sigma_n$ is finite unramified, its image~$\Sigma_n$ is open and closed, and~$\cA[n]$ acts transitively on the fibers of~$\sigma_n$. In particular, over~$\Sigma_n \cap \Hilb^{\nd}_\cA$ the morphism~$\sigma_n$ is a principal~$\cA[n]$-bundle.
\end{theorem}

The proof of \cref{Thm:SymmetricPowerLocus} will be given in \cref{Sec:ProofSymmetricPowerLocusHilb}. To conclude, we discuss an analogous statement for products which will be needed in the proof of \cref{IntroThmExtrinsic}. Let~$\cX \subset \cA$ be a closed subvariety such that~$\cX \to S$ is smooth with geometrically connected fibers of dimension~$d$. Consider the subset
\[
\Pi \subset S
\]
of images of geometric points~$\bar{s}$ of~$S$ such that~$\cX_{\bar{s}}$ is a product of (necessarily smooth, projective, connected) varieties of dimension~$> 0$.

\begin{theorem} \label{Thm:ProductLocus}
Suppose that the fibers of~$\cX \to S$ are nondivisible and have ample normal bundle. Then~$\Pi \subset S$ is closed.
\end{theorem}

This result is less precise than \cref{Thm:SymmetricPowerLocus}, but it suffices for the purposes of \cref{IntroThmExtrinsic}. Moreover, this has a simpler proof in the spirit of \cite{CadoretLiu} based on the specialization of Tannaka groups of perverse sheaves and on the characterization of products in \cite[th.~A]{JKLM}; see \cref{sec:ProofOfProductLocus}.

\subsection{Specialization} 
Let~$S$ be the spectrum of a discrete valuation ring~$R$ with generic point~$\eta$ and closed point~$s$. For simplicity we assume that $R$ is a $\bbQ$-algebra. Let~$\cA \to S$ be an abelian scheme and~$\cX \subset \cA$ a closed subscheme smooth over~$S$ with geometrically connected fibers. We want to show that if the generic fiber \[X \; = \; \cX_\eta \; \subset \; A \; = \; \cA_\eta\] is a symmetric power of a curve, then so is the special fiber. For this, we consider cycles on the cotangent bundle: For any cycle~$\PLambda$ on~$\bbP(\Omega^1_A)$ we obtain a cycle~$\sp(\PLambda)$ on~$\bbP(\Omega^1_{\cA_s})$ by taking the specialization in the sense of~\cite[sect.~20.3]{FultonIntersectionTheory} in the relative cotangent bundle. With notation as in~\cref{sec:ConormalGeometry}, we have:

\begin{lemma} \label{lemma:Specialization}
Let~$\PLambda,\PLambda_1, \PLambda_2$ be clean cycles on~$\bbP(\Omega^1_A)$.
\begin{enumerate} 
\item For all~$n\in \bbN$ we have
\begin{align*}
 \sp(\Alt^n(\PLambda))
 &\; = \; \Alt^n(\sp(\PLambda)),\\
 \sp(\PLambda_1 \circ \PLambda_2)
 &\; = \;\sp(\PLambda_1)\circ \sp(\PLambda_2).
\end{align*}
\item If~$\PLambda$ is Lagrangian, then so is~$\sp(\PLambda)$.
\end{enumerate}
\end{lemma} 

\begin{proof} 
In~(1), the compatibility of specialization with~$\circ$ follows from its compatibility with proper direct images~\cite[prop.~20.3]{FultonIntersectionTheory}. By the same argument,~$\sp$ is also compatible with the pushforward under the multiplication map~$[\beta]\colon \cA \to \cA$ for any~$\beta \in \bbN$. Compatibility with wedge powers then follows because wedge powers of clean cycles can be written in terms of products of pushforwards under multiplication maps: By~\cite[sect.~1.4]{KraemerMicrolocalII} we have
\[
 \Alt^n(\PLambda) \;=\; \sum_{\beta} m_\beta \cdot [\beta_1]_\ast(\PLambda) \circ \cdots \circ [\beta_\ell]_\ast(\PLambda),
\]
where~$\beta = (\beta_1, \dots, \beta_\ell)$ runs over partitions of degree~$n$ and the~$m_\beta \in \bbQ$ are the coefficients arising in the expansion of Schur polynomials in terms of power sums. The same formula holds for wedge powers of clean cycles on the special fiber, so it follows that~$\sp(\Alt^n(\PLambda)) = \Alt^n(\sp(\PLambda))$.

\medskip 

For~(2), recall that a cycle on~$\bbP(\Omega^1_A)$ is called Lagrangian~\cite[3.1]{KleimanConormal} if its support~$S$ is of pure dimension~$\dim A - 1$ and the contact form on~$\bbP(\Omega^1_A)$ vanishes identically on~$S$. These properties are preserved under specialization, hence the claim follows. 
\end{proof}

\begin{lemma} \label{Lemma:SpecializationSymmPowers}
Suppose that there is a smooth curve~$C \subset A$ such that the sum induces an isomorphism~$\Sym^n C \iso X$. Let~$\cC \subset \cA$ be the scheme-theoretic closure of~$C$. Then,~$\cC$ is a smooth curve over~$S$ with geometrically connected fibers, and the sum induces an isomorphism
\[
 \Sym^n \cC \;\stackrel{\sim}{\too}\; \cX.
\]
\end{lemma}

\begin{proof}
By assumption~$\Sym^n C \iso X$ via the sum morphism, hence \cref{lem:symmetric-power-criterion} shows~$\Alt^n \PLambda_C = \PLambda_X$. The compatibility of specialization with wedge powers in \cref{lemma:Specialization} then gives
\[
 \Alt^n (\sp(\PLambda_C)) \;=\; \sp(\Alt^n \PLambda_C) \;=\; \sp(\PLambda_X) \;=\; \PLambda_{\cX_s},
\]
where the last equality holds because~$\cX$ is smooth over~$S$. Note that~$\PLambda_X$ is a geometrically integral cycle.
By \cref{lemma:WedgePowerLagrangian}, the specialization~$\sp(\PLambda_C)$ is geometrically integral.
Therefore, \cref{lemma:WedgePowerLagrangian} implies that 
\[
 \sp(\PLambda_C) = \PLambda_{C'}
\]
for some geometrically integral curve~$C' \subset A$. By \cref{lem:symmetric-power-criterion} this curve is smooth and the sum map induces an isomorphism~$\Sym^n C' \iso \cX_s$. Now the fundamental cycle of~$X$ is 
\[
 [X] \; = \; \tfrac{1}{n!} \sigma_\ast [C^n],
\]
where~$\sigma \colon \cA^n \to \cA$ is the sum map. The compatibility of specialization with pushforward \cite[prop.~20.3]{FultonIntersectionTheory} yields
\[
 [\cX_s] \; = \; \tfrac{1}{n!} \sigma_\ast [\cC_s^n].
\]
Since~$\cX_s$ is the~$n$-fold sum of both~$\cC_s$ and $C'$, \cref{lemma:CurveSummandOfGeneralType} implies that~$C'$ is geometrically integral and~$\cC_s$ is generically reduced with underlying reduced subvariety~$C'$. We then conclude by \cref{LemmaOfHironaka} below.
\end{proof}

\begin{lemma} \label{LemmaOfHironaka}
Let~$f \colon \cY \to S$ be a flat morphism of finite type with $\cY$ reduced. Suppose that for every geometric point~$\bar{s}$ of~$S$ the fiber~$\cY_{\bar{s}}$ is generically reduced and its underlying reduced subscheme is integral and normal. Then~$f$ has geometrically integral fibers.
\end{lemma}

\begin{proof} First note that the generic fiber $\cY_\eta$ is geometrically integral, hence~$\cY$ is integral. Indeed, $\cY_\eta$ is geometrically irreducible by assumption and it is geometrically reduced because it is reduced and the fraction field of~$R$ is of characteristic~$0$. The proof then goes as in \cite[Th.~III.9.11]{HartshorneAG}, replacing lemma~III.9.12 there with \cite[5.12.8]{EGAIV2}. To apply the latter, note that~$R$ is universally catenary \cite[\href{https://stacks.math.columbia.edu/tag/00NM}{Lemma 00NM}]{stacks-project} and universally Japanese, as~$K$ is of characteristic zero \cite[\href{https://stacks.math.columbia.edu/tag/0335}{Proposition 0335}]{stacks-project}.
\end{proof}

\subsection{Proof of \cref{Thm:SymmetricPowerLocus}} \label{Sec:ProofSymmetricPowerLocusHilb} We may assume that~$k$ is algebraically closed. For simplicity, we drop the subscript~$n$ and write~$\sigma$,~$H$, and~$\Sigma$ instead of~$\sigma_n$,~$H_n$, and~$\Sigma_n$. The proof proceeds in several steps:

\begin{lemma} \label{Lemma:StepsSymmetricPowerLocus} The following properties hold:
\begin{enumerate}
\item The image via~$\sigma$ of any connected component of~$H$ is open in~$\Hilb^\sm_\cA$.\smallskip
\item The morphism~$\sigma$ is formally unramified, i.e.~$\Omega^1_{H / \Hilb_\cA} = 0$.\smallskip
\item Given a discrete valuation ring~$R$ with fraction field~$K$ and a commutative solid diagram
\[ 
\begin{tikzcd}
\Spec K \ar[d] \ar[r, hook] & \Spec R \ar[d] \ar[dl, dashed] \\
H \ar[r, "\sigma"] & \Hilb^\sm_\cA
\end{tikzcd}
\]
there exists a unique dotted arrow making the diagram commutative.\smallskip
\item For any~$s \in S(k)$, the action of~$\cA_s[n]$ on the fibers of~$\sigma \colon H_{n, s} \to \Sigma_s$ is transitive.
\end{enumerate}
\end{lemma}

Before proving this statement, we explain how these properties imply \cref{Thm:SymmetricPowerLocus}.

\begin{proof}[Proof of \cref{Thm:SymmetricPowerLocus}] 
To prove that~$\sigma$ is proper, let~$U \subset H$ be a connected component. The morphism~$\sigma_{\rvert U} \colon U \to \Hilb_\cA^{\sm}$ is of finite type, hence proper by the Noetherian valuative criterion \cite[\href{https://stacks.math.columbia.edu/tag/0208}{Lemma 0208}]{stacks-project} and \cref{Lemma:StepsSymmetricPowerLocus}~(3). On the other hand, by \cref{Lemma:StepsSymmetricPowerLocus} (1) the image~$V := \sigma(U)$ is open, hence a connected component of~$\Hilb_\cA^{\sm}$. By \cref{Lemma:StepsSymmetricPowerLocus} (4) we have
\[ \sigma^{-1}(V) = \cA[n] \cdot U. \]
The right-hand side is a finite union of connected components of~$H$, so the above discussion shows that~$\sigma \colon H \to \Hilb_\cA$ is proper. It remains to show that~$\sigma$ is a principal~$\cA[n]$-bundle over the locus
$
\Sigma^{\nd} := \Sigma \cap \Hilb_\cA^{\nd}.
$
If a symmetric power of a curve~$C$ in any abelian variety is nondivisible, then so is the curve~$C$. Therefore, we have an inclusion
$\sigma^{-1}(\Sigma^{\nd}) \subset H^{\nd} := H \cap \Hilb_\cA^{\nd}$. The finite \'etale~$S$-group scheme~$\cA[n]$ acts freely on the quasi-projective $S$-scheme~$H^{\nd}$, hence the quotient 
\[
H^{\nd} \too H^{\nd} / \cA[n]
\]
exists as an~$S$-scheme \cite[Exp.~V, Rem.~5.1]{SGA3} and is a principal~$\cA[n]$-bundle, hence finite \'etale. The morphism
$
\sigma^{-1}(\Sigma^\nd) / \cA[n] \to \Sigma^{\nd}
$
is finite unramified, and by \cref{Lemma:StepsSymmetricPowerLocus} (4) it is injective on~$k$-points, thus a closed embedding by \cite[\href{https://stacks.math.columbia.edu/tag/0E8M}{Tag 0E8M}]{stacks-project}.
\end{proof}

We now pass to the proof of \cref{Lemma:StepsSymmetricPowerLocus}. 

\begin{proof}[{Proof of \cref{Lemma:StepsSymmetricPowerLocus}}] (1) Let~$U \subset H$ be a connected component. Since~$U$ is of finite type over~$S$, the induced morphism~$\sigma \colon U \to \Hilb_\cA$ is of finite type, hence its image~$\sigma(U)$ is constructible. To show that~$\sigma(U)$ is open, hence conclude the proof, it suffices to show that it is stable under generalization. That is, given a~$k$-point~$h$ of~$\Hilb_\cA$ in~$\sigma(U)$, we need to prove that~$\sigma(U)$ contains~$\Spec R$ where~$R$ is the local ring of~$\Hilb_\cA$ at~$h$. To do so, consider 
\[ T : = \Spec \hat{R} \]
where~$\hat{R}$ is the completion of~$R$ along its maximal ideal. It suffices to construct a morphism~$\phi \colon T \to U$ making the following diagram commutative:
\[ 
\begin{tikzcd}
& T \ar[dl, bend right=25, "\phi"'] \ar[d]\\
U \ar[r, "\sigma"] & V
\end{tikzcd}
\]
Let~$\cX \to \Hilb_\cA$ be the universal family. By assumption~$\cX_h$ is a symmetric power of a curve, which is nonhyperelliptic because~$\cX_h$ embeds in its Albanese variety. Therefore, by \cite[th. 4.2]{KempfDeformations} there is a smooth formal curve over~$\Spf(\hat{R})$ whose~$n$-th symmetric power is isomorphic to the formal completion~$\hat{\cX}$ of~$\cX$ along~$\cX_h$. Since its canonical bundle is ample, such a formal curve can be algebraized into a smooth projective curve 
\[ \cC \; \too \; T := \Spec \hat{R}.\] Using again the formal GAGA theorem, the isomorphism~$\Sym^n \hat{\cC} \iso \hat{\cX}$ comes from a unique isomorphism of~$T$-schemes
\[  \Sym^d \cC \; \iso \; \cX_{T} := \cX \times_S T. \]
Since~$\cC \to T$ is smooth and~$\hat{R}$ is complete, any~$k$-point of~$\cC$ can be lifted to a~$T$-valued point of~$\cC$. Recalling that~$k$ is algebraically closed, fix~$p \in \cC(T)$ and consider the Abel-Jacobi embedding
\[ \quad \cC \intoo \Pic^0(\cC / T).\]
Note that the geometric fibers of~$\cC \to T$ have gonality~$> n$ because their~$n$-th symmetric power embed in their Albanese variety. It follows that the sum morphism induces an isomorphism
\[ \Sym^n \cC \; \stackrel{\sim}{\too} \; \cC + \dots + \cC.  \]
Via the identification~$\cA_T \iso \Pic^0(\cC / T)$ given by~$\Sym^n \cC \iso \cX_{T}$ we thus have
\[ j(\cX) = \cC + \dots + \cC \]
where~$j \colon \cX \into \cA_T$ is the Albanese morphism associated with~$y_0 = (p, \dots, p)$. To conclude the proof it suffices to take a suitable translate~$\cC'$ of~$\cC$ in~$\cA_T$ so that~$i(\cX)$ is the~$n$-fold sum of~$\cC'$. Namely, take 
\[ \cC' = \cC + u \qquad \textup{with} \qquad u \in \cA(T) \quad \textup{such that} \quad n u = j(x_0).\]
Note that such a point~$u \in \cA(T)$ exists because~$[n] \colon \cA \to \cA$ is a finite \'etale surjective morphism and~$\hat{R}$ is henselian. \medskip

(2) To see that the morphism~$\sigma$ is formally unramified, it suffices to show that for each~$h \in H(k)$
the~$k$-linear map~$\Omega^1_{\Hilb_A, \sigma(h)} \to \Omega^1_{\Hilb_A, h}$ is surjective where~$A:= \cA_{s}$ and~$s$ is the image of~$h$ in~$S$. This is equivalent to proving that the tangent map at~$h$,
\[ \rT_h \sigma \colon \quad \rT_h \Hilb_A = (\Omega^1_{\Hilb_A, h})^\vee \; \too \; \rT_{\sigma(h)}  \Hilb_A = (\Omega^1_{\Hilb_A, \sigma(h)})^\vee \] 
is injective. By definition, the~$k$-point~$h$ corresponds to a smooth, connected, nondivisible curve~$C \subset A$ with ample normal bundle such that the sum induces an isomorphism
\[ \Sym^n C \; \stackrel{\sim}{\too} \; X:= C + \cdots + C\]
With this notation the~$k$-point~$\sigma(h)$ corresponds to the subvariety~$X \subset A$.  A tangent vector to~$\Hilb_A$ in~$h$, seen as a~$k[\epsilon]$-point of~$\Hilb_A$, can be thought of as a first order thickening
\[\cC \; \subset \; A_{k[\epsilon]}\] 
of~$C$. Say that a scheme~$\cY$ over~$k[\epsilon]$ is \emph{constant} if it is the base change along the morphism~$k \to k[\epsilon]$ of the subscheme of~$\cY$ given by the equation~$\epsilon = 0$. Suppose that the~$n$-fold sum 
\[ \cX \; = \; \cC + \cdots + \cC \]
is constant. Since~$H \subset \Hilb_A$ is open, the sum morphism induces an isomorphism of~$k[\epsilon]$-schemes
\[ \Sym^n \cC \; \stackrel{\sim}{\too} \; \cX.\]
It follows that~$\Sym^n \cC$ is constant. Since for nonhyperelliptic curves first order deformations of symmetric powers correspond bijectively to first order deformations of the curve \cite[th. 3.6]{KempfDeformations}, it follows that~$\cC$ is constant. To conclude it remains to show that the closed embedding
\[ i \colon \quad C_{k[\epsilon]} \; \intoo \; A_{k[\epsilon]} \]
is the base change of the embedding~$i_0 \colon C \into A$. Since morphisms of abelian varieties~$\Pic^0(C) \to A$ do not deform, such a morphism~$i$ is necessarily the composite of the base change of~$i_0$ with the translation~$A_{k[\epsilon]} \to A_{k[\epsilon]}$ given by a tangent vector~$v \in \Lie A$. Moreover,  since the sum morphism induces an isomorphism~$\Sym^n \cC \iso \cX$,~$v$ must be~$n$-torsion when seen as a~$k[\epsilon]$-point of~$A$ supported at~$0$. In other words,~$v$ lies in~$\Lie A[n] = 0$ hence~$v = 0$.\medskip

(3) The morphism~$Z := \Spec R \to \Hilb^\sm_\cA$ corresponds to a closed subscheme \[\cX \; \subset \; \cA \times_S Z\] such that the morphism~$\cX \to Z$ is smooth with geometrically connected fibers. Similarly, the morphism~$\Spec K \to H$ can be seen as a smooth, geometrically connected curve~$C \subset A:= \cA \times_S \Spec K$ such that the sum induces an isomorphism
\[ \Sym^n C \; \stackrel{\sim}{\too}\; X : = \cX \times_Z \Spec K.\]
Let~$\cC \subset \cA \times_S Z$ be the scheme theoretic closure of~$C$. By \cref{Lemma:SpecializationSymmPowers} the~$Z$-scheme~$\cC$ is a smooth curve with geometrically connected fibers and the sum in~$\cA$ induces an isomorphism
\[ \Sym^n \cC \stackrel{\sim}{\too} \cX.\]
The morphism~$Z \to H$ corresponding to~$\cC$ gives the wanted dotted arrow in the statement.\medskip

(4) For~$i = 1, 2$ let~$C_i \subset A:= \cA_s$ be a smooth projective curve such that the sum induces an isomorphism \[\sigma_i \colon \Sym^n C_i \; \stackrel{\sim}{\too} \; X_i \; := \; C_i + \cdots + C_i.\]
Suppose~$X_1 = X_2$. Then by \cite{RanMartens} there is an isomorphism~$f \colon C_1 \to C_2$ such that the following diagram is commutative:
\[ 
\begin{tikzcd}[column sep=40pt, row sep=30pt]
C_1 \ar[r, "\Delta_1"] \ar[d, "f", swap]& \Sym^n C_1 \ar[r, "\sigma_1"] \ar[d, "\Sym^n f", swap] & X_1 \ar[d, equal] \\
C_2 \ar[r, "\Delta_2"] & \Sym^n C_2 \ar[r, "\sigma_2"] & X_2
\end{tikzcd}
\]
where~$\Delta_i$ is the diagonal embedding. Now~$\sigma_i \circ \Delta_i = n j_i$ where~$j_i \colon C_i \into A$ is the closed embedding. Therefore the commutativity of the above diagram implies that the morphism~$j_2 \circ f - j_1$ takes values in~$A[n]$, thus it is identically equal to an~$n$-torsion point of~$A$.
\end{proof}

\subsection{Proof of \cref{Thm:SymmPowersComponent}} \label{sec:ProofSymmPowersComp} Now \cref{Thm:SymmPowersComponent} is follows easily from \cref{Thm:SymmetricPowerLocus}. Let \[\cA \; := \; \Alb(\cX/S),\] the relative Albanese variety. To check that the symmetric power locus~$\Sigma \subset S$ is open and closed, we may reason \'etale locally on~$S$ and suppose that~$\cX \to S$ admits a section~$x$. By assumption the Albanese morphism given by~$x$ is a closed embedding~$i \colon \cX \into \cA$, thus it induces a morphism~$S \to \Hilb_\cA$ to the Hilbert scheme of~$\cA$. With notation as in \cref{Thm:SymmetricPowerLocus} we then have that
\[ \Sigma = \Sigma_d \times_{\Hilb_\cA} S\]
is open and closed. Concerning the latter claim, instead of assuming that there is a section of~$\cX \to S$, suppose that there is a translate of an Albanese morphism~$j \colon \cX_\Sigma \to \cA_\Sigma$. Again, by hypothesis~$j$ is a closed embedding, hence it induces a morphism
$f\colon \Sigma \to \Hilb_\cA$. Since the fibers of~$\cX_\Sigma$ are by assumption symmetric powers of curves, the morphism~$f$ takes values in~$\Sigma_d$. The key point is that~$f$ actually takes values in the nondivisible locus~$\Sigma_d \cap \Hilb_\cA^{\nd}$ because symmetric powers of curves are nondivisible in their Albanese variety. Indeed, viewing the latter as the Jacobian of the curve, the symmetric power in question is a summand of a theta divisor, which is a principal polarization, hence nondivisible. The wanted finite \'etale cover of~$\Sigma$ is therefore the principal~$\cA[n]$-bundle
\[ H_d \times_S \Sigma \; \too \; \Sigma = (\Sigma_d \cap \Hilb_\cA^\nd) \times_S \Sigma. \]
This concludes the proof of \cref{Thm:SymmPowersComponent}. \qed

\subsection{Proof of \cref{Thm:ProductLocus}} \label{sec:ProofOfProductLocus} 
We first prove that~$\Pi$ is constructible. Consider as in \cite[sect.~3.2]{JKLM} the Tannaka group
\[G_s \;=\; G_{\cX_{\bar{s}},\omega} \;\subset\; \GL(V_s),
\]
where~$\omega$ is any fiber functor on the Tannaka category generated by the constant perverse sheaf on~$X_{\bar{s}} \subset \cA_{\bar{s}}$ and where~$V_s \in \Rep(G_s)$ is the faithful representation defined by this perverse sheaf. By the invariance of Tannaka groups under extensions of algebraically closed fields \cite[cor.~4.4]{JKLM}, the isomorphism type of the pair~$(G_s, V_s)$ depends only on the image~$s\in S$ of the geometric point~$\bar{s}$. The derived group of the identity component is a connected semisimple group
\[
 G^\ast_s = [G_s^\circ, G_s^\circ]
\]
whose isomorphism type does not change if the subvariety~$\cX_{\bar{s}} \subset \cA_{\bar{s}}$ is replaced by a translate \cite[lemma~4.3.2]{KraemerMicrolocalII}, and by~\cite[prop.~7.4]{KWGeneric} this defines a constructible stratification of the base: For any given connected semisimple group~$H$ the locus
\[
 \{ s\in S \mid G_s^\ast \iso H \} \;\subset\; S
\]
is constructible.  So it suffices to note that the locus~$\Pi$ can be characterized by the isomorphism type of~$G_s^\ast$: Indeed~$\cX_s$ is a product if and only if the algebraic group~$G^\ast_s$ is not simple \cite[th. A]{JKLM}.\medskip

To show that~$\Pi$ is closed it suffices to show that it closed under specialization. For this consider a point~$s \in S$ which is a specialization of a point~$\eta \in \Pi$.  By \cite[lemma 7.2]{KWGeneric} we have an embedding 
\[
 G_s \;\hookrightarrow\; G_\eta
\]
as an algebraic subgroup such that via this embedding the representation~$V_\eta$ restricts to~$V_s$. Since~$\eta \in \Pi$, the group~$G_\eta$ is not simple, hence~$V_\eta$ decomposes as a tensor product of more than one representation. On the other hand, the representation~$V_s$ is minuscule because~$\cX_s$ is nondivisible \cite[cor.~1.10]{KraemerMicrolocalI}. Tensor products of representations of a simple group are never minuscule \cite[lemma 6.4]{JKLM} hence~$G_s$ cannot be simple and~$\cX_{s}$ is a product.

\section{The Albanese torsor}

In this section, we collect some general results on Albanese torsors that will be used in the proof of \cref{IntroThmIntrinsicGeneral}, particularly in cases where the varieties under consideration have no rational points. More precisely, in \cref{sec:AlbaneseTorsor} we establish a projectivity result for Albanese torsors in families; in \cref{sec:FinitenessTorsorsAbVar} we prove a Shafarevich-type finiteness result for polarized torsors under abelian varieties; and in \cref{sec:NeronModelAlbanese} we discuss basic properties of the N\'eron model of the Albanese torsor in the good reduction case. Throughout, $S$ denotes a Noetherian scheme.

\subsection{The Picard scheme} We begin with a reminder on the representability of Picard functors. Let \( X \to S \) be a smooth projective morphism with geometrically connected fibers. The \'etale sheafification of the functor  
\[
(\textup{Sch}/S) \to (\textup{Groups}), \quad T \mapsto \Pic(X_T) / \Pic(T)
\]  
is representable by an $S$-scheme $\Pic(X/S)$, which is the disjoint union of open and closed subschemes $\Pic^\Phi(X/S)$ for $\Phi \in \bbQ[z]$. These are projective over $S$ and parametrize line bundles with Hilbert polynomial $\Phi$ (with respect to some fixed relatively ample line bundle on \( X \)) \cite[th.~8.2.5]{NeronModels}. We denote by 
\[ \Pic^\tau(X/S) \; \subset \; \Pic(X/S)\]
the projective~$S$-scheme parametrizing line bundles on~$X$ with same Hilbert polynomial as the structural sheaf. Consider the subfunctor \[\Pic^0(X/S) \quad \textup{of} \quad T \mapsto \Pic^\tau(X/S)(T)\] parametrizing line bundles whose image in the fiberwise N\'eron-Severi group is trivial. Then~$\Pic^0(X/S)$ is representable by an abelian scheme if and only if the smallest open and closed subset~$U \subset \Pic^\tau(X/S)$ containing the zero section of~$\Pic^\tau(X/S)$ is smooth over~$S$ with geometrically connected fibers. Note that in such case~$\Pic^0(X/S)$ is projective, because~$\Pic^\tau(X/S)$ is. When~$S$ is reduced~$\Pic^0(X/S)$ is representable by an abelian scheme if the algebraic groups~$\Pic^0(X_s)$ for~$s \in S$ are all smooth of the same dimension \cite[prop. 9.5.20]{FGAExplained}. If~$S$ is a~$\bbQ$-scheme, then~$\Pic^0(X/S)$ is always representable by an abelian scheme \cite[rem. 9.5.21]{FGAExplained}.

\subsection{Projectivity} \label{sec:AlbaneseTorsor} 
Assume that~$\Pic^0(X/S)$ is represented by an abelian scheme, which is then projective by the above discussion. The \emph{relative Albanese scheme} is the dual abelian scheme \cite[th. 8.4.5]{NeronModels}
\[
\Alb(X/S) = \Pic^0(X/S)^\ast.
\]
Given an~$S$-scheme~$T$, any section~$x \in X(T)$ determines a unique morphism
\[
\alb_{X,x} \colon X_T \longrightarrow \Alb(X/S)_T
\]
of~$T$-schemes. However, in the absence of a global section of~$X \to S$, there is no canonical morphism~$X \to \Alb(X/S)$. Instead, there exists a natural morphism taking values in a principal~$\Alb(X/S)$-bundle over~$S$, called the \emph{Albanese torsor}, which we now introduce. The existence and representability of such a torsor have been extensively studied in greater generality, albeit with weaker conclusions; see~\cite[Exp.~VI, Th.~3.3]{FGA} and \cite{Schroer}. The key to our representability result is a concrete description of the Albanese torsor, based on the following definition:

\begin{definition} \label{def:TranslateAlbaneseMorphism}
Given an~$S$-scheme~$T$, a morphism~$f \colon X_T \to \Alb(X/S)_T$ of~$T$-schemes is a \emph{translate of an Albanese morphism} if there exist a surjective \'etale morphism~$T' \to T$ and a point~$x \in X(T')$ such that
\[
f_{T'} = \alb_{X, x} + f(x).
\]
\end{definition}
Equivalently, this identity holds for every point~$x$ of~$X$ with values in a~$T$-scheme~$T'$. Consider the functor
\[
\underline{P}(X/S) \colon (\mathrm{Sch}/S) \too (\mathrm{Sets}),
\]
which associates to an~$S$-scheme~$T$ the set of translates of an Albanese morphism over~$T$. The map
\[
X(T) \longrightarrow \underline{P}(X/S)(T), \quad x \longmapsto \alb_{X,x},
\]
is functorial in~$T$. Moreover, the translation action
\[
\underline{P}(X/S)(T) \times \Alb(X/S)(T) \longrightarrow \underline{P}(X/S)(T), \quad (f, a) \longmapsto f + a,
\]
is free, transitive, and functorial in~$T$.

\begin{proposition} \label{Prop:RepAlbaneseTorsor}
The functor~$\underline{P}(X/S)$ is representable by a scheme~$P(X/S)$ projective over $S$.
\end{proposition}

\begin{proof}
Set~$A := \Alb(X/S)$ and consider the projective~$S$-scheme~$Y := X \times_S A$. Let~$\Hilb_Y$ denote the Hilbert scheme of~$Y$, and~$\underline{\Hilb}_Y$ its functor of points. One identifies~$\underline{P}(X/S)$ with a subfunctor of~$\underline{\Hilb}_Y$ by associating to a translate of an Albanese morphism its graph. It suffices to show that~$\underline{P}(X/S) \subset \underline{\Hilb}_Y$ is represented by a closed subscheme
\[
P(X/S) \subset \Hilb_Y.
\]
In this case,~$P(X/S)$ is a principal~$A$-bundle over~$S$. Hence for each connected component~$S' \subset S$, the fiber~$P(X/S)_{S'}$ lies in a connected component of~$\Hilb_Y$ and is thus projective. To prove the claim, note that since~$\underline{P}(X/S)$ is a subsheaf of~$\underline{\Hilb}_Y$ for the \'etale topology on~$S$, we may work \'etale-locally on~$S$ and assume that~$X \to S$ admits a section~$x$. In this case, the desired closed subscheme is the image of the morphism
\[
A \longrightarrow \Hilb_Y, \quad a \longmapsto \Gamma + a,
\]
where~$\Gamma \subset Y$ is the graph of the morphism~$\alb_{X,x}$.
\end{proof}

The above discussion shows that~$P(X/S)$ is a principal bundle under the abelian scheme~$\Alb(X/S)$ and that there exists a morphism
\[
\alb_X \colon X \longrightarrow P(X/S)
\]
of~$S$-schemes, called the \emph{Albanese morphism}. Although not used in this paper, it is universal among morphisms from~$X$ to principal bundles under abelian schemes.

\begin{definition}
The principal bundle~$P(X/S)$ is called the \emph{Albanese torsor}.
\end{definition}

\subsection{Finiteness} \label{sec:FinitenessTorsorsAbVar} The importance of the projectivity and good reduction of the Albanese torsor lies in an extension of the Shafarevich conjecture for abelian varieties by Faltings~\cite{FaltingsMordell} to principal bundles. The result is somewhat well known (see, for instance, the proof of~\cite[th.~6.5]{JavanpeykarCyclic}), but we include a proof due to the lack of a suitable reference. To state and prove it, we first recall some standard facts about principal bundles under abelian schemes. Let~$A$ be a projective abelian scheme over~$S$,~$P$ a principal~$A$-bundle, and~$A^\ast$ the dual abelian scheme. There exists a unique isomorphism of~$S$-schemes
\[
f \colon \quad \NS(A/S) \; \stackrel{\sim}{\too} \; \NS(P/S)
\]
such that, for all~$S$-schemes~$T$ and all~$p \in P(T)$, the base change~$f_T$ is given by pullback along the inverse of the morphism~$\gamma_p \colon A_T \to P_T$,~$a \mapsto a + p$. Moreover, given an~$S$-point~$\lambda$ of~$\NS(A/S)$, there exists a unique isomorphism
\begin{equation} \label{PicLambdaPrincipalBundle}
g_\lambda \colon \quad \Pic^\lambda(A/S) \; \stackrel{\sim}{\too} \; \Pic^\lambda(P/S) \times_{A^\ast} P_\lambda
\end{equation}
of principal~$A^\ast$-bundles, induced by pullback along the inverse of~$\gamma_p$ as above. In this formula, $P_\lambda = P / \ker(\lambda)$ and~$\times_{A^\ast}$ denotes the contracted product of principal~$A^\ast$-bundles. Given a line bundle~$\cL$ on~$P$, its N\'eron--Severi class defines a morphism of group schemes~$\lambda \colon A \to A^\ast$. If~$S$ is connected and~$\cL$ is relatively ample, then the direct image~$\pi_\ast \cL$ under the structural morphism~$\pi \colon P \to S$ is locally free of some rank~$d$, called the \emph{degree} of~$\cL$, and we have
\[
\deg(\lambda) = d^2.
\]
Moreover,~$\cL$ defines an~$S$-point of~$\Pic^\lambda(P/S)$, and hence~$g_\lambda$ induces an isomorphism of principal~$A^\ast$-bundles
\[
\Pic^\lambda(A/S) \;\xrightarrow{\sim}\; P_\lambda.
\]
In other words, the morphism~$\lambda_\ast \colon \rH^1(S, A) \to \rH^1(S, A^\ast)$ of abelian groups induced by~$\lambda$ maps~$P$ to~$\Pic^\lambda(A/S)$. Here~$\rH^1(S, G)$ denotes the abelian group of isomorphism classes of principal bundles under a finitely presented, faithfully flat, commutative group~$S$-scheme~$G$. With this notation, we have:

\begin{lemma} \label{Lem:ClassOfPolarizedAbelianTorsor}
Let~$\lambda$ be an~$S$-point of~$\NS(A/S)$ which is relatively ample. Then the group~$\rH^1(S, \ker(\lambda))$ acts transitively on the subset of~$\rH^1(S, A)$ consisting of principal~$A$-bundles~$P$ admitting a line bundle~$\cL$ with N\'eron--Severi class~$\lambda$.
\end{lemma}

\begin{proof}
The morphism~$\lambda_\ast$ fits into the long exact cohomology sequence
\[
\cdots \too \rH^1(S, \ker(\lambda)) \too \rH^1(S, A) \stackrel{\lambda_\ast}{\too} \rH^1(S, A^\ast) \too \cdots,
\]
and the result follows.
\end{proof}

We are now in a position to state the aforementioned finiteness result:

\begin{proposition} \label{Prop:FaltingsFinitenessTorsor}
Let~$d, g \geq 1$. Then, up to isomorphism, there exist only finitely many principal bundles over~$\cO_{K, \Sigma}$ under abelian schemes of relative dimension~$g$ equipped with a polarization of degree~$d$.
\end{proposition}

\begin{proof}
We follow the argument from~\cite[th.~6.5]{JavanpeykarCyclic}. Taking~$\lambda = 0$ in~\eqref{PicLambdaPrincipalBundle}, we can recover an abelian scheme from any principal bundle under it. Therefore, by the Shafarevich conjecture for abelian varieties, it suffices to prove the statement for principal bundles under a \emph{fixed} abelian scheme~$A$ over~$\cO_{K, \Sigma}$. Furthermore, there are only finitely many isogenies~$A \to A^\ast$ of degree~$d^2$ up to automorphisms of $A$, so we may assume that all polarized principal~$A$-bundles under consideration induce the same isogeny~$\lambda \colon A \to A^\ast$. Since~$\rH^1(\cO_{K, \Sigma}, \ker(\lambda))$ is finite~\cite[prop.~5.1]{GilleMoretBailly} the result follows from \cref{Lem:ClassOfPolarizedAbelianTorsor}.
\end{proof}

\subsection{N\'eron models} \label{sec:NeronModelAlbanese} Suppose that~$S = \Spec \cO_{K, \Sigma}$ for a number field~$K$ and a finite set of primes~$\Sigma$ of~$K$. We begin by recalling the following standard fact about models of abelian torsors:

\begin{lemma} \label{Lemma:ModelsOfAbelianTorsors}
Let~$\cA$ be an abelian scheme over~$S$, and let~$\cP$ be a principal~$\cA$-bundle. Denote by~$A$ and~$P$ the generic fibers of~$\cA$ and~$\cP$. Then:
\begin{enumerate}
\item the~$S$-scheme~$\cP$ is smooth and projective;
\item given a smooth~$S$-scheme~$\cT$ with generic fiber~$T$, the map~$\cP(\cT) \to P(T)$ induced by restriction is bijective;
\item any line bundle~$L$ on~$P$ extends to a line bundle~$\cL$ on~$\cP$, and if~$L$ is ample, then~$\cL$ is relatively ample.
\end{enumerate}
\end{lemma}

\begin{proof}
(1) Follows from \cite[th.~6.4.1]{NeronModels}.\smallskip

(2) The map~$\cP(\cT) \to P(T)$ is injective because~$\cT$ is smooth over~$S$, hence~$T \subset \cT$ is Zariski dense. Thus any morphism~$T \to P$ admits at most one extension~$\cT \to \cP$. By uniqueness, the existence of such an extension can be checked \'etale locally, and therefore we may assume that~$\cP \to S$ admits a section. In this case we can identify~$\cP$ with~$\cA$, and the extension exists by \cite[cor.~8.4.6]{NeronModels}.\smallskip

(3) By \cref{Lem:ExtensionLineBundles} below, the only issue is the relative ampleness of~$\cL$. For any~$s \in S$, we have~$h^0(P, L^{\otimes n}) \le h^0(\cP_s, \cL^{\otimes n}_s)$ for all~$n \ge 0$.
Thus, the line bundle~$\cL_s$ on~$\cP_s$ is big because~$L$ is ample. Since big line bundles on abelian varieties are ample, we conclude that~$\cL_s$ is ample.
\end{proof}

\begin{lemma} \label{Lem:ExtensionLineBundles} Let~$\cY \to S$ be a smooth projective morphism. Then any line bundle on the generic fiber of~$\cY$ extends to a line bundle on~$\cY$. 
\end{lemma}

\begin{proof} See for instance \cite[answer 9.5.7]{FGAExplained}.
\end{proof}

Now let~$\cX \to S$ be a smooth projective morphism with geometrically connected fibers, and let~$X$ denote the generic fiber of~$\cX$. Even in this favorable situation, the functor~$\Pic^0(\cX/S)$ may fail to be representable by an abelian scheme.\footnote{For example, consider a family of smooth Enriques surfaces with singular or supersingular (in the sense of Bombieri--Mumford \cite[\S3]{BombieriMumford}) reduction at a point~$s \in S$ with residue characteristic~$2$; see \cite[th.~4.10]{LiedtkeLiftEnriques}.} To circumvent this issue, we consider the N\'eron model
\[
\cA \longrightarrow S \qquad \text{of the abelian variety} \qquad A := \Pic^0(X/K).
\]

\begin{lemma} \label{lemma:NeronModelPic0}
The N\'eron model~$\cA$ of~$A$ is a projective abelian scheme whose dual~$\cA^\ast$ is the N\'eron model of the abelian variety~$A^\ast = \Alb(X/K)$.
\end{lemma}

\begin{proof}
By \cite[th.~6.4.1]{NeronModels} it suffices to show that~$\cA \to S$ has abelian reduction at every point~$s \in S$. We may thus replace~$S$ with the spectrum of the discrete valuation ring~$R := \cO_{S, s}$.  Let~$\ell$ be a prime different from the residue characteristic of~$R$, and let~$\bar{K}$ be an algebraic closure of~$K$. Then we have a natural isomorphism
\[
\rH^1_{\et}(A_{\bar{K}}, \bbQ_\ell) \; \iso \; \rH^1_{\et}(X_{\bar{K}}, \bbQ_\ell)
\]
of representations of~$\Gal(\bar{K} / K)$. In particular, the~$\ell$-adic Tate module of~$A$ is unramified, as it is dual to~$\rH^1_{\et}(A_{\bar{K}}, \bbQ_\ell)$. Hence~$\cA$ has abelian reduction by the N\'eron-Ogg-Shafarevich criterion \cite[th.~7.4.5]{NeronModels}.
\end{proof}

\begin{lemma} \label{prop:NeronModelAlbaneseTorsor}
The Albanese torsor~$P := P(X/K)$ extends uniquely to a principal bundle~$\cP$ under the dual abelian scheme~$\cA^\ast$.
\end{lemma}

\begin{proof} If such an extension exists, then it is uniquely determined by the N\'eron extension property in \cref{Lemma:ModelsOfAbelianTorsors} (2). Thus by \cite[th.~7.2.1]{NeronModels}  such an extension $\cP$ can be constructed \'etale locally on $S$, so we may assume that~$\cX \to S$ has a section. Then~$P \iso A^\ast$ and in this case $\cP := \cA^\ast$ is the desired extension.
\end{proof}

\begin{definition}
The principal~$\cA^\ast$-bundle~$\cP$ is called the \emph{N\'eron model} of the Albanese torsor~$P(X/K)$. We write~$\alb_\cX \colon \cX \to \cP$ for the unique extension of the Albanese morphism~$\alb_X \colon X \to P$, whose existence is guaranteed by \cref{Lemma:ModelsOfAbelianTorsors}.
\end{definition}

\section{Proof of theorems~\ref{IntroThmIntrinsicSurfaces}, \ref{IntroThmIntrinsicGeneral} and \ref{IntroThmExtrinsic}}
\label{sec:Applications}

In this section we explain how to deduce the theorems~\ref{IntroThmIntrinsicSurfaces}, \ref{IntroThmIntrinsicGeneral}, \ref{IntroThmExtrinsic} from the main \cref{Thm:MainTheoremForFamiliesWithBigMonodromy}. Throughout this section~$K$ denotes a number field and~$\Sigma$ a finite set of primes in~$K$. We begin with some preliminaries needed in the proof.

\subsection{Parameter spaces for canonically polarized and very irregular varieties} \label{sec:ParameterSpace}Given a polynomial~$P \in \bbQ[z]$ of degree~$d$, in this section we construct a parameter space for canonically polarized varieties that embed in their Albanese variety and with Hilbert polynomial~$P$ with respect to the canonical bundle. To this end we fix an integer
\[n \ge 1\] 
with the following property: For any smooth projective variety~$X$ over a field~$k$, any ample line bundle~$L$ on~$X$ with Hilbert polynomial~$P$ and any integer~$m \ge n$, the line bundle~$L^{\otimes m}$ is very ample and~$\rH^i(X,L^{\otimes m}) = 0$ for all~$i\ge1$. The existence of such an integer is guaranteed by Matsusaka's Big Theorem. Let~$\pi \colon X \to S$ be a smooth proper morphism of Noetherian schemes and~$\cL$ a relatively ample line bundle  on~$X$ with Hilbert polynomial~$P$ fiberwise. Then, by our choice of~$n$, the~$\cO_S$-module~$\cE := \pi_\ast \cL^{\otimes n}$ is locally free of rank~$P(n)$, the morphism~$\pi^\ast \cE \to \cL^{\otimes n}$ given by adjunction is surjective and the induced morphism~$X \to \bbP(\cE^\vee)$ is a closed embedding. \medskip

By a \emph{family of canonically polarized varieties} over a Noetherian scheme~$S$ we mean a smooth proper morphism~$\pi \colon X \to S$ with geometrically connected fibers and relatively ample canonical bundle
\[ \omega_{X / S} := \det \Omega^1_{X/S}.\]
We say that such a family has Hilbert polynomial $P$ if all of its fibres have Hilbert polynomial $P$ with respect to the canonical bundle.

\begin{lemma} \label{Lemma:ParameterSpaceCanonicallyPolarized} The functor associating to a Noetherian scheme~$S$ the set of couples made of
\begin{itemize}
\item a family~$\pi \colon X \to S$ of canonically polarized varieties with Hilbert polynomial~$P$, and
\item an isomorphism of~$\cO_S$-modules $\cO_S^{P(n)} \iso \pi_\ast \omega_{X/S}^{\otimes n},$
\end{itemize}
is representable by a  quasiprojective scheme over~$\bbZ$. \end{lemma}

\begin{proof} The sought-for scheme is a principal~$\GL_{P(n)}$-bundle over an open subset of the locus of the Hilbert scheme of subschemes in~$\bbP^{P(n)-1}$ with Hilbert polynomial~$Q(z) = P(nz)$~by standard arguments similar to those in~\cite[th.~1.46]{ViehwegModuli}.
\end{proof}

\begin{definition} \label{def:FamilyCanonicallyPolarized} Let~$\pi \colon X \to S$ be a family of canonically polarized varieties over a Noetherian $\bbQ$-scheme Noetherian $S$. With notation as in \cref{sec:AlbaneseTorsor}, under these hypotheses we can consider the Albanese torsor $P(X/S)$ and the Albanese morphism
$ \alb_{X} \colon X \to P(X/S)$. We say that the family $\pi \colon X \to S$ is \emph{Albanese embedded} if 
the morphism $\alb_X$ is a closed embedding.
\end{definition}

\begin{lemma} \label{Lemma:ParameterSpaceAlbaneseEmbedded} The functor associating to a Noetherian scheme~$S$ the set of couples made of
\begin{itemize}
\item an Albanese embedded family~$\pi \colon X \to S$ of canonically polarized varieties with Hilbert polynomial~$P$, and
\item an isomorphism of~$\cO_S$-modules $\cO_S^{P(n)} \iso \pi_\ast \omega_{X/S}^{\otimes n},$
\end{itemize}
is representable by a  quasiprojective scheme over~$\bbQ$. \end{lemma}

\begin{proof} Let $M$ be the scheme obtained by taking the base change to $\bbQ$ of the scheme given by \cref{Lemma:ParameterSpaceCanonicallyPolarized}. The universal family $ \pi \colon \cX \to M$ is a family of canonically polarized varieties. Since $M$ is a $\bbQ$-scheme, we can consider the Albanese morphism $\alb_\cX \colon \cX \to P(\cX / M)$. The locus $M' \subset M$ where $\alb_X$ is a closed embedding is open. Clearly~$M'$ represents the functor in the statement, hence concludes the proof.
\end{proof}

\subsection{Models of canonically polarized varieties} \label{sec:ModelsOfVeryIrregularVarieties} To prove theorems \ref{IntroThmIntrinsicGeneral} and \ref{IntroThmExtrinsic} we need a few facts about models of canonically polarized varieties. For a smooth scheme~$\cX$ over~$\cO_{K, \Sigma}$ let~$\omega_{\cX}$ its relative canonical bundle. We are interested in the following properties:

\begin{lemma} \label{Lemma:IsomorphismsOverAlgebraicClosure}  Let~$X$ and~$Y$ be canonically polarized varieties over~$K$ with good reduction. Then, the following facts hold:
\begin{enumerate}
\item Let~$\cX$ and~$\cY$ be good models of~$X$ and~$Y$. Then, any isomorphism of~$K$-schemes~$X \to Y$ extends uniquely to an isomorphism~$\cX \to \cY$.\smallskip
\item Up to isomorphism, there is a unique good model of~$X$.\smallskip
\item Let~$\bar{K}$ be an algebraic closure of~$K$. Then, up to~$K$-isomorphism there are only finitely many canonically polarized varieties over~$K$ with good reduction that are isomorphic to~$X$ over~$\bar{K}$.
\end{enumerate}
\end{lemma}

\begin{proof} (1) follows from \cite[cor.~1]{MatsusakaMumford} because canonically polarized varieties are not ruled, (2) is a consequence of (1), and (3) is \cite[lemma~4.1]{JavanpeykarLondon}.
\end{proof}

To state the next lemma, notice that a finite \'etale cover of a canonically polarized variety is also canonically polarized, and we can ask whether as such it has good reduction.

\begin{lemma} \label{EtaleCoversCanonicallyPolarized} Let $d \ge 1$ and let $X$ be a canonically polarized variety over $K$. Then up to isomorphism there are only finitely many finite \'etale covers of~$X$ of degree $\le d$ with good reduction.
\end{lemma}

\begin{proof} By \cref{Lemma:IsomorphismsOverAlgebraicClosure} it suffices to show that, over an algebraic closure $\bar{K}$ of~$K$, there are only finitely many isomorphism classes of finite \'etale covers of $X_{\bar{K}}$ of degree~$\le d$. By embedding $\bar{K}$ into $\bbC$, this follows from the fact that the topological fundamental group of $X(\bbC)$ is finitely generated.
\end{proof}

Let $A$ be an abelian variety over $K$ with good reduction, so $A$ extends to an abelian scheme $\cA$ over $\cO_{K, \Sigma}$. Let $\cP$ be a principal $\cA$-bundle with generic fiber $P$. Let $X \subset P$ be a canonically polarized subvariety with good reduction and let $\cX$ be a good model over $\cO_{K, \Sigma}$. In this situation, the stabilizer of $X$ is defined as the subgroup $G \subset A$ preserving $X$ inside $P$. We assume that the order of $G$ is not divisible by any of the primes in $\Sigma$ so that the Zariski closure $\cG \subset \cA$ is a finite \'etale groups scheme over $\cO_{K, \Sigma}$. By \cref{Lemma:ModelsOfAbelianTorsors} the closed embedding $X \into P$ extends to a morphism $\cX \to \cP$ because $\cX$ is smooth. Such an extension is not necessarily a closed embedding, so we cannot consider the stabilizer of $\cX$. However, we have the following replacement:

\begin{lemma} \label{QuotientGoodModelByStabilizer} With the notation above, the translation action of $G$ on $X$ extends to a free action of $\cG$ on $\cX$ and the quotient $\cX / \cG$ is a good model of the canonically polarized variety $X / G$.
\end{lemma}

\begin{proof} The isomorphism $G \times X \to G \times X$, $(a, x) \mapsto (a, x + a)$ extends by \cref{Lemma:IsomorphismsOverAlgebraicClosure}~(1) to an isomorphism $\cG \times \cX \to \cG \times \cX$, yielding the wanted extension of the action. The natural morphism $\cX \to \cP$ is equivariant under the action of $\cG$. Therefore the action of $\cG$ on $\cX$ must be free, because so is the action of $\cG$ on $\cA$. The canonical bundle $\omega_{\cX}$ is ample and admits an equivariant linear action of $\cG$. Together with the fact that $\cF$ is finite \'etale, this implies that the quotient $\cY := \cX / \cG$ exists and it is smooth projective. Moreover, the quotient morphism $\cX \to \cY$ is finite \'etale, thus the canonical bundle $\omega_\cY$ is ample. It follows that $\cY$ is a good model of the canonically polarized variety $Y := X / G$.
\end{proof}

As an application we prove \cref{IntroThmIntrinsicGeneral} for very irregular varieties which are isogenous to a symmetric power of a curve inside their Albanese variety. The following argument was pointed out to us by A.\ Javanpeykar and works for any symmetric power of a curve in the sense of \cref{sec:ExtrinsicResults} (without assuming that it embeds in its Albanese variety):

\begin{proposition} \label{Prop:FinitenessSymPowersCurves} Fix integers $n \ge 1$, $g \ge 2$. Then, up to isomorphism there are only finitely many canonically polarized varieties over $K$ of dimension~$n$ with good reduction which are symmetric powers of a curve of genus $g$.
\end{proposition}

\begin{proof}
By \cref{lemma:NeronModelPic0} the Albanese variety of any variety as in the statement has dimension~$g$ and good reduction outside~$\Sigma$. By the Shafarevich conjecture for abelian varieties~\cite{FaltingsMordell}, it therefore suffices to prove the finiteness of the set~$\cF$ of isomorphism classes of varieties as in the statement whose Albanese variety is isomorphic to a fixed abelian variety~$A$ with good reduction outside~$\Sigma$. Over an algebraic closure~$\bar{K}$ of~$K$, the set~$\cC$ of isomorphism classes of smooth projective connected curves whose Jacobian is isomorphic to~$A_{\bar{K}}$ is finite. Indeed, by Torelli's theorem this follows from the fact that up to automorphisms there are only finitely many principal polarizations on $A$. Moreover, the Albanese variety of a symmetric power of a curve is isomorphic to the Jacobian of the curve. It follows that for any~$X \in \cF$, we have~$X_{\bar{K}} \iso \Sym^n C$ for some~$C \in \cC$. Since $\cC$ is finite, we conclude by \cref{Lemma:IsomorphismsOverAlgebraicClosure} (3).
\end{proof}

Let $X$ be a very irregular variety over $K$ and $P$ its Albanese torsor. Then, the Albanese morphism $X \to P$ is a closed embedding with ample normal bundle. As above, the stabilizer $G$ of~$X$ is the subgroup of $\Alb(X/K)$ preserving~$X$ inside $P$. We say that $X$ is \emph{isogenous to a symmetric power of a curve} if $X/G$ is a symmetric power of a curve; this slightly generalizes the notion in \cref{sec:ExtrinsicResults} because we do not require~$X$ to be embedded in its Albanese variety, but only in the Albanese torsor. With these preliminaries, we have the following:

\begin{corollary} \label{Cor:ShafIsogeneousSymmetricPower} Fix $\chi \in \bbZ$. Then, up to isomorphism there are only finitely many very irregular varieties over $K$ with good reduction and topological Euler characteristic~$\chi$ that are isogenous to a symmetric power of a curve.
\end{corollary}

\begin{proof} This follows from \cref{Prop:FinitenessSymPowersCurves}, \cref{QuotientGoodModelByStabilizer} and \cref{EtaleCoversCanonicallyPolarized}: Indeed a finite \'etale cover of degree $d$ of a $n$-th symmetric power of curve of genus $g$ has topological Euler characteristic \[\chi = (-1)^n d \binom{2g - 2}{n}\qquad \textup{for} \quad n \le 2g - 2,\] hence once $\chi$ is fixed there are only finitely many possibilities for $n$, $d$ and $g$ because the very irregular assumption forces $n < g$ by \cref{RootOfVeryIrregularIsVeryIrregular} below.
\end{proof}

\begin{lemma} \label{RootOfVeryIrregularIsVeryIrregular} Let $X$ be a very irregular variety over an algebraically closed field $k$ of characteristic $0$. Let $G$ be the stabilizer of $X$ in $\Alb(X)$. Then, 
\begin{enumerate}
\item the variety $Y:=X/G$ is very irregular, \smallskip 
\item the quotient map $\pi \colon X \to Y$ induces an isogeny $\Alb(X) \to \Alb(Y)$. 
\end{enumerate}
\end{lemma}

\begin{proof}
It suffices to prove (2). The tangent map of $\Alb(X)\to \Alb(Y)$ is the dual of the differential $\rd \pi  \colon \rH^0(Y, \Omega^1_Y) \to  \rH^0(X, \Omega^1_X)$. Since $X \to Y$ is a principal bundle under $G$, the linear map $\rd \pi$ is injective with image the subspace of $G$-invariant one forms 
\begin{equation} \label{Eq:InvariantOneForms}
\rH^0(X, \Omega^1_X)^G  \; \subset \; \rH^0(X, \Omega^1_X).
\end{equation} On the other hand $G \subset \Aut(X)$ is the subgroup acting trivially on the Lie algebra of $\Alb(X)$, hence on $\rH^0(X, \Omega^1_X)$. The inclusion \eqref{Eq:InvariantOneForms} is then an identity, which concludes the proof.
\end{proof}

We record a consequence to be used in the proof of \cref{IntroThmIntrinsicGeneral}:

\begin{corollary} \label{Lemma:VeryIrregularNotIsogenousToProduct}
In the situation of~\cref{RootOfVeryIrregularIsVeryIrregular}, the subvariety $X \subset \Alb(X)$ is not isogenous to a product.
\end{corollary}

\begin{proof} By \cref{RootOfVeryIrregularIsVeryIrregular} (1) the variety $Y=X/G$ cannot be a product.
\end{proof}

\subsection{Proof of \cref{IntroThmExtrinsic}} 
We go back to the notation in \cref{IntroThmExtrinsic}. Let~$A$ be an abelian variety of dimension~$g$ over~$K$. By enlarging $\Sigma$ we may suppose that it contains the primes of bad reduction of~$A$. Fix~$P \in \bbQ[z]$ of degree~$d < (g-1)/2$ and an ample line bundle~$L$ on~$A$.\medskip

\emph{Step 1.} We begin by proving the statement when the subvarieties in question are nondivisible. Since~$\Sigma$ contains the primes of bad reduction, the abelian variety~$A$ extends to an abelian scheme~$\cA$ over~$\cO_{K, \Sigma}$. The line bundle~$L$ extends to a relatively ample line bundle~$\cL$ on~$\cA$. By enlarging $\Sigma$ we may assume that the ring $\cO_{K, \Sigma}$ is a principal ideal domain. By standard Hilbert scheme arguments, the Hilbert polynomial of the subvarieties $X \subset A$ in question with respect to their canonical bundle $\omega_X$ takes only finitely many values. Similarly, there are only finitely many possibilities for the Hilbert polynomial of the graph $\Gamma \subset X \times A$ of the closed embedding $X \into A$ with respect to the ample line bundle $\omega_X \boxtimes L$. By treating each of these cases individually, we may assume that the Hilbert polynomials of $X$ and $\Gamma$ are respectively some fixed polynomial $\Phi, \Psi \in \bbQ[z]$. \medskip

As in \cref{sec:ParameterSpace}, choose an integer $n \ge 1$ such that, for any smooth projective variety~$X$ over a field~$k$, any ample line bundle~$M$ on~$X$ with Hilbert polynomial~$\Phi$ and any integer~$m \ge n$, the line bundle~$M^{\otimes m}$ is very ample and~$\rH^i(X,M^{\otimes m}) = 0$ for all~$i\ge1$. Then, the functor associating to a Noetherian $\cO_{K, \Sigma}$-scheme~$S$ the set of triples made of
\begin{itemize}
\item a family~$\pi \colon X \to S$ of canonically polarized varieties with Hilbert polynomial~$\Phi$, 
\item an isomorphism of~$\cO_S$-modules $\cO_S^{\Phi(n)} \iso \pi_\ast \omega_{X/S}^{\otimes n}$, and
\item a morphism $f \colon X \to \cA_S$ such that the fibers of $\pi$ have Hilbert polynomial $\Psi$ with respect to the ample line bundle $\omega_{X/S} \otimes f^\ast \cL$,
\end{itemize}
is representable by a  quasiprojective scheme~$\cH$ over~$\cO_{K, \Sigma}$. Indeed, this is a combination of \cref{Lemma:ParameterSpaceCanonicallyPolarized} and the standard theory of the scheme of morphisms \cite[5.6.2]{FGAExplained}.\medskip

The subvarieties $X \subset A$ that we are interested in can be seen as $\cO_{K, \Sigma}$-valued points of $\cH$, by taking the unique morphism $\cX \to \cA$ extending the closed embedding $X \into A$ and by choosing a basis of the free module $\rH^0(\cX, \omega_\cX^{\otimes n})$, where $\cX$ is a good model of $X$. Note that this last operation is possible because $\cO_{K, \Sigma}$ is principal.
 Write~$H$ for the generic fiber of~$\cH$. By \cref{Thm:SymmetricPowerLocus,,Thm:ProductLocus} the locus
\[ H' \; \subset \; H\]
where the morphism $f$ is a closed embedding and moreover the fibers of the universal family over~$H$ are nondivisible, have an ample normal bundle, are neither products nor symmetric powers of a curve, and satisfy \eqref{SkullInequality} and \eqref{Eq:NumericalConditions}, is open. By the above discussion, it suffices to show that the set
\[ H'(K) \cap \cH(\cO_{K, \Sigma}) \]
is finite up to translation by points in~$A(K)$. Let~$\cT \subset \cH$
be an integral closed subscheme of the Zariski closure of~$H'(K) \cap \cH(\cO_{K, \Sigma})$ in~$\cH$ whose generic fiber
\[ T \; \subset \; H\]  contains $T \cap H'(K) \cap \cH(\cO_{K, \Sigma})$ as a dense subset. Since~$\cT$ is closed in~$\cH$, note that we have
\[ \cT(\cO_{K, \Sigma}) = \cH(\cO_{K, \Sigma}) \cap T(K).\]
By Noetherian induction, it suffices to show that for some nonempty open subset \[ U \subset T':= T \cap H'\] the subvarieties in~$A$ parametrized by~$t \in U(K)$ all lie in the same orbit under translation by points in~$A(K)$. To do so, consider a desingularization
\[ \pi \colon  S \too T,\]
that is, a proper birational morphism with~$S$ smooth over~$K$. Enlarging~$\Sigma$ we may assume that~$\pi$ extends to a proper birational morphism~$\cS \to \cT$ with~$\cS$ flat separated over~$\cO_{K, \Sigma}$. By construction~$\cT(\cO_{K, \Sigma}) \subset T$ is Zariski dense hence so is~$\cS(\cO_{K, \Sigma}) \subset S$. Let 
\[ f \colon \cX \too A \times S\]
be the pullback of the universal family over~$H$. By construction the morphism $f$ is a closed embedding over $S' := \pi^{-1}(T')$ and we identify $\cX \times_S S'$ with its image. The key point is that the fiber~$\cX_{\bar{\eta}}$ at any geometric generic point~$\bar{\eta}$ of~$S$ is a subvariety of $A_{S, \bar{\eta}}$ which is constant up to translation. Indeed, if it were not to be the case, \cref{Thm:MainTheoremForFamiliesWithBigMonodromy} together with the Big Monodromy Criterion in \cref{subsec:bigmonodromy} would imply that~$\cS(\cO_{K, \Sigma})$ is not Zariski dense, a contradiction. Now that~$\cX_{\bar{\eta}}$ is constant up to translation, the family~$\cX$ over $S'$
is obtained by translating a fixed subvariety~$X \subset A$ along a morphism~$a \colon S' \to A$ \cite[cor. 4.8]{JKLM} (in loc.~cit.~this is stated for families over an algebraically closed field, but the argument applies in general). In particular, 
\[ \cX_s = X + a(s) \quad \textup{for all} \quad s \in S'(K).\]
The open subset~$U \subset T'$ where the birational map~$\pi \colon S' \to T'$ is an isomorphism then does the job. 
\medskip

\emph{Step 2.} We now pass to the general case. Consider the set~$\cF$ of subvarieties as in the statement of \cref{IntroThmExtrinsic}. Namely, the set~$\cF$ of smooth, geometrically connected subvarieties~$X \subset A$ with ample normal bundle, Hilbert polynomial~$P$ with respect to~$L$, good reduction, that are neither isogenous to a product nor isogenous to a symmetric power of a curve, and that satisfy \eqref{Eq:NumericalConditions} and~\eqref{SkullInequality}. Let \[H \; \subset \; \Hilb_A\] be the locus where the universal family~$\cX \to \Hilb_A$ is smooth with geometrically connected fibers and Hilbert polynomial~$P$ with respect to~$L$. The topological Euler characteristic~$\chi_\top(\cX_h)$ for $h \in H$ is locally constant and therefore takes finitely many values because~$H$ is of finite type over~$K$. For each~$h\in H$ the integer~$\chi_\top(\cX_h)$ is divisible by~$|\Stab_A(\cX_{h})|$ hence the set
\[ \cG \; := \; \{ \Stab_A(\cX_h) \mid h \in H(K)\} \]
is finite. Enlarging~$\Sigma$ we may suppose that it contains all the places above the prime divisors of~$\chi_\top(\cX_h)$ for all~$h\in H$. Then for~$G \in \cG$ the abelian variety~$A/G$ has dimension~$g$ and good reduction. Similarly, for all~$h \in H(K)$ the subvariety~$\cX_h\subset A$ has good reduction by \cref{QuotientGoodModelByStabilizer}, thus so does
\[ \cX_h/\Stab_A(\cX_h) \intoo A/\Stab_A(\cX_h).\]
Since~$\cG$ is finite it suffices to prove the statement for varieties~$X\in \cF$ whose stabilizer~$\Stab_A(X)$ equals a fixed~$G \in \cG$. In this case, pick an ample line bundle~$M$ on the abelian variety~$B:= A/G$. The line bundle~$\pi^\ast M$ on~$A$ is ample where~$\pi \colon A \to B$ is the projection. Moreover the function~$H\to \bbQ[z]$ associating to~$h\in H$ the Hilbert polynomial of~$\cX_h$ with respect to~$\pi^\ast M$ takes finitely many values. Again, by treating each polynomial separately, we may suppose that all varieties~$X \in \cF$ have the same Hilbert polynomial~$Q$ with respect to~$\pi^\ast M$. Then the subvarieties~$X/ G\into B$ are nondivisible and have Hilbert polynomial~$Q/|G|$ with respect to~$M$. 
It follows from the nondivisible case applied to~$B$ that the set
\[ \{ X/G \subset B \mid X\in \cF \}/B(K)\]
is finite. Since the group morphism~$A(K) \to B(K)$ has finite cokernel, this concludes the proof. \qed

\subsection{Proof of \cref{IntroThmIntrinsicGeneral}} We are finally in position to prove \cref{IntroThmIntrinsicGeneral}.  The polynomial~$P$ being fixed, we pick an integer~$n \ge 1$ as in \cref{sec:ParameterSpace}. Consider the parameter space~$M := M_{P, n}$ constructed in \cref{Lemma:ParameterSpaceAlbaneseEmbedded} and let
\[ \pi \colon \cX \too M\]
be the universal family over~$M$, which is an Albanese embedded family of canonically polarized varieties in the sense of \cref{def:FamilyCanonicallyPolarized}. In particular, we can consider the Albanese torsor~$P(\cX / M)$ defined in \cref{sec:AlbaneseTorsor}. The Albanese morphism
\[ \alb_{\cX} \colon \cX \too P(\cX / M)\]
is a closed embedding by construction. Replacing~$M$ by a suitable union of its connected components we may assume that any fiber~$X$ of~$\cX \to M$ satisfies the inequality~$h^0(X, \Omega^1_X) \ge 2d + 2$ as well as \eqref{SkullInequality} and \eqref{Eq:NumericalConditions} with~$A = \Alb(X)$. Let
\[ M_{\textup{irr}} \; \subset \; M\]
the open subset where the Albanese morphism has ample normal bundle. A variety~$V$ as in the statement of \cref{IntroThmIntrinsicGeneral} gives rise to a~$K$-point of~$M$ after having chosen a basis of~$\rH^0(V, \omega_V^{\otimes n})$.  
We consider the subset \[ \cF \; \subset \; M_{\textup{irr}}(K)\]
corresponding to varieties with good reduction outside~$\Sigma$ in the sense of \cref{sec:IntrinsicResults}. We need to prove that the varieties~$\cX_t$ for~$t \in \cF$ are finitely many up to isomorphism of~$K$-schemes.

\medskip

\emph{The case of trivial Albanese torsor.} We first prove \cref{IntroThmIntrinsicGeneral} with the additional hypothesis that, for the varieties~$X$ in question, the Albanese torsor~$P(X/K)$ admits a~$K$-rational point, which yields an isomorphism with the Albanese variety:
\[ P(X/K) \; \iso \; \Alb(X).\]
Via such an identification, the Albanese morphism can be seen as a morphism
\[ f \colon X \to \Alb(X)\]
which by construction is a translate of an Albanese morphism (\cref{def:TranslateAlbaneseMorphism}). Recall that by definition the Albanese torsor 
\[ N := P(\cX / M) \too M\]
parametrizes, for any~$M$-scheme~$S$, the set of morphisms~$\cX_S \to \Alb(\cX / M)_S$ that are translates of an Albanese morphism. In particular, over $N$ there is a universal translate of an Albanese morphism
\[ \cY  \; \too \; \Alb(\cY / N)\]
which is by hypothesis a closed embedding, where~$\cY := \cX_N$ is the universal family over~$N$. In other words, letting
\[ N_{\textup{irr}} \; := \; M_{\textup{irr}} \times_{M} N \quad \subset \quad N,\]
the varieties with trivial Albanese torsor appearing in the statement of \cref{IntroThmIntrinsicGeneral} can be seen as points in~$N_{\textup{irr}}(K)$ for which~$\cY_t$ has good reduction.\medskip

Varieties isogenous to symmetric powers of curves in their Albanese variety have already been treated in \cref{Cor:ShafIsogeneousSymmetricPower}, so we just need to handle varieties that are not of this kind. More precisely, let
\[ \cG \subset N_{\textup{irr}}(K)\]
be the preimage of~$\cF \subset M_{\textup{irr}}(K)$. We need to show that the varieties corresponding to points in~$\cG$ which are not isogenous to symmetric powers of curves are finitely many up to isomorphism. To do so, we reduce to working inside a fixed ambient abelian variety. The universal family~$\cY \to N$ is projective, thus the abelian scheme~$\Pic^0(\cY/N) \to N$ is equipped with a relatively ample line bundle; see  \cref{sec:AlbaneseTorsor}. We consider the relative ample line bundle~$\cL$ induced on its dual abelian scheme~$\Alb(\cY/N)$. The degree of~$\cL$ is uniformly bounded on~$N$ because it is constant on each connected component. For~$t \in \cF$ let~$\cA_t$ be the N\'eron model over~$\cO_{K, \Sigma}$ of the abelian variety~$\Alb(\cY_t)$. Since~$\cY_t$ has good reduction outside~$\Sigma$, \cref{lemma:NeronModelPic0} implies that~$\cA_t$ is an abelian scheme. The Shafarevich conjecture for polarized abelian varieties proved by Faltings~\cite{FaltingsMordell} then implies that the set
\[ \{ (\Alb(\cY_t), \cL_{t}) \mid t \in \cF \} \]
is finite up to isomorphism of polarized abelian varieties over~$K$. By treating each isomorphism class separately, we may assume
\[ (\Alb(\cY_t), \cL_{t}) \; \iso \; (A,L) \qquad \textup{for all} \quad t \in \cF\]
for a fixed polarized abelian variety~$(A,L)$ over~$K$. The Hilbert polynomial of the subvariety~$\cY \into \Alb(\cY/N)$ with respect to the polarization~$\cL$ takes finitely many values, as it is constant on connected components. Therefore, treating each of these polynomials individually, we may assume that this Hilbert polynomial is fixed. We conclude by~\cref{IntroThmExtrinsic}, which applies here because the varieties in $\cG$ embed in their Albanese varieties with ample normal bundle, are  by \cref{Lemma:VeryIrregularNotIsogenousToProduct} not isogenous to products in their Albanese variety, and are by assumption not isogenous to symmetric powers of curves.\medskip

\emph{The general case.} We now prove \cref{IntroThmIntrinsicGeneral} in general. To do so, it suffices to show that there is a finite extension~$K'$ of~$K$ such that, for all varieties~$X$ in the statement, the Albanese torsor~$\Alb(X/K)$  admits a~$K'$-rational point. Indeed, if it is the case, then we can apply the previous step over~$K'$, and with a finite set of primes in~$K'$ containing those lying above~$\Sigma$. Then, we have that the varieties under consideration are finitely many up to isomorphism of~$K'$-schemes and the conclusion follows from \cref{Lemma:IsomorphismsOverAlgebraicClosure}. To construct such an extension, by \cref{Prop:RepAlbaneseTorsor} the Albanese torsor~$p \colon P(\cX / M) \to M$ admits a relatively ample line bundle~$\cL'$. The locally free~$\cO_M$-module~$p_\ast \cL'$ has bounded rank, since~$M$ has only finitely many connected components. For~$t \in \cF$ let~$\cP_t$ denote the N\'eron model of the Albanese torsor~$P(X_t/K)$. By \cref{Lemma:ModelsOfAbelianTorsors} the ample line bundle induced by~$\cL'$ on~$P(X_t/K)$ extends to an ample line bundle~$\cL'_t$ on~$\cP_t$ with the same degree. It follows from \cref{Prop:FaltingsFinitenessTorsor} that the Albanese torsors~$\cP_t$ for~$t \in \cF$ are only finitely many up to isomorphism of~$\cO_{K, \Sigma}$-schemes. Therefore, it suffices to take a sufficiently  big finite extension~$K'$ of~$K$ over which each of them admits a~$K'$-rational point. This concludes the proof.
\qed

\subsection{Proof of \cref{IntroThmIntrinsicSurfaces}} For a smooth projective surface~$X$ the Riemann-Roch theorem together with Noether's formula imply
\[ \chi(X, \omega_X^{\otimes n}) = \tfrac{1}{2}n (n-1) c_1(X)^2 + \tfrac{1}{12} (c_1(X)^2 + c_2(X)).\]
If~$X$ is of general type, then~$c_1(X)^2 \le 3 c_2(X)$ by the Bogomolov-Miyaoka-Yau inequality. In particular once~$c_2(X)$ is fixed there are only finitely many possibilities for the Hilbert polynomial of~$X$ with respect to the canonical bundle. We then conclude by \cref{IntroThmIntrinsicGeneral}. \qed

\section{Buildings and filtrations on reductive groups}
\label{sec:buildings}

In this section we gather some generalities about reductive groups that will be used throughout the rest of the paper. Let~$k$ be a perfect field.

\subsection{Reductive groups}\label{sec:NotationReductiveGroups} For an affine algebraic group~$G$ over~$k$ let~$\rad G$ be the unipotent radical of~$G^\circ$. In this paper~$G$ is \emph{reductive} if~$\rad G$ is trivial. Note that~$G$ is not necessarily connected. When~$k$ is of characteristic~$0$, an affine algebraic group over~$k$ is reductive if and only its category of representations is semisimple. Let~$G$ be a reductive group over~$k$. Given a cocharacter~$\lambda \colon \Gm \to G$ consider the algebraic subgroup~$P_\lambda \subset G$ whose points with values in a~$k$-scheme~$S$ are those~$g \in G(S)$ such that the limit
\[ \lim_{t \to 0} \lambda(t) g \lambda(t)^{-1}\]
exists, that is, the morphism~$\lambda g \lambda^{-1} \colon \Gm \times S \to G$ extends to~$\bbA^1 \times S$. The identity component of~$P_\lambda$ is a parabolic subgroup of~$G^\circ$. Consider the subgroup~$U_\lambda$  of~$P_\lambda$ made of those~$g$ for which~$\lim_{t \to 0} \lambda(t) g \lambda(t)^{-1}$ is the neutral element of~$G$. The subgroup~$U_\lambda$ is connected by definition and coincides with the unipotent radical~$\rad P_\lambda$ of~$P_\lambda$. If~$L_\lambda$ is the centralizer of the image of~$\lambda$ in~$G$, then the projection onto~$P_\lambda / \rad P_\lambda$ induces an isomorphism 
\[ L_\lambda \stackrel{\sim}{\too} P_\lambda / \rad P_\lambda. \] See \cite[prop. 5.2]{MartinReductiveSubgroups} for these facts. When~$G$ is connected, parabolic subgroups of~$G$ are parametrized by a proper smooth~$k$-scheme~$\Par(G)$. Its Stein factorization~$\Type(G)$ is finite \'etale over~$k$ and the preimage~$\Par_t(G)$ of~$t \in \Type(G)( k)$ via the natural map~$\Par(G) \to \Type(G)$ is a geometrically connected component. A parabolic subgroup~$P \subset G$ is said of type~$t(P)= t$ if the point~$[P]\in \Par(G)( k)$ defined by~$P$ lies in~$\Par_t(G)$. 

\subsection{The rational building} 
We recall the notion of the rational (or vectorial) building  of a reductive group~$G$ over~$k$; when~$G$ is connected this can be found in~\cite[\S 2]{RousseauBuilding}.  Write~$X_\ast(G)$ for the set of cocharacters of~$G$. 

\begin{definition}
The \emph{rational building}~$\cB(G,  k)$ of~$G$ is the quotient
\[ \cB(G,  k) \; :=\; (X_\ast(G) \times ( \bbN \smallsetminus \{ 0\})) / \sim\]
where~$(\lambda, n) \sim (\lambda', n')$ if there is~$g \in P_\lambda( k)$ such that~$\lambda'^n = g \lambda^{n'} g^{-1}$. 
\end{definition}

The equivalence class of~$(\lambda, n)$ is written~$[ \tfrac{\lambda}{n} ]$. 
For~$a, b \in \bbZ$ with~$b \ge 1$ and~$x = [\tfrac{\lambda}{n}]$ write~$\tfrac{a}{b}x = [\tfrac{\lambda^a}{bn}]$. The group~$G( k)$ acts on the rational building by conjugating cocharacters. The stabilizer of a point~$x \in \cB(G,  k)$ is~$P_x( k)$ where~$P_x := P_\lambda$ and~$x = [\tfrac{\lambda}{n}]$. We write~$t(x)$ for the type of the parabolic subgroup~$P_x^\circ \subset G^\circ$. The construction of the rational building is functorial: for a morphism of algebraic groups~$f \colon G \to G'$ with~$G'$ reductive, composing with~$f$ gives rise to map \[f_\ast \colon \cB(G,  k) \too \cB(G', k).\]

\begin{lemma} The natural map~$i \colon \cB(G^\circ,  k) \to \cB(G, k)$ is bijective.
\end{lemma}

\begin{proof} Since any cocharacter takes values in~$G^\circ$ the map~$i$ is clearly surjective. To see that~$i$ is injective let~$(\lambda, n)$ and~$(\lambda', n') \in X_\ast(G)$ be couples made of a cocharacter and a positive integer such that there is~$g \in P_\lambda( k)$ with~$\lambda'^n = g \lambda^{n'} g^{-1}$. We have to prove that we can choose such an element~$g$ in the identity component of~$P_\lambda$. The image~$\bar{g}$ of~$g$ in~$L_\lambda \iso P_\lambda / \rad P_\lambda$ commutes with~$\lambda$ thus \[\lambda'^n = g \lambda^{n'} g^{-1} = g \bar{g}^{-1} \lambda^{n'} \bar{g} g^{-1}.\] This concludes the proof because~$g \bar{g}^{-1} \in \rad P_\lambda \subset P_\lambda^\circ$.
\end{proof}

The map~$f_\ast$ is injective resp. surjective if and only if $f\colon G^\circ \to G'^\circ$ has finite kernel resp. is surjective. The \emph{apartment} associated with a maximal split torus~$T$ of~$G$ is the subset~$\cB(T,  k) = X_\ast(T) \otimes_{\bbZ} \bbQ \subset \cB(G,  k)$.

\begin{example} \label{Ex:BuildingGL} Assume~$G = \GL(V)$ for a finite dimensional~$k$-vector space~$V$. For a point~$x = [\tfrac{\lambda}{n}]\in \cB(\GL(V),  k)$ let~$V  = \bigoplus_{i \in \bbZ} V_i$
be the decomposition of~$V$ in isotypical factors under the action of~$\lambda$; that is~$V_i$ is the subspace where~$\lambda(t)$ acts as~$v \mapsto t^i v$. For~$\alpha \in \bbQ$ set
\[ \rF^\alpha V = \bigoplus_{i \ge \alpha n} V_i.\]
This defines a separated exhaustive nonincreasing filtration~$F^\bullet V$ on~$V$ indexed by rational numbers. This procedure sets up a~$\GL(V)( k)$-equivariant bijection
\[ 
\cB(\GL(V),  k)  \stackrel{\sim}{\too}  
\left\{ 
\begin{array}{c}
\textup{separated exhaustive nonincreasing   filtrations } \\
\textup{of~$V$ indexed by rational numbers}
\end{array}
\right\}.
\]
Let~$G = \GO(V, \theta)$ resp.~$G= \GSp(V,\theta)$ for a nondegenerate pairing~$\theta$ on~$V$ which is symmetric resp. alternating. Then~$\cB(G, k) \subset \cB(\GL(V),  k)$ is the subset of \emph{self-dual} filtrations, i.e. such that there is~$c \in \bbQ$ such that~$(F^\alpha V)^\bot = F^{c-\alpha} V$ for each~$\alpha \in \bbQ$.
\end{example}

\begin{example} \label{ex:BuildingProductWeilRestriction} (1) For reductive groups~$G_1$,~$G_2$ over~$k$ we have
\[ \cB(G_1 \times G_2,  k)= \cB(G_1, k) \times \cB(G_2, k).\]
(2) Given a finite extension~$k'$ of~$k$ and~$G'$ a reductive group over~$k'$ we have
\[ \cB(G ,  k) = \cB(G',  k')\]
where~$G:= \Res_{ k'/ k} G'$ is the Weil restriction.
\end{example}

\subsection{Filtrations} \label{sec:HopfFiltrations} We will need a similar description for points in a general reductive group in terms of filtrations. To do this let~$G$ be an affine algebraic group over~$k$. Recall that a representation of~$G$ is a~$k$-vector space~$V$ together with the structure of a comodule over the Hopf algebra~$k[G] := \Gamma(G, \cO_G)$. Such a structure is given by a map~$\sigma \colon V \to  V \otimes  k[G]$ called \emph{comultiplication} henceforth.
When~$V$ is finite dimensional this is equivalent to the datum of a morphism of algebraic groups~$G \to \GL(V)$. Let~$\Rep(G)$ be the category of finite dimensional representations.  A \emph{filtration functor} on~$\Rep(G)$ is a collection~$F^\bullet = (F^\alpha)_{\alpha \in \bbQ}$ of exact functors~$F^\alpha \colon \Rep(G) \to \Vect_ k$ such that
\begin{itemize}
\item[(F1)] for each~$V\in \Rep(G)$ the collection~$F^\bullet V = (F^\alpha V)_{\alpha \in \bbQ}$ is a non-increasing, exhausting and separated filtration on~$V$; \label{PropertyF1}
\item[(F2)] the associated graded functor~$\gr^\bullet$ is exact; \label{PropertyF2}
\item[(F3)] \label{PropertyF3} for~$V, W \in \Rep(G)$ and~$\alpha \in \bbQ$,
\[ F^\alpha (V \otimes W) = \sum_{\alpha = \beta + \gamma} F^\beta V \otimes F^\gamma W.\]
\end{itemize}
In fact there is an integer~$n \ge 1$ for which~$\gr^\alpha = 0$ if~$\alpha n \not \in \bbZ$: this follows from the existence of a tensor generator for~$\Rep(G)$, see \cite[Tannakian categories, prop. 2.20 (b)]{DM82}. This is what is called an \emph{exact~$\otimes$-filtration} in \cite[IV.2.1]{Saa72} except for the fact that filtrations are only indexed by integers in there. A cocharacter~$\lambda \colon \Gm \to G$  determines a filtration functor by \cref{Ex:BuildingGL}. Conversely, every filtration functor with~$\gr^\alpha = 0$ if~$\alpha \not\in \bbZ$ arise in this way when~$\characteristic( k) =0$ \cite[prop. IV.2.2.2]{Saa72}. The datum of a filtration functor is encoded by a suitable filtration induced on~$k[G]$. More precisely, let 
$\mu \colon  k[G] \to  k[G] \otimes  k[G]$ be the comultiplication of the Hopf algebra~$k[G]$ and consider the composite morphism~$\nu$ defined by the diagram
\[ \begin{tikzcd}
 k[G] \otimes  k[G] \ar[dr, swap, bend right=20, "\nu"] \ar[r, "\mu \otimes \mu"] & ( k[G] \otimes  k[G]) \otimes ( k[G] \otimes  k[G]) \ar[d] \\[1.5em]
&  k[G] \otimes  k[G] \otimes  k[G], 
\end{tikzcd}
\]
where the vertical arrow on the right is given by~$ a \otimes f \otimes b \otimes g \mapsto  a \otimes b \otimes fg$.

\begin{definition} \label{Def:HopfFiltration}
A filtration~$k[G]^\bullet = ( k[G]^\alpha)_{\alpha \in \bbQ}$ on~$k[G]$ is  \emph{Hopf} if it is nonincreasing, exhausting, separated and for all~$\alpha \in \bbQ$ we have:
\begin{align*}
  k[G]^\alpha &\; = \; \mu^{-1} (  k[G] \otimes  k[G]^\alpha), \label{eq:HopfFiltrationComult}  \tag{HF1}\\
\sum_{\alpha = \beta + \gamma}  k[G]^\beta \otimes   k[G]^\gamma  &\; = \; \nu^{-1}( k[G] \otimes  k[G] \otimes  k[G]^\alpha) \tag{HF2}. \label{eq:HopfFiltrationMult} 
\end{align*}
\end{definition}

Let~$X = \Spec A$ be a variety over~$k$ with a \emph{right} action of~$G$. The action is defined by an~$k$-algebra morphism
~$\sigma \colon A \to A \otimes  k[G]~$. With~$\sigma$ as comultiplication~$A$ is a representation of~$G$ and  as such is an increasing union of finite dimensional representations. Given a filtration functor~$F^\bullet$ on~$\Rep(G)$, the~$k$-algebra~$A$ inherits a filtration~$F^\bullet A$. Let~$G$ act on itself by right multiplication; the arising representation on~$k[G]$ is called the \emph{regular representation}. The preceding discussion with~$A=  k[G]$ yields a filtration~$F^\bullet  k[G]$ on~$k[G]$ which is Hopf. Indeed the group law~$G \times G \to G$ is equivariant for the right action~$(x, y)g = (x, yg)$ of~$G$ on~$G\times G$, so the functoriality of the filtration and its compatibility with tensor product (F3) then give \eqref{eq:HopfFiltrationComult}. Similarly, \eqref{eq:HopfFiltrationMult} holds because~$G \times G \times G \to G \times G$,~$(g, h, x) \mapsto (gx, hx)$ is equivariant for right multiplication on the third factor of~$G \times G \times G$ and on both factors of~$G \times G$.

 \medskip
 
 Conversely, let~$k[G]^\bullet$ be a Hopf filtration on~$k[G]$ and~$V$ a representation with comultiplication~$\sigma \colon V \to  V \otimes  k[G]~$. For~$\alpha \in \bbQ$ set
\[ F^\alpha V := \sigma^{-1}(V \otimes  k[G]^\alpha).\]
Note that the construction of the filtration~$F^\bullet V = (F^\alpha V)_{\alpha \in \bbQ}$ is functorial with respect to~$G$-equivariant~$k$-linear maps. 

\begin{proposition}\label{Prop:FiltrationFunctorsAsHopfFiltrations} For any Hopf filtration on~$k[G]$, the above collection of functors~$F^\alpha$ is a filtration functor on~$\Rep(G)$. Moreover, these constructions are mutually inverse and set up a bijection 
\[ \{ \textup{filtration functors on~$\Rep(G)$} \} \stackrel{\sim}{\too} \{ \textup{Hopf filtrations on~$k[G]$} \}. \]
\end{proposition}

\begin{proof} 
Starting from a Hopf filtration~$k[G]^\bullet$ we have defined a functorial filtration~$F^\bullet V$ on every representation~$V$ of~$G$.
For a morphism~$\phi \colon V \to W$ of representations and~$\alpha \in \bbQ$ this filtration satisfies
\[
F^\alpha V = \phi^{-1}(F^\alpha W) \quad \textup{if~$\phi$ is injective}, \qquad
F^\alpha W = \phi(F^\alpha V) \quad \textup{if~$\phi$ is surjective}.\]
It follows that the functors~$F^\alpha \colon \Rep(G)\to \Vect_ k$ are exact. Moreover, the collection of functors~$F^\bullet = (F^\alpha)_{\alpha \in \bbQ}$ satisfies (F1) and (F2). To prove (F3) consider representations~$V$ and~$W$ with comultiplications~$\sigma_V \colon V \to V \otimes  k[G]$ and~$\sigma_W \colon W \to W \otimes  k[G]$ respectively. The comultiplication on~$V \otimes W$ is then defined by
\[ \begin{tikzcd}[row sep=0pt, column sep=40pt]
\sigma_{V \otimes W} \colon V \otimes W \ar[r, "\sigma_V \otimes \sigma_W"]&  V \otimes  k[G] \otimes  W \otimes  k[G]  \ar[r]& V \otimes W \otimes  k[G] 
\end{tikzcd}
\]
where the latter arrow is~$v \otimes f \otimes w \otimes g \mapsto v \otimes w \otimes fg$. Consider the following commutative diagram:
\[
\begin{tikzcd}[column sep=60pt]
V \otimes W \ar[r, "\sigma_{V \otimes W}"] \ar[d, "\sigma_V \otimes \sigma_W"'] & V \otimes W \otimes  k[G] \ar[d, "\sigma_V \otimes \sigma_W \otimes \id"]\\
(V \otimes  k[G]) \otimes (W \otimes  k[G]) \ar[d, "s"'] & (V \otimes  k[G]) \otimes (W \otimes  k[G]) \otimes  k[G] \ar[d, "s \otimes \id"]\\
V \otimes W \otimes  k[G] \otimes  k[G] \ar[r, "\id_{V \otimes W} \otimes \sigma_{ k[G] \otimes  k[G]}"]& V \otimes W \otimes  k[G] \otimes  k[G] \otimes  k[G]
\end{tikzcd}
\]
where~$s(v \otimes f \otimes w \otimes g) = v \otimes w \otimes f \otimes g$. Then by \eqref{eq:HopfFiltrationMult} we have:
\[ (V \otimes W)^\alpha = \sum_{\alpha = \beta + \gamma} V^\beta \otimes W^\gamma.\]
Therefore~$F^\bullet$ is a filtration functor. Property \eqref{eq:HopfFiltrationComult} implies~$F^\alpha  k[G] =  k[G]^\alpha$ for all~$\alpha \in \bbQ$,  so the given Hopf filtration is the one induced by the filtration functor that we constructed above. 

\medskip 

Conversely, let~$F^\bullet$ be any filtration functor on~$\Rep(G)$. For~$V\in \Rep(G)$, the comultiplication~$V\to V\otimes k[G]$ is~$G$-equivariant if we endow~$V\otimes k[G]$ with the action of~$G$ which is trivial on~$V$ and given by right multiplication on~$k[G]$. Since~$\mu$ is injective, the exactness of~$F^\bullet$ implies~$F^\alpha V = \sigma^{-1}(V \otimes F^\alpha  k[G])$ for all~$\alpha \in \bbQ$, hence the given filtration functor~$F^\bullet$ is the one constructed from the associated Hopf filtration as above.
\end{proof}

\begin{proposition} \label{Prop:FiltrationFunctorConnectedComponent} Let~$f \colon G \to H$ be a morphism of affine algebraic groups over~$k$. Then the natural map
\[ \{ \textup{filtration functors on~$\Rep(G)$}\}\; \too\; \{ \textup{filtration functors on~$\Rep(H)$}\}\]
is injective if~$f$ is injective, and is bijective if the morphism~$G^\circ \to H^\circ$ induced by~$f$ is an isomorphism.
\end{proposition}

\begin{proof} Suppose~$f$ is injective. Then~$f$ is a closed immersion and the induced morphism of~$k$-algebras~$\phi \colon  k[H] \to  k[G]$ is surjective. By \cref{Prop:FiltrationFunctorsAsHopfFiltrations} it suffices to show that the corresponding map 
\[ \{ \textup{Hopf filtrations on~$\Rep(G)$}\}\; \too\; \{ \textup{Hopf filtrations on~$\Rep(H)$}\}\]
is injective. The above map associates to a Hopf filtration~$k[G]^\bullet$ on~$k[G]$, the filtration~$k[H]^\bullet$ on~$k[H]$ induced by the action of right multiplication of~$G$ on~$H$. Now since~$f$ is~$G$-equivariant with respect to right multiplication of~$G$, we deduce that~$k[G]^\bullet$ is the image of~$k[H]^\bullet$ via~$\phi$. Thus~$k[G]^\bullet$ can be recovered from~$k[H]^\bullet$, showing that desired injectivity.

\medskip

For the second statement, it suffices to show that the natural map
\[ \{ \textup{filtration functors on~$\Rep(G^\circ)$}\}\; \too\; \{ \textup{filtration functors on~$\Rep(G)$}\}\]
is surjective, as we already know that it is injective. Given a filtration functor~$F^\bullet$ on~$\Rep(G)$ up to rescaling we may assume that~$\gr^\alpha= 0$ if~$\alpha \not\in \bbZ$. Then~$F^\bullet$ comes from a cocharacter. But a cocharacter has necessarily values in~$G^\circ$ as~$\Gm$ is connected. Thus~$F^\bullet$ comes from a filtration functor on~$G^\circ$.
\end{proof}

Given a filtration functor~$F^\bullet$ consider the algebraic subgroup~$P \subset G$ whose points with values in a~$k$-scheme~$S$ are those~$g \in G(S)$ such that \[g (F^\alpha V \otimes_ k \cO_S) = F^\alpha V \otimes_ k \cO_S \quad \textup{for all~$\alpha \in \bbQ$}.\]
When~$G$ is a reductive group, the subgroup~$P$ is parabolic \cite[Prop. IV.2.2.5]{Saa72} and, if~$F^\bullet$ is given by a cocharacter~$\lambda$, then~$P$ is the subgroup~$P_\lambda$ defined in the previous section. The filtration functors associated to cocharacters~$\lambda$,~$\lambda'$ then coincide if and only if~$\lambda'$ is conjugate to~$\lambda$ under~$P_\lambda$. Summing up:

\begin{corollary} \label{Cor:PointsInBuildingAsHopfFiltrations}
Suppose~$\characteristic( k) = 0$ and~$G$ reductive. Then the above constructions set up bijections 
\[\cB(G,  k) \stackrel{\sim}{\longleftrightarrow} \{ \textup{filtration functors on~$\Rep(G)$} \} \stackrel{\sim}{\longleftrightarrow} \{ \textup{Hopf filtrations on~$k[G]$} \}. \]
\end{corollary}

\subsection{Projection onto Levi quotients} \label{sec:ProjectionLeviQuotientBuildings}
Let~$G$ be a reductive group and let~$Q$ be a parabolic subgroup of its identity component $G^\circ$. For any point~$x \in \cB(G,  k)$ there is a cocharacter~$\lambda$ with values in~$Q$ such that~$x = [\tfrac{\lambda}{n}]$ for some~$n \ge 1$: Indeed~$P^\circ_x \cap Q$ contains a maximal split torus, and by \cite[prop. 11.6]{BorelTits} all maximal split tori are conjugate in~$P_x^\circ$. Let
\[
 \pi_Q\colon \quad Q \;\too\; Q^\ss \;:=\; Q/\rad Q
\]
be the projection onto the Levi quotient. Then the point
\[ \pr_Q(x) := [\tfrac{\pi_Q \circ \lambda}{n}] \in \cB(Q^\ss,  k) \]
does not depend on the chosen~$\lambda$. The map~$\pr_Q\colon \cB(G,  k) \to \cB(Q^\ss, k)$ defined in this way is~$Q( k)$-equivariant and surjective. 

\begin{example} Let~$V$ be a finite dimensional~$k$-vector space and~$G=\GL(V)$. The parabolic subgroup~$Q$ then is the stabilizer of a flag
\[ W_0 = 0 \subset W_1 \subset \cdots \subset W_n \subset W_{n +1} = V.\]
The reductive quotient~$Q^\ss$ is identified with~$\prod_{i= 0}^n \GL(\bar{W}_i)$ where~$\bar{W}_i = W_{i+1}/W_i$. If~$F^\bullet V$ is the filtration of~$V$ induced by~$x$, then the point~$\pr_Q(x)$ can be seen as the tuple of filtrations~$(F^\bullet \bar{W}_i)_{i}$ where
\[ F^\alpha \bar{W}_i = F^\alpha V \cap W_{i+1} \; /\; F^\alpha V \cap W_{i} \quad \text{for~$\alpha \in \bbQ$}.\]
\end{example}

\begin{remark} \label{Rmk:ProjectionAsHopfFiltration}
In general seeing points of the building as Hopf filtrations, the Hopf filtration on~$k[Q^\ss]$ corresponding to~$\pr_Q(x)$ is obtained as follows: First take the image under the surjection~$k[G] \to  k[Q]$ of the filtration on~$k[G]$ given by~$x$; then intersect the resulting filtration with~$k[Q^\ss] \subset  k[Q]$. 
\end{remark}

Fibers of~$\pr_Q$ are described in the following:

\begin{lemma} \label{Lemma:DescriptionFibersProjectionLeviFactor} For~$x \in \cB(G,  k)$ the map~$\rad Q( k) \to \cB(G,  k)~$,~$g \mapsto gx$ has values in~$\pr_Q^{-1}(\pr_Q(x))$ and induces a bijection
\[ f \colon (\rad Q / \rad Q \cap P_x) ( k) \stackrel{\sim}{\too} \pr_Q^{-1}(\pr_Q(x)).\]
In particular we have~$\Stab_{Q^\ss}(\pr_Q(x)) = ( Q \cap P_x) / (\rad Q \cap P_x)$.
\end{lemma}

\begin{proof} First of all notice that the group~$\rad Q \cap P_x$ is unipotent thus the vanishing of the first Galois cohomology group yields 
\[  (\rad Q / \rad Q \cap P_x) ( k) =  \rad Q ( k) / (\rad Q ( k) \cap P_x ( k)).\]
The~$Q( k)$-equivariance of~$\pr_Q$ implies that~$f$ has values in~$\pr_Q^{-1}(\pr_Q(x))$ and is injective. To prove that~$f$ is surjective let~$x' \in \pr_Q^{-1}(\pr_Q(x))$ and~$L, L' \subset Q$ Levi factors such that~$\cB(L,  k), \cB(L',  k) \subset \cB(G,  k)$ contain respectively~$x$ and~$x'$. Now~$L = u L' u^{-1}$ for some~$u \in \rad Q( k)$ because all Levi factors of~$Q$ are conjugated under~$\rad Q$. Therefore~$u x'$ belongs to~$\cB(L,  k)$ and \[\pr_Q(ux') = \pr_Q(x') = \pr_Q(x).\]
The projection~$\pi_Q \colon Q \to Q^\ss$ induces an isomorphism~$L \iso Q^\ss$ thus the composite map
\[ \cB(L,  k) \intoo \cB(G,  k) \stackrel{\pr_Q}{\too} \cB(Q^\ss,  k)\]
is a bijection. Therefore~$x = u x'$ as desired.
\end{proof}

\subsection{Positivity} \label{sec:PositivityParabolic} Let~$V$ be  a finite dimensional~$k$-vector space, and~$F^\bullet V$ the filtration on~$V$ associated to~$y \in \cB(\GL(V),  k)$ as in \cref{Ex:BuildingGL}. The \emph{weight} of~$y$ (or equivalently of the filtration~$F^\bullet V$) is the rational number
\[ \wt(y) := \frac{1}{\dim V}\langle x, \det \rangle = \sum_{\alpha \in \bbQ} \alpha \frac{\dim \gr^\alpha V}{\dim V}.\]
Let~$Q \subset G^\circ$ be a parabolic subgroup. The adjoint representation of~$Q$ restricts to a representation~$\rho \colon Q \to \GL(\Lie \rad Q)$ because the unipotent radical is a normal subgroup. The determinant of~$\rho$ is trivial on~$\rad Q$ thus it defines a character
\[\delta_Q \colon Q^\ss \too \Gm\]
 called the \emph{modular character} henceforth.\footnote{This differs from \cite{LS20} where the modular character is the inverse of the above.} Let~$x \in \cB(G,  k)$ and let~$L \subset G$ be a Levi factor of~$Q$ such that~$x \in \cB(L,  k)$. Unwinding the definition we find that the weight of the point~$\rho_\ast x \in \cB(\GL(\Lie \rad Q),  k)$ is
 \[ \wt(\rho_\ast x) = \langle \pr_Q(x), \delta_Q\rangle,\]
 where~$V = \Lie \rad Q$. Note that the right-hand side of the previous identity does not depend on the chosen~$L$.

 \begin{definition}  \label{Def:PositiveCouple}
The point~$x$ is \emph{$Q$-positive} if~$\langle \pr_Q(x), \delta_Q \rangle \ge 0$. In this case we write~$(x,Q) \succeq 0$. We will as well say that~$x$ is positive with respect to~$Q$ or that the couple~$(x, Q)$ is positive.
\end{definition}

\begin{example} Let~$G = \GL(V)$ for some finite dimensional~$k$-vector space~$V$ and let~$Q \subset G$ be the parabolic subgroup stabilizing a flag
\[ 0 = W_0 \subset W_1 \subset \cdots \subset W_r \subset W_{r + 1}=V \]
Then~$Q^\ss = \prod_{i = 0}^r \GL(\bar{W}_i)$ where~$\bar{W}_i = W_{i+1} / W_i$. Let~$x \in \cB(G,  k)$,  and let~$F^\bullet V$ be the associated filtration. Since the building is compatible with products by \cref{ex:BuildingProductWeilRestriction}, we have 
\[ 
 \pr_Q(x) \;=\; (\bar{x}_0, \dots, \bar{x}_r)
 \quad \text{with} \quad 
 \bar{x}_i \;\in\; \cB(\GL(\bar{W}_i),  k).
\]
The filtration on~$\bar{W}_i$ defined by~$\bar{x}_i$ is given by
\[F^\alpha \bar{W}_i \;=\; (W_{i+1} \cap F^\alpha V) / (W_i \cap F^\alpha V)
\quad \text{for} \quad \alpha \in \bbQ. \]
With this notation the point~$x$ is~$Q$-positive if and only if 
\[ \sum_{i < j} \dim \bar{W}_i \dim \bar{W}_j (\wt(F^\bullet \bar{W}_i) - \wt(F^\bullet \bar{W}_j)) \ge 0.\]
Indeed, as a vector space 
\[
 \Lie\rad Q \;=\; \bigoplus_{i<j} \Hom(\bar{W}_j, \bar{W}_i), 
\]
and hence
\[
 \det(\Lie\rad Q) \;=\; \bigotimes_{i<j} \left( \det(\bar{W}_j^\vee) \right)^{\otimes \dim \bar{W}_i} \otimes 
 \left( \det(\bar{W}_i)\right)^{\otimes \dim \bar{W}_j}.
\]
If we denote by~$\delta_i\colon Q^\ss \to \bbG_m$ the character defined by the action on~$\det(\bar{W}_i)$, it follows that
\[
 \delta_Q \;=\; \sum_{i<j} 
 \left( \dim \bar{W}_j \cdot \delta_i - \dim \bar{W}_i \cdot \delta_j
 \right)
\]
where we use additive notation for characters in the group~$\Hom(Q^\ss, \bbG_m)$.  Here we have~$\langle \pr_Q(x), \delta_i \rangle = \dim \bar{W}_i \cdot \wt(F^\bullet \bar{W}_i)$ by the definition of the weight of a filtration, hence 
\[
 \langle \pr_Q(x), \delta_Q \rangle \;=\; 
 \sum_{i<j} 
 \dim \bar{W}_i \dim \bar{W}_j 
 \left( 
 \wt(F^\bullet \bar{W}_i) - \wt(F^\bullet \bar{W}_j)
 \right) 
\]
and the claimed characterisation of~$Q$-positivity follows.
\end{example}

\begin{example} \label{Ex:PositiveCocharacterProduct} If~$G$ is the product of reductive groups~$G_i$ for~$i \in I$, then the tuple~$x= (x_i)_{i\in I}$ with~$x_i\in \cB(G_i,  k)$ and~$Q = \prod_{i \in I} Q_i$ with parabolic subgroups~$Q_i \subset G_i^\circ$. In this case
\[ \langle \pr_Q(x), \delta_Q \rangle \;=\; \sum_{i \in I}\langle \pr_{Q_i}(x), \delta_{Q_i} \rangle.\]
In particular, if~$x$ is~$Q$-positive, then~$x_i$ is~$Q_i$-positive for at least one~$i\in I$. 
\end{example}

\begin{example} \label{Ex:PositiveCocharacterWeilRestriction} 
If~$G = \Res_{ k'/ k} G'$ is the Weil restriction of a reductive group~$G'$ over a finite extension~$k'$ of~$k$, then we have~$Q = \Res_{ k'/ k} Q'$ for some parabolic subgroup~$Q' \subset G'^\circ$. The identification~$\cB(G,  k) = \cB(G',  k')$ then leads to 
\[ \langle \pr_{Q}(x), \delta_{Q} \rangle \;=\; [ k' :  k]\, \langle \pr_{Q'}(x), \delta_{Q'} \rangle.\]
In particular we see that~$x$ is~$Q$-positive if and only if it is~$Q'$-positive.
\end{example}

\subsection{Full subgroups} 
Let~$V$ be a finite dimensional~$k$-vector space and~$G\subset \GL(V)$ a reductive subgroup. Any point~$x\in \cB(G,  k)$ defines a filtration on~$V$ indexed by rational numbers; we denote the filtration steps by~$V_x^\alpha$ for~$\alpha \in \bbQ$ and the associated graded by 
\[
 \gr V_x \;=\; \bigoplus_{\alpha\in \bbQ} V_x^{(\alpha)}
 \quad \text{where} \quad 
 V_x^{(\alpha)} 
 \;=\; V_x^\alpha/V_x^{<\alpha}.
\]
By construction the filtration is preserved by the stabilizer~$Q_x \subset G$ of~$x$. Recall that this stabilizer may be disconnected; its identity component is the parabolic subgroup~$Q_x^\circ = Q_x\cap G^\circ$ of~$G^\circ$~\cite[prop. 5.2]{MartinReductiveSubgroups}. Let~$\rad Q_x$ be the unipotent radical of this identity component; the associated Levi quotient~$Q_x^\ss = Q_x/\rad Q_x$ faithfully acts on~$\gr V_x$ in a way preserving each graded piece, i.e.~we have a natural embedding
\[
 \epsilon_x\colon \quad Q_x^\ss \;\intoo\; \prod_{\alpha \in \bbQ} \GL(V_x^{(\alpha)}).
\]
For a subset~$I \subset \bbQ$ we consider the graded vector space~$\gr_I V_x := \bigoplus_{\alpha \in I} V_x^{(\alpha)}$ and denote by
\[
 \epsilon_{x, I} \colon 
 \quad Q_x^\ss \;\too\; 
 \gr_I \GL(V_x) \;:=\; 
 \prod_{\alpha \in I} \GL(V_x^{(\alpha)}),
\]
the composite of the morphism~$\epsilon_x$ with the projection to the factors that are indexed by the subset~$I$. For~$x, y \in \cB(G,  k)$, let
\[
 \Tran_G(x,y) \subset G
\]
be the closed subvariety of all~$g\in G$ with~$g x= y$. The stabilizer~$Q_x$ acts freely and transitively on~$\Tran_G(x,y)$ by right multiplication. For~$I\subset \bbQ$ we denote by
\[  \gr_I \Tran_{\GL(V)}(x,y) \subset \Iso(\gr_I V_x, \gr_I V_y )\]
the closed subvariety of all graded linear isomorphisms~$\varphi\colon \gr_I V_x \to \gr_I V_y$ such that 
\[ \varphi (\im \epsilon_{x,I}) \varphi^{-1} = \im \epsilon_{y,I}.\]

\begin{lemma} \label{transporter-as-torsor} 
The natural morphism 
\[
f\colon \quad 
\Tran_{G}(x,y) \too \gr_I \Tran_{\GL(V)}(x,y), \quad g \longmapsto \gr_I g.
\]
is invariant under the action of the radical~$\rad Q_x$ on the source. 
If~$\epsilon_{x,I}$ and~$\epsilon_{y,I}$ are injective, then the induced morphism on the quotient is a closed immersion
$$ \bar{f} \colon \quad \Tran_{G}(x,y)/\rad Q_x \;\intoo\; \gr_I \Tran_{\GL(V)}(x,y).~$$
\end{lemma}

\begin{proof} 
Any~$g\in \Tran_G(x,y)$ induces an element~$\gr_I g \in \gr_I \Tran_{\GL(V)}(x,y)$, and the latter does not change if we replace~$g$ with a conjugate by an element of~$\rad Q_x$. Thus we get morphisms~$f$ and~$\bar{f}$ as claimed. Now suppose~$\epsilon_{x,I}$ and~$\epsilon_{y,I}$ are injective. To show that in this case~$\bar{f}$ is a closed immersion, it suffices to show~$\bar{f}$ is injective (since the source and target are principal homogenous spaces and algebraic subgroups are always closed). For any point~$g\in \Tran_G(x,y)$ its adjoint action gives a morphism~$\Ad(g)\colon Q_x \to Q_y$ such that the following diagram commutes:
\[
\begin{tikzcd}
Q_x \ar[r] \ar[d, swap, "\Ad(g)"] 
& Q_x^\ss \ar[r, hook, "{\epsilon_{x,I}}"] \ar[d]
& \gr_I \GL(V_x) \ar[d, "\Ad(\gr_I g)"]
\\
Q_y \ar[r] 
& Q_y^\ss \ar[r,hook,  "{\epsilon_{y,I}}"] 
& \gr_I \GL(V_y)
\end{tikzcd}
\]
If~$g,h\in \Tran_G(x,y)$ satisfy~$\gr_I g = \gr_I h$, then~$h^{-1} g \in \rad Q_x$ because~$\epsilon_{x,I}$ and~$\epsilon_{y,I}$ are injective. Therefore~$\bar{f}$ is injective.
\end{proof} 

In particular, if the morphism~$\Tran_{G}(x,y) \to \gr_I \Tran_{\GL(V)}(x,y)$  is surjective, then it is a principal bundle under~$\rad Q_x$ and therefore surjective on~$k$-points by vanishing of the first Galois cohomology group of smooth connected unipotent groups over perfect fields.

\begin{definition} \label{full-subgroup} 
If~$k$ is algebraically closed, the subgroup~$G\subset \GL(V)$ is said \emph{full} if for every~$x, y\in \cB(G,  k)$ that are conjugate under~$\GL(V)( k)$, there exists a subset~$I\subset \bbQ$ with the property that~$\epsilon_{x,I}$ and~$\epsilon_{y,I}$ are embeddings and the natural morphism
\[ \Tran_G(x, y) \too \gr_I \Tran_{\GL(V)}(x, y)\]
is surjective. If~$k$ is arbitrary, we say that the subgroup~$G\subset \GL(V)$ is \emph{full} if it becomes a full subgroup after extending scalars to an algebraic closure of~$k$.
\end{definition}

As we will see below in \cref{Ex:SymplecticOrthogonalsGroupsAreFull} it is necessary to have~$I \neq \bbQ$ in order to allow orthogonal and symplectic groups to be full.

\begin{example} 
Taking~$x=y\in \cB(G,  k)$ to be the trivial cocharacter and looking at~the action scalar matrices, one sees that~$\det\colon \GL(V) \to \bbG_m$ must restrict to a surjective character on every full subgroup. In particular~$G=\SL(V)\subset \GL(V)$ is not a full subgroup, although~$G=\GL(V)$ is full.
\end{example} 

\begin{example} \label{Ex:SymplecticOrthogonalsGroupsAreFull}
Let~$\theta$ be a nondegenerate symmetric or alternating bilinear form on~$V$, and let
\[
 G \;=\; 
 \begin{cases} 
 \GSp(V, \theta) & \text{if~$\theta$ is alternating}, \\
 \GO(V, \theta) & \text{if~$\theta$ is symmetric}.
 \end{cases} 
\]
Then~$G\subset \GL(V)$ is a full subgroup. Indeed, by definition we may assume that~$k$ is algebraically closed. Since any parabolic subgroup of~$G^\circ$ is the stabilizer of an isotropic flag, up to relabelling indices the filtration induced by any~$x\in \cB(G,  k)$ has the form
\[
 0 = V_x^0 \subset V_x^1 \subset V_x^2 \subset \cdots \subset V_x^r \subset (V_x^r)^\perp \subset \cdots \subset (V_x^1)^\perp \subset (V_x^0)^\perp = V
\]
with isotropic subspaces~$V^i_x \subset V$ and~$r = r(x)$ depending on~$x$. Here we allow that~$(V_x^r)^\perp = V_x^r$ but all the other inclusions in the above filtration are assumed to be strict. On the quotient~$U_x:=(V_x^r)^\perp /V_x^r$ the bilinear form~$\theta$ induces a nondegenerate bilinear form~$\theta_x$ which is symmetric or alternating if~$\theta$ is so, and we put
\[
 H_x \;:=\; 
 \begin{cases} 
 \GSp(U_x, \theta_x) & \text{if~$\theta$ is alternating}, \\
 \GO(U_x, \theta_x) & \text{if~$\theta$ is symmetric}.
 \end{cases} 
\]
Put
$W_x^i := V_x^i/V_x^{i-1}$ for~$i=1,\dots, r$.
Then 
\[
 \gr V_x \;=\; W_x^1 \oplus \cdots \oplus W_x^r \oplus U_x \oplus (W_x^r)^\vee \oplus \cdots \oplus (W_x^1)^\vee
\]
where by construction each direct summand sits in a different graded piece, and the Levi quotient
\[
 Q_x^\ss \;=\; \GL(W_x^1) \times \cdots \times \GL(W_x^r) \times H_x
\]
acts on~$\gr V_x$ in the natural way preserving the grading. If we choose~$I\subset \bbQ$ such that
\[\gr_I V_x \;=\; W_x^1 \oplus \cdots \oplus W_x^r \oplus U_x\]
is the direct sum of the first half of the graded pieces including the middle piece, then 
\[
 \epsilon_{x,I}\colon \quad Q_x^\ss \;\hookrightarrow\; \gr_I \GL(V_x) \;=\; 
 \GL(W_x^1) \times \cdots \times \GL(W_x^r) \times \GL(U_x)
\]
is an embedding because the action on the remaining graded pieces is determined by duality. Now let~$x,y\in \cB(G,  k)$ be two points conjugate under~$\GL(V)( k)$. Then in particular the associated filtrations have the same length, so they are of the above shape with~$r=r(x)=r(y)$. Then by definition any point of~$\gr_I \Tran_{\GL(V)}(x,y)$ is a direct sum~$\bar{g} = \bar{g}_1 \oplus \cdots \oplus \bar{g}_r \oplus \bar{h}$ of isomorphisms
\[ 
 \bar{g}_i\colon \quad W_x^i 
 \;\stackrel{\sim}{\longrightarrow}\; W_y^i
  \quad \text{and} \quad 
  \bar{h}\colon \quad U_x \;\stackrel{\sim}{\longrightarrow}\; U_y
  \quad 
  \text{with} 
  \quad 
  \bar{h} H_x \bar{h}^{-1} = H_y.
\]
Now by Witt's theorem  \cite[th.~III.3.9]{ArtinGeometricAlgebra}, there exists an element~$g\in G( k)$ such that 
\begin{itemize} 
\item $g(V_x^i) \;=\; V_y^i$ and hence also~$g((V_x^i)^\perp) = (V_y^i)^\perp$ for~$i=1,\dots, r$, and \smallskip
\item $g$ induces the given isomorphisms~$\bar{g}_1, \dots, \bar{g}_r,\bar{h}$ between graded pieces. \smallskip
\end{itemize} 
Then~$g\in \Tran_G(x,y)$ lifts the element~$\bar{g} \in \gr_I \Tran_{\GL(V)}(x,y)$ as required.\medskip

Note that it is essential that we used only half of the graded pieces. To make it simple suppose~$x = y$ so we can drop the subscript~$x$ in the notation above. Assume moreover that the filtration associated with~$x$ is made of one single Lagrangian subspace~$V^1 = W$. In this case
\[ \epsilon \colon Q^\ss = \GL(W) \intoo \GL(W) \times \GL(W^\vee).\]
The normalizer of~$Q^\ss$ in~$\GL(W) \times \GL(W^\vee)$ is strictly larger than~$Q^\ss$ as it contains the center~$\Gm \times \Gm$. Composing~$\epsilon$ with the projection onto one of the two factors of~$\GL(W) \times \GL(W^\vee)$ permits to get rid of this problem.
\end{example}

\section{Realizations with extra structures} \label{sec:Realizations}

In this section we introduce the \'etale and de Rham realizations to which the Lawrence-Venkatesh method will be applied, that we call respectively \emph{\'etale data} and \emph{de Rham data}. A particular accent is put on the extra symmetries that these data satisfy, as they will be crucial for the method to work. We formulate Faltings' finiteness theorem in the context of semisimple pure \'etale data and discuss the effect of semisimplification on de Rham data. Let~$k$ be a field of characteristic~$0$ and~$p$ a prime number.

\subsection{Semilinear automorphisms and Weil restriction} Let~$Z$ be a finite \'etale scheme over~$k$. The~$k$-algebra~$R = \Gamma(Z, \cO_Z)$ is the product for~$z \in Z$ of finite extensions~$k_z$ of~$k$. An~$\cO_Z$-module~$V$ on~$Z$ is determined by its fibers, thus it is equivalent to the datum for each~$z \in Z$ of a~$k_z$-vector space~$V_z$; in particular it is always locally free. We say that a~$k$-linear endomorphism~$f$ of~$V$ is \emph{semilinear} if there is an automorphism~$\sigma$ of the~$k$-algebra~$R$ such that~$f(\lambda x) = \sigma(\lambda) f(x)$ for any~$x \in V$ and~$\lambda \in R$. In general~$\sigma$ is not unique. Nonetheless when~$f$ is bijective there is a natural choice for~$\sigma$: on the support of~$V$ it is uniquely determined by~$f$ and outside it can be taken to be the identity. When~$V$ is a vector bundle on~$Z$, semilinear automorphisms can be seen as~$k$-points of the algebraic group~$\GL^\s_{Z/ k}(V)$ over~$k$ constructed as follows. First we have an inclusion 
\[ \Res_{Z/ k} \GL(V) \subset \GL(V_0)\]
where~$V_0$ is~$V$ seen merely as~$k$-vector space. Then~$\GL^\s_{Z/ k}(V)$ is defined as the normalizer of~$\Res_{Z/ k} \GL(V)$ in~$\GL(V_0)$. We have an exact sequence
\begin{equation} \label{SESSemilinearGL} 1\too \Res_{Z/ k} \GL(V) \too \GL^\s_{Z/ k}(V) \too \Aut_ k(Z) \end{equation}
of algebraic groups over~$k$. Conjugation by a semilinear automorphism~$f$ of~$V$  induces an automorphism of~$\GL(V)$ as a~$k$-scheme such that the following diagram commutes:
\[
\begin{tikzcd}[column sep =45pt]
\GL(V) \ar[r,"g \mapsto fg f ^{-1}"] \ar[d] & \GL(V) \ar[d]\\
Z \ar[r, "\Spec \sigma"]& Z
\end{tikzcd}
\]
where~$\sigma$ is the associated automorphism of~$R$. We say that~$f$  normalizes a closed subscheme~$X \subset \GL(V)$ if~$f X f^{-1} = X$.
The datum of a~$Z$-scheme~$X$ is equivalent to the datum of a~$k_z$-scheme~$X_z$  for each~$z \in Z$; in particular all~$Z$-schemes are flat. In what follows we will adopt freely both points of view. For this reason we can perform on group~$Z$-schemes of finite type the usual operations available on algebraic groups (quotients, normalizers, etc.). Moreover we have the following relation between Weil restrictions 
\[ \Res_{Z/ k} X = \prod_{z \in Z}\Res_{ k_z/ k} X_z \]
as soon as they exists (for example if~$X$ is affine).  We profit of the setup to prove the following result that will be repeatedly used in this section:

\begin{lemma} \label{lemma:NormalizerAndWeilRestriction} Let~$V$ be a vector bundle over~$Z$,~$G \subset \GL(V)$ a subgroup~$Z$-scheme and~$N$ its normalizer. Then~$\Res_{Z/k}N \subset \GL^\s_{Z/k}(V)$ is of finite index in the normalizer of~$\Res_{Z/k} G$.
\end{lemma}

\begin{proof} For a vector bundle~$E$ on~$Z$ write~$E_0$ when seen as a~$k$-vector space, so \[\Lie \Res_{Z/k} H = (\Lie H)_0\] for a smooth group~$Z$-scheme~$H$. In order to prove the statement it suffices to show that~$\Res_{Z/k}N$ and the normalizer~$N'$ of~$\Res_{Z/k} G \subset \GL^\s_{Z/k}(V)$ have same Lie algebra. For, we have the following chain of identities:
\begin{align*}
\Lie \Res_{Z/k}N
&= (\Lie N)_0  \\
&= \{ X \in \Lie \GL(V) \mid [X, \Lie G] \subset \Lie G \}_0\\
&= \{ X \in \Lie \GL(V)_0 \mid [X, (\Lie G)_0] \subset (\Lie G)_0 \} \\
&= \{ X \in \Lie \GL(V)_0 \mid [X, \Lie \Res_{Z/k} G] \subset \Lie \Res_{Z/k} G \}   \\
&= \Lie N'.
\end{align*}
 This concludes the proof.
\end{proof}

We say that a group~$Z$-scheme~$G$ is \emph{reductive} if~$G_z$ is a reductive affine algebraic group over~$k_z$ for each~$z \in Z$. Note that this differs from the notion in \cite{SGA3} where reductive group schemes are assumed to have connected fibers. Let~$G$ be a reductive group~$Z$-scheme. We write~$G^\circ$ for the group subscheme of~$G$ whose fiber at~$z$ is~$G_z^\circ$. We also define
\[ \cB(G,Z) := \cB(\Res_{Z/ k} G,  k) = \prod_{z \in Z}\cB(\Res_{ k_z/ k} G_z,  k) = \prod_{z \in Z}\cB(G_z,  k_z), \]
where the equalities come from \cref{ex:BuildingProductWeilRestriction}. For a point~$x= (x_z)_{z \in Z}$ in~$\cB(G,Z)$ we write~$\Stab_G(x)$ for the~$Z$-scheme whose fiber at~$z$ is~$\Stab_{G_z}(x_z)$. This formalism will be applied in two contexts below.
\subsection{\'Etale data}
Take~$k = \bbQ_p$. Let~$Z$ be a finite \'etale scheme over~$\bbQ_p$ and~$\Gamma$ a topological group acting continuously on~$Z$ via~$\bbQ_p$-scheme isomorphisms. By a~\emph{semilinear representation} of~$\Gamma$ on~$Z$ we mean a couple~$(V, \rho)$ where~$V$ is a vector bundle on~$Z$ and~$\rho \colon \Gamma \to \GL^\s_{Z/ \bbQ_p}(V)$ is a continuous group morphism such that the following diagram is commutative:
\[
\begin{tikzcd}
\Gamma \ar[d,equal] \ar[r, "\rho"] & \GL_{Z/\bbQ_p}^\s(V) \ar[d]\\
\Gamma \ar[r] & \Aut(Z/ \bbQ_p)
\end{tikzcd}
\]
Geometrically the commutativity of the above diagram means that the action of~$\Gamma$ on~$V$ is equivariant with respect the action on~$Z$. A closed subscheme~$X \subset \GL(V)$ is \emph{normalized} by~$\Gamma$ if~$\rho(\gamma)$ normalizes~$X$ for each~$\gamma \in \Gamma$. 

\begin{definition}
A \emph{$\Gamma$-\'etale datum on~$Z$} is a triple~$D = (V,G,\rho)$ where
\begin{itemize}
\item $(V, \rho)$ is a semilinear representation of~$\Gamma$ on~$Z$,\smallskip
\item $G \subset \GL(V)$ is a reductive group~$Z$-scheme normalized by~$\Gamma$.
\end{itemize}
If the scheme~$Z$ is clear from the context we will call~$D$ a~$\Gamma$-\'etale datum. The~$\Gamma$-\'etale datum~$D$ is \emph{semisimple} if the Zariski closure of~$\im(\rho \colon \Gamma \to \GL(V_0))$ is a reductive subgroup, where~$V_0$ stands for~$V$ as a~$\bbQ_p$-vector space.

\medskip

A \emph{filtered} resp.~\emph{graded~$\Gamma$-\'etale datum} is a~$\Gamma$-\'etale datum~$D = (V,G,\rho)$ together with a filtration resp.~a grading where
\begin{itemize}
\item a \emph{filtration} is a~$\Gamma$-fixed point~$x \in \cB(G,Z)$;
\item a \emph{grading} is a structure of graded vector space~$V_z = \bigoplus_{\alpha \in \bbQ} V_z^{(\alpha)}$ on~$V_z$  for each~$z \in Z$  such that~$G_z \subset \prod_{\alpha \in \bbQ}\GL(V_z^{(\alpha)})$ and~$\Gamma$ preserves the subspace
\[ \sum_{\gamma \in \Gamma} V_{\gamma z}^{(\alpha)} 
\quad \text{for each} \quad \alpha \;\in\; \bbQ. \] 
We denote such a graded~$\Gamma$-\'etale datum by~$(V^{(\bullet)}, G, \rho)$. 
\end{itemize}
\end{definition}

An \emph{isomorphism} between~$\Gamma$-\'etale data~$(V, G,\rho)$ and~$(V', G', \rho')$ is a~$\Gamma$-equivariant isomorphism of vector bundles
\[f \colon \quad V \;\stackrel{\sim}{\longrightarrow} \; V'
\quad \text{with} \quad G' = fGf^{-1}.\]
By an isomorphism of filtered~$\Gamma$-\'etale data we mean an isomorphism of~$\Gamma$-\'etale data which preserves the corresponding points in the building. An isomorphism between graded~$\Gamma$-\'etale data~$(V^{(\bullet)}, G, \rho)$ and~$(V'^{(\bullet)}, G', \rho')$ is an isomorphism between the underlying~$\Gamma$-\'etale data that preserves the grading in the sense that~$f(V_z^{(\alpha)}) = V'^{(\alpha)}_z$ for all~$\alpha \in \bbQ$ and all~$z\in Z$.
The \emph{pull-back} of~$\Gamma$-\'etale data along a~$\Gamma$-equivariant finite \'etale morphism~$Z' \to Z$ is defined in the obvious way. 

\begin{example} \label{Ex:LocSysAbVar} Let~$A$ be an abelian variety over a field~$K$ of characteristic~$0$. Given an integer~$r \ge 1$ consider the~$p$-adic sheaf
\[ E := [r]_\ast \bbQ_{p, A}\]
on~$A$ where~$[r]$ is the multiplication by~$r$ on~$A$ and~$\bbQ_{p, A}$ the trivial~$p$-adic sheaf on~$A$. To decompose~$E$ let~$\bar{K}$ be an algebraic closure of~$K$ and~$A[r](\bar{K})$ the~$r$-torsion subgroup of~$A(\bar{K})$. We identify~$A[r](\bar{K})$ with a constant group scheme over~$\bbQ_p$ and consider its Cartier dual
\[ Z := A[r](\bar{K})^\ast = \Hom(A[r](\bar{K}), \bbG_{m}).\]
The Galois group~$\Gamma = \Gal(\bar{K}/K)$ acts naturally on~$Z$ via its action on~$A[r](\bar{K})$. The points of~$Z$ with values in a~$\bbQ_p$-algebra~$B$ are morphisms~$A[r](\bar{K}) \to B^\times$ of  (abstract) groups. For~$z \in Z$ let~$k_z$ be the residue field at~$z$ and 
\[ \chi_z \colon A[r](\bar{K}) \too k_z^\times\]
the corresponding character. By precomposing~$\chi_z$ with the canonical projection
\[ \pi_1^\et(A_{\bar{K}},0) = \projlim_{n\in\bbN} A[n](\bar{K}) \too A[r](\bar{K})\]
we obtain a character of geometric \'etale fundamental group of~$A_{\bar{K}}$ hence a~$p$-adic rank~$1$ local system on~$A_{\bar{K}}$. Then
\[ E_{\rvert A_{\bar{K}}} = \bigoplus_{z \in Z} L_z.\]
Let~$f \colon X \to A$ be a morphism of varieties over~$K$ with~$X$ smooth of dimension~$d$. Then the~$\bbQ_p$-vector space 
\[ V := \rH_\et^d(X_{\bar{K}}, f^\ast E) = \bigoplus_{z \in Z} \rH_\et^d(X_{\bar{K}}, f^\ast L_z) 
\]
is finite dimensional, inherits a natural structure of vector bundle over~$Z$ and the natural~$\bbQ_p$-linear action of~$\Gamma$ on~$V$ is semilinear. For each~$z \in Z$ set
\[
G_z = 
\begin{cases}
\GL(V_z) & \textup{if~$X$ is not symmetric up to translation,} \\
\GO(V_z, \theta_z) & \textup{if~$X$ is symmetric up to translation and~$d$ is even,} \\
\GSp(V_z, \theta_z) & \textup{if~$X$ is symmetric up to translation and~$d$ is odd}.
\end{cases}
\]
Here \emph{symmetric up to a translation}  means that there is an automorphism~$\iota$ of~$X$ and a point~$a \in A(K)$ such that~$f(\iota(x)) = a-f(x)$ for all~$x \in X$. Poincar\'e duality then induces a perfect pairing~$\theta_z \colon V_z \otimes V_z \to L_{z, a}$. Since the action of~$\Gamma$ preserves the Poincar\'e pairing we conclude that 
\[ D = (V, G, \rho)\]
is a~$\Gamma$-\'etale datum over~$Z$, where~$\rho$ is the representation defining the action of~$\Gamma$.
\end{example}

\begin{definition} 
For any filtered~$\Gamma$-\'etale datum~$D=(V,G,\rho,x)$, we define the associated graded 
\[
 \gr D \;=\; (\bar{V}^{(\bullet)}, \bar{G}, \bar{\rho})
\]
as follows: For any~$z\in Z$, let~$V_z^\bullet$ be the filtration induced by~$x_z$. Let~$Q_z \subset~G_z$ be the stabilizer of~$x_z$, and consider\smallskip
\[ \bar{G}_z := Q_z / \rad Q_z \intoo \prod_{\alpha} \GL(\bar{V}_z^{(\alpha)}) 
\quad \text{where} \quad 
\bar{V}_z^{(\alpha)} = \gr^{\alpha} V_z.
\]
Since~$x \in \cB(G,Z)$ is fixed under~$\Gamma$, the subspace~$\sum_{\gamma \in \Gamma} V_{\gamma z}^{\alpha}$ is stable under~$\Gamma$ for each~$\alpha \in \bbQ$. Thus~$\sum_{\gamma \in \Gamma} V_{\gamma z}^{(\alpha)}$ inherits a natural linear action of~$\Gamma$. The induced representation~$\bar{\rho} \colon \Gamma \to \GL^\s_{Z/\bbQ_p}(\bar{V})$
is semilinear, where \[\bar{V} = \bigoplus_{z \in Z} \gr V_z.\]
\end{definition}

\subsection{Finiteness} Let~$K$ be a number field and~$\bar{K}$ an algebraic closure of~$K$. Here we consider \'etale data in the above sense with~$\Gamma = \Gal(\bar{K}/K)$. For a place~$v$ of~$K$ let~$\bbF_v$ the residue field at~$v$. Fix a finite subset~$\Sigma$ of places of~$K$.

\begin{definition}
Let~$V_0$ be a finite dimensional~$\bbQ_p$-vector space and~$w \in \bbZ$. A continuous representation~$\rho \colon \Gamma \to \GL(V_0)$ is said to be \emph{pure of weight~$w$} (with respect to~$\Sigma$) if it is unramified outside~$\Sigma \cup \{ p \}$ and for each prime~$v$ not in~$\Sigma \cup \{ p \}$ the characteristic polynomial of any Frobenius element at~$v$ has integral coefficients and roots of absolute value~$|\bbF_v|^{-w/2}$. A~$\Gamma$-\'etale datum~$(V, G, \rho)$ is \emph{pure of weight~$w$} if so is the representation~$\Gamma \to \GL(V_0)$ where~$V_0$ is~$V$ seen a~$\bbQ_p$-vector space.
\end{definition}

Let~$W$ be a vector bundle on~$Z$ and~$H \subset \GL(W)$ a reductive group~$Z$-scheme. The first result of this section is the following form of Faltings' finiteness of pure semisimple Galois representations:

\begin{proposition} \label{lemma:FaltingsFiniteness}  For~$w \in \bbZ$, there are  up to isomorphism only finitely many~$\Gamma$-\'etale data with~$(G, V) \iso (H, W)$ that are semisimple and pure of weight~$w$.
\end{proposition}

\begin{proof} We correct an inaccuracy in the proof of \cite[lemma 2.6]{LV}, where the reduction to~\cite[th. 7.1]{Richardson} is wrong. We may assume~$(G,V) = (H,W)$. By \cite[th.~3.1]{DeligneBourbakiFaltings}  up to isomorphism there are  only finitely many continuous representations~$\Gamma \to \GL(W_0)$ that are semisimple and pure of weight~$w$ with respect to~$\Sigma$, where~$W_0$ stands for~$W$ seen as~$\bbQ_p$-vector space. Let~$\rho$ be such a representation. Let~$N$ be the normalizer of the subgroup~$H \subset \GL(W)$. It suffices to check that there are only finitely many~$N(Z)$-conjugacy classes in the isomorphism class of~$\rho$.  The image of~$\rho$ is contained in the normalizer~$N'$ of~$\Res_{Z/\bbQ_p} H \subset \GL(W_0)$. Let~$L \subset N'$ be the Zariski closure of the image of~$\rho$. It suffices to show that 
\[ \cC \; := \; \{ L' \subset N' \mid L' = g L g^{-1}, g \in \GL(W_0)(\bbQ_p)\}\]
is a finite union of~$N(Z)$-orbits. Now the subgroup~$\Res_{Z/\bbQ_p} N \subset N'$ is of finite index by \cref{lemma:NormalizerAndWeilRestriction}. Therefore it suffices to prove that~$\cC$ is a finite union of~$N'(\bbQ_p)$-orbits. The representation~$\rho$ is semisimple by hypothesis, thus the identity component of~$L$ is reductive. According to \cref{Prop:DisconnectedRichardsonThm}, the set
\[ \{ L' \subset (N' \times_{\bbQ_p} {\bar{\bbQ}_p} )\mid L' = g L g^{-1}, g \in \GL(W_0)(\bar{\bbQ}_p)\} \]
is a finite union of~$N'(\bar{\bbQ}_p)$-orbits, where~$\bar{\bbQ}_p$ is an algebraic closure of~$\bbQ_p$. We then conclude by \cite[ch. III, \S 4.4, th.~5]{SerreCohomologieGaloisienne}. 
\end{proof}

\begin{corollary} \label{Prop:GradedFaltingsFiniteness} Let~$w \in \bbZ$ and~$\cI \subset \cB(H,Z) / H^\circ(Z)$ a finite subset. Then up to isomorphism there are only finitely many graded~$\Gamma$-\'etale data on~$Z$ arising as the graded of a filtered \'etale data~$D = (V, G, \rho, x)$ such that 
\begin{itemize}
\item $D$ is pure of weight~$w$,
\item $\gr D$ is semisimple, and
\item there is an isomorphism~$(V,G) \iso (W, H)$ sending~$x$ in~$\cI$. 
\end{itemize}
\end{corollary}

\begin{proof} We may assume that~$\cI = \{ [x] \}$ for some~$x \in \cB(H, Z)$. In this case, up to isomorphism we may write any filtered~$\Gamma$-\'etale datum in question as a quadruple~$(W, H, \rho, x)$ where now only~$\rho$ is variable. Let~$Q$ be the stabilizer of~$x$ and~$W^\bullet$ the filtration induced on~$W$ by~$x$. We conclude by applying \cref{lemma:FaltingsFiniteness} with~$Q^\ss$ and~$\gr W^\bullet$.
\end{proof}

In order to apply \cref{Prop:GradedFaltingsFiniteness} it will be important to have a finite subset~$\cI$ permitting to construct the semisimplification of all representations in question:

\begin{lemma} \label{Lemma:GoodSetOfIndices}  Suppose that the subgroup~$H \subset \GL(W)$ is of finite index in its normalizer. Then there exists a finite subset \[\cI \subset \cB(H,Z) / H(Z) \] with the following property: For every semilinear representation 
\[ \rho \colon \Gamma \to \GL_{Z/\bbQ_p}^\s(W) \] normalizing~$H$ there is~$x \in \cB(H, Z)$ fixed under~$\Gamma$ with image in~$\cI$ such that the graded of the filtered~$\Gamma$-\'etale datum~$(W, H, \rho, x)$ is semisimple.
\end{lemma}

\begin{proof} Let~$N \subset \GL(W)$ be the normalizer of~$H$ and~$N' \subset \GL_{Z/\bbQ_p}^\s (W)$ the normalizer of~$\Res_{Z / \bbQ_p}H$. Since~$H \subset N$ is of finite index the subgroup~$\Res_{Z/\bbQ_p}N \subset N'$ is of finite index by \cref{lemma:NormalizerAndWeilRestriction}. Therefore, we have the following chains of identities:
\[ \Res_{Z/\bbQ_p} H^\circ = \Res_{Z/\bbQ_p} N^\circ = N'^\circ, \qquad \cB(H, Z)= \cB(N, Z) = \cB(N', \bbQ_p).\]
On the other hand the set 
\[ \Par(H^\circ)(Z) / H(Z) = \prod_{z \in Z} \Par(H^\circ_z)( k_z) / H_z( k_z), \]
is finite, where~$k_z$ is the residue field at~$z$. Let~$\cP \subset \Par(H^\circ)(Z)$ be a finite set of representatives of the preceding quotient. According to \cite[lemma 2.1]{BateMartinRoehrle} for each parabolic subgroup~$P \in \cP$ there is~$x_P \in \cB(H, Z)$ such that~$\Stab_{N'}(x_P)$ is the normalizer of~$\Res_{Z/\bbQ_p}P$ in~$N'$. We claim that the set
\[ \cI= \{ [x_P] \mid P \in \cP \} \subset \cB(H, Z)/ H(Z)\]
does the job. Indeed let~$\rho \colon \Gamma \to \GL_{Z/ k}^\s(W)$ be a semilinear representation normalizing~$H$. By \cite[corollary 3.5]{BateMartinRoehrle} there is a point~$y \in \cB(H, Z) = \cB(N', \bbQ_p)$ fixed under~$\Gamma$ such that the graded of the filtered~$\Gamma$-\'etale datum~$(W, H, \rho, y)$ is semisimple. Then~$\Stab_H(y)^\circ = h P h^{-1}$ for some~$P \in \cP$ and~$h \in H(Z)$. By construction the stabilizer of~$x := h x_P$ is the normalizer of~$\Res_{Z/\bbQ_p} \Stab_H(y)^\circ$. Since~$\Gamma$ fixes~$y$ the image of~$\rho$ is contained in~$\Stab_{N'}(x)$ hence~$\Gamma$ fixes~$x$ too. Moreover
\[ \rad \Res_{Z/\bbQ_p} \Stab_H(y) = \rad \Stab_{N'}(x)\]
thus the graded of the filtered~$\Gamma$-\'etale datum~$(W,H,\rho, x)$ is semisimple by \cite[example 4.9]{BateMartinRoehrle}. This concludes the proof.
\end{proof}

\subsection{de Rham data} \label{sec:deRhamData} Let~$R = \cR[\tfrac{1}{p}]$ where~$\cR$ is a finite \'etale~$\bbZ_p$-algebra. Equivalently~$R$ is a finite product of finite unramified extensions of~$\bbQ_p$. The ring~$\cR$ coincides with the ring of Witt vectors of its reduction modulo~$p$ and therefore inherits a Frobenius endomorphism~$\sigma$.  We denote again by~$\sigma$ the unique ring endomorphism of~$R$ extending~$\sigma$. Let~$V$ be a vector bundle  on~$Z = \Spec R$. In this section we always suppose semilinear endomorphisms~$\phi$ of~$V$ to be~$\bbQ_p$-linear endomorphism  such that~$\phi(\lambda v)=\sigma(\lambda) \phi(v)$ for all~$\lambda \in R$ and~$v \in V$. 

\begin{definition}
A \emph{de Rham datum on~$Z$} is a quadruple~$D = (V, G, h, \phi)$ where 
\begin{itemize}
\item $V$ is a vector bundle over~$Z$,
\item $G \subset \GL(V)$ is a reductive group~$Z$-scheme,
\item $h \in \cB(G, Z)$,
\item $\phi$ is a bijective semilinear endomorphism of~$V$ normalizing~$G$.
\end{itemize}
If the scheme~$Z$ is clear from the context we will simply call~$D$ a de Rham datum. By a~\emph{filtered} resp. \emph{graded de Rham datum}  we mean a de Rham datum~$D=(V,G,h,\phi)$ together with a filtration resp. a grading, where
\begin{itemize}
\item a \emph{filtration} is a point~$x \in \cB(G,Z)$ fixed under~$\phi$;
\item a \emph{grading} is the datum for each~$z \in Z$ of a structure of graded~$k_z$-vector space~$V_z = \bigoplus_{\alpha \in \bbQ} V_z^{(\alpha)}$ such that the graded pieces are stable under~$\phi_z$ and such that
\[ G_z \subset \prod_{\alpha \in \bbQ} \GL(V_z^{(\alpha)}). \]
\end{itemize}
\end{definition}

An \emph{isomorphism} between de Rham data~$(V, G,h,\phi)$ and~$(V', G', h', \phi')$ is an isomorphism of vector bundles
\[f \colon \quad V \;\stackrel{\sim}{\longrightarrow} \; V'
\quad \text{with} \quad G' = fGf^{-1}, \quad \phi' = f \phi f^{-1}, \quad h' = F_\ast h,\]
where~$F \colon G \to G'$ is the conjugation by~$f$. By an isomorphism of filtered  de Rham data we mean an isomorphism of de Rham data which preserves the given points in the building. An isomorphism between graded de Rham data~$(V^{(\bullet)}, G, h, \phi)$ and~$(V'^{(\bullet)}, G', h', \phi')$ is an isomorphism between the underlying de Rham data that preserves the grading in the sense that~$f(V_z^{(\alpha)}) = V'^{(\alpha)}_z$ for all~$\alpha \in \bbQ$ and all~$z \in Z$. The \emph{pull-back} of de Rham data along a finite \'etale morphism~$Z' \to Z$ is defined in the obvious way. 

\begin{example}\label{Ex:FlatBundlesOnAbelianVarieties} Let~$K$ be a finite unramified extension of~$\bbQ_p$ with ring of integers~$\cO_K$ and~$\cA$ an abelian scheme over~$\cO_K$. Given an integer~$r \ge 1$ nondivisible by~$p$ consider the vector bundle
\[ \cE := [r]_\ast \cO_{\cA} \]
on~$\cA$ where~$[r]$ is the multiplication by~$r$ on~$\cA$. The morphism~$[r]$ is finite \'etale thus the vector bundle~$\cE$ comes equipped with an integrable connection~$\nabla$ induced by the canonical derivation on~$\cO_\cA$. To decompose~$\cE$ let~$\cA^\natural$ be the universal vector extension of the dual abelian scheme~$\cA^\ast$. From a modular point of view~$\cA^\natural$ parametrizes rank~$1$ connections on~$\cA$, that is, couples~$(\cL, \nabla)$ made of a line bundle~$\cL$ on~$\cA$ and  a (necessarily integrable) connection~$\nabla$ on~$\cL$. Tensor product of line bundles together with a connection endow~$\cA^\natural$ with a structure of group scheme over~$\cO_K$. 
Consider the universal rank~$1$ connection~$(\cL,\nabla)$ on~$\cA^\natural \times \cA$. The closed subscheme  of~$r$-torsion points~$\cZ =  \cA^\natural[r]$ is finite \'etale over~$\cO_K$. Then
\[ (\cE, \nabla) = \pi_\ast (\cL, \nabla) \quad \textup{where~$\pi \colon \cZ \times  \cA \to \cA$ is the second projection}. \]
Let~$f \colon \cX \to \cA$ be a morphism of finite type~$\cO_K$-schemes with~$\cX$ smooth of relative dimension~$d$. Set~$\cY := \cZ \times_{\cO_K} \cX$ and consider the~$d$-th relative de Rham cohomology group
\[ \cV:= \cH_\dR^d(\cX/\cO_K; f^\ast(\cE, \nabla)) = \cH^d_{\dR}(\cY/ \cZ; (f \times \id)^\ast(\cL, \nabla)).\] 
Suppose that~$\cV$ is locally free as well as the graded pieces of the Hodge filtration on~$\cV$. Let~$\tilde{\cX}$ be the pull-back of~$\cX$ along the morphism~$[r]$. The crystalline-de Rham comparison theorem furnishes a canonical isomorphism of~$\cO_K$-modules
\[ \cV \iso \rH^d_{\cris}(\tilde{\cX}_p/ \cO_K)\]
where the subscript~$p$ stands for the reduction modulo~$p$. By functoriality of crystalline cohomology~$\cV$ inherits a Frobenius operator~$\phi$. Let~$A$,~$X$,~$Y$ and~$Z$ be the generic fiber respectively of~$\cA$,~$\cX$,~$\cY$ and~$\cZ$. Then
\[ V := \cV \otimes_{\cO_K} K = \cH^d_{\dR}(Y/ Z; \cL, \nabla) \]
is a vector bundle over~$Z$ and as~$K$-vector spaces we have \[V_0 =  \rH^d_{\dR}(X/K; \cE, \nabla).\] The Frobenius operator~$\phi$ extends to a semilinear automorphism of~$V$ denoted~$\phi$ again.  For each~$z \in Z$ set
\[
G_z = 
\begin{cases}
\GL(V_z) & \textup{if~$X$ is not symmetric up to translation,} \\
\GO(V_z, \theta_z) & \textup{if~$X$ is symmetric up to translation and~$d$ is even,} \\
\GSp(V_z, \theta_z) & \textup{if~$X$ is symmetric up to translation and~$d$ is odd}.
\end{cases}
\]
As in \cref{Ex:LocSysAbVar} \emph{symmetric up to a translation}  means that there is an automorphism~$\iota$ of~$X$ and a point~$a \in A(K)$ such that~$f(\iota(x)) = a -f(x)$ for all~$x \in X$. Poincar\'e duality then induces a perfect pairing~$\theta_z \colon V_z \otimes V_z \to L_{z, a}$. The Frobenius operator~$\phi$ normalizes~$G$ because the crystalline-de Rham comparison isomorphism is compatible with the Poincar\'e pairing and the Frobenius operator preserves it up to a scalar multiple. The vector bundle~$V$ comes with the Hodge filtration
\[ F^\bullet \; : \quad V \; =\; F^0 \;  \supset \; F^1 \; \supset\; \cdots \; \supset\; F^d\; \supset\; F^{d + 1}\; =\; 0, \]
where~$F^i = \im(\R^d \pi_{Z \ast} (\Omega^{\bullet \ge i}_{Y/Z} \otimes_{\cO_Y} \pi_A^\ast \cL) \to V )$ for the projection~$\pi_Z \colon Y \to Z$. 
Moreover, if~$X$ is symmetric up to a translation, then the Hodge filtration is autodual with respect to the Poincar\'e pairing in the sense that for each~$i$ we have~$(F^i)^\perp = F^{d+1-i}$  (this can be checked on fibers at geometric points of~$Z$, so it follows from the statement over the complex numbers in~\cref{Prop:HodgeDecompositionIsOrthogonalWRTPoincarePairing}). 
Thus the Hodge filtration~$F^\bullet$ defines a point~$h \in \cB(G,Z)$, in other words the quadruple 
\[ D = (V,G,h,\phi)\]
is a de Rham datum on~$Z$. 
\end{example} 

\begin{definition}
The associated graded~$\gr D = (\bar{V}, \bar{G}, \bar{h}, \bar{\phi}, \bar{V}^{(\bullet)})$ of a filtered de Rham datum~$D= (V,G,h,\phi,x)$ is defined as follows. For each point~$z \in Z$ let~$V_{z}^\bullet$ be the filtration which is induced by~$x_z$, and let~$Q_z \subset G_z$ be the parabolic subgroup stabilizing~$x_z$. Then we put
\[ \bar{G}_z := Q_z / \rad Q_z \intoo \prod_{\alpha \in \bbQ} \GL(V^{(\alpha)}_z) \]
where~$V^{(\alpha)}_z := \gr^\alpha V_z$, and we define~$\bar{h}_z := \pr_{Q_z}(h_z)$ and~$\bar{\phi}_z := \gr \phi_z$.
\end{definition}

\subsection{From graded to filtered data} For~$i = 1, 2$ consider a filtered de Rham datum~$D_i = (V_{i}, G_{i}, h_i, \phi_i, x_{i})$. We want to understand the relation between~$D_1$ and~$D_2$ when the associated graded data are isomorphic. As a reference datum we fix~$(V, G, \varphi)$
where
\begin{itemize} 
\item $V$ is a vector bundle on~$Z$,
\item $G\subset \GL(V)$ is a reductive subgroup scheme over~$Z$,
\item $\varphi\colon V\to V$ is a semilinear endomorphism normalizing~$G$.
\end{itemize} 
Suppose that for~$i=1,2$ we are given isomorphisms of vector bundles
\[ 
 \tau_i \colon \quad V_i \;\stackrel{\sim}{\longrightarrow}\; 
 V 
 \quad \text{such that} \quad 
 G \;=\; \tau_{i} G_{i} \tau_{i}^{-1} \quad \text{and} \quad \phi \;=\; \tau_{i} \phi_{i} \tau_{ i}^{-1}.
 \]
Via these isomorphisms we will identify the points~$x_i, h_i \in \cB(G_i,  k)$ with points in~$\cB(G, Z)$ which we again denote~$x_i, h_i$. Let~$Q_i \subset G$ denote the stabilizer of~$x_{i}$, with Levi quotient~$Q_{i}^\ss = Q_{i}/ \rad Q_{i}$. With notation as in~\cref{sec:ProjectionLeviQuotientBuildings} consider the projection
\[
\pr_i \colon \quad 
\cB(G,Z) \;=\; \prod_{z \in Z}\cB(G_z,  k_z) 
\; \too \; \cB(Q_i^\ss, Z) \;=\; \prod_{z \in Z}\cB(Q_{i,z}^\ss,  k_z).
\] 
In general an isomorphism between graded cannot be lifted to an isomorphism of the original filtered one. Nonetheless the filtered data satisfy a weaker relation (to be compared with \cref{Def:BalancedRelation}):

\begin{proposition}  \label{isomorphic-graded-de-rham}
If in the above setup the graded de Rham data~$\gr D_1$ and~$\gr D_2$ are isomorphic and for every~$z\in Z$ the subgroup~$G_z \subset \GL(V_z)$ is full in the sense of~\cref{full-subgroup}, then there exists~$g\in G(Z)$ such that 
\begin{align*}
x_2 =g x_1, &&  g \phi g^{-1} \phi^{-1} \in \rad Q_2, && \pr_2(h_2) = \pr_2(g h_1).
\end{align*}
\end{proposition}

\begin{proof} By reasoning pointwise we may assume~$Z = \Spec k$ for a finite unramified extension~$k$ of~$\bbQ_p$. Via the isomorphisms~$f_{i}$ we make the identifications
\[
 V_i = V ,
 \quad 
 G_i = G
 \quad \text{and} \quad 
 \phi_i = \phi 
 \quad \text{for} \quad i = 1,2.
\]  
Let~$V_{i}^\alpha\subset V$ be the steps of the filtration given by~$x_{i}$ and let~$V_i^{(\alpha)} = V_i^{\alpha}/V_i^{<\alpha}$ be the associated graded pieces, for each~$\alpha \in \bbQ$. By assumption there exists an isomorphism~$\gr D_1 \iso \gr D_2$ of graded de Rham data, which by definition consists of a collection of~$K$-linear isomorphisms~$\bar{g}^{(\alpha)} \colon V_{1}^{(\alpha)} \to V_{2}^{(\alpha)}$ whose direct sum~$\bar{g}$ has the following properties:\smallskip 
\begin{enumerate} 
\item $Q_{2}^\ss = \bar{g} Q_{1}^\ss \bar{g}^{-1}$, \smallskip 
\item $\bar{\phi} = \bar{g} \bar{\phi} \bar{g}^{-1}$, \smallskip 
\item $\pr_{2}(h_{2}) = \bar{g}_\ast \pr_{1} (h_{1})$. \smallskip
\end{enumerate}
We now use the assumption that the subgroup~$G\subset \GL(V)$ is full.
Note that the filtrations on~$V$ induced by the two points~$x_1,x_2$ are conjugate under~$\GL(V)( k)$ since their filtered pieces have the same dimension. The corresponding weights~$\alpha \in \bbQ$ also match, thus that the two points~$x_1, x_2$ are conjugate under~$\GL(V)( k)$.  
By the definition of full subgroups we then have a finite subset~$I\subset \bbQ$ such that
\[
 \epsilon_{x_i, I} \colon 
 \quad Q_i^\ss \;\too\; \prod_{\alpha \in I} \GL(V_i^{(\alpha)})
\]
is an embedding and~$\Tran_G(x_1, x_2) \to \gr_I \Tran_{\GL(V)}(x_1, x_2)$ is surjective. It follows that the map on~$k$-points
\[
\Tran_G(x_1, x_2)( k) \too \gr_I \Tran_{\GL(V)}(x_1, x_2)( k)
\]
is also surjective by the remark after~\cref{transporter-as-torsor}. By  property (1) the point~$\bar{g}$ belongs to the transporter~$\gr_I \Tran_{\GL(V)}(x_1, x_2)( k)$, hence this gives an element \[g\in \Tran_G(x_1, x_2)( k) \subset G( k)\] lifting~$\bar{g}$. By definition of the transporter~$\Tran_G(x_1, x_2)$ we have~$gx_1 = x_2$. To prove~$g\varphi g^{-1} \varphi^{-1} \in \rad Q_2$
note that~$g\varphi g^{-1} \varphi^{-1}$ belongs to~$Q_2$ because~$\phi$ fixes~$x_1$ and~$x_2$. The wanted assertion then follows from condition (2) and the injectivity of~$\epsilon_{x_2, I}$. Finally from (3) we get~$\pr_2(h_2) = \pr_2(gh_1)$ which concludes the proof.
\end{proof}

\section{$p$-adic Hodge theory} \label{sec:PAdicHodgeTheory}

We now discuss how to pass from \'etale to de Rham data via~$p$-adic Hodge theory, and how global purity on the \'etale side leads to positivity in the sense of \cref{sec:PositivityParabolic} on the de Rham side. 

\subsection{Crystalline representations}
Let~$K$ be a finite unramified extension of~$\bbQ_p$ with ring of integers~$\cO_K$, and let~$\bar{K}$ be an algebraic closure of~$K$ with absolute Galois group~$\Gamma = \Gal(\bar{K}/K)$. The ring of crystalline periods is an integral domain
\[B:= B_{\cris}\] 
 with an algebra structure over the maximal unramified extension of~$\bbQ_p$. It comes equipped with an action of~$\Gamma$, a Frobenius operator~$\phi$ and a (non-increasing, exhausting, separated) filtration~$h^\bullet B = (h^\alpha B)_{\alpha \in \bbZ}$ stable under~$\Gamma$. For the construction of~$B$ and the properties cited below see for example \cite{FontaineOuyang} or \cite[II.9]{BrinonConrad}. 

\begin{definition} Let~$V$ be a~$\bbQ_p$-vector space endowed with a linear action of~$\Gamma$. When~$V$ is finite dimensional, we say that the representation~$V$ is \emph{crystalline} if it is continuous and~$\dim_{K} (V \otimes_{\bbQ_p} B)^\Gamma = \dim_{\bbQ_p} V$. When~$V$ is infinite dimensional, we say that~$V$ is \emph{crystalline} if it is an increasing union of finite dimensional crystalline subrepresentations. 
\end{definition}

We then set \[V_{\dR} := (V \otimes_{\bbQ_p} B)^\Gamma, \qquad h^\alpha V := (V \otimes_{\bbQ_p} h^\alpha B)^\Gamma,\] so  we have a natural isomorphism
$ V \otimes_{\bbQ_p} B \iso V_{\dR} \otimes_K B$. The filtration~$h_V := h^\bullet V$ is non-increasing, exhausting and separated; it is called the \emph{Hodge filtration} of~$V_{\dR}$. The operator~$\phi$ induces a bijective semilinear endomorphism~$\phi_V$ of~$V_{\dR}$ which we call again the Frobenius operator. The class of crystalline representations is stable under subquotients and tensor products.  For a morphism of crystalline representations~$f \colon V \to W$ the resulting~$K$-linear map~$f_{\dR} \colon V_{\dR} \to W_{\dR}$ is compatible with the induced Frobenius operators on~$V_{\dR}$ and~$W_{\dR}$. The functor~$V \rightsquigarrow V_{\dR}$ is exact and if~$f$ is injective resp. surjective then~$h_V$ is the restriction to~$V$ of~$h_W$ resp.~$h_W$ is the quotient of~$h_V$. We will refer to this property as \emph{exactness of the Hodge filtration}. For crystalline representations~$V$ and~$W$ we have a canonical isomorphism \[(V \otimes_{\bbQ_p} W)_{\dR} = V_\dR \otimes_K W_{\dR}\] compatible with the Frobenius operators and the Hodge filtrations. When~$V$ is finite dimensional the filtration~$h_V$ can be seen as a point of~$\cB(\GL(V_{\dR}), K)$ by \cref{Ex:BuildingGL}.

\begin{example} \label{Ex:GaloisRepWithFiniteImageIsCristalline}An unramified representation~$\rho \colon \Gamma \to \GL(V)$ is crystalline; see for instance \cite[prop. 7.17]{FontaineOuyang}. When the image of~$\rho$ is finite this can be seen by the following argument. In this case~$\ker \rho = \Gamma' := \Gal(\bar{K}/K')$ for some finite unramified extension~$K'$ of~$K$ and
\begin{align*}
(V \otimes_{\bbQ_p} B)^\Gamma 
&= [(V \otimes_{\bbQ_p} B)^{\Gamma'}]^{\Gamma/\Gamma'} \\
&= (V \otimes_{\bbQ_p} B^{\Gamma'})^{\Gal(K'/K)} \\
&= (V \otimes_{\bbQ_p} K')^{\Gal(K'/K)}.
\end{align*}
So by Galois descent~$\dim_K (V\otimes_{\bbQ_p} B)^\Gamma = \dim_K (V \otimes_{\bbQ_p} K')^{\Gal(K'/K)} = \dim_{\bbQ_p} V$.

\end{example}

Let~$A$ be a~$\bbQ_p$-algebra with a crystalline action of~$\Gamma$ compatible with the ring structure of~$A$, that is~$\gamma(ab) = (\gamma a)(\gamma b)$ for all~$a,b \in A$ and~$\gamma \in~\Gamma$. A~$\bbQ_p$-linear action of~$\Gamma$ on an~$A$-module~$M$ is \emph{semilinear} if for all~$\gamma \in~\Gamma$,~$a \in A$ and~$m \in M$ we have~$\gamma(am) = (\gamma a)(\gamma m)$. Given an~$A$-module~$M$ with a crystalline semilinear action of~$\Gamma$ the~$K$-vector space~$M_{\dR}$ inherits a natural structure of~$A_{\dR}$-module. For a~$\Gamma$-equivariant morphism of~$A$-modules~$f \colon M \to N$ the induced  map~$f_\dR \colon M_\dR \to N_\dR$ is a morphism of~$A_\dR$-modules.

\begin{lemma} \label{lemma:deRhamificationModules} The so-defined functor
\[
\left\{
\begin{array}{c}
\textup{$A$-modules with a crystalline} \\ \textup{semilinear action of~$\Gamma$}
\end{array}
\right\} \; \too \;
\left\{ 
\begin{array}{c}
\textup{$A_{\dR}$-modules}
\end{array}
\right\}
\] 
is exact, symmetric monoidal and compatible with pull-back via~$\Gamma$-equivariant morphisms of affine~$\bbQ_p$-schemes with a crystalline action of~$\Gamma$. An~$A$-module~$M$ with a crystalline semilinear action of~$\Gamma$ is finitely generated if and only if~$M_{\dR}$ is.
\end{lemma}

\begin{proof} The exactness comes from that of the functor~$V \mapsto V_{\dR}$. Given~$A$-modules~$M$ and~$N$  with a crystalline semilinear action of~$\Gamma$, we have
\begin{align*} 
(M \otimes_A N)_{\dR} \otimes_K B &\iso (M \otimes_A N) \otimes_{\bbQ_p} B \\
&= (M \otimes_{\bbQ_p} B) \otimes_{A \otimes_{\bbQ_p} B} (N\otimes_{\bbQ_p} B) \\
&\iso (M_{\dR} \otimes_K B) \otimes_{A_{\dR} \otimes_K B} (N_{\dR} \otimes_K B) \\
&= (M_{\dR} \otimes_{A_{\dR}} N_{\dR}) \otimes_K B
\end{align*}  
Passing to~$\Gamma$-invariants we obtain a natural isomorphism
\[ \alpha_{M,N} \colon (M \otimes_A N)_{\dR} \stackrel{\sim}{\too} M_{\dR} \otimes_{A_{\dR}} N_{\dR}. \]
This also show that the functor~$M \mapsto M_{\dR}$ is compatible with pull-back. It remains to prove that the functor is symmetric. For, notice that the commutativity of the diagram
\[
\begin{tikzcd}
M_{\dR} \otimes_{A_{\dR}} N_{\dR} \ar[r, "\alpha_{M,N}"] \ar[d, "x \otimes y \mapsto y \otimes x"'] & (M \otimes_A N)_{\dR} \ar[d, "x \otimes y \mapsto y \otimes x"]\\
N_{\dR} \otimes_{A_{\dR}} M_{\dR} \ar[r, "\alpha_{N,M}"] & (N \otimes_A M)_{\dR}
\end{tikzcd}
\]
can be tested after extending scalars to~$B$ where it becomes clear. The last statement follows from the fact that finite generation can be tested after extending scalars to~$B$; see  \cite[\href{https://stacks.math.columbia.edu/tag/08XD}{Theorem 08XD}]{stacks-project}.
\end{proof}

Let~$X$ be an affine~$\bbQ_p$-scheme endowed with an action of~$\Gamma$ defined by a group morphism~$\sigma\colon \Gamma \to \Aut(X)$. Then \[\Gamma \times \Gamma(X,\cO_X) \too \Gamma(X,\cO_X),\quad (\gamma,f) \longmapsto \sigma(\gamma^{-1})^\ast f\] defines a linear action of~$\Gamma$ on~$\Gamma(X,\cO_X)$ compatible with the ring structure.

\begin{definition}
We say that the action of~$\Gamma$ on~$X$ is \emph{crystalline} if~$\Gamma(X,\cO_X)$ is a crystalline representation. If the action of~$\Gamma$ on~$X$ is crystalline, then the~$K$-scheme \[X_{\dR} := \Spec \Gamma(X,\cO_X)_{\dR}\] comes with a natural isomorphism~$X \times_{\bbQ_p} B \iso X_{\dR} \times_K B$
of~$B$-schemes. The Frobenius operator on~$\Gamma(X, \cO_X)_{\dR}$ induces an automorphism~$\phi_X$ of~$X_{\dR}$ as a~$\bbQ_p$-scheme. We write~$h_X$ for the Hodge filtration on~$\Gamma(X,\cO_X)_{\dR}$.
\end{definition}

The construction is functorial: A~$\Gamma$-equivariant morphism~$f \colon X \to Y$ between affine~$\bbQ_p$-schemes endowed with a crystalline action of~$\bbQ_p$ induces a morphism of~$K$-schemes~$ f_\dR \colon X_\dR \to Y_\dR$.

\begin{lemma} \label{lemma:DeRhamificationScheme} The so-defined functor
\[
\begin{array}{rcl}
{\left\{ 
\begin{array}{c}
\textup{affine~$\bbQ_p$-schemes} \\
\textup{with a crystalline action of~$\Gamma$}
\end{array}
\right\}}
&\too&
{\left\{ 
\begin{array}{c}
\textup{affine~$K$-schemes} 
\end{array}
\right\}} \\
X & \longmapsto & X_{\dR}
\end{array}
\]
is compatible with fiber products. For any~$\Gamma$-equivariant morphism~$f \colon X \to Y$ of affine~$\bbQ_p$-schemes with a crystalline action of~$\Gamma$ we have:
\begin{enumerate}
\item \smallskip For any property~$\sfP$  of scheme morphisms that can be checked after faithfully flat base change,~$f$ satisfies~$\sfP$ if and only if~$f_{\dR}$ does.
\item The following square is commutative:
\[
\begin{tikzcd}
X_\dR \ar[r, "f_\dR"] \ar[d, "\phi_X"']& Y_\dR \ar[d, "\phi_Y"] \\
X_{\dR} \ar[r, "f_\dR"] & Y_{\dR}
\end{tikzcd}
\]
\end{enumerate}
\end{lemma}

\begin{proof} The compatibility with fiber products holds by the compatibility with pull-back in~\cref{lemma:deRhamificationModules}. Statement (1) is clear since the injections~$\bbQ_p \into B$ and~$K \into B$ are faithfully flat ring morphisms, and (2) is just the functoriality of the Frobenius endomorphism. 
\end{proof}

\begin{example} \label{Ex:CristallineActionOnFiniteScheme}Let~$X$ be a finite \'etale~$\bbQ_p$-scheme with an unramified action of~$\Gamma$. Since~$\Aut(X)$ is finite, the representation~$\rho \colon \Gamma \to \GL(\Gamma(X,\cO_X))$ has finite image. As in \cref{Ex:GaloisRepWithFiniteImageIsCristalline} write~$\ker \rho = \Gal(\bar{K}/K')$ for some finite unramified extension~$K'$ of~$K$. Then
\[ X_{\dR}  = X_{K'} / \Gal(K'/K)\]
where~$\Gal(K'/K)$ acts on~$X_{K'} = \Spec \Gamma(X,\cO_X) \otimes_{\bbQ_p} K'$ via its diagonal action on \emph{both}~$\Gamma(X, \cO_X)$ and~$K'$. Moreover, we have the set-theoretical identity
\[ X_{\dR} = X(\bar{K}) / \Gamma\]
where~$\Gamma$ acts on~$X(\bar{K})$ via its diagonal action on both~$X$ and~$\bar{K}$. With this in mind, the degree of a point in~$X_{\dR}$ is the length of the corresponding orbit under~$\Gamma$. 
\end{example}

It will be useful to have a handy condition to prove that the action on an affine scheme is crystalline.

\begin{lemma} \label{Lemma:CharCristallineActions} Let~$X= \Spec A$ be an affine~$\bbQ_p$-scheme  with a crystalline action of~$\Gamma$ and let~$Y$ be an affine~$X$-scheme of finite type  with an action of~$\Gamma$ such that the structural morphism~$Y \to X$ is~$\Gamma$-equivariant. Then, the action on~$Y$ is crystalline if and only if there are an~$A$-module~$M$ of finite type with a semilinear crystalline action of~$\Gamma$ and a~$\Gamma$-equivariant closed embedding~$Y \into \Spec (\Sym_A M)$ of~$X$-schemes.
\end{lemma}

\begin{proof} Suppose first that the action of~$\Gamma$ on~$R= \Gamma(Y, \cO_Y)$ is crystalline. Let~$V \subset R$ be a finite dimensional subrepresentation generating~$R$ as an~$A$-algebra. The~$A$-submodule~$M \subset R$ generated by~$V$ is stable under~$\Gamma$ and finitely generated as an~$A$-module. Then~$M$ does the job. Conversely let~$M$ be an~$A$-module with a crystalline semilinear action of~$\Gamma$ and consider a~$\Gamma$-equivariant surjective morphism of~$A$-algebras~$f\colon S:= \Sym_A M \to R$. For each integer~$n\ge0$ the representation  \[S_n := \bigoplus_{i=0}^n \Sym^i_A M\] is crystalline, thus so~$R_n := f(S_n) \subset R$. Since~$S = \bigcup_{n=0}^\infty S_n$ and~$\phi$ is surjective we have~$R = \bigcup_{n=0}^\infty R_n$ which concludes the proof.
\end{proof}

\begin{lemma} \label{Lemma:WeilResPadicHodge} Let~$X \to T$ and~$\pi \colon T \to S$ be~$\Gamma$-equivariant morphisms between affine varieties over~$\bbQ_p$ endowed with a crystalline action of~$\Gamma$. Suppose that~$\pi$ is finite flat. Then the natural action of~$\Gamma$ on~$\Res_{T/S} X$ is crystalline and
\[ (\Res_{T/S} X)_{\dR} = \Res_{T_{\dR}/S_\dR} X_{\dR}.\]
\end{lemma}

\begin{proof} By \cref{Lemma:CharCristallineActions} there are a coherent~$\cO_T$-module $V$ with a crystalline semilinear action on~$\Gamma(T,V)$ and a closed embedding~$X \into \Spec_T{\Sym V}$ which is $\Gamma$-equivariant. Then the morphism induced by Weil restriction
\[ j \colon \Res_{T/S} X \intoo \Res_{T/S} (\Spec_T{\Sym V}) = \Spec_S {\Sym \pi_\ast T}\]
is a closed immersion. The group~$\Gamma$ acts naturally on~$\Res_{T/S} X$ and the closed immersion~$j$ is~$\Gamma$-equivariant. \Cref{Lemma:CharCristallineActions} then implies that the action of~$\Gamma$ on~$\Res_{T/S} X$ is crystalline. Now note that there is a natural morphism of~$S_{\dR}$-schemes \[f \colon (\Res_{T/S} X)_{\dR} \too \Res_{T_{\dR}/S_{\dR}} X_{\dR}.\]
Indeed by universal property of the Weil restriction such a morphism~$f$ is given by a morphism of~$T_{\dR}$-schemes~$(\Res_{T/S} X)_{\dR} \times_{S_\dR} T_{\dR} \to X_{\dR}$ still denoted~$f$. Similarly the identity of~$\Res_{T/S} X$ is given by a morphism of~$T$-schemes
\[ g \colon \Res_{T/S} (X) \times_{S} T \too X.\]
The morphism~$g$ is~$\Gamma$-equivariant and we set
\[ f := g_{\dR}\colon (\Res_{T/S} X)_{\dR} \times_S T_{\dR} = (\Res_{T/S} (X) \times_{S} T)_{\dR}   \too X_{\dR}\]
where the identity comes from \cref{lemma:DeRhamificationScheme}. In order to conclude the proof it suffices to show that~$f$ is an isomorphism after extending scalars to~$B$. The construction of the Weil restriction is compatible with scalars extension thus
\begin{align*} (\Res_{T/S} X)_{\dR} \times_K B &\iso \Res_{T/S} (X) \times_{\bbQ_p} B = \Res_{T_B /S_B} (X_B), \\
(\Res_{T_{\dR}/S_{\dR}} X_{\dR}) \times_K B &= \Res_{T_{\dR, B}/S_{\dR,B}} (X_{\dR,B}) \iso \Res_{T_B /S_B} (X_B).
\end{align*}
The endomorphism of~$\Res_{T_B /S_B} (X_B)$ induced by~$f$ via the preceding isomorphisms is just the identity, which concludes the proof.
\end{proof}

\begin{remark}[A variation on the construction] \label{ex:NotForgettingCoefficients} Let~$V$ be a~$K$-vector space with a crystalline~$K$-linear action of~$\Gamma$. Write~$V_0$ for~$V$ seen simply as a~$\bbQ_p$-vector space. Then~$V_{0, \dR} = V_{\dR} \otimes_{\bbQ_p} K$ where
\[V_{\dR} := (V \otimes_K B)^\Gamma.\]
The construction of~$V_{\dR}$ is functorial with respect to~$K$-linear~$\Gamma$-equivariant maps. This gives rise to a functor
\[ 
\left\{ \begin{array}{c}
\textup{$K$-vector spaces with} \\
\textup{a crystalline~$K$-linear action of~$\Gamma$}
\end{array}
\right\}\too
\left\{ \begin{array}{c}
\textup{$K$-vector spaces}
\end{array}
\right\}
 \]
which is exact and compatible with tensor product. The Hodge filtration is again exact in the above sense and compatible with tensor product. Similarly, for an affine scheme~$X= \Spec A$ over~$K$ with a crystalline action of~$\Gamma$ by~$K$-scheme isomorphisms we define
\[ X_{\dR} = \Spec{(A \otimes_K B)^\Gamma}.\]
This construction is functorial with respect to~$\Gamma$-equivariant morphisms of~$K$-schemes, hence gives rise to a functor
\[ 
\left\{ \begin{array}{c}
\textup{$K$-schemes with} \\
\textup{a crystalline~$K$-linear action of~$\Gamma$}
\end{array}
\right\}\too
\left\{ \begin{array}{c}
\textup{$K$-schemes}
\end{array}
\right\}
 \]
which is compatible with fibered products.
\end{remark}

\subsection{Affine group schemes with crystalline actions} 
Let~$Z=\Spec R$ be an affine~$\bbQ_p$-scheme endowed with a crystalline action of~$\Gamma$ and let~$V$ be a vector bundle on~$Z$ together with a crystalline semilinear action on the~$R$-module~$\Gamma(Z, V)$. In this section we write~$\bbV(V) = \Spec \Sym V^\vee$ for the total space of~$V$. Consider a subgroup~$Z$-scheme~$G \subset \GL(V)$  normalized by~$\Gamma$. The action by conjugation of~$\Gamma$ on~$G$ is crystalline. Indeed the closed embedding
\[ G \intoo \bbV(\End V \oplus \cO_Z ), \quad g \longmapsto (g, \det(g)^{-1})\]
is~$\Gamma$-equivariant with respect to the action on~$\End(V\oplus \det V)$ given by
\[
\gamma.(f, \lambda) = ( \rho(\gamma) f \rho(\gamma)^{-1}, \lambda), \qquad \textup{for } \gamma \in\Gamma, f \in \End(V), \textup{ and } \lambda \in \det V,
\]
hence we conclude by \cref{Lemma:CharCristallineActions}. The group law and the inversion on~$G$ are equivariant morphisms for the action of~$\Gamma$ and the neutral element of~$G$ is fixed under~$\Gamma$. Thus~$G_{\dR}$ is group scheme over~$Z_{\dR}$, and actually a subgroup
\[ G_{\dR} \intoo \GL(V)_{\dR} =\GL(V_{\dR})\] 
because~$G \into\GL(V)$ is~$\Gamma$-equivariant when~$\Gamma$ acts on~$\GL(V)$ by conjugation.

\begin{lemma} \label{Lemma:StupidPropertiesGroupsWithCristallineAction} In the above setup suppose~$Z$ is finite \'etale. Then the following hold:
\begin{enumerate}
\item For any normal subgroup~$H \subset G$ normalized by~$\Gamma$ the action by conjugation of~$\Gamma$ on~$G/H$ is crystalline and
\[(G/H)_{\dR} = G_{\dR}/H_{\dR}.\]
\item The unipotent radical~$\rad G \subset G$ is normalized by~$\Gamma$ and 
\[ (\rad G)_{\dR} = \rad(G_{\dR}).\]
\item For any Hopf filtration~$x$ on~$\bbQ_p[\Res_{Z/\bbQ_p}G]$ stable under~$\Gamma$ there is a unique Hopf filtration on~$K[\Res_{Z_{\dR}/K} G_{\dR}]$ such that~$x \mapsto x_{\dR}$ via 
\[\Res_{Z/\bbQ_p}G \times_{\bbQ_p} B \; \iso\;  \Res_{Z_{\dR}/K}G_{\dR} \times_{K} B.
\]
\end{enumerate}
\end{lemma}

\begin{proof} (1) The projection~$\pi\colon G \to G/H$ is~$\Gamma$-equivariant and~$H$-invariant, thus it induces an~$H_{\dR}$-invariant morphism~$\pi_{\dR} \colon G_{\dR} \to (G/H)_{\dR}$. By property of the quotient this factors through a morphism~$G_{\dR}/H_{\dR} \to (G/H)_{\dR}$ which is seen to be an isomorphism after extending scalars to~$B$.

\medskip

(2) The subgroup~$(\rad G)_{\dR}$ is connected normal and unipotent. Therefore we have an inclusion~$\rad G)_{\dR} \subset \rad (G_{\dR})$ which is seen to an equality after extending scalars to~$B$.

\medskip

(3) Write~$H=\Res_{Z/\bbQ_p} G$ so that~$H_{\dR} = \Res_{Z_{\dR}/K} G_{\dR}$. Let~$\bbQ_p[H]^\bullet$ be an Hopf filtration on~$\bbQ_p[H]$ stable under~$\Gamma$. The morphisms~$\mu$ and~$\nu$ in the \cref{Def:HopfFiltration} of an Hopf filtration  are~$\Gamma$-equivariant with respect to the action by conjugation of~$\Gamma$ on~$H$. It follows that the filtration of~$K[H_{\dR}]$ defined by
\[ K[H_{\dR}]^\alpha= \bbQ_p[H]^\alpha_{\dR}, \quad \alpha \in \bbQ,\]
is Hopf and it does the job by construction.
\end{proof}

Now we reset notation and work over~$Z = \Spec \bbQ_p$. In this case crystalline actions on affine algebraic groups admit the following characterization:

\begin{lemma}Let~$G$ be an affine algebraic group over~$\bbQ_p$ and~$\rho \colon \Gamma \to G$ a group morphism. Then the following are equivalent:
\begin{enumerate}
\item The action of~$\Gamma$ on~$G$ by left multiplication is crystalline.
\item The action of~$\Gamma$ on~$G$ by right multiplication is crystalline.

\item There is a finite dimensional faithful representation~$i \colon G\into\GL(V)$ such that~$i \circ \rho \colon \Gamma \to \GL(V)$ is crystalline.
\item Given a representation~$V$ of~$G$ the induced linear action of~$\Gamma$ is crystalline.
\end{enumerate}
If any of the above equivalent conditions holds, then the action of~$\Gamma$ on~$G$ by conjugation is crystalline.
\end{lemma}

\begin{proof} (1) and (2) are equivalent because the inversion morphism~$G \to G$ is~$\Gamma$-equivariant when~$\Gamma$ acts by left multiplication on the source and by~$(\gamma, g) \mapsto g \gamma^{-1}$ on the target. The implication (4)~$\Rightarrow (3)$ is clear. The implication (2)~$\Rightarrow$ (4) follows from the fact that any representation of~$G$ is union of finite dimensional subrepresentations and any finite dimensional representation of~$G$ is isomorphic to a subrepresentation of some finite direct sum of the regular representation \cite[cor. 4.13]{MilneAlgebraicGroups}. For the implication (3)~$\Rightarrow$ (2) notice that the closed embedding
\[ i\colon G \intoo \bbV(\End V \oplus \bbQ_p), \quad g \longmapsto (g, \det g^{-1})\]
is~$\Gamma$-equivariant when~$\Gamma$ acts on~$\End V \oplus \det V^\vee$ by~$(f, \lambda) \mapsto (f\gamma^{-1}, \det \gamma \lambda)$. For the last statement, we already saw in the paragraph preceding \cref{Lemma:StupidPropertiesGroupsWithCristallineAction} that the conjugation action is crystalline. 
\end{proof}

We are ready to state and prove the main result of this section, which is the existence of a canonical Hopf filtration for an affine group with a crystalline action:

\begin{proposition} \label{Lemma:HodgeFiltrationComesFromASubgroup} Let~$G$ be an affine algebraic group over~$\bbQ_p$ and~$\rho \colon \Gamma \to G$ a group morphism such that the action of~$\Gamma$ on~$G$ by right multiplication is crystalline. Then,
\begin{enumerate}
\item there exists a unique Hopf filtration~$h_G$ on~$K[G_{\dR}]$ such that for any representation~$V$ of~$G$ the Hodge filtration~$h_V$ on~$V_{\dR}$ is the one induced by~$h_G$; 
\item given a morphism of affine algebraic groups~$f \colon G \to H$ the action of~$\Gamma$ on~$H$ by right multiplication is crystalline and~$h_H$ is the filtration induced by~$h_G$; 
\item if~$G$ is reductive and~$Q \subset G^\circ$ is a parabolic subgroup normalized by~$\Gamma$, then~$h_G$ is induced by a unique Hopf filtration~$h_Q$ on~$K[Q_{\dR}]$ and the Hopf filtration~$h_{Q^\ss}$  induced by~$h_Q$ via~$Q_{\dR} \to Q_{\dR}^\ss = Q_{\dR}/ \rad Q_{\dR}$~is
\[ h_{Q^\ss} = \pr_Q(h_G).\]
\end{enumerate}
\end{proposition}

\begin{proof} (1) The uniqueness is clear: Any faithful representation~$G \into \GL(V)$  induces an injective map
\[ \{ \textup{Hopf filtrations on~$K[G_{\dR}]$}\} \too  \{\textup{filtrations on~$V_{\dR}$} \} \]
by \cref{Prop:FiltrationFunctorConnectedComponent}. To show the existence we proceed in several steps:

\medskip

\emph{Construction of the filtration.} Write~$L$ for~$G$ endowed with the action of~$\Gamma$ by left multiplication, and write~$P = L_{\dR}$. The group law~$G \times_{\bbQ_p} L \to L$ on~$G$ is equivariant with respect to the action of~$\Gamma$ by conjugation on~$G$ and by left multiplication on~$L$. This induces an action 
\[ m\colon G_{\dR} \times_K P = G_{\dR} \times_K L_{\dR} \too L_{\dR} = P\]
making~$P$ into a principal bundle under~$G_{\dR}$. Moreover the actions of~$\Gamma$ on~$L$ given by~$(\gamma,g) \mapsto \gamma g$ and~$(\gamma, g)\mapsto g \gamma^{-1}$ commute. Therefore~$P = L_{\dR}$ inherits an action of~$\Gamma$ by~$K$-schemes isomorphisms. We claim that this action of~$\Gamma$ on~$P$ is crystalline and
\[P_{\dR} = G_{\dR}.\]
Here and afterwards for any~$K$-scheme~$X = \Spec A$ with a crystalline~$K$-linear action of~$\Gamma$ we write~$X_{\dR}= \Spec {(A \otimes_K B)^\Gamma}$ as in \cref{ex:NotForgettingCoefficients}. Once this claim is proved we set~$h_G$ to be the filtration on~$K[G_{\dR}]$  induced by the Hodge filtration on~$P_{\dR}$ via the above identification. 

To prove the claim pick a faithful representation~$i \colon G \into \GL(V)$. Then~$i$ is~$\Gamma$-equivariant both for the actions of left and right multiplication. By looking first at the action by left multiplication, the morphism~$i$ induces a closed embedding
\[i \colon P= L_{\dR}\intoo \Iso_K(V \otimes_{\bbQ_p} K, V_{\dR}). \]
Here for finite dimensional~$K$-vector spaces~$W$ and~$W'$ we write~$\Iso_K(W, W')$ for the~$K$-scheme whose points with values in a~$K$-algebra~$A$ are isomorphisms of~$A$-modules~$W \otimes_K A \to W' \otimes_K A$. Note that~$P_B$ is the orbit under composition by~$G_{\dR, B}$ of the~$B$-point of~$\Iso_K(V \otimes_{\bbQ_p} K, V_{\dR})$ given by the canonical isomorphism
\[ V \otimes_{\bbQ_p} B \iso V_{\dR} \otimes_K B.\]
Now the action of~$\Gamma$ on~$G$ and~$\GL(V)$ by multiplication on the right commutes with the action by multiplication on the left, therefore it induces an action of~$\Gamma$ on~$P$ and~$ \Iso_K(V \otimes_{\bbQ_p} K, V_{\dR})$ by~$K$-scheme isomorphisms. Explicitly, the action on~$\Gamma$ on~$ \Iso_K(V \otimes_{\bbQ_p} K, V_{\dR})$ is just given via its linear action on~$V$. In particular the morphism
\[ \Iso_K(V \otimes_{\bbQ_p} K, V_{\dR}) \intoo \bbV(H \oplus \det H^\vee) \quad \textup{where~$H = \Hom_K(V \otimes_{\bbQ_p} K, V_{\dR})$}\]
is~$\Gamma$-equivariant where~$\Gamma$ acts on the target via its linear action on~$V$. This shows that the action of~$\Gamma$ on~$P$ is crystalline. The induced morphism~$i$ is~$\Gamma$-equivariant with respect to these actions and therefore induces a closed embedding
\[ i_{\dR} \colon P_{\dR} \intoo \Iso_K(V_{\dR}, V_{\dR}) = \GL(V_{\dR}).\]
Note that~$P_{\dR,B}$ is the orbit under composition by~$G_{\dR, B}$ of~$\id \in \GL(V_{\dR, B})$. In other words~$P_{\dR} = G_{\dR}$ as desired.

\medskip 

\emph{Hopf property.} Note that the morphism~$m$ is~$\Gamma$-equivariant when~$\Gamma$ acts as above on~$P$ and trivially on~$G_{\dR}$. The induced morphism
\[ m_{\dR} \colon G_{\dR} \times_K P_{\dR} = G_{\dR} \times_K G_{\dR} \too P_{\dR} = G_{\dR}\] 
is the group law of~$G_{\dR}$. Let~$\mu \colon K[G_{\dR}] \to K[G_{\dR}] \otimes_K K[G_{\dR}]$ be the~$K$-algebra morphism induced by~$m_{\dR}$. Since~$\mu$ is injective, the exactness of the Hodge filtration and its compatibility with tensor product imply
\[ h^\alpha K[G_{\dR}] = \mu^{-1}(K[G_{\dR}] \otimes_K h^\alpha K[G_{\dR}]) \quad \textup{for all~$\alpha\in\bbQ$},\]
that is, the filtration~$h_G$ satisfies \eqref{eq:HopfFiltrationComult}. To prove \eqref{eq:HopfFiltrationMult} consider the morphism
\[
\begin{tikzcd}[row sep=0pt]
n \colon G_{\dR} \times_K G_{\dR} \times_K P \ar[r] & G_{\dR} \times_K P \times_K G_{\dR}\times_K P \ar[r, "\mu \times \mu"]& P \times_K P,
\end{tikzcd}
\]
where the first map is given by~$(g,h,x) \mapsto (g,x,h,x)$ and the second one is given by~$(g,x,h,y) \mapsto (gx, hy)$.
The above composite morphism is~$\Gamma$-equivariant when~$\Gamma$ acts on~$P$ as above and trivally on~$N_{\dR}$. The morphism of~$K$-algebras \[\nu \colon K[G_{\dR}] \otimes K[G_{\dR}] \too K[G_{\dR}] \otimes K[G_{\dR}]  \otimes K[G_{\dR}]\]  induced by~$n_{\dR}$ via the equality~$P_{\dR} = G_{\dR}$ is the one appearing in \eqref{eq:HopfFiltrationMult}  for~$G_{\dR}$. Since~$\nu$ is injective, again the exactness of the Hodge filtration and its compatibility with tensor product imply
\begin{align*} 
\nu^{-1}(K[G_{\dR}] \otimes K[G_{\dR}]  \otimes h^\alpha K[G_{\dR}]) &= h^\alpha(K[G_{\dR}] \otimes K[G_{\dR}]) \\
&= \sum_{\alpha= \beta + \gamma} h^\beta K[G_{\dR}] \otimes h^\gamma K[G_{\dR}],
\end{align*}
for all~$\alpha\in \bbQ$. This concludes the proof that the filtration~$h_G$ is Hopf.

\medskip 

\emph{Universal property.} Any representation of~$G$ is the union of its finite dimensional subrepresentations so it suffices to prove the statement for any finite dimensional representation~$\rho \colon G \to \GL(V)$. The morphism of schemes \[\bbV(V^\vee) \times_{\bbQ_p} L \too \bbV(V^\vee), \quad (f,g) \longmapsto fg\] is~$\Gamma$-equivariant when~$\Gamma$ acts on~$\bbV(V^\vee) \times_{\bbQ_p} L$ by~$(f,g)\mapsto (f \gamma^{-1}, \gamma g)$ and trivially on~$\bbV(V^\vee)$. Therefore it induces a morphism of~$K$-schemes
\[ s \colon  \bbV(V_{\dR}^\vee) \times_K P = \bbV(V_{\dR}^\vee) \times_K L_{\dR} \too \bbV(V^\vee \otimes_{\bbQ_p}K).\]
The above action of~$\Gamma$ on~$\bbV(V^\vee) \times_{\bbQ_p} L$ commutes with the action of~$\Gamma$ induced by right multiplication on~$L$. Therefore~$s$ is~$\Gamma$-equivariant when~$\Gamma$ acts on~$P$ as above, on~$V^\vee$ via the representation contragradient to~$\rho$ and trivially on~$\bbV(V_{\dR}^\vee)$. The induced morphism 
\[s_{\dR} \colon \bbV(V_{\dR}^\vee) \times_K G_{\dR} = \bbV(V_{\dR}^\vee) \times_K P_{\dR}  \too \bbV(V_{\dR}^\vee)\]
is the dual action of~$G_{\dR}$ on~$V_{\dR}$. The morphism~$s_{\dR}$ is defined by the~$K$-linear~map
\[ \sigma \colon V_{\dR} \too V_{\dR} \otimes K[G_{\dR}]\]
defining the action of~$N_{\dR}$ on~$V_{\dR}$. Since~$\sigma$ is injective, once again the exactness of the Hodge filtration and its compatibility with tensor product imply
\[ h^\alpha V_{\dR} = \sigma^{-1}(V\otimes h^\alpha K[G_{\dR}]),\qquad \textup{for all }\alpha\in\bbQ.\] 
The right-hand side of the previous identity is by definition the filtration induced by the Hopf filtration~$h_V$; see the discussion preceding \cref{Prop:FiltrationFunctorsAsHopfFiltrations}. Thus~$h_V$ is the filtration induced by~$h_G$.

\medskip

(2) Let~$\sigma \colon H \into \GL(W)$ be a faithful representation of~$H$. Then~$\sigma \circ f \circ \rho$ is a crystalline representation and~$\sigma$ is~$\Gamma$-equivariant for the right multiplication by~$\Gamma$. It follows that the action by right multiplication of~$\Gamma$ on~$H$ is crystalline. Moreover
the induced map
\[ \{ \textup{Hopf filtrations on~$K[H_{\dR}]$}\} \too  \{\textup{filtrations on~$W_{\dR}$} \} \]
is injective by \cref{Prop:FiltrationFunctorConnectedComponent}. The Hopf filtration~$h_H$ and the filtration induced by~$h_G$ on~$K[H_{\dR}]$ have both image~$h_W$, thus they coincide.

\medskip

(3) Let~$N \subset G$ be the normalizer of the parabolic subgroup~$Q$. Since~$\Gamma$ normalizes~$Q$ the image of~$\rho$ is contained in~$N$. It follows from (2) that the filtrations~$h_G$ is induced by the Hopf filtration~$h_N$ of~$K[N_{\dR}]$. Now~$Q = N^\circ$ thus the induced map
\[ \{ \textup{Hopf filtrations on~$K[Q_{\dR}]$}\} \too  \{ \textup{Hopf filtrations on~$K[N_{\dR}]$}\} \]
is bijective by \cref{Prop:FiltrationFunctorConnectedComponent}. Therefore~$h_N$ comes from a Hopf filtration on~$K[Q_{\dR}]$. The last statement follows from \cref{Rmk:ProjectionAsHopfFiltration}.
\end{proof}

\subsection{From \'etale to de Rham} The goal of this section is showing that~$p$-adic Hodge theory transforms~$\Gamma$-\'etale data into de Rham data. Let~$Z$ be a finite \'etale~$\bbQ_p$-scheme with an action of~$\Gamma$. Suppose that the action is unramified and that \[R:= \Gamma(Z, \cO_Z)\] is a product of unramified extensions of~$\bbQ_p$. The~$K$-scheme~$Z_{\dR}$ is finite \'etale over~$K$ by \cref{lemma:DeRhamificationScheme} and by \cref{Ex:CristallineActionOnFiniteScheme} the~$K$-algebra~$R_{\dR}$ is then a product of unramified extensions of~$\bbQ_p$. 

\begin{definition} A~$\Gamma$-\'etale datum~$D = (V, G, \rho)$ on~$Z$ is \emph{transportable} if the representation~$V$ is crystalline and the image of~$\rho \colon \Gamma \to \GL(V_0)$ is contained in a subgroup of~$\GL(V_0)$ containing~$\Res_{Z/\bbQ_p} G$ as a finite index subgroup.
\end{definition}

When the representation~$V$ is crystalline, if the subgroup~$G \subset \GL(V)$ is of finite index in its normalizer, then~$D$ is transportable by \cref{lemma:NormalizerAndWeilRestriction}. Given a~$\Gamma$-\'etale datum~$(V,G,\rho)$ on~$Z$ we let~$\Gamma$ act on~$G$ by conjugation and consider the group scheme~$G_{\dR}$ over~$Z_{\dR}$.

\begin{proposition} \label{prop:DescentViaPAdicHodgeTheory} Let~$D = (V,G,\rho)$ be a transportable~$\Gamma$-\'etale datum on~$Z$. Then
\[ D_\dR := (V_{\dR}, G_{\dR}, \phi_V, h_V)\]
is a de Rham datum on~$Z_{\dR}$ and the following properties hold:
\begin{enumerate} 
\item For any filtration~$x$ on~$D$ there is a unique filtration~$x_{\dR}$ on~$D_\dR$ such that
$\cB(\Res_{Z/\bbQ_p}G, \Frac B) \iso \cB(\Res_{Z_\dR/K}G_{\dR}, \Frac B)$ carries~$x$ to~$x_{\dR}$; \smallskip
\item For any grading~$V^{(\bullet)}$ on~$D$ there is a unique grading~$V^{(\bullet)}_\dR$ on~$D_\dR$ such that 
the isomorphism~$V \otimes_{\bbQ_p} B \iso V_{\dR} \otimes_{K} B$ carries~$V^{(\bullet)} \otimes_{\bbQ_p} B$ to~$V^{(\bullet)}_\dR \otimes_K B$.  
\smallskip
\end{enumerate}
The above constructions are functorial and give rise to a commutative diagram
\[
\begin{tikzcd}
{\left\{ \begin{array}{c}
\textup{transportable} \\
\textup{$\Gamma$-\'etale data on~$Z$}
\end{array} \right\}} \ar[r, "{\dR}"]
& {\left\{ \begin{array}{c}
\textup{de Rham data on~$Z_{\dR}$}
\end{array} \right\}} 
\\
{\left\{ \begin{array}{c}
\textup{filtered transportable} \\
\textup{$\Gamma$-\'etale data on~$Z$}
\end{array} \right\}} \ar[u,"\textup{forgetful}"] \ar[d, "\gr"'] \ar[r, "{\dR}"]
& {\left\{ \begin{array}{c}
\textup{filtered} \\
\textup{de Rham data on~$Z_{\dR}$}
\end{array} \right\}} \ar[u, "\textup{forgetful}"'] \ar[d, "\gr"]
\\
{\left\{ \begin{array}{c}
\textup{graded transportable} \\
\textup{$\Gamma$-\'etale data on~$Z$}
\end{array} \right\}} \ar[r, "{\dR}"]
& {\left\{ \begin{array}{c}
\textup{graded} \\
\textup{de Rham data on~$Z_{\dR}$}
\end{array} \right\}} 
\end{tikzcd}
\]
where arrows in the above categories are taken to be isomorphisms.
\end{proposition}

\begin{proof}[{Proof of \cref{prop:DescentViaPAdicHodgeTheory}}] In order to prove that~$D_{\dR}$ is a de Rham datum we have to show that:

\medskip

(a)~$V_{\dR}$ is a vector bundle over~$Z_{\dR}$: This is  \cref{lemma:deRhamificationModules}. 

\medskip

(b)~$G_{\dR} \subset \GL(V_{\dR})$ is a reductive group~$Z_{\dR}$-scheme: Being reductive can be tested after extending scalars to~$B$ where the property is true because~$G \times_{\bbQ_p} B$ is reductive over~$Z \times_{\bbQ_p} B$.

\medskip

(c) the Frobenius operator~$\phi_V$ on~$V_{\dR}$ is semilinear and normalizes~$G_{\dR}$ in the sense of \cref{sec:deRhamData}: Indeed~$\phi_G = \phi_{\GL(V) \rvert G}$ by functoriality of the construction of the Frobenius operator. On other hand~$\phi_{\GL(V)}$ is the conjugation by~$\phi_V$.

\medskip

(d) The Hodge filtration~$h_V$ on~$V_{\dR}$ belongs to~$$\cB(G_{\dR}, Z_{\dR}) \subset \cB(\GL(V_{\dR}), Z_{\dR}):$$ Since~$D$ is transportable there is a subgroup~$N \subset\GL(V_0)$ containing the image of~$\rho$ and having~$\Res_{Z/\bbQ_p} G$ as a subgroup of finite index, where~$V_0$ stands for~$V$ as~$\bbQ_p$-vector space. By \cref{Lemma:HodgeFiltrationComesFromASubgroup} (1) the Hodge filtration~$h_V$ comes from a Hopf filtration~$h_N$ on~$K[N_{\dR}]$. Now the group~$N_{\dR}$ is reductive and contains~$\Res_{Z_{\dR}/K} G_{\dR}$ as a subgroup of finite index. In particular the induced map
\[ \cB(G_{\dR}, Z_{\dR}):= \cB(\Res_{Z_{\dR}/K} G_{\dR}, K) \too \cB(N_{\dR}, K)\]
is bijective, so~$h_N$ comes from a necessarily unique point in~$\cB(G_{\dR}, Z_{\dR})$.

\medskip

Now statement (1) is \cref{Lemma:StupidPropertiesGroupsWithCristallineAction} (3). For statement (2) note that by \cref{Ex:CristallineActionOnFiniteScheme} any point~$z \in Z_{\dR}$ can be seen as a~$\Gamma$-orbit~$Z' \subset Z(\bar{K})$. The grading on~$V_{\dR}$ is then defined by
\[ V_{\dR, z}^{(\alpha)} := \left(\sum_{z' \in Z'} V_{z'}^{(\alpha)} \otimes_{\bbQ_p} B \right)^\Gamma \]
for~$\alpha \in \bbQ$. It is clear by \cref{lemma:deRhamificationModules} that isomorphisms of (filtered, graded)~$\Gamma$-\'etale data are transformed into isomorphisms of (filtered, graded) de Rham data. The commutativity of the upper square is clear, while that of the lower one follows from the exactness of the functor~$W \rightsquigarrow (W_{\dR}, \phi_W, h_W)$ plus \cref{Lemma:StupidPropertiesGroupsWithCristallineAction} (1) and (2) and \cref{Lemma:HodgeFiltrationComesFromASubgroup} (3).
\end{proof}

\begin{example} \label{Ex:ComparisonEtaleLocalSystemDeRhamFlatBundles} Let~$\cA$ be an abelian scheme over~$\cO_K$ with generic fiber~$A$ and let~$r \ge 1$ be an integer nondivisibile by~$p$. As in \cref{Ex:LocSysAbVar} we see~$A[r](\bar{K})$ as constant group scheme over~$\bbQ_p$  and consider its Cartier dual
\[ Z:= A[r](\bar{K})^\ast = \Hom(A[r](\bar{K}), \Gm).\]
Since~$p \not \mid r$ the action of~$\Gamma$ on~$Z$ is unramified and~$R =\Gamma(Z,\cO_Z)$ is a product of unramified extensions of~$\bbQ_p$. Moreover
\[ \Ker (\Gamma \to \Aut Z)=\Gal(\bar{K}/K') \]
where~$K' \subset \bar{K}$ is the unramified extension of~$K$ over which all~$r$-torsion points are rational. Therefore~$A[r]_{K'}$ is the constant group scheme with value~$A[r](K')$,
\[ Z_{K'} = A[r]^\ast_{K'}, \]
and  by \cref{Ex:CristallineActionOnFiniteScheme} we conclude that 
\[ Z_{\dR} = Z_{K'} / \Gal(K'/K) = A[r]^\ast. \]
Let~$A^\ast$ be the dual abelian variety and~$A^\natural$ the universal vector extension of~$A^\ast$. The morphism~$\pi \colon A^\natural \to A^\ast$ of algebraic groups forgetting the connection on a line bundle induces an isomorphism
\[ A^\natural[r] \stackrel{\sim}{\too} A^\ast[r] = A[r]^\ast= Z_{\dR}. \]
Let~$E$ be the~$p$-adic \'etale local system  of \cref{Ex:LocSysAbVar} and~$(\cE, \nabla)$ the flat vector bundle  of \cref{Ex:FlatBundlesOnAbelianVarieties} on~$A$. Then
\[ E \otimes_{\bbQ_p} \cO_{A}^\et \iso \cE^\et\]
where the superscript `$\et$' stands for the associated \'etale sheaf. Via this isomorphism the connection~$\nabla$ on~$\cE^\et$ is induced by the derivation on~$\cO_A^\et$. Let~$f \colon \cX \to \cA$ be a morphism of finite type~$\cO_K$-schemes with~$\cX$ smooth of relative dimension~$d$ and generic fiber~$X = \cX_K$. By \cite[th. 5.6]{FaltingsCristallineEtale} we have an isomorphism
\[ \rH^d_{\et}(X_{\bar{K}}, f^\ast E)_{\dR} \iso \rH^d_{\dR}(X / K; f^\ast( \cE, \nabla))\]
of vector bundles over~$Z_{\dR}$ compatible with the Poincar\'e pairing. In other words the~$\Gamma$-\'etale datum~$D = (V, G, \rho)$ on~$Z$ of \cref{Ex:LocSysAbVar} is transportable because~$G$ coincides with its normalizer in $\GL(V)$ and the de Rham datum~$D_{\dR}$ on~$Z_{\dR}$ is isomorphic to that in \cref{Ex:FlatBundlesOnAbelianVarieties}.
\end{example}

\subsection{Positivity from global purity} Recall that a number field~$F$ is CM if it is a totally imaginary quadratic extension of a totally real number field~$F_\bbR$. Let~$K$ be a number field. If~$K$ contains a CM subfield, then there is a largest one~$K^{\textup{CM}} \subset K$. Following \cite[2.4]{LV} a place~$v$ of~$K$ is called \emph{friendly} if~$v$ is unramified and, in case~$K$ contains a CM subfield,~$v$ lies above a place of~$(K^{\textup{CM}})_\bbR$ which is inert in~$K^{\textup{CM}}$. Let~$v$ be a~$p$-adic place of~$K$,~$K_v$ the completion of~$K$ at~$v$,~$\bar{K}_v$ an algebraic closure of~$K_v$ and~$\bar{K}$ the algebraic closure of~$K$ in~$\bar{K}_v$. We have a natural inclusion
\[ \Gamma_v = \Gal(\bar{K}_v/K_v)\intoo \Gamma = \Gal(\bar{K}/K).\]
Let~$Z$ be a finite \'etale~$\bbQ_p$-scheme such that~$R:= \Gamma(Z, \cO_Z)$ is a product of unramified extensions of~$\bbQ_p$. We assume that~$Z$ is equipped with a continuous action of~$\Gamma$ and that the induced action of~$\Gamma_v$ is unramified. Given a~$\Gamma$-\'etale datum on~$Z$ we obtain a~$\Gamma_v$-\'etale datum by restriction of the underlying representations; hence in the crystalline case we get an associated de Rham datum on~$Z_\dR$ by \cref{prop:DescentViaPAdicHodgeTheory}. For a point~$z\in Z_\dR$ we denote by~$k_z$ the residue field of~$z$.

\begin{proposition} \label{Prop:PositivityFromGlobalPurity} Suppose~$v$ is friendly and
let~$D = (V,G,\rho,x)$ be a filtered~$\Gamma$-\'etale datum with~$G \subset \GL(V)$ of finite index in its normalizer. Assume that the underlying representation is pure of some weight, and crystalline at all~$p$-adic places. Let~$D_\dR = (V_\dR,G_\dR, \phi, h, x_\dR)$ be the associated de Rham datum at~$v$. Then there exists~$z \in Z_{\dR}$ such that 
\[h_z \in \cB(G_{\dR, z},  k_z) \quad \text{is positive with respect to} \quad \Stab_{G_{\dR, z}}(x_{\dR, z})^\circ.\] 
\end{proposition}

\begin{proof} Let~$Q = \Res_{Z/\bbQ_p} \Stab_G(x)$ be the Weil restriction of the stabilizer of~$x$.
Since~$x$ is fixed under the action of~$\Gamma$, the representation~$\rho \colon \Gamma \to \GL(V_0)$ takes values in the normalizer~$N$ of the subgroup~$Q \subset \GL(V_0)$, where~$V_0$ stands for~$V$ seen as a~$\bbQ_p$-vector space. 
The action of~$N$ by conjugation on~$Q$ preserves its identity component and its unipotent radical. Therefore the adjoint action of~$Q$ on the Lie algebra~$W:= \Lie \rad Q$ extends to a representation
\[ \sigma \colon N  \too \GL(W).\] 
The composite representation~$\sigma \circ \rho \colon \Gamma \to \GL(W)$ is pure of weight zero: Indeed, it is a subrepresentation of the adjoint representation~$\End(V_0)$ of~$\Gamma$, and the latter is pure of weight zero because~$V_0$ is pure of some weight. Note that the restriction of~$D$ to~$\Gamma_v$ is transportable because~$G \subset \GL(V)$ is of finite index in its normalizer. Since the place~$v$ is friendly, it then follows by \cite[lemma 2.9]{LV} that the weight of the Hodge filtration~$h_W$ induced by~$p$-adic Hodge theory on~$W_\dR = \Lie \rad Q_{\dR}$
is 
\begin{equation} \label{Eq:HodgeFiltrIsWeight0}\wt(h_W) = 0. \end{equation}
By \cref{Lemma:HodgeFiltrationComesFromASubgroup} (1) the filtrations~$h=h_V$ and~$h_W$ are both induced by the Hopf filtration~$h_N$ on~$K[N_{\dR}]$. By \cref{Lemma:HodgeFiltrationComesFromASubgroup} (3) the Hopf filtration on~$K[N^\ss_{\dR}]$  induced by~$h_N$ is \[\pr_{Q^\circ_{\dR}}(h_V)\] where~$N^\ss_{\dR} = N_\dR / \rad Q_{\dR}~$: Indeed~$Q^{\circ, \ss}_{\dR} = Q^\circ_{\dR} / \rad Q_{\dR}$ is a subgroup of finite index in~$N_{\dR}^\ss$ hence Hopf filtrations on Hopf algebras of these groups can be identified by \cref{Prop:FiltrationFunctorConnectedComponent}. Now the representation~$\sigma$ is~$\Gamma$-equivariant and the induced morphism
\[ \sigma_{\dR} \colon  N_{\dR} \too \GL(W_{\dR})\]
is the adjoint representation. Its determinant~$\det \sigma_{\dR}$ factors through the modular character
\[ \delta \colon N_{\dR}^\ss \too \GL(\det W_{\dR}).\]
Therefore \eqref{Eq:HodgeFiltrIsWeight0} implies
\[ \langle \pr_{Q^\circ_{\dR}}(h_V), \delta \rangle = \wt(h_W) = 0.\]
In other words the point~$h \in\cB(G_{\dR}, Z_{\dR})$ is positive with respect to the parabolic subgroup 
\[ Q_{\dR}^\circ = \prod_{z \in Z_{\dR}} \Res_{ k_z/\bbQ_p} \Stab_{G_{\dR, z}} (x_{\dR, z})^\circ.\]
Hence there is~$z \in Z_{\dR}$ such that~$h_z \in \cB(G_z, k_z)$ is positive with respect to the parabolic subgroup~$\Res_{ k_z/\bbQ_p} \Stab_{G_{\dR, z}} (x_{\dR, z})^\circ$; see \cref{Ex:PositiveCocharacterProduct}. By \cref{Ex:PositiveCocharacterWeilRestriction} this is equivalent to saying that~$h_z$ is positive with respect to~$\Stab_{G_{\dR, z}} (x_{\dR, z})^\circ$ which concludes the proof.
\end{proof}

\section{Combinatorics on flag varieties} \label{sec:OrbitsFlagVarieties}

In this section we discuss a dimension estimate on orbits of flag varieties that will allow us to apply the Bakker-Tsimerman theorem later on. Let~$k$ be a perfect field.

\subsection{The adjoint polygon} Let~$x \in \cB(\GL(V),  k)$ where~$V$ is a~$k$-vector space of dimension~$d$, and let~$F^\bullet V$ be the corresponding filtration. Denote by~$\alpha_1 > \cdots > \alpha_n$ the rational numbers for which~$\gr^{\alpha_i} V \neq 0$.  The \emph{polygon} of~$x$ is the continuous piecewise affine function~$\upsilon_x \colon [0, d] \to \bbR$ such that~$\upsilon_x(0) = 0$ and, for~$i = 1, \dots,n$, the slope of~$\upsilon_x$ on the interval~$[\dim F^{\alpha_{i-1}} V, \dim F^{\alpha_i} V]$ is~$\alpha_i$ where~$\dim F^{\alpha_{0}} V = 0$. The function~$\upsilon_x$  is by definition concave and satisfies~$\upsilon_x(d) = d \wt(x)$,~$\upsilon_{cx} = c \upsilon_x$ for all rational numbers~$c \ge 0$,~$\upsilon_{gx} = \upsilon_x$ for all~$g \in \GL(V)( k)$ and \[\upsilon_{-x}(t) = \upsilon_x(d - t) - d \wt(x).\]

\begin{definition} Let~$G$ be a reductive group over~$k$ and~$x \in \cB(G,  k)$. The \emph{adjoint polygon} is the function 
\[ \Upsilon_{G,x} := \upsilon_{\ad_\ast x} \colon \quad [0, \dim G] \; \too \; \bbR\]
where~$\ad \colon G \to \GL(\Lie G )$ is the adjoint representation. If the group is clear from the context we will omit it.
\end{definition}

The adjoint polygon~$\Upsilon_{G,x}$  is concave and satisfies~$\Upsilon_{G,cx} = c \Upsilon_{G,x}$ for all rational numbers~$c \ge 0$,~$\Upsilon_{G,gx} = \Upsilon_{G,x}$ for all~$g \in G( k)$,~$\Upsilon_{G,x} = \Upsilon_{G^\circ,x}$ and \[\Upsilon_{G,x}(t) = \Upsilon_{G,x}(d - t),\]
thus~$\Upsilon_{G,x}(\dim G) = 0$. The above identity follows from~$\ad_\ast(-x) = \ad_\ast x$. In \eqref{Eq:ExpressionAdjointPolygon} we will prove an explicit formula of~$\Upsilon_{G, x}$ in terms of roots.

\subsection{The main dimension estimate} \label{sec:StatementOrbitInFlag} 
Let~$G$ be reductive group and~$\alpha$ an automorphism of~$G$ as an algebraic group. In practice~$\alpha$ will be the conjugation by a Frobenius operator obtained by comparison between crystalline and de Rham cohomology. To state the main result we need some more notation.

\begin{definition} \label{Def:BalancedRelation} Let~$x, x' \in \cB(G,  k)$ and~$Q, Q' \subset G^\circ$ be parabolic subgroups. The couples~$(x, Q)$ and~$(x', Q')$ are \emph{$\alpha$-equivalent} if there is a point~$g \in G( k)$ such that
\begin{align*} 
Q' = g Q g^{-1}, && 
g \alpha(g)^{-1} \in \rad Q', &&
 \pr_{Q'}(gx) = \pr_{Q'}(x'),
\end{align*}
where~$\pr_{Q'} \colon \cB(G,  k) \to \cB(Q^\ss,  k)$ is the projection defined in \cref{sec:ProjectionLeviQuotientBuildings}. If this is the case we write~$(x,Q) \sim_\alpha (x', Q')$.
\end{definition}

 If~$(x, Q) \sim_\alpha (x', Q')$, then~$\alpha (Q) =Q$ if and only if~$\alpha(Q') = Q'$. In our application the above equivalence relation will arise by taking semisimplification of global Galois representations and then applying~$p$-adic Hodge theory in a prime of good reduction, see \cref{isomorphic-graded-de-rham}. Since we will not be able to have big monodromy for a tuple of characters that form a full local Galois orbit, we will have to project onto certain direct factors. To account for this, consider an epimorphism~$f \colon G \to \bar{G}$  of reductive groups. Let
\[
\begin{array}{rcl}
\Par(G^\circ) &\too& \Par(\bar{G}^\circ), \\[0.3em]
P &\longmapsto& \bar{P} := f(P)
\end{array}
\qquad 
\begin{array}{rcl}
\cB(G,  k) &\too& \cB(\bar{G},  k), \\[0.3em]
x &\longmapsto& \bar{x} := f_\ast x
\end{array}
\]
be the induced maps. Note that with the above notation we have~$\bar{P}_x = P_{\bar{x}}$.

\begin{definition} \label{Def:BizarreEqRel}Let~$x \in \cB(G,  k)$ and~$Q \subset G^\circ$ be a parabolic subgroup stable under~$\alpha$. Consider 
\[ 
[ x,Q ]_{\alpha} \; := \; \left\{ (P^\circ_{x'}, Q') \mid  (x, Q) \sim_\alpha (x', Q')   \textup{ and }  (\bar{x}' , \bar{Q}') \succeq 0\right\} \; \subset \; \Par(G^\circ)^2 
\]
where~$(\bar{x}', \bar{Q}') \succeq 0$ means that~$\bar{x}'$ is~$\bar{Q}'$-positive in the sense of \cref{Def:PositiveCouple}. Note that the notion of~$\alpha$-equivalence depends on the group~$G$ and not only on its identity component. We will also write~$[x, Q]_{G, \alpha} := [x,Q]_\alpha$ to emphasize this.
\end{definition}

With notation as in~\cref{sec:NotationReductiveGroups}, consider the locus~$\Par_{t(x)}(G^\circ) \subset \Par(G^\circ)$ of parabolics of type~$t(x)$. Let~$\frH \subset \Par_{t(x)}(G^\circ)$ be a smooth subvariety such that the composite map 
\[ \frH \intoo \Par_{t(x)}(G^\circ) \ontoo \Par_{t(\bar{x})}(\bar{G}^\circ)\]
is a smooth proper morphism with equidimensional fibers. The main result of this section is an upper bound for the codimension of the Zariski closure of~$p([ x,Q]_\alpha) \cap \frH$ in~$\frH$ where
\[ p \colon \Par(G^\circ) \times  \Par(G^\circ) \too  \Par(G^\circ), \quad (P, Q) \longmapsto P, \] 
is the first projection. Note that~$[ x,Q ]_\alpha$ is stable under the action of the elements of~$G( k)$ fixed under~$\alpha$ so the dimension of the latter will naturally play a role. In fact we obtain a bound in terms of the 
semisimple part~$\alpha^\ss$ in the Jordan-Chevalley decomposition of~$\alpha$
seen as an element of the algebraic group~$\Aut(G)$. Let \[C := C_{G}(\alpha^\ss) = \{ g \in G : \alpha^\ss(g) = g \} \subset G.\]

\begin{theorem} \label{Thm:BoundInFlag} Let~$x \in \cB(G, k)$ and~$Q \subset G^\circ$ a parabolic subgroup with~$\alpha(Q) =~Q$ and~$(\bar{x}, \bar{Q}) \succeq 0$. Let~$n \ge  1$ be an integer such that 
\begin{gather}
n \le \dim \Par_{t(\bar{x})}(\bar{G}^\circ), \label{Eq:EasyInequality} \\
 \Upsilon_{\bar{G},\bar{x}}(n -1) + \Upsilon_{\bar{G},\bar{x}}(n - 1+ \tfrac{1}{2}(\dim \bar{P}_x^\ss - \rk \bar{G})) < \max \Upsilon_{\bar{G},\bar{x}}.  \label{Eq:HardInequalityInPractice}
\end{gather}
Then the Zariski closure of~$p([ x,Q ]_\alpha) \cap \frH$ in~$\frH$ has codimension~$\ge n - \dim C$.
\end{theorem}

As they stand inequalities \eqref{Eq:EasyInequality} and \eqref{Eq:HardInequalityInPractice} are in general hard to verify. In practice we will apply the theorem in a situation where the quotient~$\bar{G}$ decomposes as a high enough power of some given reductive group~$H$, in which case~\cref{Thm:BoundInFlag} leads to the following bound: 

\begin{corollary} \label{Cor:UniformBoundInFlag} Fix a finite subset~$I \subset \bbQ$ and integers~$n, d\ge 1$. Then there exists in integer~$N_0 = N_0(I,n, d) \ge 1$ independent of~$k$ with the following property: Suppose we are given
\begin{itemize}  \setlength\itemsep{0.2em}
\item a reductive group~$H$ of dimension~$d$,
\item a point~$y \in \cB(H,  k)$ such that all slopes of the adjoint polygon
$\Upsilon_{H, y}$ lie in~$I$ and
 \[ 2 \dim \Stab_H(y)^\ss < \dim H + \rk H,\]
\item a reductive group~$G$ over~$k$ and~$\alpha \in \Aut(G)$ with
\[
 \dim C_{G}(\alpha^\ss) \le \dim H,
\]
\item a surjective morphism~$f\colon G\to \bar{G} := H^N$ for some integer~$N\ge N_0$, 
\item a point~$x\in \cB(G,  k)$ with~$f_*(x)=\Delta_*(y)$ for the diagonal~$\Delta \colon H\hookrightarrow \bar{G}$, 
\item a parabolic subgroup~$Q\in \Par(G^\circ)$ stable under~$\alpha$ and such that \[(\bar{x}, \bar{Q}) := (f_*(x), f(Q)) \succeq 0,\] 
\item a subvariety~$\frH \subset \Par_{t(x)}(G^\circ)$ such that~$\frH \to \Par_{t(\bar{x})}(\bar{G}^\circ)$ is a smooth proper morphism with equidimensional fibers.
\end{itemize} 
Then the Zariski closure of~$p( [x,Q]_\alpha) \cap \frH$ in~$\frH$ has codimension~$\ge n - \dim H$.
\end{corollary}

\begin{proof} \label{Rmk:HardInequalityInPractice} 
Put~$\epsilon = (n - 1)/N$, then \eqref{Eq:EasyInequality} and \eqref{Eq:HardInequalityInPractice} will be implied by the two inequalities
 \begin{gather*}
\epsilon + \tfrac{1}{N}\le \dim \Par_{t(y)}(H^\circ), \\
 \Upsilon_{H,y}(\epsilon) + \Upsilon_{H,y}(\epsilon + \tfrac{1}{2}(\dim P_y^\ss - \rk H)) < \max \Upsilon_{H,y}.  
\end{gather*}
Since we fixed the dimension of~$H$ and the set of slopes note that there are only finitely many possibilities for~$\Upsilon_{H, y}$,~$\dim \Par_{t(y)}(H^\circ)$,~$\dim P_y^\ss$ and~$\rk H$. For~$N$ big enough the first of the above inequalities is always satisfied, and the second one is equivalent to~$\tfrac{1}{2} (\dim P_y^\ss - \rk G) < \dim \rad P_y$: Indeed the function~$\Upsilon_{H,y}$ is increasing on the interval~$[0, \dim \rad P_y]$ by the formula for the adjoint polygon in terms of weights in \eqref{Eq:ExpressionAdjointPolygon} below; since we have~$\dim \rad P_y = \tfrac{1}{2}(\dim H - \dim P_y^\ss)$, the second inequality then becomes equivalent to~$2 \dim P_y^\ss < \dim H + \rk H$.
 \end{proof}
 
The rest of this section will be devoted to the proof of \cref{Thm:BoundInFlag}.

\subsection{Strategy of the proof} One easily reduces to the case where the group~$G$ is connected: Indeed, it is easy to see that~$[x, Q]_{G, \alpha}$ is finite union of sets~$[x', Q']_{G^\circ, \alpha}$ with~$x' = gx$ and~$Q'= g Q g^{-1}$ for representatives~$g \in G( k)$ of~$G( k)/G^\circ( k)$. Since moreover the assumptions of \cref{Thm:BoundInFlag} only depend on the identity component of~$G$, we will for the rest of the proof assume that~$G$ is \emph{connected}. Extending scalars only makes the set~$[x, Q]_\alpha$ bigger, so we will also assume~$k$ to be \emph{algebraically closed} and hence identify varieties with the set of their~$k$-points.

\medskip 

Now recall that the definition of~$[x,Q]_\alpha$ involves two conditions: A positivity condition and a condition on~$\alpha$-equivalence. 
For the proof of \cref{Thm:BoundInFlag} we will construct a morphism from~$\Par(G^\circ) \times \Par(G^\circ)$ to another variety such that the dimension of the image~$[x,Q]_\alpha^\ss$ of~$[x,Q]_\alpha$ will be bounded using~$\alpha$-equivalence only; the dimension of the fibers will instead be bounded using positivity, for which the automorphism~$\alpha$ will no longer play a role.

\medskip

To construct the morphism, let~$\cQ$ be the universal parabolic subgroup over the variety of parabolic subgroups~$\Par(G)$. The quotient~$\cQ^\ss := \cQ / \rad \cQ$ is a reductive group scheme over~$\Par(G)$ with connected fibers. Therefore the scheme of its parabolic subgroups is represented by a smooth and proper~$\Par(G)$-scheme \[\Par(\cQ^\ss) \too \Par(G).\] 
The points of~$\Par(\cQ^\ss)$ correspond to ordered pairs~$(Q_1, Q_2)$ of parabolic subgroups~$Q_2 \subset G$ and~$Q_1 \subset Q_2^\ss$. Let~$Q \subset G$ be a parabolic subgroup with~$\alpha(Q) = Q$, and let~$x \in \cB(G , k)$. Set
\[  [x, Q]_\alpha^\ss  :=  \left\{ (\Stab_{Q'^\ss}(y), Q') \left| 
\begin{array}{l}
\medskip Q' = g Q g^{-1}, y = \pr_{Q'}(gx), \\
g \in G \textup{ with } g \alpha(g)^{-1} \in \rad Q'
\end{array} 
\right. \right\}  \subset \Par(\cQ^\ss).
 \]
To see how this set is related to~$[x, Q]_\alpha$ notice that the set-theoretical map
\[ r \colon \Par(G) \times \Par(G) \too \Par(\cQ^\ss), \quad (P, Q') \longmapsto (P \cap Q' / P \cap \rad Q', Q') \]
 induces a map
\[ r \colon [x, Q]_\alpha \too [x, Q]_\alpha^\ss. \]
Indeed, for any parabolic subgroup~$Q' \subset G$ and~$x' \in \cB(G,  k)$, \cref{Lemma:DescriptionFibersProjectionLeviFactor} gives
\[ \Stab_{Q'^\ss}(\pr_{Q'}(x')) \; = \; \Stab_G(x') \cap Q \; / \; \Stab_{G}(x') \cap \rad Q.\]
Notice that if the couple~$(x, Q)$ is positive, then the map~$r \colon  [x, Q]_\alpha \to [x, Q]_\alpha^\ss$ is surjective even though the definition of the target does not involve any positivity condition. We will bound the dimension of the target in terms of~$\alpha$ only:

\begin{proposition} \label{prop:DimensionInTermsOfCentralizer} 
With notation as above, the Zariski closure of~$[x, Q]^\ss_\alpha$ in~$\Par(\cQ^\ss)$ has dimension 
\[ \dim \overline{[x, Q]^\ss} \;\le\; \dim C_G(\alpha^\ss). \]
\end{proposition}

The proof of this will be given in~\cref{sec:BoundsOnTheBase}. It then remains to bound the fibers of~$r\colon [x, Q] \to [x, Q]^\ss$. 
To do this, we consider for a parabolic subgroup~$Q' \subset G$ the commutative square
\[ 
\begin{tikzcd}
\cB(G,  k) \ar[d, "\pr_{Q'}"'] \ar[r, "x \mapsto \bar{x}"] & \cB(\bar{G},  k) \ar[d, "\pr_{\bar{Q}'}"] \\
\cB(Q'^\ss,  k) \ar[r, "y \mapsto \bar{y}"] & \cB(\bar{Q}'^\ss,  k)
\end{tikzcd}
\]
where~$\bar{Q}'$ is the image of~$Q'$ via the projection~$G \to \bar{G}$. The fibers of the map~$r$ will be bounded in terms of the sets 
\[
 [\bar{x}]_{\bar{Q}', \bar{y}} \;  := \; \{P_{\bar{x}'} \mid \bar{x}' \in \bar{G}\bar{x} \textup{ is~$\bar{Q}'$-positive and~$\pr_{\bar{Q}'}(\bar{x}')= \bar{y}$} \} \; \subset \; \Par(\bar{G})
\]
for~$y \in \cB(Q,  k)$. In~\cref{sec:BoundsOnTheFiber} we will show:

\begin{proposition} \label{Prop:CodimensionWithFixedParabolic} With notation as above, let~$n \ge 1$ be such that \eqref{Eq:EasyInequality} and \eqref{Eq:HardInequalityInPractice} hold.
Then the Zariski closure of~$[\bar{x}]_{\bar{Q}', \bar{y}}$ in~$\Par_{t(\bar{x})}(\bar{G})$ has codimension~$\ge n$.
\end{proposition}

Assume now that we have~$Q' = g Q g^{-1}$ and~$y= \pr_{Q'}(gx)$ for some~$g \in G$ such that~$g^{-1} \alpha(g) \in \rad Q$. Then~$(P_y, Q')\in [x, Q]^\ss_{\alpha}$ for~$P_y = \Stab_{Q'^\ss}(y)$, and for the composite morphism~$$q\colon  \quad\frH \into \Par(G) \onto \Par(\bar{G})$$ we have an inclusion
\[ q\bigl( \, r^{-1}(P_y, Q') \cap [x, Q]_\alpha \, \bigr) \; \subset \; {[\bar{x}]}_{\bar{Q}', \bar{y}} \times \{ \bar{Q}' \}. \]
The situation is summarized in the following diagram, where~$\pr\colon \Par(G)\to \Par(\bar{G})$ is the pushforward under the projection~$G\to \bar{G}$:
\[
\begin{tikzcd}
\Par(\cQ^\ss) & {[x, Q]^\ss_\alpha}  \\
 \Par(G) \times \Par(G) \ar[u, "r"] & {[x, Q]_\alpha} \ar[u, "r"'] \\
\frH \times \Par(G) \ar[d, "q \times \pr"'] \ar[u, hook] & r^{-1}(P_y, Q') \cap [x, Q]_\alpha \cap (\frH \times \Par(G) ) \ar[u, hook] \ar[d, "q \times \pr"] \\
\Par(\bar{G}) \times \Par(\bar{G}) &  \displaystyle  { [\bar{x}]}_{\bar{Q}', \bar{y}} \times \{ \bar{Q}' \}  
\end{tikzcd}
\]
We can now prove \cref{Thm:BoundInFlag} assuming \cref{prop:DimensionInTermsOfCentralizer,,Prop:CodimensionWithFixedParabolic}:

\begin{proof}[{Proof of \cref{Thm:BoundInFlag}}] By \cref{Lemma:EqClassIsConstructible} below, the set~$[x, Q]_\alpha \subset  \Par(G) \times  \Par(G)$ is constructible, thus it suffices to show
\[\dim {[x, Q]_\alpha \cap (\frH \times \Par(G))} = \dim \frH - n +  \dim C_{G}(\alpha).\] 
By \cref{prop:DimensionInTermsOfCentralizer} the Zariski closure of~$[x,Q]_{\alpha}^\ss$ has dimension~$\le \dim C_{G}(\alpha)$. In general the set-theoretical map~$r$ is not underlying a morphism of varieties as the intersection of parabolics may jump. Nonetheless~$\Par(G) \times \Par(G)$ can be decomposed into a finite collection~$\cS$ of pairwise disjoint locally closed subsets such that for each~$S \in \cS$ the map~$r\colon S \to \Par (\cQ)$ is a morphism of varieties. Therefore  for each point~$(P_y, Q') \in [x, Q]_{\alpha}^\ss$ the fiber
\[ F:= r^{-1}(P_y, Q') \cap [x, Q]_\alpha \cap (\frH \times \Par(G) ) \]
is constructible and it suffices to show that~$F$ has dimension~$\le \dim \frH - n$. As mentioned above, we have
\[ F \; \subset \; q^{-1}([\bar{x}]_{\bar{Q}', \bar{y}}) \times \{ Q' \}. \]
By \cref{Prop:CodimensionWithFixedParabolic} the Zariski closure of~$[\bar{x}]_{\bar{Q}', \bar{y}}$ in~$\Par_{t(\bar{x})}(\bar{G})$ has codimension~$\ge n$ . By hypothesis the morphism~$q \colon \frH \to \Par_{t(\bar{x})}(\bar{G})$ is proper smooth with equidimensional fiber. Therefore the Zariski closure of~$q^{-1}([\bar{x}]_{\bar{Q}', \bar{y}})$ has codimension~$\ge n$, hence the Zariski closure of~$F$ in~$\frH \times \{ Q' \}$ has dimension~$\le \dim \frH - n$ as desired.
\end{proof}

\begin{lemma} \label{Lemma:EqClassIsConstructible} The set~$[x, Q]_\alpha$ is constructible.
\end{lemma}

\begin{proof} The set~$[x, Q]_\alpha$ is the intersection of the sets
\begin{align*}
X\; &:=\; \{ (P_{x'}, Q') \mid (x', Q') \sim_\alpha (x, Q) \}, \\
Y \; &:= \; \{ (P_{x'}, Q') \mid (\bar{x}', \bar{Q}') \succeq 0 ,  x' \in Gx\},
\end{align*}
thus it suffices to show that~$X$ and~$Y$ are constructible. Now~$X$ is the image of the morphism
\[
 \{ g\in G \mid g^{-1}\alpha(g) \in\rad Q\} \too \Par(G)^2, \qquad
 g \longmapsto (g P_x g^{-1}, gQ g^{-1}),
 \]
and is therefore constructible. The subset~$Y$ is the preimage of 
\[ \bar{Y} \; := \;  \{ (P_{\bar{x}'}, \bar{Q}') \mid (\bar{x}' ,{\bar{Q}'}) \succeq 0 , \bar{x}' \in \bar{G} \bar{x}\} \; \subset \; \Par(\bar{G}) \times \Par(\bar{G})\]
under the projection~$\Par(G) \times  \Par(G) \to \Par(\bar{G}) \times \Par(\bar{G})$, so it suffices to show~$\bar{Y}$ is constructible. From now on the group~$G$ does not play a role any more, hence to simplify notation we write~$G$,~$x$,~$Q$,~$Y$ etc. instead of~$\bar{G}$,~$\bar{x}$,~$\bar{Q}$,~$\bar{Y}$ etc. Then~$Y$ is the image of the morphism
\[
 G \times \{ (P_{x'}, Q) \in Y \} \too \Par(G)^2, \qquad
  (g, P, Q) \longmapsto (g P g^{-1}, g Q g^{-1}), \]
so it suffices to prove that~$\{ (P_{x'}, Q) \in Y \}\subset \Par(G) \times \{ Q \}$ is constructible. To do so let~$x = [\tfrac{\lambda}{n}]$ for a cocharacter~$\lambda$  with values in a Levi factor~$L$ of~$Q$. For~$u \in \rad Q$ a point~$x' \in \cB(G,  k)$ is~$Q$-positive if and only if~$ux'$ is. Since all Levi factors of~$Q$ are conjugates under~$\rad Q$, this implies
\[\{  (P_{x'}, Q)\in Y \} = \rad Q \cdot \{ (P_{x'}, Q)\in Y \mid x' \in \cB(L,  k)\}.\] 
The set~$G x \cap \cB(L,  k)$ is a finite union of~$L$-orbits. To see this, we may assume that~$\lambda$ is injective. In this case, every point~$y \in G x \cap \cB(L,  k)$ is of the form~$y = [\tfrac{\mu}{n}]$ for some injective cocharacter~$\mu$  of~$L$ conjugate to~$\lambda$ under~$G$. Then the claimed finiteness follows from \cref{Prop:DisconnectedRichardsonThm} applied to the images of~$\lambda$ and~$\mu$. In view of this, it suffices to show that for~$y \in \cB(L, k)$ the set \[\{ P_{g y} \mid g \in L, (gy, Q) \succeq 0\} \subset \Par(G)\] is constructible. Let~$V:= \Lie \rad Q$ and~$\rho \colon Q \to \GL(V)$ the adjoint representation. Since the weight of a filtration on~$V$ is invariant under~$\GL(V)$, we have
\[ \wt(\rho_\ast (g y)) = \wt(\rho(g) \rho_\ast y) = \wt(\rho_\ast y)
\quad \text{for~$g\in L$}. \]
So~$\{ P_{g y} \mid g \in L, (gy,{Q})\succeq 0\}$ is either empty or an~$L$-orbit, hence constructible.    \end{proof}

\subsection{Bounds on the base} \label{sec:BoundsOnTheBase}

Now we go back to the setting in \cref{sec:StatementOrbitInFlag}. For the proof of \cref{Prop:CodimensionWithFixedParabolic} we will first reduce to the case where the automorphism~$\alpha$ is semisimple; in this case we have the following result:

\begin{lemma} \label{Lemma:FinitelyManyOrbitsCentralizer} Suppose~$\alpha$ semisimple. Let~$\Par(G)^\alpha \subset \Par(G)$ be the subvariety of parabolic subgroups stable under~$\alpha$. Then~$C_{G}(\alpha)$ acts naturally on~$\Par(G)^\alpha$ and \[ |\Par(G)^\alpha /C_{G}(\alpha)| < +\infty. \]
\end{lemma}

\begin{proof} Let~$Q \subset G$ be a parabolic subgroup stable under~$\alpha$. Its orbit under~$C_G(\alpha)$ has dimension~$\dim C_G(\alpha) - \dim C_Q(\alpha)$ where~$C_Q(\alpha) = C_G(\alpha) \cap Q$. Therefore, to conclude the proof it suffices to show that the local dimension of~$\Par(G)^\alpha$ at~$Q$ is~$\le \dim C_G(\alpha) - \dim C_Q(\alpha)$. In turn, it suffices to show that
\[ \dim \rT_{[Q]} \Par(G)^\alpha \le \dim C_G(\alpha) - \dim C_Q(\alpha) \]
where~$\rT_{[Q]} \Par(G)^\alpha$ stands for the tangent space of~$\Par(G)^\alpha$ at the point~$[Q]$. 
Consider the morphism~$\sigma \colon G \to \Par(G)$ given by~$g \mapsto [g Q g^{-1}]$. The tangent map~$\rT_e \sigma$ is~$\alpha$-equivariant and induces an isomorphism
\[ \rT_{[Q]} \Par(G) \iso \Lie G / \Lie Q. \]
Now~$\rT_{[Q]} \Par(G)^\alpha$ is made of those tangent vectors to~$\Par(G)$ in~$[Q]$ that are fixed under the action of~$\alpha$. Since~$\alpha$ is semisimple, we have
\[ (\Lie G / \Lie Q)^\alpha = (\Lie G)^\alpha / (\Lie Q)^\alpha\]
where the superscript~$\alpha$ stands for the subspace of~$\alpha$-invariants. The identity 
$(\Lie H)^\alpha = \Lie C_H(\alpha)$ for~$H = G, Q$ then permits to conclude the proof.
\end{proof}

We will then pass from the centralizer in a parabolic subgroup to a centralizer in its Levi quotient as follows.
For a parabolic subgroup~$Q \subset G$ fixed under~$\alpha$ set~$C_Q(\alpha) := C_{G}(\alpha) \cap Q$. The radical~$\rad Q$ is also stable under~$\alpha$, hence~$\alpha$ induces an automorphism of the Levi quotient~$Q^\ss := Q / \rad Q$. Let~$C_{Q^\ss}(\alpha) \subset Q^\ss$ be the subgroup of fixed points of this automorphism, then we have:

\begin{lemma} \label{Lemma:SemiSimplifiedCentralizer} Suppose~$\alpha$ semisimple. Given a parabolic subgroup~$Q \subset G$ stable under~$\alpha$ the image of~$\pi_Q \colon C_{Q}(\alpha) \to C_{Q^\ss}(\alpha)$ is a subgroup of finite index.
\end{lemma}

\begin{proof} It suffices to show that the tangent map~$p = \Lie \pi_Q \colon \Lie Q \to \Lie Q^\ss$ restricts to a surjection~$\Lie C_{Q}(\alpha) \to \Lie C_{Q^\ss}(\alpha)$. The tangent action of~$\alpha$ on~$\Lie Q$ preserves~$\Lie \rad Q$ and the projection~$p$ is~$\alpha$-equivariant. Since~$\alpha$ is semisimple, there is a unique complement~$L$ of~$\Lie \rad Q$ in~$\Lie Q$ which is stable under~$\alpha$ and the projection~$p \colon L \to Q^\ss$ is an isomorphism. The Lie algebra of~$C_{Q}(\alpha)$ consists of the vectors in~$\Lie Q$ fixed by~$\alpha$, and similarly for~$C_{Q^\ss}(\alpha)$. In particular~$p$ induces an isomorphism \[\{ v \in L : \alpha(v) = v\} \to \Lie C_{Q^\ss}(\alpha)\] which concludes the proof.
\end{proof}

\begin{proof}[Proof of \cref{prop:DimensionInTermsOfCentralizer}] 
The subset~$[x,Q]_\alpha^\ss \subset \Par(\cQ^\ss)$ is constructible because it is the image of the subvariety \[T := \{ g \in G \mid g^{-1} \alpha(g) \in \rad Q\} \subset G\] via the morphism~$T \to \Par(\cQ^\ss)$ defined by
\[g\; \longmapsto \; (Q_1, Q_2) 
\;\; \text{with} \;\; 
\begin{cases} 
Q_2 \;=\; gQ g^{-1} \\
Q_1 \;=\; g (Q \cap P_x) g^{-1}/ g (\rad Q \cap P_x)g^{-1} \; \subset\;  Q_2^\ss.
\end{cases}
\]
Here~$g (Q \cap P_x) g^{-1}/ g (\rad Q \cap P_x)g^{-1}$ is the stabilizer of the point~$\pr_{gQg^{-1}}(gx)$ by \cref{Lemma:DescriptionFibersProjectionLeviFactor}. Now we reduce to the case where~$\alpha$ is semisimple: Indeed, let~$g \in G$ such that~$g^{-1} \alpha(g) \in \rad Q$, then
\[ \{ \beta \in \Aut(G) \mid \beta(Q) = Q, g^{-1} \beta(g) \in \rad Q\} \subset \Aut(G) \] is an algebraic subgroup. By functoriality of the Jordan-Chevalley decomposition the semisimple part~$\beta = \alpha^\ss$ of~$\alpha$ belongs to such a subgroup and~$[x, Q]_{\alpha} \subset [x, Q]^\ss_{\beta}$, so for the rest of the proof we can assume that~$\alpha = \alpha^\ss$ is semisimple. This being said, let
\[ \pi \colon \Par(\cQ^\ss) \too \Par(G), \quad (P, Q) \longmapsto Q\]
be the structure morphism of the~$\Par(G)$-scheme~$\Par(\cQ^\ss)$. The set~$[x, Q]_{\alpha}^\ss$ is stable under~$C_G(\alpha)$, thus so is its image
\[ \pi([x, Q]_{\alpha}^\ss) \subset \Par(G)\]
because~$\pi$ is~$G$-equivariant. Moreover, since~$Q$ is stable under~$\alpha$, each parabolic subgroup in~$\pi([x, Q]_{\alpha}^\ss)$ is stable under~$\alpha$. By hypothesis the automorphism~$\alpha$ is semisimple, thus
\[ | \pi([x, Q]_{\alpha}^\ss) / C_G(\alpha)|<+\infty \]
by \cref{Lemma:FinitelyManyOrbitsCentralizer}. Therefore it suffices to bound the dimension of~$\pi^{-1}(\cC) \cap [x, Q]_{\alpha}^\ss$ for each~$C_{G}(\alpha)$-orbit~$\cC \subset \pi([x, Q]_{\alpha}^\ss)$. Fix~$Q' \in \cC$ and note that 
\[ \dim \cC = \dim C_{G}(\alpha) - \dim C_{Q'}(\alpha). \]
With notation as in \cref{Lemma:SemiSimplifiedCentralizer} the set~$\pi^{-1}(Q') \cap [x, Q]_\alpha^\ss$ is a single~$C_{Q'^\ss}(\alpha)$-orbit, hence a finite union of~$C_{Q'}(\alpha)$-orbits by the same lemma because we assumed~$\alpha$ to be semisimple. In particular,
\[ \dim \pi^{-1}(Q') \cap [x, Q]_\alpha^\ss \le \dim C_{Q'}(\alpha).\]
In sum, the fibers of~$\pi \colon [x, Q]_\alpha^\ss \to \Par(G)$ have dimension~$\le \dim C_{Q'}(\alpha)$ while its image has dimension~$\le \dim C_{G}(\alpha) - \dim C_{Q'}(\alpha)$. Therefore 
\[ \dim {[x, Q]_\alpha^\ss} \le \dim C_{G}(\alpha) - \dim C_{Q'}(\alpha) + \dim C_{Q'}(\alpha) = \dim C_{Q'}(\alpha),\]
which concludes the proof.
\end{proof}

\subsection{Interlude on roots} \label{sec:Roots} Let~$B \subset G$ be a Borel subgroup and~$T \subset B$ a maximal torus. For a  connected subgroup~$H\subset G$ containing~$T$, a root of~$H$ is a nontrivial character~$\chi \colon T \to \Gm$ for which the vector subspace
\[ (\Lie H)_\chi := \{ v \in \Lie H : \ad_H(t).v = \chi(t).v, t \in T( k)\}\]
is nonzero, where~$\ad_H \colon T \to \GL(\Lie H)$ is the adjoint representation of~$H$.  Let~$\Phi_H$ be the set of roots of~$H$. For~$\chi \in \Phi_{G}$ the vector space~$(\Lie G)_\chi$ has dimension~$1$. A root~$\chi \in \Phi_G$ is \emph{positive} if it belongs to~$\Phi_B$. If~$\chi$ is not positive, then~$- \chi$ is positive. A root is \emph{simple} if it is positive and it cannot be written as the sum of two positive roots. Let~$\Delta_G$ be the set of simple roots and
\[ \Delta_{H} := (- \Phi_H) \cap \Delta_G = \{ \chi \in \Delta_G : - \chi \in \Phi_H\}. \]
Let~$N$ be the normalizer of~$T$ in~$G$. The group~$N( k)$ acts by conjugation on the group of cocharacters~$X_\ast(T)$ and the group of characters~$X^\ast(T)$ of~$T$. These actions are linear, extend naturally to~$\rX_\ast(T) \otimes \bbQ$ and~$\rX^\ast(T) \otimes \bbQ$ and factor through an action of the Weyl group~$W = N( k) / T( k)$. The action of~$W$ respects the natural pairing: for~$w \in W$,~$x \in \rX_\ast(T) \otimes \bbQ$ and~$\chi \in \rX^\ast (T) \otimes \bbQ$ we have~$\langle w x, \chi \rangle= \langle x, w^{-1} \chi \rangle$.  The set of roots~$\Phi_G$ is stable under the action of~$W$. For~$n \in N( k)$ having image~$w$ in~$W$, the group~$n P_x n^{-1} = P_{wx}$ only depends on~$w$.

\medskip

With this notation we are able to give an explicit expression for the adjoint polygon of a point~$x$ lying in the apartment~$\cB(T,  k) = \rX_\ast(T) \otimes \bbQ \subset \cB(G,  k)$. Suppose that~$B$ is contained in~$P:= P_x$. Then
\[ \Phi_P = \{ \chi \in \Phi_G : \langle x, \chi\rangle \ge 0 \}.\] 
Let~$\ad \colon G \to \GL(\Lie G)$ the adjoint representation and consider the decomposition~$\Lie G = \bigoplus_{\chi \in \Phi_G \cup \{0 \}} (\Lie G)_\chi$  of~$\Lie G$ in isotypical components with respect to~$T$. The filtration~$F^\bullet \Lie G$ defined by~$\ad_\ast x \in \cB(\GL(\Lie G),  k)$ is given, for~$\alpha \in \bbQ$, by
\[ F^\alpha \Lie G = \bigoplus_{\langle x, \chi \rangle \ge \alpha} (\Lie G)_\chi. \]
Let~$\{ 1, \dots, \dim G \} \to \Phi_G \cup \{ 0 \}$,~$i \mapsto \chi_i$ be a surjection such that~$\langle x, \chi_i \rangle \ge \langle x, \chi_{i+1} \rangle$ and \[ | \{ i : \chi_i =0 \}| = \dim (\Lie G)_0 = \dim T.\]
 With this notation, for all~$t\in [0, \dim G]$, we have
\begin{equation} \label{Eq:ExpressionAdjointPolygon}
\Upsilon_{G, x}(t) = \sum_{i = 1}^{\lfloor t \rfloor } \langle x, \chi_i \rangle + \langle x, \chi_{\lfloor t \rfloor + 1}  \rangle (t - \lfloor t \rfloor).
\end{equation}
Thus~$\Upsilon_{G,x}$ is increasing on~$[0, \dim \rad P]$, constant on~$[\dim \rad P, \dim P]$ and decreasing on~$[\dim P, \dim G]$, hence~$\Upsilon_{G,x}$ is nondecreasing on~$[0, \dim P]$ and \[\max \Upsilon_{G,x} = \sum_{\chi \in \Phi_P} \langle \chi, x \rangle.\]

We take profit of this setup to state a refinement of the Bruhat decomposition that will be used afterwards. Given parabolic subgroups~$P,Q \subset G$ containing the Borel subgroup~$B$ consider 
\[ W_{Q, P} = \{ w \in W : w \Delta_P \subseteq \Phi_B, w^{-1} \Delta_{Q} \subseteq \Phi_B \}. \]
As usual, for~$n \in N( k)$, the coset~$Q nP$ only depends in the image~$w \in W$ of~$n$ and is denoted by~$Q w P$.

\begin{lemma} \label{Prop:BruhatDecomposition} We have~$G = Q W_{Q, P} P := \bigcup_{w \in W_{Q, P}} Q w P$.
\end{lemma}

\begin{proof} See the discussion after (11.12) in \cite{LV}.
\end{proof}

\subsection{Bounds on the fiber} \label{sec:BoundsOnTheFiber}
In this section we are going to prove \cref{Prop:CodimensionWithFixedParabolic}. Since only the group~$\bar{G}$ plays a role, we reset notation and let~$G$ be a reductive group over~$k$. For~$x \in\cB(G,  k)$, a parabolic subgroup~$Q \subset G$ and~$y \in \cB(Q^\ss, k)$ we consider the set
\[
 [x]_{Q, y} \;  := \; \{P_{x'} \mid x' \in G( k)x \textup{ is~$Q$-positive and~$\pr_{Q}(x') = y$} \}.
\]
Let~$n \ge 1$ be an integer such that
\begin{gather} 
n \le \dim \Par_{t(x)}(G), \label{Eq:EasyInequalityIntermediate} \\
\label{Eq:HardInequalityIntermediate} \Upsilon_{G,x}(n -1) + \Upsilon_{G,x}(n - 1+ \tfrac{1}{2}(\dim P_x^\ss - \rk G)) < \max \Upsilon_{G,x}.
\end{gather}
We are going to prove that the Zariski closure of~$[x]_{Q,y}$ in~$\Par_{t(x)}(G)$ has codimension~$\ge n$. Note that \eqref{Eq:EasyInequalityIntermediate} and \eqref{Eq:HardInequalityIntermediate} imply
\begin{equation}  \label{eq:HypothesisAdjointPolygonIntermediateStatement} 
\Upsilon_{G, x}(n -1) +  \Upsilon_{G, x}(n - 1+ \dim P_x - \dim Q) < \max  \Upsilon_{G, x},
\end{equation}
and it is rather this weaker inequality that we are going to use. Indeed~$\Upsilon_{G,x}$ is nondecreasing on~$[0, \dim P_x]$ and
\[ n- 1 + \dim P_x - \dim Q \le n-1+  \tfrac{1}{2}(\dim P_x^\ss - \rk G) \le \dim P_x.\]
Since~$[x]_{Q,y} = [ gx ]_{Q,y}$ for all~$g \in G$ we may suppose that~$P:= P_x$ and~$Q$ share a Borel subgroup~$B$. Consider  a maximal torus~$T \subset B$,  the normalizer~$N$ of~$T \subset G$ and the Weyl group~$W = N / T$. In this setup we can adopt the notation introduced in \cref{sec:Roots}. The point~$x$ then belongs to the apartment~$\cB(T,  k) \subset \cB(G, k)$. Note that~$N$ acts on~$\cB(T,  k)$ by conjugation and the action factors through~$W$ because~$T$ is commutative. Since a parabolic subgroup coincides with its own normalizer, the map~$\sigma \colon G / P \to  \Par_{t(x)}(G)$,~$g \mapsto gP g^{-1}$ is injective. For~$w \in W$ set
\[ C_w := \sigma(Qn)  \; \subset \; \Par_{t(x)}(G)\]
where~$n \in N$ maps to~$w \in W$; this does not depend on the chosen~$n$. With the notation of \cref{Prop:BruhatDecomposition} we have that~$\Par_t(G)$ is the union of~$C_w$ for~$w \in W_{Q, P}$. Now the subset 
\[[ x]_Q := \{ P_{x'} \mid x' \in Gx \textup{ is~$Q$-positive}\} = \bigcup_{y \in \cB(Q^\ss,  k)} [x]_{Q,y}\] is stable under the action of~$Q$. Therefore for~$w \in W_{Q, P}$ either
\[ C_w \subset [x]_Q \qquad \textup{or} \qquad C_w \cap [x]_Q = \emptyset.
\]
 Moreover~$C_w \subset [x]_Q$ if and only if~$wx$ is~$Q$-positive. Let~$ W_{Q, P}^+ \subset W_{Q, P}$ be the subset of~$w$ such that~$wx$ is~$Q$-positive. With this notation we have
\[ [ x]_{Q,y} = \bigcup_{w \in W^+_{Q, P}} C_w \cap [x]_{Q,y}. \]
Arguing by contradiction, suppose that the Zariski closure of~$[x]_{Q, y}$ has codimension~$\le n - 1$. The above decomposition implies that  there is~$w \in W^+_{Q, P}$ such that the Zariski closure of~$C_w \cap [x]_{Q,y}$ has codimension $\le n - 1$. Since
\[  C_w \cap  [x]_{Q,y} = \{ P_{qwx} \mid q \in Q \textup{ such that } \pr_Q(qwx) = y\} \neq \emptyset \]
there is~$q \in Q$ such that~$\pr_Q(qwx) = y$. Applying \cref{Lemma:DescriptionFibersProjectionLeviFactor} to~$qwx$ we obtain the identities
\[ \dim \overline{ C_w \cap  [x]_{Q,y}} = \dim (\rad Q / \rad Q \cap P_{qwx}) = \dim (\rad Q / \rad Q \cap P_{wx}), \]
where the latter comes from the isomorphism~$\rad Q \cap P_{wx} \iso \rad Q \cap P_{qwx}$ given by conjugation by~$q$. The roots appearing in the isotypical decomposition of \[\Lie \rad Q / \Lie (\rad Q \cap P_{wx})\] are of the form~$- \chi$ for
\[ \chi \in \Phi^{w}_{Q, P} = \{ \chi \in \Phi_G \smallsetminus \Phi_Q : -w^{-1}  \chi \not\in \Phi_{P} \}\]
because~$\Lie \rad Q = \bigoplus_{\chi \in \Phi_G \smallsetminus \Phi_Q} (\Lie G)_{- \chi}$. Thus~$\dim \overline{ C_w \cap  [x]_{Q,y}} =  |\Phi_{Q, P}^w|$ so that the hypothesis we are trying to prove fallacious becomes 
\begin{equation} \label{eq:AbsurdHypothesis2} 
 \dim G/P - |\Phi_{Q, P}^w| \le n -1.
\end{equation}
Now consider the modular character~$\delta \colon Q \to \Gm$ of~$Q$. When restricted to~$T$ we have
\[ \delta_{\rvert T} = - \sum_{\chi \in \Phi_G \smallsetminus \Phi_Q} \chi.  \]
The point~$wx$ is~$Q$-positive thus~$\langle wx, \delta \rangle \ge 0$. Since~$\langle wx,  \delta \rangle = \langle x, w^{-1}\delta \rangle$ this inequality can be rewritten as
\begin{equation} \label{eq:AbsurdEqualityModularCharacter}   \sum_{\chi \in (\Phi_G \smallsetminus \Phi_Q) \smallsetminus \Phi_{Q, P}^w} \langle x, - w^{-1}\chi \rangle \ge \sum_{\chi \in \Phi_{Q, P}^w} \langle x, w^{-1}\chi \rangle.
\end{equation}
The previous equality will lead to a contradiction. On the one hand each summand on the right-hand side of \eqref{eq:AbsurdEqualityModularCharacter} is positive: roots~$\chi \in \Phi_{Q, P}^w$ satisfy~$-w^{-1} \chi \not \in \Phi_P$ which is equivalent to~$\langle x, w^{-1}\chi \rangle > 0$. In particular by the explicit expression for the adjoint polygon given in \eqref{Eq:ExpressionAdjointPolygon} we have
\[ \sum_{\chi \in \Phi_{Q, P}^w} \langle x, w^{-1}\chi \rangle  \ge \max \Upsilon_{G,x} - \Upsilon_{G,x}(n - 1).\]
because by \eqref{eq:AbsurdHypothesis2} there at most~$n-1$ roots~$\chi \not \in \Phi_{Q, P}^w$ such that~$\langle wx, \chi \rangle > 0$.

On the other hand, the left-hand side of \eqref{eq:AbsurdEqualityModularCharacter} has 
\[ |(\Phi_G \smallsetminus \Phi_Q) \smallsetminus \Phi_{Q, P}^w| = \dim G / Q - |\Phi_{Q, P}^w| \]
summands. Now \eqref{Eq:ExpressionAdjointPolygon} says that the value of~$\Upsilon_{G,x}$ at some integer~$t \in [0, \dim G]$ is the sum of the~$t$ biggest values of~$\langle x, \alpha\rangle$ for~$\alpha \in \Phi_G \cup \{ 0 \}$, with the convention that the trivial cocharacter is counted~$\dim T$ times. In particular, we obtain the trivial upper bound 
\[ 
\sum_{\chi \in (\Phi_G \smallsetminus \Phi_Q) \smallsetminus \Phi_{Q, P}^w} \langle x, - w^{-1}\chi \rangle \le \Upsilon_{G,x}(\dim G / Q - |\Phi_{Q, P}^w|).
 \]
By \eqref{eq:AbsurdHypothesis2} we have~$\dim G / Q - |\Phi_{Q, P}^w| \le \dim P - \dim Q + n - 1$. The inequality~\eqref{eq:AbsurdEqualityModularCharacter} then yields
\begin{align*} \max \Upsilon_{G,x} - \Upsilon_x(n - 1) &\le \Upsilon_{G,x}(\dim G / Q - |\Phi_{Q, P}^w|) \\ &\le  \Upsilon_{G,x}(\dim P - \dim Q + n - 1),
\end{align*}
because~$\Upsilon_x$ is nondecreasing on~$[0, \dim P]$ and~$\dim P - \dim Q + n - 1 \le \dim P$ by the assumption~$n \le \dim G - \dim P$. This contradicts \eqref{eq:HypothesisAdjointPolygonIntermediateStatement} and concludes the proof.
\qed

\section{Reminder on period mappings} 

For convenience of the reader we collect in this section some general facts about differential Galois groups and period mappings that will be used in the proof of the main theorem. Let~$S$ be a smooth geometrically connected variety over a field~$k$ of characteristic~$0$. 

\subsection{Differential Galois group} Following Katz \cite[1.1]{KatzCalculation} an \emph{algebraic differential equation} on~$S$ is a pair~$(\cV, \nabla)$ made of a  vector bundle~$\cV$ on~$S$ and an integrable connection~$\nabla$. To ease notation we denote the pair simply by~$\cV$ if the connection is clear from the context. Let~$\DE(S)$ be the category of differential equations on~$S$ with arrows given by parallel~$\cO_S$-linear morphisms. The category~$\DE(S)$ is abelian and~$k$-linear with internal homs and a tensor product given by the usual formula. This
makes~$\DE(S)$ into a Tannakian category. For~$s \in S( k)$ the functor~$\omega_s \colon \cV \mapsto \cV_s$ taking the fiber at~$s$ is a fiber functor for~$\DE(S)$. For an algebraic differential equation~$\cV$ on~$S$ let~$\langle \cV \rangle$ be the smallest tensor full subcategory of~$\DE(S)$ containing~$\cV$.

\begin{definition}
 The \emph{differential Galois group}~$\Gal(\cV, s)$ of~$\cV$ with base point~$s$ is the Tannaka group of the Tannaka category~$\langle \cV \rangle$ together with the fiber functor~$\omega_s$.
 \end{definition}
 
The differential Galois group~$\Gal(\cV, s)$ is the algebraic subgroup of~$\GL(\cV_s)$ whose points with values in a~$k$-scheme~$S$ are those~$\cO_{S}$-linear automorphisms~$\phi$ of~$\cV_s \otimes_ k \cO_S$ such that~$\phi(\cW_s \otimes_ k \cO_S) = \cW_s \otimes_ k \cO_S$ for all integers~$n \ge 1$,~$a, b \in \bbN^n$ and algebraic subdifferential equations \[\cW \subset \cV^{a, b}:= \bigoplus_{i= 1}^n \cV^{\vee \otimes a_i} \otimes \cV^{\otimes b_i}.\] The differential Galois group is invariant under extension of scalars: given a field extension~$k'$ of~$k$, we have
\[ \Gal(\cV', s')= \Gal(\cV,s) \times_ k  k'\]
where~$(\cV', \nabla')$ is the pull-back of~$\cV$ along~$S' := S \times_ k  k' \to S$ and~$s' \in S'( k')$ is the point deduced from~$s$; see \cite[Prop. 1.3.2]{KatzCalculation}. 

\begin{definition} For an~$S$-scheme~$f \colon T\to S$ consider the fiber functors
\[\omega_T, \omega_{T,s} \colon \quad \DE(S) \; \too \;  (\cO_T\textup{-mod})\]
 defined by~$\omega_T(\cW) = f^\ast \cW$ and~$\omega_{T, s}(\cW)= \cW_s \otimes_ k \cO_T$. The \emph{Picard-Vessiot principal bundle}~$\PV(\cV, s)$ of~$\cV$ with base point~$s$ is the~$S$-scheme whose points with values in an~$S$-scheme~$f \colon T \to S$ are isomorphisms~$\omega_{T}\to \omega_{T, s}$ of fiber functors on~$\langle \cV \rangle$.
\end{definition}

A point~$\PV(\cV, s)$ with values in an~$S$-scheme~$f \colon T \to S$ is an~$\cO_T$-linear isomorphism~$\phi \colon f^\ast \cV \to \cV_s \otimes_ k \cO_T$ such that \[\phi(f^\ast \cW) = \cW_s \otimes_ k \cO_T\] for all algebraic subdifferential equations~$\cW \subset \cV^{a, b}$ with~$a, b \in \bbN^n$. This shows in particular that~$\PV(\cV, s)$ is representable by a closed subscheme of the frame bundle of~$\cV$. By design~$\PV(\cV, s)$ is a principal bundle under the action by composition of~$\Gal(\cV, s)$. Applying the definition with~$T$ being~$\PV(\cV, s)$ gives a `universal' isomorphism
\[ u_{\cV, s} \colon p^\ast \cV \stackrel{\sim}{\too} \cV_s \otimes_ k \cO_{\PV(\cV, s)}\]
where~$p \colon \PV(\cV, s) \to S$ is the structural morphism.  Consider the formal completion~$\hat{\cO}_{S, s}$ of~$\cO_{S, s}$ and the natural morphism~$\iota_s \colon \Spec \hat{\cO}_{S, s} \to S$. The formal analogue of the Frobenius integrability theorem states that evaluation at~$s$ induces an isomorphism
\[ (\iota_s^\ast \cV)^\nabla := \Ker \iota_s^\ast \nabla \stackrel{\sim}{\too} \cV_s \]
of~$k$-vector spaces. By Nakayama's lemma it follows that the natural map
\[ (\iota_s^\ast \cV)^\nabla \otimes_ k \hat{\cO}_{S, s} \stackrel{\sim}{\too} \iota_s^\ast \cV\]
is an isomorphism of~$\hat{\cO}_{S, s}$-modules. Combining the previous two isomorphisms we obtain a third one
\begin{equation} \label{Eq:IsomorphismFrobeniusIntegrability}\iota_{\cV, s} \colon \iota_s^\ast \cV \stackrel{\sim}{\too} \cV_s \otimes_ k \hat{\cO}_{S, s} \end{equation}
which is parallel with respect to the connection~$\id \otimes \rd$ on the target. It follows that for any~$n \ge 1$, any~$a, b \in \bbN^n$ and any  algebraic subdifferential equation of~$\cW \subset \cV^{a, b}$ we have~$\iota_{\cV, s}(\iota_s^\ast \cW) = \cW_s \otimes_ k \hat{\cO}_{S, s}$. Hence~$\iota_{\cV, s}$ can be seen as a morphism 
\[ \iota_{\cV, s} \colon \Spec \hat{\cO}_{S, s} \too \PV(\cV, s)\]
whose composition with the structural morphism~$p \colon \PV(\cV, s) \to S$ is~$\iota_s$. Suppose that as an~$\cO_S$-module~$\cV$ comes with a filtration
\[ \cF^\bullet \; : \quad \cV \; =\; \cF^0 \;  \supset \; \cF^1 \; \supset\; \cdots \; \supset\; \cF^n\; \supset\; \cF^{n + 1}\; =\; 0, \]
where each graded piece~$\gr_i \cF^\bullet = \cF^i / \cF^{i+1}$ is locally free of rank~$r_i$ say. Then the filtration~$u_{\cV, s}(\cF^\bullet)$ on~$\cV_s \otimes_ k \cO_{\PV(\cV, s)}$ determines a morphism of~$k$-schemes
\begin{equation} \label{Eq:PeriodMapAssociatedWithAFiltration}
\phi_{\cV, s, \cF^\bullet} \colon \PV(\cV, s) \too \Flag_r(\cV_s)
\end{equation}
where~$\Flag_r(\cV_s)$ is the variety of flags of~$\cV_s$ of type~$r = (r_0,\dots, r_n)$. The above morphism is equivariant under the action of~$\Gal(\cV, s)$. We will be interested in the composite map:
\[ 
\begin{tikzcd}[column sep=35pt]
\hat{\Phi}_{\cV, s, \cF^\bullet} \colon \Spec \hat{\cO}_{S, s} \ar[r, "{\iota_{\cV, s}}"] & \PV(\cV, s) \ar[r, "{\phi_{\cV, s, \cF^\bullet}}"]& \Flag_{r}(\cV_s).
\end{tikzcd}
\]

\subsection{The analytic picture} Suppose~$k = \bbC$ and let~$S^\an$ be the complex manifold associated with~$S$. The previous constructions can be carried out in the analytic framework, giving rise to the Tannakian category~$\DE(S^\an)$ of holomorphic differential equations. For a holomorphic differential equation we still write~$\Gal(\cV, s)$ for the corresponding Tannaka group and~$\PV(\cV, s)$ for the Picard-Vessiot principal bundle. 
The sheaf of germs of horizontal sections~$\cV^\nabla := \ker \nabla$ is a local system on~$S^\an$, which can be seen as a finite dimensional representation of the topological fundamental group~$\pi_1(S^\an, s)$. This induces an equivalence of the category~$\DE(S^\an)$ with the one of finite dimensional representations of~$\pi_1(S^\an, s)$. In particular, for a holomorphic differential equation~$\cV$ the Tannaka group~$\Gal(\cV, s)$ coincindes with \emph{algebraic monodromy group}~$\Mon(\cV, s)$, that is, the Zariski closure of the image of the monodromy representation
\[ \rho_{\cV, s} \colon \pi_1(S^\an, s) \too \GL(\cV_s).\]
Let~$\nu \colon \tilde{S} \to S^\an$ be a universal cover of the topological space~$S^\an$ and~$\tilde{s} \in \tilde{S}$ a point over~$s$. The sheaf~$\nu^\ast \cV^\nabla$ is constant, thus the evaluation at each~$t \in \tilde{S}$ induces an isomorphism
\[ \Gamma(\tilde{S}, \nu^\ast \cV^\nabla) \stackrel{\sim}{\too} (\nu^\ast \cV)_{t} = \cV_{\nu(t)}.\]
In particular the natural morphism~$\Gamma(\tilde{S}, \nu^\ast  \cV^\nabla) \otimes_\bbC \cO_{\tilde{S}} \to \nu^\ast \cV$
is an isomorphism. Combining the previous isomorphisms we obtain a parallel isomorphism of holomorphic differential equations~$\nu_{\cV, \tilde{s}} \colon \nu^\ast \cV \to \cV_s \otimes_\bbC \cO_{\tilde{S}}$
with respect the connection~$\id \otimes \rd$ on the target. The previous isomorphism can be seen as a holomorphic map 
\[ \nu_{\cV, \tilde{s}} \colon \tilde{S} \too \PV(\cV, s). \]
whose composition with the structural morphism~$p \colon \PV(\cV, s) \to S^\an$ is~$\nu$. The map~$\nu_{\cV, \tilde{s}}$ is equivariant under the natural action of~$\pi_1(S^\an, s)$ on~$\tilde{S}$ and the action on~$\PV(\cV, s)$ defined by monodromy representation
\[ \rho_s  \colon \pi_1(S^\an, s) \too \Mon(\cV, s) = \Gal(\cV, s).\]
Suppose that as an~$\cO_S$-module~$\cV$ comes with a holomorphic filtration~$\cF^\bullet$ whose graded is locally free. Via~$u_{\cV, s} \colon p^\ast \cV \to \cV_s \otimes_\bbC \cO_{\PV(\cV, s)}$ the filtration~$u_{\cV, s}(\cF^\bullet)$ on~$\cV_s \otimes_\bbC \cO_{\PV(\cV, s)}$ determines a holomorphic map
\[ \phi_{\cV, s, \cF^\bullet} \colon \PV(\cV, s) \too \Flag_r(\cV_s)^\an\]
where~$\Flag_r(\cV_s)$ is the variety on flags of~$\cV_s$ of type~$r = (r_0,\dots, r_n)$. The previous morphism is equivariant under the action of the differential Galois group~$\Gal(\cV, s)$ and we consider the composite morphism:
\[ 
\begin{tikzcd}[column sep=35pt]
\tilde{\Phi}_{\cV, s, \cF^\bullet} \colon \tilde{S} \ar[r, "{\nu_{\cV, \tilde{s}}}"] & \PV(\cV, s) \ar[r, "{\phi_{\cV, s, \cF^\bullet}}"]& \Flag_{r}(\cV_s)^\an.
\end{tikzcd}
\]
\begin{remark} \label{Rmk:SimplyConnectedNbh} Let~$\Omega \subset S^\an$ be a simply connected neighbourhood of~$s$. The preimage~$\nu^{-1}(\Omega)$ is a disjoint union of copies of~$\Omega$ indexed by points in the fiber~$\nu^{-1}(s)$. Having fixed such a point permits to identify~$\Omega$ with an open neighbourhood of~$\tilde{s}$ and consider the `restriction' to~$\Omega$ of the above maps:
\begin{align*}
 \iota_{\cV, \Omega, s} := \nu_{\cV, \tilde{s} \rvert \Omega} \colon \Omega &\too \PV(\cV, s), \\
 \Phi_{\cV, \Omega, s,\cF^\bullet} := \tilde{\Phi}_{\cV, s, \cF^\bullet} \colon \Omega &\too \Flag_r(\cV_s)^\an.
 \end{align*}
\end{remark}

The analytification functor~$\DE(S)\to \DE(S^\an)$ is exact, faithful,~$\bbC$-linear and compatible with tensor product. In particular, for an algebraic differential equation~$\cV$ we have closed immersions
\[ \Gal(\cV^\an, s) \subset \Gal(\cV, s), \qquad  \PV(\cV^\an, s) \subset \PV(\cV, s)^\an. \]
Moreover, according to Deligne \cite[Theorem II.5.9]{Del70} it induces an equivalence of categories~$\DE(S)^\rs \iso \DE(S^\an)$
where~$\DE(S)^{\rs} \subset \DE(S)$ is the full subcategory of algebraic differential equations with regular singular points. Therefore the above closed immersions are isomorphisms as soon as the algebraic differential equation~$\cV$ has regular singular points. 

\begin{remark} \label{Rmk:FormalGermPeriodMap} Let~$\Omega \subset S^\an$ be a  simply connected open neighbourhood~$\Omega$ of~$s$. Consider the holomorphic map~$\Omega \to \PV(\cV^\an, s) \subset \PV(\cV, s)^\an$ where the first map is~$\iota_{\cV^\an, \Omega, s}$. By construction, via the isomorphism~$ \hat{\cO}_{S, s} \iso \hat{\cO}^\an_{S, s}$ the formal germ at~$s$ of the preceding map is~$\iota_{\cV, s}$.
\end{remark}

\subsection{The case of variations of Hodge structures} Keep the notation introduced in the preceding subsection. Let~$\cV$ be a holomorphic differential equation on~$S^\an$ and~$\cF^\bullet$ a filtration on~$\cV$ whose graded is locally free. Suppose that the triple~$(\cV, \cF^\bullet)$ underlies a pure polarized integral variation of Hodge structures~$\cH$ on~$S^\an$. The associated the holomorphic map 
\[ 
\tilde{\Phi} := \tilde{\Phi}_{\cV,s, \cF^\bullet} \colon \tilde{S} \too \Flag_{r}(\cV_s)^\an
\]
is called the \emph{complex period map}.

\begin{proposition} \label{Prop:ComplexBakkerTsimerman} With the notation above,
\begin{enumerate}
\item the algebraic group~$\Gal(\cV, s)^\circ$ is reductive;
\item the subgroup of~$P\subset \Gal(\cV, s)^\circ$ stabilizing~$\cF^\bullet_s$ is parabolic;
\item the image of~$\tilde{\Phi}$ is a Zariski-dense subset of~$\frH := \Gal(\cV, s)^\circ / P$;
\item Let~$Z \subset \frH$ be a subvariety with~$\dim S + \dim Z \le \dim \frH$
and~$Z'$ an analytic irreducible component of~$\tilde{\Phi}^{-1}(Z^\an)$. Then~$\nu(Z')$ is not Zariski-dense in~$S$.
\end{enumerate}
\end{proposition}

\begin{proof} (1) See \cite[th. 4.2.6]{DeligneHodgeII}. \medskip

(2) The statement remains unaffected by the changes of the base point~$s$, thus by \cite[th. 1]{AndreGenericMumfordTate} we may assume that~$G := \Gal(\cV, s)^\circ$ is a normal subgroup of the Mumford-Tate group~$\MT(\cH_s)$ of the Hodge structure~$\cH_s$. With the notation of \cref{sec:NotationReductiveGroups} the parabolic subgroup of~$\GL(\cV_s)$ stabilizing~$\cF^\bullet_s$ is~$P_\mu$ where~$\mu$ is the Hodge cocharacter. By construction~$\mu$ is contained in~$\MT(\cH_s)$ thus it normalizes~$G$. Therefore~$\mu$ induces by conjugation a cocharacter of~$\Aut G$ and in turn a cocharacter of~$G / Z$  by connectedness where~$Z \subset G$ is the center. Some multiple~$\mu^r$ then lifts to a cocharacter of~$G$. Finally the subgroup~$P \subset G$ is the parabolic associated with the cocharacter~$\mu^r$. \medskip

(3) By loc.~cit.~we may pick~$s$ such that~$G:=\Gal(\cV, s)^\circ \unlhd \MT(\cH_s)^\mathrm{der}$ is a normal subgroup in the derived group of the Mumford-Tate group. Up to isogeny any connected normal subgroup in a connected semisimple group splits off as a direct factor, so there is a connected semisimple subgroup~$R \subset \MT(\cV, s)^\mathrm{der}$ such that the morphism
\[
 m\colon \quad G\times R \;\too\; \MT(\cH_s)^\mathrm{der}, \quad 
 (g,r) \;\longmapsto\; g\cdot r
\]
is an isogeny. Fix a point~$\tilde{s} \in \tilde{S}$ above~$s$. For any subgroup~$M\subset \MT(\cH_s)$ let \[D(M) = M\cdot \tilde{\Phi}(\tilde{s})\subset \Flag_r(\cV_s)^\an\] denote the corresponding orbit in the flag variety. Then by~\cite[prop.~15.3.13 and th.~15.3.14]{CMSP17} the isogeny~$m$ induces a surjective holomorphic map
\[
 D(G) \times D(R) \;\too\; D(\MT(\cH_s))
\] 
with finite fibers, and the period map factors as
\[
\begin{tikzcd} 
\tilde{S} \ar[r, "\tilde{\Phi}"] \ar[d, dashed, swap, "{\exists \Psi}"] & \Flag_r(\cV_s)^\an \\
D(G) \times D(R) \ar[r] & D(\MT(\cV, s)) \ar[u, hook]
\end{tikzcd}
\]
where the projection to the second factor~$\pr_2\circ \Psi \colon \tilde{S} \to D(R)$ is a constant map. So the image of~$\tilde{\Phi}$ is contained in a single~$G$-orbit; it is then obviously dense in that orbit since it is stable under the action of~$\pi_1(S, s)$ and since by definition~$G$ is the identity component of the Zariski closure of the image of~$\pi_1(S, s)$. \medskip

(4) This is a consequence of the Ax-Schanuel property for variations of Hodge structures \cite[th. 1.1]{BakkerTsimerman}. In order to apply it, set 
\begin{align*} X:= S, && \check{D}:= \frH, && V:= S \times Z, && W:= \im(\nu \times \tilde{\Phi} \colon \tilde{S} \to S^\an \times \frH^\an).
\end{align*}
Note that the closed analytic subspace~$W \subset S^\an \times \frH^\an$ coincides with the fibered product considered in \emph{loc.~cit.} By design the image of~$Z'$ in~$S^\an \times \frH^\an$ via~$\nu \times \tilde{\Phi}$ is contained in some analytic irreducible component~$U \subset V^\an \cap W$. If 
\[ \codim_{S\times \frH}U < \codim_{S\times \frH}V + \codim_{S\times \frH}W\]
then~$\nu(Z')$ is not Zariski-dense by \emph{loc.~cit.} If this is not the case, the converse of the previous inequality yields
\[ \dim S + \dim \frH - \dim U \ge \dim \frH - \dim Z + \dim \frH\]
because 
$\codim_{S\times \frH}U = \codim_\frH Z$ and~$\codim_{S\times \frH}W = \dim \frH$. Notice that the latter holds because the holomorphic map~$\nu \times \tilde{\Phi}$ has discrete fibers. The hypothesis \[\dim S + \dim Z \le \dim \frH\] then yields~$\dim U = 0$ forcing~$Z'$ to be a singleton. 
\end{proof}

\subsection{Back to formal} We go back to the case when~$k$ is an arbitrary field of characteristic~$0$. Let~$\cV$ be an algebraic differential equation over~$S$ and~$\cF^\bullet$ a filtration on~$\cV$ whose graded is locally free. With the notation of \eqref{Eq:PeriodMapAssociatedWithAFiltration} consider the morphism 
\[
\begin{tikzcd}[column sep=35pt]
\hat{\Phi} \colon \Spec \hat{\cO}_{S, s} \ar[r, "{\iota_{\cV ,s}}"] & \PV(\cV, s) \ar[r, "{\phi_{\cV, s, \cF^\bullet}}"] & \Flag_r(\cV_s).
\end{tikzcd}
\]

\begin{definition} \label{Def:UnderlyingAVHS}
We say that the triple~$(\cV, \cF^\bullet)$ \emph{underlies a variation of integral pure polarized Hodge structures} if the differential equation~$(\cV, \nabla)$ has regular singular points, and there are a subfield~$k_0 \subset  k$ finitely generated over~$\bbQ$ and an embedding~$k_0 \into \bbC$ such that
\begin{enumerate}
\item $S$,~$\cV$,~$\nabla$,~$\cF^\bullet$ come by base change from objects~$S_0$,~$\cV_0$,~$\nabla_0$,~$\cF^\bullet_0$ over~$k_0$;
\item the triple~$(\cV_{0, \bbC}, \nabla_{0, \bbC}, \cF^\bullet_{0, \bbC})$ underlies a variation  of integral polarized pure Hodge structures on the complex manifold~$S_{0, \bbC}^\an$. 
\end{enumerate}
\end{definition}

When this is the case the morphism~$\hat{\Phi}$ is called the \emph{formal period map}.

\begin{proposition} \label{Prop:FormalBakkerTsimerman} Assume that~$(\cV, \cF^\bullet)$ underlies a variation of integral pure polarized Hodge structures. Then,
\begin{enumerate}
\item the algebraic group~$\Gal(\cV, s)^\circ$ is reductive;
\item the subgroup~$P \subset \Gal(\cV, s)^\circ$ stabilizing~$\cF^\bullet_s$ is parabolic;
\item the scheme-theoretic image of~$\hat{\Phi}$ is~$\frH :=  \Gal(\cV, s)^\circ / P$;
\item let~$Z \subset \frH$ be a subvariety with~$\dim S + \dim Z \le \dim \frH$.
Then the scheme-theoretic image of~$\hat{\Phi}^{-1}(Z) \into \Spec \hat{\cO}_{S,s} \to S$ is not~$S$.
\end{enumerate}
\end{proposition}

\begin{proof} Each statement is deduced from the corresponding one in \cref{Prop:ComplexBakkerTsimerman} first by descending from~$\bbC$ to~$k_0$ and then extending from~$k_0$ to~$k$. \medskip

(1) and (2) are deduced directly. \medskip

(3) Since the construction of the scheme-theoretic image is compatible with extension of scalars, we may assume~$k =  k_0$. In this case for ease notation we drop the subscript~$0$. Let~$\nu \colon \tilde{S} \to S^\an$ be a universal cover of~$S$ and consider the complex period map~$\tilde{\Phi} \colon \tilde{S} \to \Flag_r(\cV_{\bbC, s})^\an$. With the notation of \cref{Rmk:SimplyConnectedNbh} for a simply connected open neighbourhood~$\Omega \subset S^\an$ of~$s$ let~$\Phi \colon \Omega \to \Flag_r(\cV_{\bbC, s})^\an$ be the restriction to~$\Omega$ of the complex period map~$\tilde{\Phi}$. The Zariski closure of the image~$\tilde{\Phi}$ coincides with the Zariski closure of the image of~$\Phi$. In turn this equals the scheme-theoretic image of the formal germ of~$\Phi$ at~$s$. By \cref{Rmk:FormalGermPeriodMap} the formal germ of~$\Phi$ at~$s$ is the morphism
\[ \Spec \hat{\cO}_{S_\bbC, s} \too \Spec (\hat{\cO}_{S, s} \otimes_ k \bbC) \stackrel{\hat{\Phi}_\bbC}{\too}   \Flag(\cV_{\bbC, s})\]
where the first is given by the natural injection~$\hat{\cO}_{S, s} \otimes_ k \bbC \into \hat{\cO}_{S_\bbC, s}$ and~$\hat{\Phi}_\bbC$ deduced from~$\hat{\Phi}$ by extending scalars to~$\bbC$. It follows that the scheme-theoretic image of the formal germ of~$\Phi$ at~$s$ is the scheme-theoretic image of~$\hat{\Phi}_\bbC$. We then conclude because the construction of the scheme-theoretic image is compatible with extension of scalars. \medskip

(4) Let~$k'_0 \subset  k$ be a finitely generated extension of~$k_0$ over which~$Z$ is defined. The embedding~$k_0 \into \bbC$ can be extended to~$k_0'$ allowing to replace~$k_0$ by~$k_0'$ and suppose~$Z$ defined over~$k_0$. As in (3) then we can assume~$k =  k_0$ and we borrow the notation introduced above. \Cref{Prop:ComplexBakkerTsimerman} (4) implies that no analytic irreducible component of~$\Phi^{-1}(Z_\bbC^\an)$ is Zariski-dense in~$S_\bbC$. The preimage of~$\hat{\Phi}^{-1}(Z)$ via~$\Spec \hat{\cO}_{S_\bbC, s} \to \Spec \hat{\cO}_{S, s}$ is the formal germ of the finitely many analytic irreducible components of~$\Phi^{-1}(Z^\an_\bbC)$ passing through~$s$. Therefore we conclude again by compatibility of the scheme-theoretic image with extension of scalars.
\end{proof}

\subsection{The~$p$-adic analogue} Suppose~$k$ is a complete valued extension of~$\bbQ_p$ and let~$S^\an$ denote the Berkovich analytification of~$S$. Let~$\cV$ an algebraic differential equation over~$S$. Via the natural isomorphism~$\hat{\cO}_{S,s} \iso \hat{\cO}^\an_{S, s}$ \cite[th.~3.4.1]{Ber90} we can see~$\iota_{\cV, s}$ as the formal germ of an isomorphism between the~$k$-analytic vector bundles~$\cV^\an$ and~$\cV_s \otimes_ k \cO_S^\an$.

\begin{lemma} The formal germ~$\iota_{\cV, s}$ converges on some open neighbourhood~$\Omega \subset S^\an$ of~$s$ to an isomorphism
\[ \iota_{\cV, \Omega, s} \colon \cV^\an_{\rvert \Omega} \stackrel{\sim}{\too} \cV_s \otimes_ k \cO_{\Omega}\]
which is parallel with respect to the connection~$\id \otimes \rd$ on the target.
\end{lemma}

\begin{proof} Let~$z_1, \dots, z_d \in \cO_{S, s}$ be local parameters where~$d = \dim S$. Since~$S$ is smooth the rigid-analytic analogue of the implicit function theorem assures the existence of an open neighbourhood of~$s$ on which the~$z_i$ define a~$k$-analytic isomorphism with an open disc \[\{ x \in \bbA^{d, \an}_ k \mid |t_i| < r_i , i = 1,\dots, d\} \subset \bbA^{d, \an}_ k.\] 
Here~$r_1, \dots, r_d > 0$ and~$t_1, \dots, t_d$ are the coordinates on~$\bbA^d_ k$. Now~$\iota_{\cV, s}$ admits an explicit expression. Identify~$\hat{\cO}_{S, s}$ with~$k[\![z_1, \dots, z_d ]\!]$ and for~$i = 1, \dots, d$ consider the~$k$-linear sheaf endomorphism~$D_i:= \nabla(\partial/\partial z_i) \colon \cV \to \cV$. Then we have 
\[ \iota_{\cV, s} \colon \iota_s^\ast \cV \too \cV_s \otimes_ k \hat{\cO}_{S,s}, \qquad  v \longmapsto \sum_{n \in\bbN^d} D^n(v)_s \otimes \frac{z^n}{n!}, \]
where for a~$d$-tuple~$n = (n_1, \dots, n_d)$ of non-negative integers we used the notation~$z^n = z_1^{n_1} \cdots z_d^{n_d}$,~$n! = n_1! \cdots n_d!$ and~$D^n = D_1^{n_1} \circ \cdots \circ D_d^{n_d}$.\footnote{Notice that the order in the latter composition  is irrelevant: the endomorphisms~$D_i$ commute because the derivations~$\partial/\partial z_i$ do and the connection~$\nabla$ is integrable.} The above power series has positive radius of convergence whence the existence of a neighbourhood~$\Omega$ as in the statement. The parallel character of~$\iota_{\cV, \Omega, s}$ is deduced from the one of~$\iota_{\cV, s}$.
\end{proof}

\begin{example} \label{Ex:ConvergenceOfIsocrystal} Suppose that~$k$ is a finite unramified extension of~$\bbQ_p$ with ring of integers~$\cO_ k$. Let~$S$ be a separated smooth~$\cO_ k$-scheme of finite type with generic fiber~$S_ k$,~$f \colon \cX \to S$ a proper smooth morphism and~$\pi \colon \cY \to \cX$ be a finite \'etale morphism. The vector bundle~$\pi_\ast \cO_\cY$ comes with a natural integrable connection~$\nabla$ induced by the canonical derivation on~$\cO_\cY$. Consider a subvector bundle~$\cE \subset \pi_\ast \cO_\cY$ stable under~$\nabla$. Let~$i,j,q \ge 0$ and suppose that the Hodge to de Rham spectral sequence
\[ R^j f_\ast (\Omega^i_{\cX/S} \otimes \cE) \implies \cH_{\dR}^{i + j}(\cX/S; \cE, \nabla) = \rR f^{i+j}_\ast (\Omega^\bullet_{\cX / S} \otimes \cE) \]
degenerates, where~$\Omega^\bullet_{\cX / S} \otimes \cE$ is the de Rham complex of~$(\cE, \nabla)$, and that the quasicoherent sheaves
\[ \cV := \cH_{\dR}^q(\cX/S; \cE, \nabla) = \rR f^q_\ast (\Omega^\bullet_{\cX / S} \otimes \cE), \qquad R^j f_\ast (\Omega^i_{\cX/S} \otimes \cE), \]
are locally free of finite rank. The vector bundle~$\cV$ comes equipped with the Gauss-Manin connection that we still denote~$\nabla$. Let~$S_\eta \subset S_ k^\an$  be the \emph{Raynaud generic fiber} of~$S$. It is the compact subset whose points with values in a complete valued extension~$k'$ of~$k$ are 
\[ \cS_\eta( k') = \im(\cS(\cO_{ k'}) \into S( k')) \]
where~$\cO_{ k'} \subset  k'$ is the ring of integers. Let~$\bbF$ be the residue field of~$k$ and~$\bbF'$ that of~$k'$. The reduction map~$\cS(\cO_{ k'})\to \cS(\bbF')$ defines an anticontinuous map
\[ \red \colon \cS_\eta \too \tilde{\cS} := \cS \times_{\cO_ k} \bbF. \]
A point~$s \in \cS(\cO_ k)$ being fixed, the open subset 
\[ \Omega := \red^{-1}(\red(s))\]
is called the~$p$-adic disc around~$s$. It owes its name to the fact that formal smoothness of~$\cS$ implies that~$\Omega$ is isomorphic as~$k$-analytic space to some power of open unit disc. Comparison with crystalline cohomology shows that~$\iota_{\cV, \nabla, s}$ converges on the whole~$\Omega$; see \cite[prop. 3.1.1 and section 7]{KatzTravauxDeDwork}.
\end{example}

Let~$\Omega \subset S^\an$ be an open neighbourhood of~$s$ on which it exists a parallel isomorphism~$\iota_{\cV, \Omega, s} \colon \cV^\an_{\rvert \Omega} \to \cV_s \otimes_ k \cO_{\Omega}$ as in the lemma. Then~$\iota_{\cV, \Omega, s}$ can be seen as a morphism~$\Omega \to \PV(\cV, s)^\an$ of~$k$-analytic spaces. Suppose that~$\cV$ comes with a filtration~$\cF^\bullet$ whose graded is locally free. As usual this determines a morphism of analytic spaces
\[ 
\begin{tikzcd}[column sep=35pt]
\Phi_p := \Phi_{\cV, \Omega, s, \cF^\bullet} \colon \Omega \ar[r, "{\iota_{\cV, \Omega, s}}"] & \PV(\cV, s)^\an \ar[r, "{\phi_{\cV, s, \cF^\bullet}}"]& \Flag_{r}(\cV_s)^\an.
\end{tikzcd}
\]
When~$(\cV, \cF^\bullet)$ underlies a variation of integral pure polarized Hodge structures in the sense of \cref{Def:UnderlyingAVHS} we call~$\Phi_p$ the \emph{$p$-adic period map}.

\begin{proposition} \label{Prop:PAdicBakkerTsimerman} Suppose that~$(\cV, \cF^\bullet)$ underlies a variation of integral pure polarized Hodge structures. Then,
\begin{enumerate}
\item the algebraic group~$\Gal(\cV, s)^\circ$ is reductive;
\item the subgroup~$P \subset \Gal(\cV, s)^\circ$ stabilizing~$\cF^\bullet_s$ is parabolic;
\item the Zariski closure of the image of~$\Phi_p$ is~$\frH :=  \Gal(\cV, s)^\circ / P$;
\item let~$Z \subset \frH$ be a subvariety with~$\dim S + \dim Z \le \dim \frH$ and~$Z'$ an analytic irreducible component of~$\Phi^{-1}_p(Z^\an)$. Then~$Z'$ is not Zariski-dense in~$S$.
\end{enumerate}
\end{proposition}

\begin{proof} This is deduced by applying \cref{Prop:FormalBakkerTsimerman} to~$k_0$ and then extending scalars to~$k$. As in the proof of \cref{Prop:FormalBakkerTsimerman} the point is that the formal germ of~$\Phi_p$ is the formal period map~$\Phi$ associated~$(\cV_0, \nabla_0, \cF_0^\bullet)$, or rather the extension to~$k$ of the latter. There is a little hiccup with (4): the~$k$-analytic subspace~$Z'$ may not contain~$s$. In any case~$Z'$ contains a point~$s'$ defined over a finite extension~$k'$ of~$k$. We may replace~$k$ by~$k'$ and~$k_0$ by a finite extension so that~$s'$ is~$k$-rational. In this case~$\iota_{\cV, \Omega, s}$ induces an isomorphism  of~$k$-vector spaces~$\alpha \colon \cV_s \to \cV_{s'}$. The~$p$-adic period map~$\Phi'_p$ centered at~$s'$ is then~$\beta \circ \Phi_p~$
where~$\beta \colon \Flag_{r}(\cV_s) \to \Flag_r(\cV_{s'})$ is the isomorphism induced by~$\alpha$. In particular~$Z'$ is an analytic irreducible component of \[ \Phi'^{-1}_p(\beta(Z^\an)) = (\beta \circ \Phi_p)^{-1}( \beta(Z^\an)) = \Phi_p^{-1}(Z^\an).\]
This permits to apply \cref{Prop:FormalBakkerTsimerman} (4) to the formal germ of~$\Phi'_p$ and the subvariety~$\beta(Z)$. Since the Zariski-closure of~$Z'$ only depends on its formal germ, this concludes the proof.
\end{proof}

\section{The Lawrence-Venkatesh method: Proof of \cref{Thm:MainTheoremForFamiliesWithBigMonodromy}} \label{sec:ProofOfNonDensityThm}

We now put together the results from the preceding sections: We apply~$p$-adic Hodge theory to a suitable \'etale datum and look at the arising~$p$-adic period map to prove the main \cref{Thm:MainTheoremForFamiliesWithBigMonodromy} from the introduction.

\subsection{Statement} Let~$K$ be a number field and~$\Sigma$ a finite set of places of~$K$. Let~$S$ be a smooth integral separated~$\cO_{K, \Sigma}$-scheme of finite type and~$A$ an abelian scheme of dimension~$g$ over~$\cO_{K, \Sigma}$. Let~$\cX \to A_S := A \times S$ be a morphism of $S$-schemes where~$\cX$ is smooth projective over~$S$ with geometrically connected fibers of dimension~$d$:
\[
\begin{tikzcd}
&  \cX \ar[d] \ar[dl, bend right=30, swap] \ar[dr, bend left=30]  & \\
A & A_S \ar[l, swap, "\pr_A"] \ar[r, "\pr_S"]  & S
\end{tikzcd}
\]
\Cref{Thm:MainTheoremForFamiliesWithBigMonodromy} is a consequence of the following theorem:
\begin{theorem} \label{Thm:MainTheoremForFamiliesWithBigMonodromyModel} Suppose that\smallskip
\begin{itemize}
\item the generic fiber of~$\cX \to A_S$ has big monodromy for most tuples of torsion characters; \smallskip
\item every geometric fiber~$X$ of~$\cX \to S$ satisfies \eqref{SkullInequality}.\smallskip 
\end{itemize}
Then~$S(\cO_{K, \Sigma})$ is not Zariski-dense in~$S$.
\end{theorem}
The proof of this statement will take the entire section.

\subsection{Numerics} \label{sec:Numerics} Let~$X$ be a geometric generic fiber of~$\cX \to S$ and
\[
e = (-1)^d \chi_{\top}(X)
\]
where~$\chi_{\top}(X)$ is the topological Euler characteristic of~$X$. Consider the \emph{split} reductive group~$H$ over~$\bbQ$ defined as
\[
H = 
\begin{cases}
\GL_e& \textup{if~$X$ is not symmetric up to translation,} \\
\GO_e & \textup{if~$X$ is symmetric up to translation and~$d$ is even,} \\
\GSp_e & \textup{if~$X$ is symmetric up to translation and~$d$ is odd}.
\end{cases}
\]
Pick~$y \in \cB(H, \bbQ)$ so that the associated filtration~$F^\bullet W$ on~$W = \bbQ^e$  satisfies
\[
\dim \gr^\alpha W =
\begin{cases}
(-1)^{d-\alpha} \chi(X, \Omega^\alpha_X), & \textup{if~$\alpha = 0,\dots,d$},\\
0 & \textup{otherwise}.
\end{cases}
\]
This is possible since~$h_\alpha := (-1)^{d-\alpha} \chi(X, \Omega^\alpha_X) \ge 0$ by \cref{Prop:NumericHodgeSymmetry}. Let $P_y \subset H$ be the subgroup stabilizing $y$. An explicit computation of~$\dim P_y^\ss$ in terms of the dimensions~$h_\alpha = \dim \gr^\alpha W$ case by case for each~$H$ shows the implication
\[
 2\sum_\alpha h_\alpha^2 =
 2 (-1)^d\chi(\Omega^d_{X \times X}) \le \chi_{\top}(X\times X) = e^2
  \implies  2 \dim P_y^\ss < \dim H + \rk H.
\]
Let~$N = N_0(\{0,1,\dots, d\}, \dim S_K + \dim H, \dim H) \ge 1$ be as in \cref{Cor:UniformBoundInFlag}.

\subsection{Construction of the \'etale datum} Let~$A^\natural$ be the moduli space of rank~$1$ connections on~$A$ and~$(\cL,\nabla)$ the universal rank~$1$ connection on~$A \times A^\natural$. Up to enlarging~$\Sigma$ we may assume that the Hodge to de Rham spectral sequence
\[ \rR^j (\pi_S \times \id_{A^\natural})_\ast (\Omega^i_{\cX \times A^\natural/S \times A^\natural} \otimes \cL) \implies \cH^{i + j}_{\dR}(\cX \times A^\natural/S \times A^\natural; \cL,\nabla)\]
degenerates and that the quasicoherent sheaves over~$A^\natural \times S$,
\[ \cH^i_{\dR}(\cX \times A^\natural/S \times A^\natural; \cL,\nabla), \quad \rR^j (\pi_S \times \id_{A^\natural})_\ast (\Omega^i_{\cX \times A^\natural/S \times A^\natural} \otimes \cL),\]
are locally free of finite rank for all~$i, j \in \bbN$, where~$\pi_S \colon \cX \to S$ is the projection. 

\medskip

Consider  a friendly place~$v$ of~$K$ lying over a prime number~$p$ such that~$\Sigma$ contains no~$p$-adic place. This is possible because there are infinitely many friendly places on~$K$. Let~$\bar{K}_v$ be an algebraic closure of the~$v$-adic completion~$K_v$ of~$K$ and~$\bar{K}$ the algebraic closure of~$K$ in~$\bar{K}_v$. This induces an injection 
\[ \Gamma_v :=\Gal(\bar{K}_v/ K_v) \; \intoo \; \Gamma := \Gal(\bar{K}/K).\]
Let~$\bbF_v$ be the residue field at~$v$. The field~$\bbF_v$ is finite thus so is the set~$S(\bbF_v)$. In particular in order to prove that~$S(\cO_{K, \Sigma})$ is not Zariski dense it suffices to prove that each fiber of the reduction map~$S(\cO_{K, \Sigma}) \to S(\bbF_v)$
is not Zariski dense. Let \[F \subset S(\cO_{K, \Sigma})\] be one of these fibers. If $\cX_{\bar{\eta}}$ is not symmetric up to translation, discard the fibers $\cX_s$ for $s \in F$ that are symmetric up to translation, since being symmetric up to translation is a closed condition.\medskip

Fix an embedding~$\sigma \colon \bar{K} \into \bbC$ and let~$\pi_1(A_\sigma, 0)$ be the topological fundamental group of~$A_\sigma$. For a character~$\chi \colon \pi_1(A_\sigma, 0) \to \bbC^\times$ let~$L_\chi$ be the corresponding rank~$1$ local system on~$A_\sigma$. When~$\chi$ is torsion of order~$r$ it comes from a character~$A[r](\bar{K}) \to \mu_r(\bar{K})$ where~$\mu_r(\bar{K}) \subset \bar{K}$ is the set of~$r$-th roots of unity. We let~$\Gamma$ act on such torsion characters via its action on~$A[r](\bar{K})$ and~$\mu_r(\bar{K})$. By  \cite[lemmas 4.8 and 4.9]{LS20} there is an integer~$r \ge 1$ nondivisible by~$p$ and a character~$\chi_0 \colon A[r](\bar{K}) \to \mu_r(\bar{K})$ of order~$r$   such that the following properties hold:
\begin{enumerate}
\item $\rH^i(\cX_{s}(\bbC), \pi^\ast L_{\gamma \chi}) = 0$ for all~$\chi \in \Gamma.\chi_0$,~$s \in S_\sigma(\bbC)$ and~$i\neq d$;
\smallskip
\item for each~$\chi \in \Gamma.\chi_0$ there are~$\chi_1, \dots, \chi_N \in \Gamma_v \chi$ pairwise distinct such that~$\cX_\sigma \to A_{S,\sigma}$ has big monodromy for the~$N$-tuple~$\underline{\chi} = (\chi_1, \dots, \chi_N)$.
\end{enumerate}
Here~$\pi \colon \cX \to A$ is the projection. Property (2) will be crucial to bound the dimension of the centralizer of the crystalline Frobenius; see the end of \cref{Sec:EndOfTheProof} and~\eqref{Eq:KeyLowerBoundDegree}.
We see~$A[r](\bar{K})$ as a constant group scheme over~$\bbQ$ and consider its Cartier dual 
\[ A[r](\bar{K})^\ast = \Hom(A[r](\bar{K}), \bbG_{m, \bbQ})\]
Any character~$\chi \colon A[r](\bar{K}) \to \mu_r(\bar{K})$ induces a point of~$A[r](\bar{K})^\ast$ with values in the~$r$-th cyclotomic extension of~$\bbQ$. This permits to see~$\Gamma.\chi_0$ as a closed subscheme of~$A[r](\bar{K})^\ast$ whose geometric points are of the form~$\chi \in \Gamma.\chi_0^m$ for some integer~$m$ prime to~$r$. We define~$Z$ to be the base change to~$\bbQ_p$ of this closed subscheme. For each~$s \in F$ apply \cref{Ex:LocSysAbVar} with~$X= \cX_{s} \times_{\cO_{K, \Sigma}} K$ to obtain a~$\Gamma$-\'etale datum
\[ (V_s, G_s, \rho_s)\]
on~$Z$ by pullback. For~$s, s' \in F$ the couples~$(V_s, G_s)$ and~$(V_{s'}, G_{s'})$ are isomorphic: Indeed, in the symmetric case Poincar\'e duality gives an isomorphism of local systems~$V\to V^\vee$ inducing the bilinear form~$\theta_s$ on the fiber over each point~$s$, so the isomorphism type of the group preserving this bilinear form is independent of the point~$s$. Fix a vector bundle~$W$ on~$Z$, a reductive group~$Z$-scheme~$H$ and for each~$s \in \Omega$ an isomorphism~$(V_s, G_s) \iso (W,H)$. Let~$\cI \subset \cB(H, Z) / H(Z)$ be a subset as in the statement of \cref{Lemma:GoodSetOfIndices}. For each~$s \in F$ there is~$x_s \in \cB(G_s, Z)$ fixed under~$\Gamma$ and mapped in~$\cI$ via the isomorphism~$(V_s, G_s) \iso (W,H)$ with the property that the graded of the filtered~$\Gamma$-\'etale datum
\[D_s:= (V_s, G_s, \rho_s, x_s)\] 
is semisimple.

\subsection{The Lawrence-Venkatesh strategy}
By the version of Faltings' finiteness of pure global Galois representations given in \cref{Prop:GradedFaltingsFiniteness}, the image of the composite map
\[
\begin{tikzcd}[row sep=0pt]
F \ar[r]
& { \left\{ \begin{array}{c} \textup{filtered~$\Gamma$-\'etale}\\\textup{data on~$Z$} \end{array} \right\} / \iso} \ar[r,"\textup{graded}"]
& { \left\{ \begin{array}{c} \textup{graded~$\Gamma$-\'etale}\\\textup{data on~$Z$} \end{array} \right\} / \iso} \\
s \ar[r, mapsto]& {[D_s]} \ar[r, mapsto] & {[\gr D_s]}
\end{tikzcd}
\]
is finite (here the brackets stand for the isomorphism class). So to show that~$F$ is not Zariski dense, it suffices to show that the fibers of the map \[ F \ni s \longmapsto [\gr D_s]\] are not Zariski dense. 
For this we continue the preceding diagram and apply~$p$-adic Hodge theory to pass to the de Rham side:
\[
\begin{tikzcd}
{ \left\{ \begin{array}{c} \textup{filtered~$\Gamma$-\'etale}\\\textup{data on~$Z$} \end{array} \right\} } \ar[d, "\textup{restriction to~$\Gamma_v$}"']\ar[r,"\textup{graded}"] \ar[d]
& { \left\{ \begin{array}{c} \textup{graded~$\Gamma$-\'etale}\\\textup{data on~$Z$} \end{array} \right\}} \ar[d, "\textup{restriction to~$\Gamma_v$}"] \\
{ \left\{ \begin{array}{c} \textup{filtered~$\Gamma_v$-\'etale}\\\textup{data on~$Z$} \end{array} \right\}} \ar[r,"\textup{graded}"] \ar[d, "{D \mapsto D_{\dR}}"']
& { \left\{ \begin{array}{c} \textup{graded~$\Gamma_v$-\'etale}\\\textup{data on~$Z$} \end{array} \right\}} \ar[d, "{D \mapsto D_{\dR}}"] \\
{ \left\{ \begin{array}{c} \textup{filtered de Rham}\\\textup{data on~$Z_{\dR}$} \end{array} \right\}} \ar[r,"\textup{graded}"]
& { \left\{ \begin{array}{c} \textup{graded de Rham}\\\textup{data on~$Z_{\dR}$} \end{array} \right\}}
\end{tikzcd}
\]
This is possible because~$r$ is not divisible by~$p$, so the action of~$\Gamma_v$ on~$Z$ is unramified, and the ring~$\Gamma(Z, \cO_Z)$ is the product of unramified extensions of~$\bbQ_p$. It is then enough to prove that the fibers of the map \[F \ni s \longmapsto [(\gr D_s \rvert_{\Gamma_v})_{\dR}]\] are not Zariski dense, where for a filtered or graded~$\Gamma$-\'etale datum~$D$ we wrote~$D \rvert_{\Gamma_v}$ for its restriction to~$\Gamma_v$. Let~$A^\natural$ be the moduli space of rank~$1$ connections on~$A$ and let~$(\cL, \nabla)$ be the universal rank~$1$ connection on~$A \times A^\natural$. By \cref{Ex:ComparisonEtaleLocalSystemDeRhamFlatBundles} we have that~$Z_{\dR}$ is a closed subscheme of~$A^\natural_{K_v}$ made of points of order~$r$ and
 \[ (D_s \rvert_{\Gamma_v})_{\dR} \iso ( V_{\dR, s}:= \cH^d_{\dR}(\cX_s \times Z_{\dR}/Z_{\dR}; \pi^\ast(\cL,\nabla)), G_{ \dR, s}, \phi_s, h_s).\]
Recall that~$\phi_s$ is the Frobenius operator acting on de Rham cohomology via the de Rham-crystalline comparison theorem. Moreover, 
\[
G_{\dR, s} = 
\begin{cases}
\GL(V_{\dR, s}) & \textup{if~$\cX$ is not symmetric up to translation,} \\
\GO(V_{\dR, s}, \theta_s) & \textup{if~$\cX$ is symmetric up to translation and~$2 \mid d$,} \\
\GSp(V_{\dR, s}, \omega_s) & \textup{if~$\cX$ is symmetric up to translation and~$2 \nmid d$,}
\end{cases}
\]
where~$\theta_s \colon V_{\dR, s} \otimes V_{\dR, s}\to \cL_{a(s)} \otimes \cO_{Z_\dR}$ is the pairing induced by Poincar\'e duality if~$\cX = - \cX +a$ for some~$a \colon S \to A$. With this notation the point~$h_s$ in~$ \cB(G_{\dR, s}, Z_{\dR})$ is the one inducing the Hodge filtration on~$V_{\dR, s}$. 

\subsection{Transliteration via parallel transport} 
\label{sec:ParallelTransport} Let~$\cO_{K, v} \subset K_v$ be the ring of integers and~$\bbF_v$ the residue field of~$K_v$. Consider~$S_v := S \times_{\cO_{K, \Sigma}} \cO_{K, v}$ and~$S_{v, \eta}$ the Raynaud generic fiber of~$S_v$. As in \cref{Ex:ConvergenceOfIsocrystal} we see~$S_{v, \eta}$ as a compact subset of the Berkovich analytic space~$S_{K_v}^\an$ attached to the  variety~$S_{K_v}$ over~$K_v$. In this framework we have an anti-continuous reduction map
\[ \red \colon S_{v, \eta} \too \tilde{S}_v = S_v \times_{\cO_{K,v}} \bbF_v.\]
Fix~$o \in F$. The open subset~$\Omega := \red^{-1}(\red(o))$ is such that
\[\Omega(K_v) = \{ s \in S (\cO_{K, v}) \mid s \equiv o \pmod{p}\}.\]
In particular~$F$ is contained in~$\Omega$. The relative~$d$-th de Rham cohomology group
\[ \cV := \cH_{\dR}^d(\cX \times Z_{\dR}/ S \times Z_{\dR}; \pi^\ast (\cL, \nabla))\]
is a coherent sheaf over~$S \times Z_{\dR}$ and comes equipped with the Gauss-Manin connection~$\nabla_{\textup{GM}}$. Actually note that~$\cV$ is locally free by the choice of~$\Sigma$ at the beginning of the proof. For each~$s \in F$, we have
\[ \cV_{s} := \cV_{\rvert \{s \} \times Z_{\dR}} = V_{\dR, s}. \]
Applying the discussion in \cref{Ex:ConvergenceOfIsocrystal} over any point of~$Z_{\dR}$ the Gauss-Manin connection can be integrated over~$\Omega \times Z_{\dR}^\an$ to give an analytic isomorphism of vector bundles
\[ \tau \colon \cV_{\rvert \Omega \times Z^\an_{\dR}} \stackrel{\sim}{\too} \cV_{o} \otimes \cO_{\Omega \times Z^\an_{\dR}} \quad \textup{where} \quad  V :=  \cV_o = \cV_{\rvert\{ o\} \times Z_{\dR}}.\]
For~$s \in F$ the isomorphism of vector bundles~$f_s \colon \cV_s = V_{\dR, s} \to \cV_{o}= V_{\dR, o}$ is such that
\[ 
 G := G_{\dR, o}\;=\; \tau_{s} G_{\dR, s} \tau_{s}^{-1} \qquad \text{and} \qquad \phi := \phi_o \;=\; \tau_{s} \phi_{s} \tau_{s}^{-1}.
 \]
This comes from the compatibility of parallel transport with Poincar\'e duality and the isomorphism given by the de Rham-crystalline comparison theorem \cite[th. 7.1]{BerthelotOgus}. The isomorphism~$\tau$ permits to see~$x_s, h_s \in \cB(G_{\dR, s}, Z_{\dR})$ as points in~$\cB(G, Z_{\dR})$. For each~$s \in F$ write~$Q_s$ for the stabilizer of~$x_s$ in~$G$. This being set up, if~$s, s' \in F$ are such that we have an isomorphism of graded de Rham data
\[ (\gr D_s \rvert_{\Gamma_v})_{\dR} \; \iso \; (\gr D_{s'} \rvert_{\Gamma_v})_{\dR} \]
over~$Z_{\dR}$, then by \cref{isomorphic-graded-de-rham} there is~$g \in G(Z_{\dR})$ such that
\begin{align*}
x_{s'} =g x_s, &&  g \phi g^{-1} \phi^{-1} \in \rad Q_{s'}, && \pr_{s'}(h_{s'}) = \pr_{s'}(g h_{s}),
\end{align*}
where~$\pr_{s'} \colon \cB(G, Z_{\dR}) \to \cB(Q_{s'}^{\circ, \ss}, Z_{\dR})$ is the projection. Notice that we can apply \cref{isomorphic-graded-de-rham} because the subgroup~$G \subset \GL(V)$ is full. Let~$\alpha$ be the automorphism of~$\Res_{Z_{\dR}/K_v} G$ given by conjugation by~$\phi$. Then the above implies that the couples 
\[ (h_{s}, \Res_{Z_{\dR}/K_v} Q_{s}^\circ) \quad \textup{and} \quad  (h_{s'}, \Res_{Z_{\dR}/K_v} Q_{s'}^\circ)\]
are~$\alpha$-equivalent in the sense of \cref{Def:BalancedRelation}. In order to apply \cref{Thm:BoundInFlag} we still miss the positivity assumption on the above couples. For each~$s \in S$ the~$\Gamma$-\'etale datum~$D_s$ is pure of weight~$d$ thus \cref{Prop:PositivityFromGlobalPurity} implies that there is~$z \in Z_{\dR}$ such that 
\[h_{s,z} \in \cB(G_{z},  k_z) \quad \text{is positive with respect to} \quad Q_{s,z}^\circ,\]
where~$k_z$ is the residue field at~$z$. The point~$z$ depends a priori on~$s \in F$, but in order to show that~$F$ is not Zariski dense we may suppose that it does not depend on~$s$ because~$Z_{\dR}$ is finite. In this case is suffices to show that the fibers of the map
 \[F \ni s \longmapsto [(\gr D_s \rvert_{\Gamma_v})_{\dR}\rvert_{\{ z\}}]\]  
are not Zariski dense where for a graded de Rham datum~$D$ on~$Z_{\dR}$ we wrote~$D_{\rvert \{ z \}}$ for its restriction to~$\{ z \}= \Spec  k_z$. 

\subsection{The period mapping} \label{Sec:EndOfTheProof} At this stage the scheme~$Z_{\dR}$ plays no role any more, thus in what follows we replace~$Z_{\dR}$ by the singleton~$\{ z\} = \Spec  k$ where~$k =  k_z$. Accordingly we have~$G= G_z$,~$V = V_z$, etc. from now on. For a variety~$T$ over~$k$ write~$\Res T$ for its Weil restriction to~$K_v$ whenever it exists. Notice that~$S \times Z_{\dR}$ then becomes the variety~$S_ k$ over~$k$ obtained from~$S$ by extending scalars to~$k$. Let~$\cV_0$ be the push-forward of~$\cV$ along the projection~$S_{ k} \to S_{K_v}$ and~$\nabla_0$ be the connection on~$\cV_0$ induced by the Gauss-Manin connection on~$\cV$. Let
\[ \Gal(\cV_0, o) \subset \Res G\]
be the differential Galois group for the Gauss-Manin connection on~$\cV_0$. Let~$\cF^\bullet$ be the Hodge filtration on~$\cV_0$. The pair~$(\cV_0,\cF^\bullet)$ underlies a variation of integral pure polarized Hodge structures \cite[th. II.7.9]{Del70}. By \cref{Prop:PAdicBakkerTsimerman} (2) the subgroup~$P \subset \Gal(\cV_0, o)^\circ$ stabilizing the flag~$\cF^\bullet_o$ is a parabolic subgroup, hence we have an inclusion
\[ \frH := \Gal(\cV_0, o)^\circ  /P \intoo \Par_{t} (\Res G^\circ) \]
where~$t$ is the type of the parabolic subgroup of~$\Res G^\circ$ stabilizing the Hodge filtration~$h_o \in \cB(G,  k)$. Consider the associated~$p$-adic period mapping
\[ \Phi_p \colon \Omega \too \frH^\an. \]
By \cref{Ex:CristallineActionOnFiniteScheme} the point~$z$ corresponds to the~$\Gamma_v$-orbit of some~$\chi \in \Gamma.\chi_0^m$ with~$m$ an integer prime to~$r$. With this in mind (plus the extension~$k$ of~$K_v$ being Galois) for a variety~$T$ over~$k$ we have a natural isomorphism
\begin{equation} \label{eq:WeilRestrictionAsProduct}(\Res T)_ k \iso T^{\Gamma_v \chi} \end{equation}
whenever the Weil restriction of~$T$ exists. By property (2) in the choice of~$\chi_0$ there are~$\chi_1, \dots, \chi_N \in \Gamma_v \chi$ pairwise distinct such that~$\cX_\sigma \to A_{S,\sigma}$ has big monodromy for the~$N$-tuple~$\underline{\chi} = (\chi_1, \dots, \chi_N)$. In particular,
\begin{equation} \label{Eq:KeyLowerBoundDegree} [k : K_v] = |\Gamma_v . \chi| \ge N.\end{equation}
Consider the projection onto the factors corresponding to~$\chi_1, \dots, \chi_N \in \Gamma_v \chi$:
\[ \pi \colon (\Res G)_ k \too \bar{G} := G^N.\]
The hypothesis of having big monodromy implies that~$\pi$ induces a surjection
\[ \Gal(\cV,o)_ k \ontoo \bar{G}. \]
Let~$\bar{t}$ the type of parabolics of~$\bar{G}^\circ$ induced by the type~$t$ of~$G^\circ$. Then the composite morphism
\[ \frH_ k \intoo \Par_t(\Res G^\circ)_ k \ontoo \Par_{\bar{t}}(\bar{G}^\circ) \]
is proper smooth with equidimensional fibers. For each~$s \in F$ let~$h_{s,  k}$ be the point of~$\cB(\Res G,  k)$ induced by the point~$h_s$ of~$\cB(\Res G, K_v) = \cB(G,  k)$. Via the isomorphism \eqref{eq:WeilRestrictionAsProduct} we have~$\cB(\Res G,  k) = \cB(G,  k)^{\Gamma_v \chi}$ and the point~$x_{s,  k}$ is the diagonal embedding of~$h_s$; moreover~$h_{s,  k}$ is positive with respect to the parabolic subgroup 
\[ (\Res Q)_ k^\circ \iso (Q_s^\circ)^{\Gamma_v \chi} \]
because each coordinate of~$h_{s,  k}$ with respect to this product decomposition is~$Q_{s,  k}^\circ$-positive. It follows that the point~$\bar{h}_{s} \in\cB(\bar{G},  k)$ induced by~$h_{s,  k}$ is positive with respect to~$\bar{Q}_{s}^\circ$ where~$\bar{Q}_s$ is the image in~$\bar{G}$ of~$(\Res Q)_ k$. Let
\[ F_{s} = \{ s' \in F \mid (\gr D_{s' }\rvert_{\Gamma_v})_{\dR}\rvert_{\{ z\}} \iso (\gr D_s \rvert_{\Gamma_v})_{\dR}\rvert_{\{ z\}}\}\]
Now a point~$s \in F$ we have \[\Phi_p(s) = P_{h_s}^\circ \]
 as points of~$\Par (G^\circ)$, so the above discussion shows that~$F_s$ is contained in the preimage via the~$p$-adic period mapping~$\Phi_p$ of the subset 
\[E_s := \pr_{1}([ h_{s,  k}, (\Res Q_s)_ k^\circ]_\alpha) \cap \frH_ k\] where~$\pr_1 \colon \frH \times \Par(\Res G^\circ)_ k \to \frH$ is the projection onto the first factor and, with the notation of \cref{Def:BizarreEqRel},
\[ [ h_{s,  k}, (\Res Q_s)_ k^\circ]_\alpha := \{ (P_{x'},Q) \mid (x', Q') \sim_\alpha (h_{s,  k}, (\Res Q_s)_ k^\circ), \; (\bar{h}', \bar{Q}')  \succeq 0 \}. \]
The slopes of the function~$\Upsilon_{G,h_s}$ lie in~$\{ 0, \dots, d\}$ and with the notation of \cref{sec:Numerics} we have~$\dim G =  \dim H$ and~$\dim \Stab_G(h_s)^\ss = \dim\Stab_H(y)$ by \cref{Prop:NumericHodgeSymmetry}. Let~$C$ be the subgroup of~$\Res G$ of elements fixed by the semisimplification of~$\alpha$. Then by \cite[lemma 5.33]{LS20} we have
\[ \dim C \le \dim G = \frac{\dim \Res G}{[ k : K_v]}. \]
By \eqref{Eq:KeyLowerBoundDegree} the extension~$k$ of~$K_v$ has degree~$\ge N$. The choice of~$N$ implies that we can apply \cref{Cor:UniformBoundInFlag} to~$\Res G$ with~$n = \dim S_K + \dim H$ and obtain that the Zariski closure of~$E_s$ in~$\frH_ k$ has codimension~$\ge \dim S_K$. The~$p$-adic version of the Bakker-Tsimerman theorem given in \cref{Prop:PAdicBakkerTsimerman} (4) implies that
\[F_s \subset \Phi_p^{-1}(E_s) \;  \subset \; S_v\]
is not Zariski-dense. This concludes the proof. \qed

\appendix

 \section{Conormal geometry on abelian varieties} 
 \label{sec:ConormalGeometry}
 
In this appendix we characterize products of subvarieties and symmetric powers of curves in abelian varieties in terms of conormal geometry.
  
 \subsection{Gauss maps and conormal varieties} Let~$A$ be an abelian variety of dimension~$g$ over a field~$k$ of characteristic zero. The cotangent bundle~$\Omega^1_A$ is trivial with fiber~$(\Lie A)^\vee=\rH^0(A, \Omega^1_A)$. The \emph{Gauss map} of a subvariety~$\PLambda \subset \bbP(\Omega^1_A)$ is the composite morphism
\[
 \gamma_{\PLambda} \colon \quad \PLambda \;\intoo\; \bbP(\Omega^1_A) \; \iso \; A \times \bbP_A \;\too\; \bbP_A,
\]
where the second map is the second projection. The scheme theoretic image of the first projection~$\PLambda \to A$ will be called the \emph{base} of~$\PLambda$. If~$\PLambda$ is integral of dimension~$g - 1$, we say that the subvariety~$\PLambda$ is \emph{clean} if its Gauss map is dominant (or equivalently generically finite); otherwise we say that~$\PLambda$ is \emph{negligible}. More generally, a cycle~$\PLambda$ of pure dimension~$g -1$ on~$\bbP(\Omega_A^1)$ is called~\emph{clean} resp. \emph{negligible} if every irreducible component of its support is clean resp. negligible in the previous sense. By the Gauss map of a cycle~$\PLambda$ we mean that of its support. The notion of base is extended to cycles by linearity.
\medskip

The main example is the \emph{conormal variety} of an integral subvariety~$X\subset A$, that is, the Zariski closure
\[
 \PLambda_X \;\subset\; \bbP(\Omega^1_A)
\]
of the projective conormal bundle to the smooth locus of~$X$. If~$\Stab_A(X)$ denotes the stabilizer of~$X$ in~$A$, then we have the following equivalences:
\[ \PLambda_X \textup{ is clean} \iff \Stab(X) \textup{ is finite} \iff X \textup{ is of general type}.\]
Indeed, the first one is \cite[th. 1]{WeissauerArxiv2015a}, the second follows from Ueno's fibration theorem \cite[th.~3.10]{UenoCompositio} \cite[th.~3]{Abr94}. If~$X$ is not necessarily integral, consider its fundamental cycle 
\[ 
 [X] \;=\; a_1 [X_1] + \cdots + a_n [X_n]
\]
where the~$X_i \subset A$ are the irreducible components of~$X$ and~$a_i \in \bbN$. We then define the conormal variety as the cycle
\[
 \PLambda_X \;=\; a_1 \PLambda_{X_1} + \cdots + a_n \PLambda_{X_n}
\]
of pure dimension~$g -1$. By the Gauss map of~$X$ we mean that of~$\PLambda_X$.\medskip

Finally, recall that a cycle~$\PLambda$ on~$\bbP(\Omega^1_A)$ is called \emph{Lagrangian} if its support~$S$ is of pure dimension~$g - 1$ and the contact form on~$\bbP(\Omega^1_A)$  vanishes identically on~$S$ \cite[3.1]{KleimanConormal}. Linear combinations of conormal varieties are Lagrangian \cite[prop. 3.2]{KleimanConormal} and the converse holds since~$k$ is of characteristic~$0$ \cite[cor. 3.5]{KleimanConormal}.

\subsection{The ring of clean cycles} We define the group of clean cycles~$\cC(A)$ to be the quotient of the free abelian group generated by cycles~$\PLambda$ of pure dimension~$g-1$ on~$\bbP(\Omega^1_A)$ modulo the subgroup of negligible ones. The projection onto the quotient induces an isomorphism
\[
   \bigoplus_{\PLambda} \;\bbZ\cdot \PLambda \; \stackrel{\sim}{\too} \;  \cC(A),
\]
where the direct sum runs over clean integral subvarieties~$\PLambda \subset \bbP(\Omega^1_A)$. 

\medskip 

The group~$\cC(A)$ is endowed with a ring structure: Let~$\PLambda_1, \PLambda_2 \subset \bbP(\Omega^1_A)$ be clean integral subvarieties.  Let~$U\subset \bbP_A$ be a nonempty open subset over which
$
 \PLambda_{i \rvert U} := \PLambda_i \times_{\bbP_A} U
$
is finite flat for~$i =1, 2$. Define the convolution \[\PLambda_1 \circ \PLambda_2 \;\in\; \cC(A)\] to be the pushforward of the fundamental cycle of
\[
\overline{\PLambda_{1 \rvert U} \times_U \PLambda_{2 \rvert U}}
\;  \subset \; A\times A \times \bbP_A 
\]
under the sum morphism~$A\times A \times \bbP_A \to A\times \bbP_A$. We extend this product~$\circ$ bilinearly to a product
\[
\circ\colon \quad \cC(A) \times \cC(A) \;\too\; \cC(A),
\]
which endows~$\cC(A)$ with a natural ring structure. Note that if~$\PLambda_1$ and~$\PLambda_2$ have generically \'etale Gauss maps, then so does their convolution.

\subsection{Wedge powers} Let~$n \ge 1$ be an integer and~$\PLambda_1, \dots, \PLambda_n \subset \bbP(\Omega^1_A)$ clean integral subvarieties. Let~$U \subset \bbP(\Omega^1_A)$ be a nonempty open subset over which \[ \PLambda_{i \rvert U} := \PLambda_i \times_{\bbP_A} U\] is finite flat for~$i =1, \dots, n$. The wedge product
\[ \PLambda_1 \wedge \cdots \wedge \PLambda_n \in \cC(A)\]
is the image in~$\cC(A)$ of the pushforward of the fundamental cycle of
\[
\overline{(\PLambda_{1 \rvert U} \times_U \cdots \times_U \PLambda_{n \rvert U}) \smallsetminus (\Delta \times \bbP_A)}
 \; \subset \; A^n \times \bbP_A 
\]
under the sum morphism~$A^n \times \bbP_A \to A\times \bbP_A$ where~$\Delta \subset A^n$ is the big diagonal. We extend the wedge product to clean cycles by multilinearity. Note that if the Gauss maps of~$\PLambda_1, \dots, \PLambda_n$ are generically \'etale, then so is the Gauss map of their wedge product. For a clean cycle~$\PLambda$ we define
\[ \Alt^n \PLambda \; := \; \tfrac{1}{n!} \, \PLambda \wedge \cdots \wedge \PLambda \; \in \; \cC(A). \]

\begin{lemma} \label{lemma:WedgePowerLagrangian}
Let~$\PLambda$ be a clean effective cycle on~$\bbP(\Omega^1_A)$ with base~$Y$, and let~$n \ge 1$ be such that~$\Alt^n \PLambda$ has positive dimensional base~$X$. If~$\Alt^n \PLambda$ is geometrically integral, then so is~$\PLambda$ hence~$Y$.
\end{lemma}

\begin{proof} We may suppose~$k$ algebraically closed.  Pick an open dense~$U \subset \bbP_A$ over which the Gauss map of~$\PLambda$ is finite flat. For the integrality we argue as in \cite[lemma~5.18]{JKLM}. Fix~$v \in U(k)$ and let~$x_1, \dots, x_d\in A(k)$ be the points in the fiber of the Gauss map~$\PLambda \to \bbP_A$ counted with multiplicities so that 
$\PLambda_{v} = [x_1] + \cdots + [x_d]$
as zero-dimensional cycles on~$A$. By taking the fiber over~$v$ in~$\Alt^n \PLambda$ we then get
\[
 (\Alt^n \PLambda)_v \;=\; \tfrac{1}{n!} \sum_{1\le i_1 < \cdots < i_n \le d} [ x_{i_1} + \cdots + x_{i_n} ].
\]
Since~$\Alt^n \PLambda$ is integral, we identify it with its support. By hypothesis there is a nonempty open subset~$V \subset U$ over which the Gauss map of~$\Alt^n \PLambda$ is finite \'etale. For~$v \in V(k)$ then all the summands~$x_{i_1} + \cdots + x_{i_n}$ in the above equation are pairwise distinct. But~$\dim X > 0$ forces~$n<d$, hence in this case the points~$x_1, \dots, x_d$ are pairwise distinct as well. Since this holds for the fiber of the Gauss map~$\PLambda \to \bbP_A$ over generic~$v$, we get that~$\PLambda$ is reduced. In passing, note that~$x_1, \dots, x_d$ being pairwise distinct implies that the Gauss map of~$\PLambda$ is generically \'etale.
Now if~$\PLambda$ were not integral, say~$\PLambda = \PLambda_1 + \PLambda_2$ for clean effective cycles~$\PLambda_1, \PLambda_2$, then
\[
 \Alt^n \PLambda \;=\; 
 \!\!\!\sum_{n_1 + n_2 = n} \Alt^{n_1} \PLambda_1 \circ \Alt^{n_2} \PLambda_2.
\]
The sum on the right hand side contains at least two nontrivial summands because~$\deg(\PLambda_1 \to \bbP_A) + \deg(\PLambda_2\to \bbP_A) = d > n$, see loc.~cit. This would contradict the integrality~$\Alt^n \PLambda$.
\end{proof} 

In the above setting, we can characterize symmetric powers of curves on abelian varieties in terms of their conormal varieties:

\begin{lemma} \label{lem:symmetric-power-criterion}
Let~$X\subset A$ be a smooth geometrically irreducible subvariety of dimension~$n$ with generically \'etale Gauss map, and let~$C \subset A$ be a curve whose irreducible components are of general type 
when endowed with their reduced structure. Then, the following are equivalent:
\begin{enumerate} 
\item $\Alt^n \PLambda_{C} = \PLambda_{X}$.
\item The sum induces an isomorphism
$\Sym^n C \iso X$.
\end{enumerate} 
Moreover, if these properties hold,~$C$ is smooth and geometrically irreducible.
\end{lemma}

\begin{proof} If (1) holds, then the definitions imply that~$X$ is the sum of~$n$ copies of the curve~$C \subset A$. The sum morphism therefore induces a morphism
\[
 \sigma\colon \quad \Sym^n C \;\too\; X,
\]
To see that~$\sigma$ is an isomorphism, consider the abelian subvariety 
$ B \subset A$ generated by the curve, i.e.~the smallest abelian subvariety containing the difference of any two points on the curve. After a translation both~$C$ and~$X$ are contained in~$B$, and their conormal varieties in~$A$ are obtained from the respective conormal varieties in~$B$ by taking closures of the preimages under the projection
\[f \colon \quad B\times (\bbP_A \smallsetminus \bbP(V)) \; \too \; B\times \bbP_B,\]
where~$V = (\Lie A/B)^\vee$.  By construction the formation of wedge powers of clean cycles commutes with taking closure of the preimages under~$f$. Hence for the proof of (2) we may replace~$A$ by~$B$ and therefore assume that~$C$ generates~$A$ in what follows. In this case, the Gauss map
\[
 \gamma_C\colon \quad \PLambda_{C} \;\too\; \bbP_A
\]
is a finite morphism. By our assumption
\[ \PLambda_{X} \; = \; \Alt^n \PLambda_{C}, \]
it follows that~$\gamma_X\colon \PLambda_{X} \to \bbP_A$ is also finite. Now consider the following commutative diagram which comes from our assumption~$\Alt^n \PLambda_C = \PLambda_X$ and from the definition of wedge powers of cycles:
\[
\begin{tikzcd}[column sep=40pt]
 \Sym^n C  \ar[d, swap, "\sigma"] &  \PLambda_C^{[n]}  \ar[d, "\tilde{\sigma}"]  \ar[l, swap, "\pi_{C,n}"]  \ar[r,"\gamma_{C,n}"] &  \bbP_A  \ar[d, equals] \\
X &  \PLambda_X \ar[l, "\pi_X"] \ar[r, swap, "\gamma_X"]  & \bbP_A
\end{tikzcd} 
\]	
The fibers of~$\pi_X$ have pure dimension~$N=\codim_A X - 1$. Moreover, by the above diagram~$\tilde{\sigma}$ is a finite morphism since~$\gamma_{C,n}$ and~$\gamma_X$ are finite. So all the fibers of~$\pi_X \circ \tilde{\sigma}$ have dimension~$N$. Since~$\sigma$ is generically finite, it follows that the generic fiber of~$\pi_{C, n}$ also has dimension~$N$. The semicontinuity of the fiber dimension for proper morphisms then implies that the morphism 
$\sigma \colon \Sym^n C \to X$ is finite. On the other hand the identity~$\Alt^n \PLambda_{C} = \PLambda_{X}$ implies that~$\sigma$ is birational. The smoothness of~$X$ then forces~$\sigma$ to be an isomorphism.

\medskip

Conversely, if (2) holds, then the curve~$C$ is integral and smooth; see for instance \cite[prop. B.3]{JKLM}. To prove~$\PLambda_X = \Alt^n \PLambda_C$ we may assume~$k$ algebraically closed. Then, given pairwise distinct points~$x_1, \dots, x_n \in C(k)$, the tangent space to~$X$ at~$x=x_1+\cdots + x_n$ is
\[
 T_{X,x} \;=\; T_{C,x_1} \oplus \cdots \oplus T_{C,x_n}
\]
where all the occurring tangent spaces are viewed as subspaces of~$\Lie A$. In particular, the conormal of~$X \subset A$ at~$x$ is the intersection of the conormal subspaces of~$C \subset A$ at~$x_i$ for~$i= 1, \dots, n$. Unwinding the definitions, this implies the wanted identity~$\PLambda_X = \Alt^n \PLambda_C$. 
\end{proof}

In the above lemma one cannot drop the assumption that all irreducible components of the curve are of general type. Indeed,  condition~$(1)$ is unaffected if we add to~$C$ extra components that are not of general type, whereas condition~$(2)$ excludes the existence of such extra components. In fact, any curve summand of an iterated sum of a curve $C$ has to be $C$. More precisely:

\begin{lemma} \label{lemma:CurveSummandOfGeneralType}
Let~$X\subset A$ be a geometrically integral subvariety of general type of dimension~$n$ which is the~$n$-fold sum of a subvariety $C \subset A$ of pure dimension~$1$. Suppose that $X$ has fundamental cycle
\[ [X] \; = \; \tfrac{1}{n!} \sigma_\ast [C^n], \]
where~$\sigma \colon A^n \to A$ is the sum map and that $X$ is also the~$n$-fold sum of some irreducible component~$C' \subset C$ endowed with its reduced structure. Then,~$C'$ is geometrically integral and~$C$ is generically reduced with underlying reduced subvariety~$C'$.
\end{lemma}

\begin{proof} We may assume~$k$ algebraically closed. The hypothesis~$[X] = \tfrac{1}{n!} \sigma_\ast [C^n]$ forces~$C'$ to enter with multiplicity one in~$[C]$ hence~$C$ is generically reduced along~$C'$. Also, it implies that~$C'$ is the unique component for which the sum map~$C' \times Y \to X$ is generically finite, where~$Y$ is the sum of~$n-1$ copies of~$C'$. If there is another irreducible component~$C''\subset C$, then after a translation we may assume~$0\in C''(k)$, so that
\[ Y \;\subset\; C''+Y \;\subset\; C + Y \;=\; X, \]
The second inclusion must be strict because the sum map~$C''\times Y \to X$ is not generically finite. Thus~$Y = C'' + Y$ because~$\dim Y = n-1$ by generic finiteness of the sum map~$C'\times Y \to X$. In particular,~$Y$ is stable under translations by all points in~$C''$. Since~$Y$ is a summand of~$X$, we have the following chain of inclusion
\[ C'' \; \subset \; \Stab_A(Y) \; \subset \; \Stab_A(X),\]
contradicting that~$X$ is of general type. \end{proof}

\section{Conjugacy classes of disconnected reductive pairs}

In this appendix we verify the fact about conjugacy classes that was used in the proof of \cref{lemma:FaltingsFiniteness}. Let~$k$ be an algebraically closed field of characteristic zero. To simplify notation, we identify algebraic varieties over~$k$ with the set of their~$k$-rational points. Let~$V$ be a finite dimensional~$k$-vector space. For algebraic subgroups~$H \subset G \subset \GL(V)$ consider the set
\[ \cC_G(H) := \{ H' \subset G \mid H' = g H g^{-1} \; \text{for some}\; g \in \GL(V)\}. \]
The group~$G$ acts naturally on~$\cC_G(H)$ by conjugation.

\begin{proposition} \label{Prop:DisconnectedRichardsonThm} If~$G$ and~$H$ are reductive, then the set~$\cC_G(H) /G$ is finite.
\end{proposition}

\begin{proof} When~$H$ is finite this is \cite[Cor. 16.2]{Vinberg}. We next deal with the case when~$H$ is connected: Then~$\cC_{G}(H) = \cC_{G^\circ}(H)$, hence we may assume that~$G$ is also connected; in this case it suffices to show for the adjoint representation~$\ad$ of~$\GL(V)$ that 
\[ \{ \mathfrak{h}' \mid \mathfrak{h}' = \ad(g).\Lie H \subset \Lie G, g \in \GL(V)\}\]
is a finite union of~$G$-orbits, which holds by~\cite[th. 7.1]{Richardson}.\medskip

When~$H$ is arbitrary, we combine the previous two cases as follows. By \cref{Lemma:RepresentingGroupOfComponents} below, there is a finite subgroup~$F\subset G$ such that the projection~$F \to G/G^\circ$ is surjective. From the already proven result in the connected case we know that the set~$\cC_G(H^\circ)$ consists of finitely many~$G$-orbits. Fix such an~$G$-orbit, say the~$G$-orbit of~$g_0 H^\circ g_0^{-1} \subset G$. Writing~$H' = g_0 H g_0^{-1}$ it suffices to prove that the set
\[ \cC':= \{g H' g^{-1} \mid g \in N_{\GL(V)}(H'^\circ), g H' g^{-1} \subset G \} \]
consists of finitely many~$N_G(H'^\circ)$-orbits, where~$N_G(H'^\circ)$ resp.~$N_{\GL(V)}(H'^\circ)$ is the normalizer of~$H'^\circ$ in~$G$ resp.~$\GL(V)$. Replacing~$H$ by~$H'$ we may assume~$H= H'$. Now for~$g \in N_{\GL(V)}(H^\circ)$ we have~$ gHg^{-1} = \langle H^\circ, gFg^{-1} \rangle$ because~$H = \langle H^\circ, F\rangle$. Since the identity component~$H^\circ$ is contained in~$G$, we have that \[ gHg^{-1}\subset G \iff gFg^{-1}\subset G \iff gFg^{-1} \subset N_G(H^\circ).\]
The latter equivalence holds because~$g$ and~$F$ are contained in~$N_{\GL(V)}(H^\circ)$: the first by assumption and the second because any subgroup of~$H$ normalizes the identity component~$H^\circ$. It follows that we have an inclusion
\[ \cC' = \{g F g^{-1} \mid g \in N_{\GL(V)}(H^\circ), g F g^{-1} \subset N_G(H^\circ) \} \subset \cC_{N_G(H^\circ)}(F),\]
yielding an injection~$\cC' / N_G(H^\circ) \into \cC_{N_G(H^\circ)}(F) / N_G(H^\circ)$.
The group~$N_G(H^\circ)$ is reductive and~$F$ is finite, thus~$\cC_{N_G(H^\circ)}(F)$ is a finite union of~$N_G(H^\circ)$-orbits by the finite case. This concludes the proof.
\end{proof}

\begin{lemma} \label{Lemma:RepresentingGroupOfComponents} Given  a reductive group~$G$, there is a finite subgroup~$F \subset G$ such that the projection~$F \to G/G^\circ$ is surjective. 
\end{lemma}

\begin{proof} We reproduce the argument 
in \cite{MO13}. 
First we reduce to the case when~$G^\circ$ is a torus: Let~$T \subset G^\circ$ be a maximal torus and~$N_G(T) \subset G$ the normalizer of~$T$. The projection~$N_G(T) \to G/G^\circ$ is surjective. Indeed given~$g \in G$ the subgroup~$g T g^{-1}$ is a maximal torus of~$G^\circ$. Since maximal (split) tori of a connected reductive group are all conjugated, there is~$g_0 \in G^\circ$ such that~$g T g^{-1} = g_0 T g_0^{-1}$ hence~$h = g_0^{-1} g$ lies in the normalizer of~$T$. In particular~$g \equiv h$ modulo~$G^\circ$. Since~$N_G(T)/T$ is finite, we may replace~$G^\circ$ by~$T$ and~$G$ by~$N_G(T)$. In this case write~$T = G^\circ$ and~$W= G / G^\circ$. We claim that the extension~$1 \to T \to G \to W \to 1$ is the pushout of an extension
\[ 1 \too T[n] \too \tilde{G} \too W \too 1\]
along~$T[n] \into T$ where~$n= |W|$ and~$T[n] \subset T$ is the~$n$-torsion subgroup. Since the group~$\tilde{G}$ is finite, the subgroup~$F := \im(\tilde{G} \to G)$ is finite and does the job. To prove the claim, remark that~$T$ is commutative so we can see it as a~$W$-module by conjugation. For each~$m \in \bbN$ the short exact sequence~$1 \to T[m] \to T \to T \to 1$ of~$W$-modules induces an exact sequence
\[ \Ext^1(W,T[m]) \too \Ext^1(W,T) \stackrel{m}{\too} \Ext^1(W,T) \]
of abelian groups. These abelian groups are~$n$-torsion because~$n = |W|$. Thus by taking~$m = n$ one sees that the isomorphism class~$[G] \in \Ext^1(W,T)$ is in the image of~$\Ext^1(W,T[m]) \to \Ext^1(W,T)$. This concludes the proof.
\end{proof}

\section{The Hodge filtration on twisted cohomology}

In this appendix we gather some general facts about the Hodge filtration on the cohomology of torsion local systems of rank one. Let~$Y$ be a smooth projective complex algebraic variety with a free action of a finite abelian group~$G$, and denote by~$\pi \colon Y \to X=Y/G$ the quotient. For a character~$\chi \in G^\ast = \Hom(G, \bbG_m)$, we denote by~$L_\chi$ the rank one local system on~$Y(\bbC)$ whose sections on an open subset~$U \subset Y(\bbC)$ are the locally constant functions~$f \colon \pi^{-1}(U) \to \bbC$ such that
\[ 
 f(gx) = \chi(g)f(x) \quad \text{for all} \quad g \in G, x \in \pi^{-1}(U).
\]
We have a direct sum decomposition~$\pi_\ast \bbC_Y = \bigoplus_{\chi \in G^\ast} L_\chi$ inducing a decomposition, for each~$q \ge 0$,
\[ \rH^q(Y, \bbC) = \bigoplus_{\chi \in G^\ast} \rH^q (X, L_\chi). \]
Similarly, let~$\cL_\chi$ be the line bundle on~$X$ whose sections on an open subset~$U \subset X$ are those~$f \in \Gamma(\pi^{-1}(U, \cO_Y))$ with~$\sigma^\ast f = \chi \otimes  f$ for the group action~$\sigma \colon G \times Y \to Y$. Then we have a decomposition
\[ \pi_\ast \cO_Y = \bigoplus_{\chi \in G^\ast} \cL_\chi.\]
\begin{proposition} \label{Prop:HodgeDecompositionCharacter}
For~$q \in \bbN$ and any character~$\chi \in G^\ast$, the Hodge decomposition of~$\rH^q(Y, \bbC)$ induces a decomposition
\[ \rH^q(X, L_\chi) = \bigoplus_{i + j = q} \rH^j(X, \Omega_X^i \otimes \cL_\chi).\]
\end{proposition}

\begin{proof} 
For every coherent vector bundle~$\cE$ on~$X$ and~$j\ge 0$, the projection formula shows that
\[ \rH^j(Y, \pi^\ast \cE) = \bigoplus_{\chi \in G^\ast} \rH^j(X, \cE \otimes \cL_\chi). \]
Taking~$\cE = \Omega_Y^i$ and using that~$\pi^\ast \Omega^i_X \to \Omega^i_Y$ is an isomorphism because~$\pi$ is finite \'etale, we get
\[ \rH^j(Y, \Omega_Y^i) = \bigoplus_{\chi \in G^\ast} \rH^j(X, \Omega^i_X \otimes \cL_\chi).\]
We now insert this in the Hodge decomposition
\[ \alpha \colon \rH^q(Y, \bbC)   \stackrel{\sim}{\too} \bigoplus_{i + j = q} \rH^j(Y, \Omega_Y^i).\]
The group~$G$ acts linearly on~$\rH^q(Y, \bbC)$ and~$\rH^j(Y, \Omega_Y^i)$, and~$\alpha$ is~$G$-equivariant for these actions since the Hodge decomposition is functorial. The claim then follows by looking at eigenspaces: For~$\chi \in G^\ast$, the~$\chi$-eigenspace of~$\rH^q(Y, \bbC)$ is~$\rH^q (X, L_\chi)$, while the~$\chi$-eigenspace of~$\rH^j(Y, \Omega_Y^i)$ is~$\rH^j(X, \Omega^i_X \otimes \cL_\chi)$. 
\end{proof}

Let~$F^\bullet \rH^q(Y, \bbC)$ be the Hodge filtration of~$\rH^q(Y, \bbC)$, given in terms of the Hodge decomposition by
\[ F^p \rH^q(Y, \bbC) = \bigoplus_{i = p}^q \rH^{q - i}(Y, \Omega_Y^i). \]
\Cref{Prop:HodgeDecompositionCharacter} shows that for~$\chi \in G^\ast$, the Hodge filtration of~$\rH^q(X, L_\chi)$ is
\[ F^p \rH^q(X, L_\chi) = \bigoplus_{i = p}^q \rH^{q - i}(Y, \Omega_X^i \otimes \cL_\chi). \]
We are interested in situation where generic vanishing applies:

\begin{example} \label{ex:semismall}
Consider a smooth complex projective variety~$X$ of dimension~$d$ with a morphism~$f\colon X\to A$ to an abelian variety. Recall that~$f$ is~\emph{semismall} if \[\dim X\times_A X \le \dim X.\] For instance this is the case if~$f$ is finite.  By ~\cite[prop.~4.2.1]{DeCataldoMigliorini09} the morphism~$f$ is semismall if and only if~$P=Rf_*\bbC_X[d]$ is a perverse sheaf, because~$X$ is smooth and proper.
 Semismall morphisms are in particular generically finite. For our purpose we can weaken this as follows: Let us say that~$f\colon X\to A$ is~\emph{cohomologically semismall} if for all~$i\neq 0$ the perverse cohomology sheaves
$
{^p\hspace*{-0.2em}\cH}^i(P)
$ 
have Euler characteristic zero. In this case, generic vanishing~\cite{KWVanishing, SchnellHolonomic} says that most rank one local systems~$L$ on~$A(\bbC)$ satisfy
\[ \rH^q(X, f^*L) \;=\; \rH^{q-d}(A, P\otimes L) \;=\; 0 \quad \text{for all} \quad q\;\neq\; d.
\]
If~$L$ is of finite order, then~$L$ embeds as a direct summand in the direct image~$p_*\bbC_B$ for some isogeny~$p\colon B\to A$ of abelian varieties. We can use the setting discussed at the beginning of this section by writing the given variety as the quotient~$X=Y/G$ of the smooth variety~$Y=X\times_A B$ by the free action of the finite group~$G=\ker(p)$. Then~$f^*L \iso L_\chi$ for some character~$\chi \in G^\ast$ and therefore
\[
 \rH^q(X, L_\chi) \;=\; 0 
 \quad \text{for all} \quad q \;\neq\; d.
\]
In general, such a vanishing property can be used to write Hodge numbers as Euler characteristics of sheaves of differential forms:
\end{example}

\begin{proposition} \label{Prop:NumericHodgeSymmetry} 
Let~$d := \dim X$ and let~$\chi \in G^\ast$ be such that~$\rH^q(X, L_\chi)=0$ for all~$q \neq d$. Then for~$i = 0, \dots, d$ we have:
\begin{align}
\tag{1} \dim \rH^{d- i }(X, \Omega^i_X \otimes \cL_\chi) &= (-1)^{d- i}\chi(X, \Omega^i_X), \\
\tag{2} \dim \rH^{d- i }(X, \Omega^i_X \otimes \cL_\chi) &= \dim \rH^{i}(X, \Omega^{d-i}_X \otimes \cL_\chi), \\
\tag{3} \dim \rH^d(X, L_\chi) &= \sum_{i = 0}^d (-1)^{d- i}\chi(X, \Omega^i_X).
\end{align}
\end{proposition}

\begin{proof} (1)~$\rH^q(X, L_\chi)=0$ for~$q \neq d$ implies~$\rH^j(X, \Omega^i_X \otimes \cL_\chi)=0$ for~$i + j \neq d$ by \cref{Prop:HodgeDecompositionCharacter}. In particular,
\begin{align*} \chi(X, \Omega^i_X \otimes \cL_\chi) &= \sum_{j \in \bbN} (-1)^j \dim \rH^j(X, \Omega^i_X \otimes \cL_\chi) \\ &= (-1)^{d-i} \dim \rH^{d-i}(X, \Omega^i_X \otimes \cL_\chi).
\end{align*}
But~$\chi(X, \Omega^i_X \otimes \cL_\chi) = \chi(X, \Omega^i_X)$ by Hirzebruch-Riemann-Roch since~$c_1(\cL_\chi)$ vanishes.

\smallskip 

(2) follows from (1) since~$\chi(X, \Omega^i_X) = (-1)^d\chi(X, \Omega^{d-i}_X)$ by Serre duality.

\smallskip

(3) follows from (1) together with \cref{Prop:HodgeDecompositionCharacter}.
\end{proof}

In the situation of \cref{Prop:NumericHodgeSymmetry} the Hodge filtration is compatible with the Poincar\'e pairing. To explain what we mean by this, note that the Poincar\'e pairing induces a perfect pairing 
\[
\theta \colon \quad \rH^d(X, L_\chi) \otimes \rH^d(X, L_{\chi^{-1}}) \too \bbC. 
\]

\begin{corollary} \label{Prop:HodgeDecompositionIsOrthogonalWRTPoincarePairing} 
Let~$\chi\in G^\ast$ be a character with~$H^q(X, L_\chi)=0$ for all~$q\neq d$. Then the orthocomplement of~$F^i H^d(X, L_\chi)$ with respect to the bilinear form~$\theta$ is 
\[ (F^i H^d(X, L_\chi))^\bot = F^{d + 1 - i} H^d(X, L_{\chi^{-1}}). \]
\end{corollary}

\begin{proof} 
By construction of the Poincar\'e pairing we have for each~$p,q$ a commutative diagram 
\[
\begin{tikzcd} 
 H^d(X, L_\chi) \otimes H^d(X, L_{\chi^{-1}}) \ar[r, "\theta"] & \bbC \\
 H^{d-p}(X, \Omega_X^p \otimes \cL_\chi) \otimes H^{d-q}(X, \Omega_X^q \otimes \cL_{\chi^{-1}}) \ar[r, "\smile"] \ar[u, hook]
 & H^{2d-p-q}(X, \Omega_X^{p+q}) \ar[u, swap, "\tr"]
\end{tikzcd} 
\]
where~$\mathrm{tr}$ is the trace map if~$p+q=d$, and the zero map otherwise. Hence the pairing between
\[
 F^i H^d(X, L_\chi) \;=\; 
 \bigoplus_{p\ge i}H^{d-p}(X, \Omega_X^p \otimes \cL_{\chi})
\] 
and 
\[
 F^{d+1-i} H^d(X, L_{\chi^{-1}}) \;=\; 
 \bigoplus_{q\ge d+1-i} H^{d-q}(X, \Omega_X^q \otimes \cL_{\chi^{-1}})
\]
is zero since it only involves terms with~$p+q\ge i+d+1-i>d$. So we have an inclusion
\[
 F^{d+1-i} H^d(X, L_{\chi^{-1}})
 \;\subset\;
 (F^i H^d(X, L_\chi))^\bot
\]
and to show that the two spaces are equal, it suffices to check that they have the same dimension. Since~$\theta$ is nondegenerate, the dimension of the orthocomplement is given by
\begin{eqnarray*}
 \dim (F^i H^d(X, L_\chi))^\bot
 &\;=\;& \dim H^d(X, L_\chi) - \dim F^i H^d(X, L_\chi)
 \\
 &\;=\;& \sum_{p=0}^{i-1} (-1)^{d-p} \chi(X, \Omega_X^p) 
 \\
 &\;=\;& \sum_{q=d+1-i}^d (-1)^q \chi(X, \Omega_X^{d-q}) 
 \\
 &\;=\;& \sum_{q=d+1-i}^d (-1)^{d-q} \chi(X, \Omega_X^q) 
 \\
 &\;=\;& \dim F^{d+1-i} H^d(X, L_{\chi^{-1}})
\end{eqnarray*}
where the last four equalities are due to \cref{Prop:NumericHodgeSymmetry}.
\end{proof}

Part (1) of~\cref{Prop:NumericHodgeSymmetry} in particular says that~$(-1)^{d-p} \chi(X, \Omega_X^p)\ge 0$ for all~$p$, in line with the result for subvarieties of abelian varieties given by generic vanishing~\cite[cor.~1.4]{PopaSchnell}. This allows to verify the numerical assumption~\eqref{SkullInequality}  from the introduction for odd-dimensional varieties and for surfaces with nef cotangent bundle:

\begin{lemma} \label{inequality-in-odd-dimension}
Let~$X$ be a smooth projective variety of dimension~$d$ over a field of characteristic zero, and assume~$h_p := (-1)^{d-p} \chi(X, \Omega_X^p)\ge 0$ for all~$p$. Suppose either that~$d$ is odd or that~$X$ is a surface  with nef cotangent bundle. Then
\[
 (-1)^d 2\chi(X\times X, \Omega^d_{X\times X}) \;\le\; 
 \chi_{\top}(X\times X).
\]
\end{lemma} 

\begin{proof} 
Serre duality says~$h_p = h_{d-p}$ for all~$p$, so the K\"unneth decomposition for differential forms gives
\[
 \chi(X\times X, \Omega_{X\times X}^d) \;=\; \sum_{p=0}^d \chi(X, \Omega_X^p) \chi(X, \Omega_X^{d-p} )
 \;=\; 
 \sum_{p=0}^d h_p h_{d-p} 
 \;=\;
 \sum_{p=0}^d h_p^2.
\]
On the other hand, the Hodge decomposition gives~$\chi_\top(X)=h_0 + \cdots + h_d$ and therefore
\[
 \chi_\top(X\times X) \;=\; \left(\chi_\top(X)\right)^2 \;=\; 
 \left( h_0 + \cdots + h_d \right)^2
\]
It will thus be enough to show that the inequality~$Q(h_0, \dots, h_d)\ge 0$ holds for the quadratic form
\[
Q(X_0, \dots, X_d) =  \left( X_0 + \cdots + X_d \right)^2 - 2 (X_0^2 + \cdots + X_d^2).
\]
If~$d$ is odd, then the desired inequality follows from~$h_i = h_{d-i} \ge 0$ since in this case
\[
 Q(h_0, \dots, h_d) = 8 \sum_{0 \le i < j \le d/2} h_i h_j \ge 0.
\]
If~$d$ is even, then one gets
\[
Q(h_0, \dots, h_d) = - h_{d/2}^2 + 4 (h_0 + \cdots + h_{d/2 - 1})h_{d/2} + 8 \sum_{0 \le i < j \le d/2 -1} h_i h_j.
\]
In particular, for~$d=2$ the desired inequality can be reformulated as follows:
\begin{eqnarray*} 
 Q(h_0, h_1, h_2) \ge 0 
 &\;\Longleftrightarrow\;&
 h_1(4h_0-h_1) \ge 0 \\
 &\;\Longleftrightarrow\;&
 4\chi(X, \cO_X) + \chi(X, \Omega_X^1) \ge 0\\
 &\;\Longleftrightarrow\;&
 6\chi(X, \cO_X) \ge c_2(X) \\
 &\;\Longleftrightarrow\;&
 c_1^2(X) \ge c_2(X).
\end{eqnarray*}
Now~$c_1(X)^2 = c_1(\Omega_X^1)^2$ and~$c_2(X)=c_2(\Omega^1_X)$. Therefore if~$X$ has nef cotangent bundle, then~$c_1(X)^2\ge c_2(X)$ by nonnegativity of Schur polynomials of Chern classes of nef vector bundles \cite[example 8.3.10]{LazarsfeldPositivityII}.
\end{proof} 

\small

\bibliography{./../biblio}

\newcommand{\etalchar}[1]{$^{#1}$}
\providecommand{\bysame}{\leavevmode\hbox to3em{\hrulefill}\thinspace}
\providecommand{\MR}{\relax\ifhmode\unskip\space\fi MR }
\providecommand{\MRhref}[2]{%
  \href{http://www.ams.org/mathscinet-getitem?mr=#1}{#2}
}
\providecommand{\href}[2]{#2}
\begin{thebibliography}{DMOS82}

\bibitem[Abr94]{Abr94}
D.~Abramovich, \emph{Subvarieties of semiabelian varieties}, Compositio Math.
  \textbf{90} (1994), no.~1, 37--52.

\bibitem[And92]{AndreGenericMumfordTate}
Y.~Andr\'{e}, \emph{Mumford-{T}ate groups of mixed {H}odge structures and the
  theorem of the fixed part}, Compositio Math. \textbf{82} (1992), no.~1,
  1--24.

\bibitem[And96]{AndreK3}
Y.~Andr\'e, \emph{{On the Shafarevich and Tate conjectures for hyperk\"ahler
  varieties}}, {Math. Ann.} \textbf{305} (1996), no.~2, 205--248.

\bibitem[Art57]{ArtinGeometricAlgebra}
E.~Artin, \emph{Geometric algebra}, Intersci. Tracts Pure Appl. Math., vol.~3,
  Interscience Publishers, New York, NY, 1957.

\bibitem[BC09]{BrinonConrad}
O.~Brinon and B.~Conrad, \emph{{CMI} summer school notes on $p$-adic {H}odge
  theory}, 2009, available at
  \url{https://math.stanford.edu/~conrad/papers/notes.pdf}.

\bibitem[Ber90]{Ber90}
V.~G. Berkovich, \emph{Spectral theory and analytic geometry over
  non-{A}rchimedean fields}, Mathematical Surveys and Monographs, vol.~33,
  American Mathematical Society, Providence, RI, 1990.

\bibitem[BLR90]{NeronModels}
S.~Bosch, W.~L\"{u}tkebohmert, and M.~Raynaud, \emph{N\'{e}ron models},
  Ergebnisse der Mathematik und ihrer Grenzgebiete (3), vol.~21,
  Springer-Verlag, Berlin, 1990.

\bibitem[BM76]{BombieriMumford}
E.~Bombieri and D.~Mumford, \emph{Enriques' classification of surfaces in char
  p. {III}}, Invent. Math. \textbf{35} (1976), 197--232.

\bibitem[BMR20]{BateMartinRoehrle}
M.~Bate, B.~Martin, and G.~R{\"o}hrle, \emph{Semisimplification for subgroups
  of reductive algebraic groups}, Forum Math. Sigma \textbf{8} (2020), 1--10,
  Id/No e43.

\bibitem[BO78]{BerthelotOgus}
P.~Berthelot and A.~Ogus, \emph{Notes on crystalline cohomology}, Princeton
  University Press, Princeton, NJ; University of Tokyo Press, Tokyo, 1978.

\bibitem[BSS18]{BSS}
B.~Bhatt, C.~Schnell, and P.~Scholze, \emph{Vanishing theorems for perverse
  sheaves on abelian varieties, revisited}, Sel. Math., New Ser. \textbf{24}
  (2018), no.~1, 63--84.

\bibitem[BT65]{BorelTits}
A.~Borel and J.~Tits, \emph{Groupes r\'{e}ductifs}, Inst. Hautes \'{E}tudes
  Sci. Publ. Math. (1965), no.~27, 55--150.

\bibitem[BT19]{BakkerTsimerman}
B.~Bakker and J.~Tsimerman, \emph{The {A}x-{S}chanuel conjecture for variations
  of {H}odge structures}, Invent. Math. \textbf{217} (2019), no.~1, 77--94.

\bibitem[CL25]{CadoretLiu}
A.~Cadoret and H.~Liu, \emph{Variation of {Tannaka} groups of perverse sheaves
  in family}, 2025,
  \href{https://arxiv.org/abs/2505.01716}{\texttt{arXiv:2505.01716}}.

\bibitem[CMSP17]{CMSP17}
J.~Carlson, S.~M{\"u}ller-Stach, and C.~Peters, \emph{Period mappings and
  period domains}, 2nd edition ed., Camb. Stud. Adv. Math., vol. 168,
  Cambridge: Cambridge University Press, 2017.

\bibitem[CP15]{CampanaPaun}
F.~Campana and M.~P\u{a}un, \emph{Orbifold generic semi-positivity: an
  application to families of canonically polarized manifolds}, Ann. Inst.
  Fourier (Grenoble) \textbf{65} (2015), no.~2, 835--861.

\bibitem[dCM09]{DeCataldoMigliorini09}
M.~A.~A. de~Cataldo and L.~Migliorini, \emph{The decomposition theorem,
  perverse sheaves and the topology of algebraic maps}, Bull. Am. Math. Soc.,
  New Ser. \textbf{46} (2009), no.~4, 535--633.

\bibitem[Deb95]{Deb95}
O.~Debarre, \emph{Fulton-{H}ansen and {B}arth-{L}efschetz theorems for
  subvarieties of abelian varieties}, J. Reine Angew. Math. \textbf{467}
  (1995), 187--197.

\bibitem[Del70]{Del70}
P.~Deligne, \emph{\'{E}quations diff\'{e}rentielles \`a points singuliers
  r\'{e}guliers}, Lecture Notes in Mathematics, Vol. 163, Springer-Verlag,
  Berlin-New York, 1970.

\bibitem[Del71]{DeligneHodgeII}
\bysame, \emph{Th\'{e}orie de {H}odge. {II}}, Inst. Hautes \'{E}tudes Sci.
  Publ. Math. (1971), no.~40, 5--57.

\bibitem[Del74]{DeligneWeilI}
\bysame, \emph{La conjecture de {W}eil. {I}}, Inst. Hautes \'{E}tudes Sci.
  Publ. Math. (1974), no.~43, 273--307.

\bibitem[Del85]{DeligneBourbakiFaltings}
\bysame, \emph{Preuve des conjectures de {T}ate et de {S}hafarevitch (d'apr\`es
  {G}. {F}altings)}, Ast\'{e}risque (1985), no.~121-122, 25--41, Seminar
  Bourbaki, Vol. 1983/84.

\bibitem[DMOS82]{DM82}
P.~Deligne, J.~S. Milne, A.~Ogus, and K.~Shih, \emph{Hodge cycles, motives, and
  {S}himura varieties}, Lecture Notes in Mathematics, vol. 900,
  Springer-Verlag, Berlin-New York, 1982.

\bibitem[Fal83]{FaltingsMordell}
G.~Faltings, \emph{Endlichkeitss\"{a}tze f\"{u}r abelsche {V}ariet\"{a}ten
  \"{u}ber {Z}ahlk\"{o}rpern}, Invent. Math. \textbf{73} (1983), no.~3,
  349--366.

\bibitem[Fal89]{FaltingsCristallineEtale}
\bysame, \emph{Crystalline cohomology and {$p$}-adic {G}alois-representations},
  Algebraic analysis, geometry, and number theory ({B}altimore, {MD}, 1988),
  Johns Hopkins Univ. Press, Baltimore, MD, 1989, pp.~25--80.

\bibitem[FGI{\etalchar{+}}05]{FGAExplained}
B.~Fantechi, L.~G\"{o}ttsche, L.~Illusie, S.~L. Kleiman, N.~Nitsure, and
  A.~Vistoli, \emph{Fundamental algebraic geometry}, Mathematical Surveys and
  Monographs, vol. 123, American Mathematical Society, Providence, RI, 2005,
  Grothendieck's FGA explained.

\bibitem[FLTZ22]{TakamatsuHyperkahler}
L.~Fu, Z.~Li, T.~Takamatsu, and H.~Zou, \emph{Unpolarized {S}hafarevich
  conjectures for hyper-{K}{\"a}hler varieties}, 2022,
  \href{https://arxiv.org/abs/2203.10391}{\texttt{arXiv:2203.10391}}.

\bibitem[FO]{FontaineOuyang}
J.-M. Fontaine and Y.~Ouyang, \emph{Theory of $p$-adic {G}alois
  {R}epresentations},
  \href{https://www.imo.universite-paris-saclay.fr/~fontaine/galoisrep.pdf}{\texttt{www.imo.universite-paris-saclay.fr/\textasciitilde~fontaine/galoisrep.pdf}}.

\bibitem[Ful98]{FultonIntersectionTheory}
W.~Fulton, \emph{Intersection theory}, second ed., Ergebnisse der Mathematik
  und ihrer Grenzgebiete. 3. Folge. A Series of Modern Surveys in Mathematics,
  vol.~2, Springer-Verlag, Berlin, 1998.

\bibitem[GMB13]{GilleMoretBailly}
P.~Gille and L.~Moret-Bailly, \emph{Algebraic actions of arithmetic groups},
  Torsors, \'etale homotopy and applications to rational points. Lecture notes
  of mini-courses presented at the workshop ``Torsors: theory and
  applications'', Edinburgh, UK, January 10--14, 2011 and at the study group
  organised in Imperial College, London, UK in autumn 2010, Cambridge:
  Cambridge University Press, 2013, pp.~231--249.

\bibitem[GP11]{SGA3}
P.~Gille and P.~Polo (eds.), \emph{Sch\'{e}mas en groupes ({SGA} 3). {T}ome
  {I}. {P}ropri\'{e}t\'{e}s g\'{e}n\'{e}rales des sch\'{e}mas en groupes},
  Documents Math\'{e}matiques (Paris), vol.~7, Soci\'{e}t\'{e} Math\'{e}matique
  de France, Paris, 2011, S\'{e}minaire de G\'{e}om\'{e}trie Alg\'{e}brique du
  Bois Marie 1962--64. A seminar directed by M. Demazure and A. Grothendieck
  with the collaboration of M. Artin, J.-E. Bertin, P. Gabriel, M. Raynaud and
  J-P. Serre, Revised and annotated edition of the 1970 French original.

\bibitem[Gro62]{FGA}
A.~Grothendieck, \emph{Fondements de la g{\'e}om{\'e}trie alg{\'e}brique.
  {Extraits} du {S{\'e}minaire} {Bourbaki} 1957-1962}, Paris: {Secr{\'e}tariat}
  math{\'e}matique (1962)., 1962.

\bibitem[Gro65]{EGAIV2}
\bysame, \emph{{\'E}l{\'e}ments de g{\'e}om{\'e}trie alg{\'e}brique. {IV}:
  {\'E}tude locale des sch{\'e}mas et des morphismes de sch{\'e}mas.
  ({S{\'e}conde} partie)}, Publ. Math., Inst. Hautes {\'E}tud. Sci. \textbf{24}
  (1965), 1--231.

\bibitem[Har71]{Har71}
R.~Hartshorne, \emph{Ample vector bundles on curves}, Nagoya Math. J.
  \textbf{43} (1971), 73--89.

\bibitem[Har77]{HartshorneAG}
\bysame, \emph{Algebraic geometry}, Grad. Texts Math., vol.~52, Springer, Cham,
  1977.

\bibitem[IKT24]{TakamatsuMukaiGenusSeven}
T.~Ito, T.~Kanemitsu, A.~Takamatsu, and Y.~Tanaka, \emph{{A}rithmetic
  finiteness of {M}ukai varieties of genus 7}, 2024,
  \href{https://arxiv.org/abs/2409.20046}{\texttt{arXiv:2409.20046}}.

\bibitem[Jav15]{JavanpeykarLondon}
A.~Javanpeykar, \emph{{N\'eron models and the arithmetic Shafarevich conjecture
  for certain canonically polarized surfaces}}, {Bull. Lond. Math. Soc.}
  \textbf{47} (2015), no.~1, 55--64.

\bibitem[JKLM25]{JKLM}
A.~Javanpeykar, T.~Kr{\"a}mer, C.~Lehn, and M.~Maculan, \emph{The monodromy of
  families of subvarieties on abelian varieties}, Duke Math. J. \textbf{174}
  (2025), no.~6, 1045--1149.

\bibitem[JL15]{JavLoughFlag}
A.~Javanpeykar and D.~Loughran, \emph{{Good reduction of algebraic groups and
  flag varieties}}, {Arch. Math.} \textbf{104} (2015), no.~2, 133--143.

\bibitem[JL17]{JavanpeykarLoughran}
\bysame, \emph{{Complete intersections: moduli, Torelli, and good reduction}},
  {Math. Ann.} \textbf{368} (2017), no.~3-4, 1191--1225.

\bibitem[JL18]{JavLoughPisa}
\bysame, \emph{{Good reduction of Fano threefolds and sextic surfaces}}, {Ann.
  Sc. Norm. Super. Pisa, Cl. Sci. (5)} \textbf{18} (2018), no.~2, 509--535.

\bibitem[JL21]{JavLoughLondon}
\bysame, \emph{Arithmetic hyperbolicity and a stacky {C}hevalley--{W}eil
  theorem}, Journal of the London Mathematical Society \textbf{103} (2021),
  no.~3, 846--869.

\bibitem[JLM24]{JavanpeykarCyclic}
A.~Javanpeykar, D.~Loughran, and S.~Mathur, \emph{Good reduction and cyclic
  covers}, J. Inst. Math. Jussieu \textbf{23} (2024), no.~1, 463--494.

\bibitem[Kat73]{KatzTravauxDeDwork}
N.~M. Katz, \emph{Travaux de {D}work}, S\'{e}minaire {B}ourbaki, 24\`eme
  ann\'{e}e (1971/1972), Lecture Notes in Math., vol. Vol. 317, Springer,
  Berlin-New York, 1973, pp.~Exp. No. 409, pp. 167--200.

\bibitem[Kat87]{KatzCalculation}
\bysame, \emph{On the calculation of some differential {G}alois groups},
  Invent. Math. \textbf{87} (1987), no.~1, 13--61.

\bibitem[Kem81]{KempfDeformations}
G.~R. Kempf, \emph{Deformations of symmetric products}, Riemann surfaces and
  related topics: {Proc}. 1978 {Stony} {Brook} {Conf}., {Ann}. {Math}. {Stud}.
  97, 319-341 (1981)., 1981.

\bibitem[Kle84]{KleimanConormal}
S.~L. Kleiman, \emph{About the conormal scheme}, Complete intersections,
  {Lect}. 1st {Sess}. {C}.{I}.{M}.{E}., {Acireale}/{Italy} 1983, {Lect}.
  {Notes} {Math}. 1092, 161-197 (1984)., 1984.

\bibitem[KLM24]{KLM}
T.~Kr\"amer, C.~Lehn, and M.~Maculan, \emph{The {T}annaka group {$E_6$} only
  arises from cubic threefolds}, 2024,
  \href{https://arxiv.org/abs/2406.07401}{\texttt{arXiv:2406.07401}}.

\bibitem[KM83]{KollarMatsusaka}
J.~Koll{\'a}r and T.~Matsusaka, \emph{Riemann-roch type inequalities}, Am. J.
  Math. \textbf{105} (1983), 229--252.

\bibitem[Kr{\"a}21]{KraemerMicrolocalII}
T.~Kr{\"a}mer, \emph{Characteristic cycles and the microlocal geometry of the
  {G}auss map, {II}}, J. Reine Angew. Math. \textbf{774} (2021), 53--92.

\bibitem[Kr{\"a}22]{KraemerMicrolocalI}
\bysame, \emph{Characteristic cycles and the microlocal geometry of the {G}auss
  map, {I}}, Ann. Sci. \'{E}c. Norm. Sup\'{e}r. (4) \textbf{55} (2022),
  1475--1527.

\bibitem[KW15a]{KWGeneric}
T.~Kr{\"a}mer and R.~Weissauer, \emph{On the {Tannaka} group attached to the
  theta divisor of a generic principally polarized abelian variety}, Math. Z.
  \textbf{281} (2015), no.~3-4, 723--745.

\bibitem[KW15b]{KWVanishing}
\bysame, \emph{Vanishing theorems for constructible sheaves on abelian
  varieties}, J. Algebraic Geom. \textbf{24} (2015), no.~3, 531--568.

\bibitem[Laz04]{LazarsfeldPositivityII}
R.~Lazarsfeld, \emph{Positivity in algebraic geometry. {II}}, Ergebnisse der
  Mathematik und ihrer Grenzgebiete. 3. Folge. A Series of Modern Surveys in
  Mathematics, vol.~49, Springer-Verlag, Berlin, 2004, Positivity for vector
  bundles, and multiplier ideals.

\bibitem[Lic22]{LichtHyperbolicity}
P.~Licht, \emph{Hyperbolicity of the moduli of certain {Fano} threefolds}, Acta
  Arith. \textbf{206} (2022), no.~1, 75--95.

\bibitem[Lie15]{LiedtkeLiftEnriques}
C.~Liedtke, \emph{Arithmetic moduli and lifting of {Enriques} surfaces}, J.
  Reine Angew. Math. \textbf{706} (2015), 35--65.

\bibitem[LS20]{LS20}
B.~Lawrence and W.~Sawin, \emph{The {S}hafarevich conjecture for hypersurfaces
  in abelian varieties}, 2020,
  \href{https://arxiv.org/abs/2004.09046v3}{\texttt{arXiv:2004.09046v3}}.

\bibitem[LS24]{Schroer}
B.~Laurent and S.~Schr{\"o}er, \emph{Para-abelian varieties and {Albanese}
  maps}, Bull. Braz. Math. Soc. (N.S.) \textbf{55} (2024), no.~1, 39, Id/No 4.

\bibitem[Lu25]{LuSkull}
F.~Lu, \emph{Special {C}ases of the {S}hafarevich {C}onjecture for {C}omplete
  {I}ntersections in {A}belian {V}arieties}, 2025,
  \href{https://arxiv.org/abs/2506.14935}{\texttt{arXiv:2506.14935}}.

\bibitem[LV20]{LV}
B.~Lawrence and A.~Venkatesh, \emph{Diophantine problems and {$p$}-adic period
  mappings}, Invent. Math. \textbf{221} (2020), no.~3, 893--999.

\bibitem[Mar03]{MartinReductiveSubgroups}
B.~M.~S. Martin, \emph{Reductive subgroups of reductive groups in nonzero
  characteristic}, J. Algebra \textbf{262} (2003), no.~2, 265--286.

\bibitem[Mil17]{MilneAlgebraicGroups}
J.~S. Milne, \emph{Algebraic groups}, Cambridge Studies in Advanced
  Mathematics, vol. 170, Cambridge University Press, Cambridge, 2017, The
  theory of group schemes of finite type over a field.

\bibitem[MM64]{MatsusakaMumford}
T.~Matsusaka and D.~Mumford, \emph{Two fundamental theorems on deformations of
  polarized varieties}, Amer. J. Math. \textbf{86} (1964), 668--684.

\bibitem[NT22]{TakamatsuPolycurves}
I.~Nagamachi and T.~Takamatsu, \emph{The {Shafarevich} conjecture and some
  extension theorems for proper hyperbolic polycurves}, Math. Res. Lett.
  \textbf{29} (2022), no.~2, 541--558.

\bibitem[PS13]{PopaSchnell}
M.~Popa and C.~Schnell, \emph{Generic vanishing theory via mixed {H}odge
  modules}, Forum Math. Sigma \textbf{1} (2013), Paper No. e1, 60.

\bibitem[Ran81]{Ran}
Z.~Ran, \emph{On subvarieties of abelian varieties}, Invent. Math. \textbf{62}
  (1981), no.~3, 459--479.

\bibitem[Ran86]{RanMartens}
\bysame, \emph{On a theorem of {Martens}}, Rend. Semin. Mat., Torino
  \textbf{44} (1986), 287--291.

\bibitem[Ric67]{Richardson}
R.~W. Richardson, Jr., \emph{Conjugacy classes in {L}ie algebras and algebraic
  groups}, Ann. of Math. (2) \textbf{86} (1967), 1--15.

\bibitem[Rou81]{RousseauBuilding}
G.~Rousseau, \emph{Instabilit\'{e} dans les espaces vectoriels}, Algebraic
  surfaces ({O}rsay, 1976--78), Lecture Notes in Math., vol. 868, Springer,
  Berlin-New York, 1981, pp.~263--276.

\bibitem[{Sch}85]{SchollDelPezzo}
A.~J. {Scholl}, \emph{{A finiteness theorem for del Pezzo surfaces over
  algebraic number fields}}, {J. Lond. Math. Soc., II. Ser.} \textbf{32}
  (1985), 31--40.

\bibitem[Sch15]{SchnellHolonomic}
C.~Schnell, \emph{Holonomic {$\mathcal{D}$}-modules on abelian varieties},
  Publ. Math., Inst. Hautes {\'E}tud. Sci. \textbf{121} (2015), 1--55.

\bibitem[Ser94]{SerreCohomologieGaloisienne}
J.-P. Serre, \emph{Cohomologie galoisienne}, fifth ed., Lecture Notes in
  Mathematics, vol.~5, Springer-Verlag, Berlin, 1994.

\bibitem[She97]{She}
Y.~She, \emph{{The unpolarized Shafarevich Conjecture for K3 Surfaces}}, 1997,
  \url{arXiv:1705.09038}.

\bibitem[SR72]{Saa72}
N.~Saavedra~Rivano, \emph{Cat\'{e}gories {T}annakiennes}, Lecture Notes in
  Mathematics, Vol. 265, Springer-Verlag, Berlin-New York, 1972.

\bibitem[{Sta}22]{stacks-project}
The {Stacks Project Authors}, \emph{The {Stacks Project}},
  \url{https://stacks.math.columbia.edu}, 2022.

\bibitem[Tak19]{TakamatsuEnriques}
T.~Takamatsu, \emph{{On the Shafarevich conjecture for Enriques surfaces}},
  2019, \url{arXiv:1911.03419}.

\bibitem[Tak20]{TakamatsuK3}
\bysame, \emph{On a cohomological generalization of the {S}hafarevich
  conjecture for {K}3 surfaces}, Algebra Number Theory \textbf{14} (2020),
  no.~9, 2505--2531.

\bibitem[Uen73]{UenoCompositio}
K.~Ueno, \emph{Classification of algebraic varieties. {I}}, Compositio Math.
  \textbf{27} (1973), 277--342.

\bibitem[use13]{MO13}
user76758, \emph{In any {L}ie group with finitely many connected components,
  does there exist a finite subgroup which meets every component?},
  Math{O}verflow, 2013,
  \href{https://mathoverflow.net/questions/150949/}{\texttt{mathoverflow.net/questions/150949/}}.

\bibitem[Vie95]{ViehwegModuli}
E.~Viehweg, \emph{Quasi-projective moduli for polarized manifolds}, Ergeb.
  Math. Grenzgeb., 3. Folge, vol.~30, Berlin: Springer-Verlag, 1995.

\bibitem[Vin96]{Vinberg}
E.~B. Vinberg, \emph{On invariants of a set of matrices}, J. Lie Theory
  \textbf{6} (1996), no.~2, 249--269.

\bibitem[Wei15]{WeissauerArxiv2015a}
R.~Weissauer, \emph{On {S}ubvarieties of {A}belian {V}arieties with degenerate
  {G}auss mapping}, ar{X}iv, 2015,
  \href{https://arxiv.org/abs/1110.0095}{\texttt{arXiv:1110.0095}}.

\end{thebibliography}

\bibliographystyle{amsalpha}

\end{document}